\pdfoutput=1
\documentclass[11pt,twoside]{article}

\usepackage{graphicx, amsmath, amsthm, amssymb, amscd, mathtools}

\usepackage{diagrams, dsfont}

\usepackage[usetoc,nostar]{titleref}

\newcommand{\oddheader}{Topology of Privacy:  Lattice Structures and Information Bubbles}

\newcommand{\vso}{\vspace*{1pt}}
\newcommand{\vst}{\vspace*{2pt}}
\newcommand{\vsr}{\vspace*{3pt}}

\newtheorem{theorem}{Theorem}
\newtheorem{lemma}[theorem]{Lemma}
\newtheorem{corollary}[theorem]{Corollary}

\newtheorem{definition}[theorem]{Definition}

\newcounter{currentThmCount}

\newcommand{\ifig}[2]{\includegraphics[#2]{#1-eps-converted-to.pdf}}

\newcommand{\setdef}[2]{{\left\lbrace #1\;\left\vert\;#2 \right.\right\rbrace}}
\newcommand{\setdefb}[2]{{\big\lbrace #1\;\left\vert\;#2 \right.\big\rbrace}}

\newcommand{\setdefbV}[2]{{\big\lbrace #1\;\Big\vert\;#2 \big\rbrace}}

\newcommand{\vtsp}{\kern .08em}

\newcommand{\join}{\vee}
\newcommand{\meet}{\wedge}

\newcommand{\one}{\bullet}

\newcommand{\tas}{\{{\tt a}\}}

\newcommand{\ta}{{\tt a}}
\newcommand{\tb}{{\tt b}}
\newcommand{\tc}{{\tt c}}
\newcommand{\td}{{\tt d}}
\newcommand{\te}{{\tt e}}
\newcommand{\tf}{{\tt f}}
\newcommand{\tg}{{\tt g}}
\newcommand{\tth}{{\tt h}}

\newcommand{\tone}{{\tt 1}}
\newcommand{\ttwo}{{\tt 2}}
\newcommand{\tthree}{{\tt 3}}

\newcommand{\ic}{{\tt c}}
\newcommand{\ig}{{\tt g}}
\newcommand{\is}{{\tt s}}
\newcommand{\iv}{{\tt v}}

\newcommand{\gc}{{\tt gc}}
\newcommand{\gs}{{\tt gs}}
\newcommand{\gv}{{\tt gv}}
\newcommand{\cs}{{\tt cs}}
\newcommand{\cv}{{\tt cv}}
\newcommand{\sv}{{\tt sv}}

\newcommand{\cityA}{{\sc Athens}}
\newcommand{\cityB}{{\sc Berlin}}
\newcommand{\cityC}{{\sc Caen}}
\newcommand{\cityD}{{\sc Dublin}}
\newcommand{\cityE}{{\sc E{\l}k}}

\newcommand{\cA}{\hbox{\tt A}}
\newcommand{\cB}{\hbox{\tt B}}
\newcommand{\cC}{\hbox{\tt C}}
\newcommand{\cD}{\hbox{\tt D}}
\newcommand{\cE}{\hbox{\tt E}}

\newcommand{\rcityA}{\rotatebox{90}{\hbox{\cityA}}}
\newcommand{\rcityB}{\rotatebox{90}{\hbox{\cityB}}}
\newcommand{\rcityC}{\rotatebox{90}{\hbox{\cityC}}}
\newcommand{\rcityD}{\rotatebox{90}{\hbox{\cityD}}}
\newcommand{\rcityE}{\rotatebox{90}{\hbox{\cityE}}}

\newcommand{\hcityB}{\hbox{\cityB}}
\newcommand{\hcityC}{\hbox{\cityC}}
\newcommand{\hcityD}{\hbox{\cityD}}

\newcommand{\authorA}{{\tt Alice}}
\newcommand{\authorB}{{\tt Ben}}
\newcommand{\authorC}{{\tt Claire}}
\newcommand{\authorD}{{\tt David}}
\newcommand{\authorE}{{\tt Eric}}

\newcommand{\hauthorA}{\hbox{\authorA}}
\newcommand{\hauthorB}{\hbox{\authorB}}
\newcommand{\hauthorC}{\hbox{\authorC}}
\newcommand{\hauthorD}{\hbox{\authorD}}
\newcommand{\hauthorE}{\hbox{\authorE}}

\newcommand{\xadj}{\mskip-1.8mu}
\newcommand{\XxY}{{X}\xadj\times{Y}}

\newcommand{\yadj}{\mskip-2.5mu}
\newcommand{\YxY}{{Y}\yadj\times{Y}}

\newcommand{\pqadj}{\mskip-0.8mu}
\newcommand{\PxQ}{{P}\pqadj\times{Q}}

\newcommand{\Rp}{R^{\prime}}
\newcommand{\Qp}{Q^{\prime}}

\newcommand{\sphere}{\mathbb{S}}
\newcommand{\Smo}{\sphere^{\vtsp{-1}}}
\newcommand{\Szero}{\sphere^{\vtsp{0}}}
\newcommand{\Sone}{\sphere^{1}}
\newcommand{\Stwo}{\sphere^{2}}

\newcommand{\Sko}{\sphere^{\vtsp{k-1}}}
\newcommand{\Skt}{\sphere^{\vtsp{k-2}}}
\newcommand{\Snt}{\sphere^{\vtsp{n-2}}}

\newcommand{\Z}{\mathbb{Z}}

\newcommand{\dowx}{\Psi_R}
\newcommand{\dowy}{\Phi_R}
\newcommand{\dowqx}{\Psi_Q}
\newcommand{\dowqy}{\Phi_Q}

\newcommand{\dowyp}{\Phi_{R^\prime}}
\newcommand{\dowqpx}{\Psi_{Q^\prime}}
\newcommand{\dowqpy}{\Phi_{Q^\prime}}

\newcommand{\dowyone}{\Phi_{R_1}}
\newcommand{\dowytwo}{\Phi_{R_2}}

\newcommand{\qsg}{{Q(\sigma, \gamma)}}
\newcommand{\dqx}{\Psi_\qsg}
\newcommand{\dqy}{\Phi_\qsg}

\newcommand{\kmax}{{k_{max}}}
\newcommand{\Qmax}{{Q_{max}}}
\newcommand{\dowqmx}{\Psi_\Qmax}

\newcommand{\Ximply}{X_{\rm imply}}

\newcommand{\dowcx}{\Psi_C}
\newcommand{\dowcy}{\Phi_C}

\newcommand{\dowby}{\Phi_B}
\newcommand{\dowBx}{\Psi_B}
\newcommand{\dowBy}{\Phi_B}

\newcommand{\dowbpy}{\Phi_{B^\prime}}
\newcommand{\dowfx}{\Psi_F}
\newcommand{\dowfy}{\Phi_F}
\newcommand{\dowdx}{\Psi_D}
\newcommand{\dowdy}{\Phi_D}

\newcommand{\dowMx}{\Psi_M}
\newcommand{\dowMy}{\Phi_M}
\newcommand{\dowTx}{\Psi_T}
\newcommand{\dowTy}{\Phi_T}
\newcommand{\dowGx}{\Psi_G}

\newcommand{\dowMMy}{\Phi_{MM}}

\newcommand{\F}{{\mathfrak{F}}}
\newcommand{\Fdowx}{\F(\Psi_R)}
\newcommand{\Fdowy}{\F(\Phi_R)}
\newcommand{\Fdowqx}{\F(\Psi_Q)}
\newcommand{\Fdowqy}{\F(\Phi_Q)}
\newcommand{\Fdowxp}{\F(\Psi_{R^\prime})}
\newcommand{\Fdowyp}{\F(\Phi_{R^\prime})}

\newcommand{\PR}{P_R}
\newcommand{\PRplus}{P^{+}_R}
\newcommand{\PRp}{P_{R^\prime}}
\newcommand{\PQ}{P_Q}
\newcommand{\PQplus}{P^{+}_Q}
\newcommand{\PM}{P_M}
\newcommand{\PMplus}{P^{+}_M}
\newcommand{\PT}{P_T}
\newcommand{\PTplus}{P^{+}_T}
\newcommand{\PG}{P_G}
\newcommand{\PGplus}{P^{+}_G}
\newcommand{\PC}{P_C}
\newcommand{\PCplus}{P^{+}_C}

\newcommand{\PB}{P_B}

\newcommand{\PBop}{P_B^{\rm op}}

\newcommand{\topone}{\hat{1}}
\newcommand{\botzero}{\hat{0}}
\newcommand{\oneR}{\hat{1}_R}
\newcommand{\zeroR}{\hat{0}_R}
\newcommand{\oneQ}{\hat{1}_Q}
\newcommand{\zeroQ}{\hat{0}_Q}
\newcommand{\oneL}{\hat{1}_L}
\newcommand{\zeroL}{\hat{0}_L}
\newcommand{\oneC}{\hat{1}_C}
\newcommand{\zeroC}{\hat{0}_C}

\newcommand{\Qprop}{\overline{Q}}

\newcommand{\Lprop}{\overline{L}}
\newcommand{\Llt}{<_L}

\newcommand{\Llte}{\leq_L}

\newcommand{\Lor}{\vee_L}
\newcommand{\Land}{\wedge_L}

\newcommand{\comp}{{\mathfrak C}}

\newcommand{\Plte}{\leq_P}
\newcommand{\Qlte}{\leq_Q}
\newcommand{\Qgte}{\geq_Q}

\newcommand\mydefem[1]{{\underline{\em #1}}}

\newcommand{\inter}{\cap}
\newcommand{\union}{\cup}
\newcommand{\biginter}{\bigcap}
\newcommand{\bigunion}{\bigcup}

\newcommand{\sdiff}[2]{#1 \setminus #2}

\newcommand{\tX}[0]{\overline{X}}
\newcommand{\tY}[0]{\overline{Y}}

\newcommand{\xtild}{{\tilde{x}}}
\newcommand{\ytild}{{\tilde{y}}}

\newcommand{\Xp}{{X^\prime}}
\newcommand{\Yp}{{Y^\prime}}

\newcommand{\ys}{\{y\}}
\newcommand{\ysi}{\{y_i\}}
\newcommand{\ysj}{\{y_j\}}

\newcommand{\XR}{X^R}
\newcommand{\YR}{Y^R}
\newcommand{\XQ}{X^Q}
\newcommand{\YQ}{Y^Q}

\newcommand{\XRy}{X^{R}_y}
\newcommand{\YRx}{Y^{R}_x}
\newcommand{\XQy}{X^{Q}_y}
\newcommand{\YQx}{Y^{Q}_x}

\newcommand{\XRyi}{X^{R}_{y_{\scriptstyle i}}}
\newcommand{\XRyj}{X^{R}_{y_{\scriptstyle j}}}
\newcommand{\XQybar}{X^{Q}_{\ybar}}

\newcommand{\fX}{f_X}
\newcommand{\fY}{f_Y}
\newcommand{\fXg}{\fX^g}
\newcommand{\fYg}{\fY^g}

\newcommand{\fYinv}{\fY^{-1}}

\newcommand{\gX}{g_X}
\newcommand{\gY}{g_Y}
\newcommand{\hX}{h_X}
\newcommand{\hY}{h_Y}

\newcommand{\spc}[0]{\hspace*{0.2pt}}
\newcommand{\hspc}[0]{\hbox{\hspace*{0.2pt}}}
\newcommand{\hspt}[0]{\hbox{\hspace*{1pt}}}

\newcommand{\argmap}[1]{\xmapsto{\spc\,{#1}\spc}}

\newcommand{\abs}[1]{\lvert#1\rvert}

\newcommand{\norm}[1]{\lVert#1\rVert}
\newcommand{\supp}[1]{\norm{#1}}
\newcommand{\verts}[1]{\mathop{\rm verts}(#1)}

\newcommand{\Lk}{\mathop{\rm Lk}}
\newcommand{\lk}{\mathop{\rm Lk}}
\newcommand{\dl}{\mathop{\rm dl}}
\newcommand{\st}{\mathop{\rm \overline{St}}}

\newcommand{\homot}{\simeq}

\newcommand{\rslow}{\mathop{{\rm r}_{\rm slow}}}
\newcommand{\rfast}{\mathop{{\rm r}_{\rm fast}}}

\newcommand{\Mlkcard}{\hbox{\# of athletes}}
\newcommand{\Jlkcard}{\hbox{\# of musicians}}

\newcommand{\Mlkmax}[1]{\underset{\hspace*{-1pt}\hbox{\footnotesize{\hbox{athletes}}}}{\max}\,{#1}}
\newcommand{\Jlkmax}[1]{\underset{\hspace*{-1pt}\hbox{\footnotesize{\hbox{musicians}}}}{\max}\,{#1}}

\newcommand{\pelem}[2]{$(\makebox[21pt][l]{{#1}}, \makebox[21pt][r]{{#2}})$}

\newcommand{\frakA}{{\mathfrak{A}}}
\newcommand{\calA}{{\mathcal{A}}}

\newcommand{\src}{\mathop{\rm src}}
\newcommand{\DG}{\Delta_G}
\newcommand{\maxDG}{{\mathfrak{M}}}
\newcommand{\SG}{{\overline{\Delta}_G}}

\newcommand{\wrep}{\diamond}

\newcommand{\dowAx}{\Psi_A}
\newcommand{\dowAy}{\Phi_A}
\newcommand{\PA}{P_A}
\newcommand{\PAplus}{P^{+}_A}
\newcommand{\oneA}{\hat{1}_A}
\newcommand{\zeroA}{\hat{0}_A}
\newcommand{\PAop}{P_A^{\rm op}}
\newcommand{\PAplusop}{(P_A^{+})^{\rm op}}

\newcommand{\sd}{\mathop{\rm sd}}

\newcommand{\cl}{{\rm cl}}

\newcommand{\clAy}{\cl_A}
\newcommand{\clBy}{\cl_B}

\newcommand{\srceq}{\theta}

\newcommand{\Smax}{S_{\max}}
\newcommand{\Ssmax}{{\hbox{\footnotesize $S$}}_{\scriptstyle \max}}
\newcommand{\Sv}{S^{\join}}
\newcommand{\Svmax}{(S^{\join})_{\max}}

\newcommand{\Tv}{T^{\join}}

\newcommand{\bndry}[1]{{\partial{#1}}}
\newcommand{\redbndry}{\widetilde{\partial}}

\newcommand{\emptygen}{{\mathds{1}}}

\newcommand{\clsy}{\phi_R \circ \psi_R}
\newcommand{\clsx}{\psi_R \circ \phi_R}
\newcommand{\clsqy}{\phi_Q \circ \psi_Q}
\newcommand{\clsqx}{\psi_Q \circ \phi_Q}
\newcommand{\clsyp}{\phi_{R^\prime} \circ \psi_{R^\prime}}
\newcommand{\clsxp}{\psi_{R^\prime} \circ \phi_{R^\prime}}

\newcommand{\psimax}{\psi_{\Qmax}}

\newcommand{\xbar}{{\overline{x}}}
\newcommand{\ybar}{{\overline{y}}}

\newcommand{\xhat}{{\hat{x}}}
\newcommand{\yhat}{{\hat{y}}}

\newcommand{\xstar}{x^*}
\newcommand{\ystar}{y^*}

\newcommand{\np}{{$N\!P$}}
\newcommand{\MinInf}{{\sc MinInf\ }}
\newcommand{\MinInfnb}{{\sc MinInf}}

\newcommand{\calU}{{\cal U}}
\newcommand{\calW}{{\cal W}}
\newcommand{\nerveU}{{\cal N}(\calU)}

\newcommand{\roteq}{\hbox{\Huge \rotatebox{90}{$=$}}}

\newcommand{\uddots}{\mkern1mu\raise1pt
    \vbox{\kern7pt\hbox{.}}\mkern2mu
    \raise4pt\hbox{.}\mkern2mu\raise7pt\hbox{.}\mkern1mu}

\newcommand{\vp}{{\bf{p}}}
\newcommand{\vq}{{\bf{q}}}

\title{\vspace*{-0.9in}
          Topology of Privacy:\\[2pt]
  Lattice Structures and Information Bubbles\\[2pt]
               for
      Inference and Obfuscation}

\author{%
\parbox[c]{2.2in}{%
\begin{center}
Michael Erdmann\thanks{This report is based upon work supported in
part by the Air Force Office of Scientific Research under award number
FA9550-14-1-0012 and in part by the National Science Foundation under
award number IIS-1409003.  Any opinions, findings and conclusions or
recommendations expressed in this report are those of the author and
do not necessarily reflect the views of the Government, the
U.S.~Department of Defense, or the National Science Foundation.}\\
Carnegie Mellon University\\
December 12, 2017
\end{center}}}

\date{{\small \copyright\ 2017 Michael Erdmann}}

\begin{document}

\pagenumbering{roman}

\maketitle{}
\thispagestyle{empty}

\vspace*{-0.1in}
\begin{abstract}
Information has intrinsic geometric and topological structure, arising
from relative relationships beyond absolute values or types.  For
instance, the fact that two people did or did not share a meal
describes a relationship independent of the meal's ingredients.
Multiple such relationships give rise to relations and their lattices.
Lattices have topology.  That topology informs the ways in which
information may be observed, hidden, inferred, and dissembled.
Privacy preservation may be understood as finding isotropic
topologies, in which relations appear homogeneous.  Moreover, the
underlying lattice structure of those topologies has a temporal
aspect, which reveals how isotropy may contract over time, thereby
puncturing privacy.

Dowker's Theorem establishes a homotopy equivalence between two
simplicial complexes derived from a relation.  From a privacy
perspective, one complex describes individuals with common attributes,
the other describes attributes shared by individuals.  The homotopy
equivalence is an alignment of certain common cores of those
complexes, effectively interpreting sets of individuals as sets of
attributes, and vice-versa.  That common core has a lattice structure.
An element in the lattice consists of two components, one being a set
of individuals, the other being an equivalent set of attributes.  The
lattice operations join and meet each amount to set intersection in
one component and set union followed by a potentially
privacy-puncturing inference in the other component.

One objective of this research has been to understand the topology of
the Dowker complexes, from a privacy perspective.  First, privacy loss
appears as simplicial collapse of free faces.  Such collapse is local,
but the property of fully preserving both attribute and association
privacy requires a global condition: a particular kind of spherical
hole.  Second, by looking at the link of an identifiable individual in
its encompassing Dowker complex, one can characterize that
individual's attribute privacy via another sphere condition.  This
characterization generalizes to certain groups' attribute privacy.
Third, even when long-term attribute privacy is impossible, homology
provides lower bounds on how an individual may defer identification,
when that individual has control over how to reveal attributes.
Intuitively, the idea is to first reveal information that could
otherwise be inferred.  This last result highlights privacy as a
dynamic process.  Privacy loss may be cast as gradient flow.  Harmonic
flow for privacy preservation may be fertile ground for future
research.
\end{abstract}


\markright{Contents}

\tableofcontents

\cleardoublepage

\pagenumbering{arabic}

\clearpage
\section{Introduction}
\markright{Introduction}
\label{intro}

Privacy is the ability of an individual or entity to control how much
that individual or entity reveals about itself to others.  Fundamental
research into privacy seeks to understand the limits of that ability.

\vspace*{0.05in}

A brief history of privacy should include the following:

\vspace*{-0.05in}

\begin{itemize}

\item {\bf The right} to privacy as a legal principle, appearing in an
  1890 Harvard Law Review article \cite{tpr:warrenbrandeis}.  The
  article was a reaction to the then modern technology of photography
  and the dissemination of gossip via print media.

\item {\bf A demonstration} linking supposedly anonymous public
  information with other more specific public data, thereby revealing
  sensitive attributes \cite{tpr:sweeneykanon}.  The demonstration
  employed zip code, gender, and birth date to link anonymous public
  insurance summaries with voter registration data.  Doing so produced
  the health record of the governor of Massachusetts.  This privacy
  failure suggested a first form of homogenization, called {\em
  $k$-anonymity}.  Roughly, the idea was to structure databases in
  such a way that a database could respond to any query with an answer
  consisting of no fewer than $k$ individuals matching the query
  parameters.

\item {\bf The discovery} that it is impossible to preserve the
  privacy of an individual for even a single attribute in the face of
  repeated statistical queries over a population
  \cite{tpr:DinurNissim}, {\em unless} \,answers to those queries are
  purposefully perturbed with noise of magnitude on the order of at
  least $\sqrt{n}$.  Here $n$ is the size of the population.  The
  significance of this discovery is to underscore how difficult it is
  to preserve privacy while retaining information utility.

\item {\bf Netflix Prize}.  In 2006, Netflix offered a \$1M prize
  for an algorithm that would predict viewer preferences better than
  Netflix's internal algorithm.  Netflix made available some of its
  historical user preferences, in anonymized form, as a basis for the
  competition.  Once again, it turned out that one could link this
  anonymized data with other publicly available databases, resulting
  in the potential (and in some cases actual) identification of
  Netflix viewers, thereby de-anonymizing their viewing history
  \cite{tpr:netflix}.  Whereas in the earlier health example, a few
  specific observables made linking possible (global coordinates, one
  might say, namely zip code, gender, birth date), in the Netflix
  example, the intrinsic geometric structure of the database
  facilitated linking via a wide variety of observables (local
  landmarks, one might say, namely movies that were characteristic for
  each individual).  Key was sparsity of information: 8 movie ratings
  and dates were generally enough to uniquely characterize $99\%$ of
  viewers in the Netflix Prize dataset, even with errors in the
  ratings and dates.

\item {\bf Differential Privacy} \cite{tpr:dworkcacm, tpr:Dwork08}
  seeks to avoid the previous privacy failures by focusing on local
  rather than absolute privacy guarantees.  The underlying approach in
  differential privacy is for a database to answer statistical queries
  with a particular stochastic blurring.  Specifically, the
  probability that an interrogator of the database will make any
  particular inference should depend only in a very small way on
  whether any one individual does or does not have a particular
  attribute (such as even being in the database).  We might call this
  {\em stochastic homogeneity}.

\item {\bf Randomized Response}.  Differential privacy is further
  significant because it makes explicit the dynamic nature of privacy;
  there may be no enduring privacy guarantees but there are
  differential guarantees.  A particular form is {\em randomized
  response}, a technique used in the social sciences to elicit
  reliable aggregate answers to sensitive questions, asking the
  question of many people, but perturbing individual answers
  stochastically so as not to learn much about any one individual from
  any single response \cite{tpr:rrWarner}.  A version has been
  employed by Google to find malware \cite{tpr:rappor}.

\end{itemize}

Privacy has both a combinatorial component and a statistical
component.  Prior research has largely focused on statistical
techniques, both to preserve privacy and to puncture privacy.  One of
the goals of this research is to understand the combinatorial
component of privacy, leading naturally to methods from combinatorial
topology.

\vst

A desire to understand the geometry and topology of the types of
inferences revealed by the Netflix Prize formed the specific
motivation for our research initially.  Subsequently, we realized that
the lattice structure found in that geometry had broader
applicability, providing an ability to model the dynamics of privacy
more generally.

\clearpage
\section{Outline}
\markright{Outline}

The remaining sections and appendices present the following material:

\subsection*{Main Narrative:}

\begin{description}
\item[3:] Toy examples illustrating how a relation may lead to privacy
  loss in the presence of background information.  The section
  introduces the {\em doubly-labeled poset\,} associated with a
  relation, to model such inferences.  The elements of the poset are
  ordered pairs, each a set of individuals and a set of attributes.

  This section also states and discusses assumptions that hold
  throughout the report.

\item[4:] Formal description of the {\em Galois connection} associated
  with a relation.  The section first defines, for any relation, two
  simplicial complexes called {\em Dowker complexes}.  One complex
  represents sets of individuals with shared attributes, the other
  represents sets of attributes shared by individuals.  The Galois
  connection then establishes a homotopy equivalence between the
  Dowker complexes, thereby generating the relation's doubly-labeled
  poset.  The homotopy equivalence gives rise to closure operators,
  with ``closure'' in the poset modeling inference of unobserved
  attributes from observed attributes (or unobserved individuals from
  observed individuals).

  This section also defines {\em attribute privacy\,} and {\em
  association privacy}.

\item[5:] A characterization of privacy in terms of the absence of free
  faces in the relevant Dowker complex.  This section observes as well
  that the only connected relations able to preserve both attribute and
  association privacy must look either like linear cycles or like
  boundary complexes.  In particular, the number of individuals and
  attributes must be the same.

\item[6:] Conditional relations, as models for simplicial links.
  A conditional relation is much like a conditional probability
  distribution.  It might, for instance, represent the possible
  arrangement of remaining attributes among individuals, after some
  attributes have already been observed.

\item[7:] A characterization of individual and group attribute privacy
  in terms of spherical and boundary complexes for the relation that
  models the individual's or group's link in its Dowker complex.

\item[8:] A brief exploration of holes in relations, focusing on
  attribute spaces generated by bits.

\item[9:] A small example exploring the possibility of increasing
  privacy by change-of-coordinate transformations.

\item[10:] A lengthy exploration of how someone can delay identification,
  by releasing attributes selectively in a particular order.  This
  idea leads to the notion of {\em informative attribute release
  sequences}, how to find such sequences in the {\em Galois lattice},
  and the use of homology as a lower bound for the number and
  length of such sequences.

\item[11:] Computation of the homology and maximal informative attribute
  release sequences present in two relations found on the world wide
  web.  One relation describes Olympic athletes and their medals, the
  other describes jazz musicians and their bands.

\item[12:] A more general perspective of inference as motion in
  lattices, not necessarily directly derived from a relation.  This
  perspective suggests connections to randomized response techniques.

\item[13:] An examination of the ability to obfuscate strategies and/or
  goals in graphs where motions may be nondeterministic or stochastic.

\item[14:] A possible category for representing relations, along with
  an analysis of morphism properties.  The morphisms between relations
  in this category induce simplicial and therefore continuous maps
  between the relations' corresponding Dowker complexes.

  This section further shows by example how a morphism of relations,
  when it is surjective at the set level, generates the full lattice
  of the codomain's relation, via closure under lattice operations.
  \ (A general proof appears in Appendix~\ref{morphismappendix}.)

\item[15:] Some thoughts for the future, including an example that
  connects stochastic sensing to the Galois lattice.
\end{description}

\vspace*{0.05in}

\subsection*{Appendices:}

\vspace*{0.1in}

\begin{description}
\item[A:] A summary of the basic notation and definitions used in this
  report.

\item[B:] A summary of the basic tools used in this report, establishing
  the homotopy equivalences and closure operators mentioned previously.

\item[C:] Construction of links and deletions, and examination of the
 privacy properties each inherits from its encompassing relation.
 This appendix explores the significance of free faces in the Dowker
 complexes.  The appendix further proves that a relation with more
 attributes than individuals cannot preserve attribute privacy for
 every individual.

\item[D:] Proof that the problem of finding a minimal set of
 attributes from which another attribute may be inferred is
 \np-complete.  This stands in contrast to the observation that the
 problem of finding {\em some}\, set of attributes from which another
 may be inferred (or reporting that no such set exists) is computable in
 polynomial time.

\item[E:] Detailed proofs of the results claimed in
 Section~\ref{sphereprivacy}.  Also a detailed proof of the assertion
 from Section~\ref{faceshape} regarding relations that preserve
 both attribute and association privacy.

\item[F:] Detailed proofs of the connection between maximal chains in
 a relation's Galois lattice and informative attribute release
 sequences.  When such sequences are order-independent they correspond
 to spherical holes, leading to the concept of an {\em isotropic\,}
 sequence.

\item[G:] Detailed proof that homology establishes a lower bound for the
 number and length of maximal chains in a relation's Galois lattice, and
 thus for the number and length of informative attribute release
 sequences that may be used to delay identification.

\item[H:] An application of the previous results with the aim of
 obfuscating the identification of strategies for attaining goals in
 graphs with uncertain transitions.

\item[I:] Detailed proofs of the assertions of Section~\ref{category}
 regarding morphisms.

\item[J:] Some additional examples: 
   \begin{enumerate}

       \item Dunce Hat:  modeled as a relation for which the Dowker
         attribute complex is contractible but has no free attribute
         faces, meaning the relation preserves attribute privacy.

       \item Disinformation:  An example that glues together two copies
         of the M\"obius strip, thereby removing free faces and creating
         a form of homogeneity that preserves attribute privacy yet
         retains the utility of identifiability.

       \item Insufficient Representation:  If there are insufficiently
         many individuals in a relation generated by bits, attribute
         inference is possible.

       \item A Matching Example: When many individuals are being
         observed, cardinality constraints allow for inferences beyond
         those discussed in this report.

   \end{enumerate}

\end{description}

\clearpage

\vspace*{-0.35in}

\begin{center}
{\Large \bf \underline{List of Primary Symbols}}
\addcontentsline{toc}{section}{List of Primary Symbols}

\vspace*{0.2in}

\begin{tabular}{llr}
{\large \bf \underline{Symbol}} &  {\large \bf \underline{Typical Meaning}} & {\large \bf \underline{Page(s)}}\\[6pt]
$X$              &  discrete space of individuals  & \pageref{basicdefs}, \pageref{reldefApp}\\[2pt]
$Y$              &  discrete space of attributes  & \pageref{basicdefs}, \pageref{reldefApp}\\[2pt]
$R$              &  relation on $\XxY$  & \pageref{basicdefs}, \pageref{reldefApp}\\[2pt]
$X_y$            &  individuals with attribute $y$ (usually in the context of relation $R$)  & \pageref{basicdefs}, \pageref{reldefApp}\\[2pt]
$Y_x$            &  attributes of individual $x$ (usually in the context of relation $R$)   & \pageref{basicdefs}, \pageref{reldefApp}\\[2pt]
$Q$              &  another relation, often representing a link in a simplicial complex & \pageref{linkgamma}, \pageref{yLink}, \pageref{restrictedlink}\\[2pt]
$\Sigma, \Gamma$ &  generic simplicial complexes \ (sometimes merely sets) & \pageref{simpcpxdef}\\[2pt]
$\dowx$          &  complex; simplices are sets of individuals with a common attribute   & \pageref{basicdefs}, \pageref{dowkerXAppdef}\\[2pt]
$\dowy$          &  complex; simplices are sets of attributes shared by some individual  & \pageref{basicdefs}, \pageref{dowkerYAppdef}\\[2pt]
$\sigma$         &  usually a simplex representing individuals in $\dowx$  & \\[2pt]
$\gamma$         &  usually a simplex representing attributes in $\dowy$  & \\[2pt]
$\phi_R$         &  homotopy equivalence from sets of individuals to shared attributes   & \pageref{phipsiDef}, \pageref{phipsiAppdef}\\[2pt]
$\psi_R$         &  homotopy equivalence from sets of attributes to sharing individuals  & \pageref{phipsiDef}, \pageref{phipsiAppdef}\\[2pt]
$P$              &  partially ordered set (poset)   & \pageref{posetAppdef}\\[2pt]
$\F(\Sigma)$     &  face poset of the simplicial complex $\Sigma$  & \pageref{faceposetDef}, \pageref{faceposetAppdef}\\[2pt]
$\Delta(P)$      &  order complex of the poset $P$  & \pageref{ordercpxDef}, \pageref{ordercpxAppdef}\\[2pt]
$\PR$            &  doubly-labeled poset associated with relation $R$  & \pageref{PRsemantics}, \pageref{defPR}, \pageref{PRAppDeff}\\[2pt]
$L$              &  (inference) lattice  & (\pageref{inferencelattice}) \pageref{latticeAppdef}\\[2pt]
$\PRplus$        &  Galois lattice formed from $\PR$  &  \pageref{galoislattice}, \pageref{galoislatticeApp}\\[2pt]
\hspace*{-0.65in}$\{(\sigma_k, \gamma_k) < \cdots < (\sigma_0, \gamma_0)\}$ & chain of length $k$ in the lattice $\PRplus$
                                               & \hspace*{-0.35in}\pageref{chaintoiars}, \pageref{posetchains}, \pageref{posetchainAppdef}\\[2pt]
$y_1, \ldots, y_k$ & informative attribute release sequence (iars) of length $k$ (for relation $R$)\hspace*{-0.3in}  & \pageref{iars}\\[2pt]
$V$              &  set of vertices in a simplicial complex or states in a graph  & \\[2pt]
$\bndry{(V)}$    &  simplicial boundary complex with vertices $V$  & \pageref{bndrycpxDef}, \pageref{bndrycpxAppdefStart}\\[2pt]
$\Smo$           &  sphere of dimension $-1$, modeling the empty complex $\{\emptyset\}$         & \pageref{emptycomplex}\\[2pt]
$\Sone$          &  circle & \pageref{bndrycpxDef}\\[2pt]
$\Snt$           &  sphere of dimension $n\!-\!2$  & \pageref{bndrycpxDef}, \pageref{bndrycpxAppdefEnd}\\[2pt]
$C_k(\Sigma; \Z)$ &  group of simplicial $k$-chains over $\Sigma$, with integer coefficients  & \pageref{kchainAppdef}\\[2pt]
$\redbndry$      &  (family of) reduced boundary map(s) \ $C_k(\Sigma; \Z) \rightarrow C_{k-1}(\Sigma; \Z)$  & \pageref{redbndryAppdef}\\[2pt]
$\widetilde{H}_k(\Sigma; \Z)$ & reduced $k$-dimensional homology group of $\Sigma$, with integer coefficients & \pageref{homolAppDef}\\[2pt]
$G$              &  a graph, generally with nondeterministic and/or stochastic actions  & \pageref{stratcomplex}, \pageref{GandDGdefsApp}\\[2pt]
$\DG$            &  strategy complex of a graph  & \pageref{stratcomplexDescrip}, \pageref{GandDGdefsApp}\\[2pt]
$\SG$            &  source complex of a graph  & \pageref{SrcCmplxAppdef}\\[2pt]
$\homot$         &  homotopy equivalence  & \pageref{homotAppdef}\\[2pt]
$*$              &  simplicial join & \pageref{joinAppdef}\\[2pt]
$\join$          &  either topological wedge sum or lattice join & \pageref{wedgeAppdef}, \pageref{lattices}\\[2pt]
$\meet$          &  lattice meet & \pageref{lattices}\\
\end{tabular}
\end{center}

\clearpage
\section{Privacy: Relations and Partially Ordered Sets}
\markright{Assumptions and a Toy Example}
\label{privacyexamples}

Our investigation of privacy in this report will be in terms of
relations.  As we will see in this section and the next, relations
give rise to simplicial complexes, which give rise to partially
ordered sets, which expose an underlying lattice structure.  That
lattice structure makes explicit how privacy may be preserved or lost
through so-called {\em background knowledge}.  As we will see in
Section~\ref{leveraginglattices}, the lattice structure also makes
explicit how identification may be delayed by careful release of
information.

\subsection{A Toy Example: Health Data and Attribute Privacy}
\label{ToyExample}

Consider the following relation $H$, describing the results of a
hypothetical health study for four patients and three attributes.  The
patients have been anonymized and are represented simply by the set of
numbers $\{1,2,3,4\}$.  The three attributes are drawn from the set
$\{\hbox{\sc smokes}, \hbox{\sc has\_cancer}, \hbox{\sc drinks\_soda}\}$.

\vspace*{0.1in}

\noindent One can describe a relation equivalently either as a matrix
or as a set of ordered pairs:

\vspace*{0.1in}

\noindent \underline{Relation $H$ as a matrix:}

\vspace*{-0.3in}

\[
\qquad\qquad\qquad\begin{array}{c|ccc}
H & \hbox{\sc smokes} & \hbox{\sc has\_cancer} & \hbox{\sc drinks\_soda} \\[2pt]\hline
1 & \one & \one & \\
2 &      & \one & \one \\
3 &      &      & \one \\
4 &      &      & \one \\
\end{array}
\]

\vspace*{0.1in}

\noindent \underline{Relation $H$ as a set of ordered pairs:}

\vspace*{-0.1in}

\[\big\{(1, \hbox{\sc smokes}), \, (1, \hbox{\sc has\_cancer}), \,
    (2, \hbox{\sc has\_cancer}), \, (2, \hbox{\sc drinks\_soda}), \,\]
\[(3, \hbox{\sc drinks\_soda}), \, (4, \hbox{\sc drinks\_soda})\big\}.\]

\vst

\subsection*{Assumptions}
\label{assumptions}

Before discussing privacy further, we make some assumptions that hold
throughout the report:

\vspace*{-0.1in}

\addcontentsline{toc}{subsubsection}{Assumption of Relational Completeness}
\paragraph{Assumption of Relational Completeness:} \ We 
assume that any given relation {\em is not missing any observable
elements}, relative to some external (unspecified) ground
truth. \label{relationalcompleteness}

\vst

For example, if we observe that someone drinks soda and has cancer in
relation $H$, then we would conclude that we are observing individual
\#2.  We would be surprised to see that individual smoke.  If for some
reason we ever do see the individual smoke, then we would deem our
observations to be \spc{\em inconsistent\spc} with relation $H$.  ---
The meaning of inconsistency depends on context.  At top-level, an
inconsistency may mean that the relation or observation is errorful.
When making conditional observations, an inconsistency may actually
supply useful information, as we will see in Lemma~\ref{interplocal}
on page~\pageref{interplocal}.

\vsr

{\bf Comment:}\ A relation may contain extra elements, as may be
useful for disinformation.  A relation could even be missing elements
that represent valid ordered pairs, so long as those elements are
deemed to be unobservable for that relation.  For example, one may
have a time series of relations in which some attributes only become
observable at later times.  In such a setting, one may never know
whether a particular individual had a particular attribute at an
earlier time.

In the example, it could be that individual \#1 drinks soda, but that
it is impossible to observe this fact.  In that case, relation $H$
would still satisfy the assumption of relational completeness, even
though $H$ contains no entry\footnote{Terminology: We often use the
term {\tt{\char13}entry{\char13}} to mean an element of a relation, as
in a matrix, or in one of its rows or columns.} indicating that
individual \#1 drinks soda.

\addcontentsline{toc}{subsubsection}{Assumption of Observational Monotonicity}
\paragraph{Assumption of Observational Monotonicity:} \ Even though we
assume {\em relations} are complete, we do {\em not\,} assume that
{\em observations} are complete.  Instead, we assume: {\em The
observation of a particular attribute for an individual is meaningful;
lack of such an observation does not necessarily imply that the
individual fails to have the unobserved attribute.}  The motivation
for this assumption is that one may yet discover that the individual
has the attribute.
For example, suppose we observe someone (whom we know to be part of
relation $H$) drinking soda.  Even if that is all we observe, we do
{\em not\,} conclude that the individual is cancer free.  It could be
that we might yet observe the individual to have cancer.

\vst

If absence of an attribute is significant {\em and}\, that absence is
observable, then both the attribute and its negation could and perhaps
should appear explicitly in the relation as distinct mutually
exclusive attributes.  For instance, {\sc Prime} versus {\sc
Composite} might be such a pair of attributes for integers greater
than 1.

\addcontentsline{toc}{subsubsection}{Assumption of Observational Accuracy}
\paragraph{Assumption of Observational Accuracy:} \ We assume that
{\em observations are accurate}.  For instance, if we observe an
integer to be either {\sc Prime} or {\sc Composite}, then we do so
correctly.

\paragraph{Comments:}

The three assumptions above are {\em desiderata} for how the
mathematical abstractions of this report fit into the real world.
\ Some comments are in order:

\vspace*{-0.05in}

\begin{itemize}
\addtolength{\itemsep}{-1pt}

\item  In and of itself, a relation defines a particular kind of
world, a bipartite graph, and there is no external ground truth.

\item In such a world, the completeness, monotonicity, and accuracy
assumptions describe a sensor and the meaning of observations made by
the sensor.

\end{itemize}

The purpose of the assumptions in the real world is largely to ensure
consistency between different relations and with possible observations.

\vspace*{-0.05in}

\begin{itemize}
\addtolength{\itemsep}{-1pt}

\item The monotonicity assumption is important because information
generally aggregates asynchronously.  Together with the other
assumptions, this assumption means that one may view relations as
monotone Boolean functions, and thus may leverage methods from
combinatorial topology.

\item One may incorporate some errors into the relational and
observational models, for instance by blurring a relation.  For very
large integers, a relation might allow some integers to have {\em
both\,} {\sc Prime} and {\sc Composite} as attributes.  Although an
integer is one or the other, the relation admits to uncertainty by
allowing both attributes at once.  Indeed, some relations purposefully
introduce such blurring to preserve privacy, as with randomized
response \cite{tpr:rrWarner}.  In robotics, natural relational
blurring arising from noisy but environment-compatible sensors can
actually help establish the topology of a region, for instance by
dualizing sensors and landmarks \cite{tpr:ghristlipsky}.

\end{itemize}

\subsection*{Privacy Implications}
\markright{Privacy Implications}

Making the health study $H$ of page~\pageref{ToyExample} publicly
available has some privacy implications, including the following:

\vspace*{-0.05in}

\begin{itemize}
\item Suppose someone named Bob tells his friend Alice that he was
  part of the study.  Alice knows that Bob smokes everywhere he goes,
  so she can infer that he is Patient \#1 and has cancer. \ 
  (This is an example of inference in a relation using background
  knowledge.)

\item Suppose Cindy is Patient \#2.  She has full attribute privacy as
  far as relation $H$ is concerned.  In particular, as we saw already,
  Cindy can tell her friends that she was part of the health study
  while drinking soda and those friends will not be able to conclude
  that she has cancer.

\item Patients \#3 and \#4 are not only indistinguishable from each
  other but also from Cindy (patient \#2), as far as relation $H$ is
  concerned.  This is a very strong form of anonymity.  Even if one of
  them reveals that s/he drinks soda, s/he will remain
  indistinguishable from the other two patients who drink soda.

\end{itemize}

{\bf Caveat:} In the last case, if Cindy reveals that she has cancer
  and is seen to be different from the other individuals, then one may
  be able to remove her from the relation, narrowing the focus and
  creating a new relation that may allow additional inferences.
  Similar caveats hold for the other bullets.  Deletions are discussed
  further in Appendix~\ref{linksandinference}.

\vspace*{0.1in}

\paragraph{Modifying a Relation to Increase Privacy}
\label{spheres}

We can make a small change in relation $H$ that enhances privacy.
If we artificially give patient \#3 the attribute \hbox{\sc smokes},
then we obtain the following modified relation $H^\prime$:

\[
\begin{array}{c|ccc}
H^\prime & \hbox{\sc smokes} & \hbox{\sc has\_cancer} & \hbox{\sc drinks\_soda} \\[2pt]\hline
1 & \one & \one & \\
2 &      & \one & \one \\
3 & \one &      & \one \\
4 &      &      & \one \\
\end{array}
\]

\vspace*{0.1in}

Now Bob may reveal to Alice that he was part of the health study
without Alice being able to infer that he has cancer, even though she
knows that everyone knows that he smokes.  In fact, more generally,
one can no longer infer cancer from smoking, within the relation.

\vst

Such an artificial entry in the relation is a form of {\em
disinformation}.  It certainly skews statistics and utility.  It also
increases privacy.

\clearpage
\subsection{A Dual Perspective: Payroll Data and Association Privacy}
\label{payroll}

The previous example examined a relation from the perspective of {\em
attribute privacy}: we were interested in understanding how
observation of some attribute(s) implied other attribute(s), possibly
identifying an individual.  A dual perspective is {\em association
privacy}, in which one seeks to understand how some associations
between individuals imply others.

The following hypothetical ``salary'' relation $S$ has the same matrix
structure as relation $H$ did earlier, but with different semantics.
This relation represents employees $\{\hbox{Bob}, \hbox{Mary},
\hbox{Frank}, \hbox{Julie}\}$ working on secret projects
$\{\ta, \tb, \tc\}$.  Now the employee names are visible so that a
payroll clerk can disburse salaries correctly, but the actual projects
are anonymous.

\[
\begin{array}{c|ccc}
    S        & \ta  & \tb  & \tc \\[2pt]\hline
\hbox{Bob}   & \one & \one & \\
\hbox{Mary}  &      & \one & \one \\
\hbox{Frank} &      &      & \one \\
\hbox{Julie} &      &      & \one \\
\end{array}
\]

\vspace*{0.1in}

The salary relation $S$ has some implications for association privacy,
including the following:

\vspace*{-0.05in}

\begin{itemize}

\item If someone tells the payroll clerk that Julie is the lead of a
  very important project with valuable information, then the payroll
  clerk can infer that Mary and Frank have also been exposed to
  valuable information.

\item In contrast, if someone tells the payroll clerk that Bob is
  running a very important project, then the payroll clerk does not
  have enough information to conclude that Mary is also working on an
  important project.

\end{itemize}

Regarding disinformation: Observe how adding the artificial entry
$(\hbox{Julie}, \ta)$ prevents the payroll clerk from using the
relation to infer that Mary and Frank have valuable information, even
if the payroll clerk learns via background information that Julie is
the lead of a very important project with such information:

\[
\begin{array}{c|ccc}
  S^\prime   & \ta  & \tb  & \tc \\[2pt]\hline
\hbox{Bob}   & \one & \one & \\
\hbox{Mary}  &      & \one & \one \\
\hbox{Frank} &      &      & \one \\
\hbox{Julie} & \one &      & \one \\
\end{array}
\]

\clearpage
\subsection{Privacy Preservation and Loss: A Poset Model}
\markright{Privacy Preservation and Loss: A Poset Model}
\label{privacyposet}

\begin{figure}[h]
\begin{center}
\hspace*{0.4in}\begin{minipage}{1.5in}{$\begin{array}{c|ccc}
    R        & \ta  & \tb  & \tc \\[2pt]\hline
 1  & \one & \one & \\
 2  &      & \one & \one \\
 3 &      &      & \one \\
 4 &      &      & \one \\
\end{array}$}
\end{minipage}
\hspace*{0.125in}
\begin{minipage}{4in}
\ifig{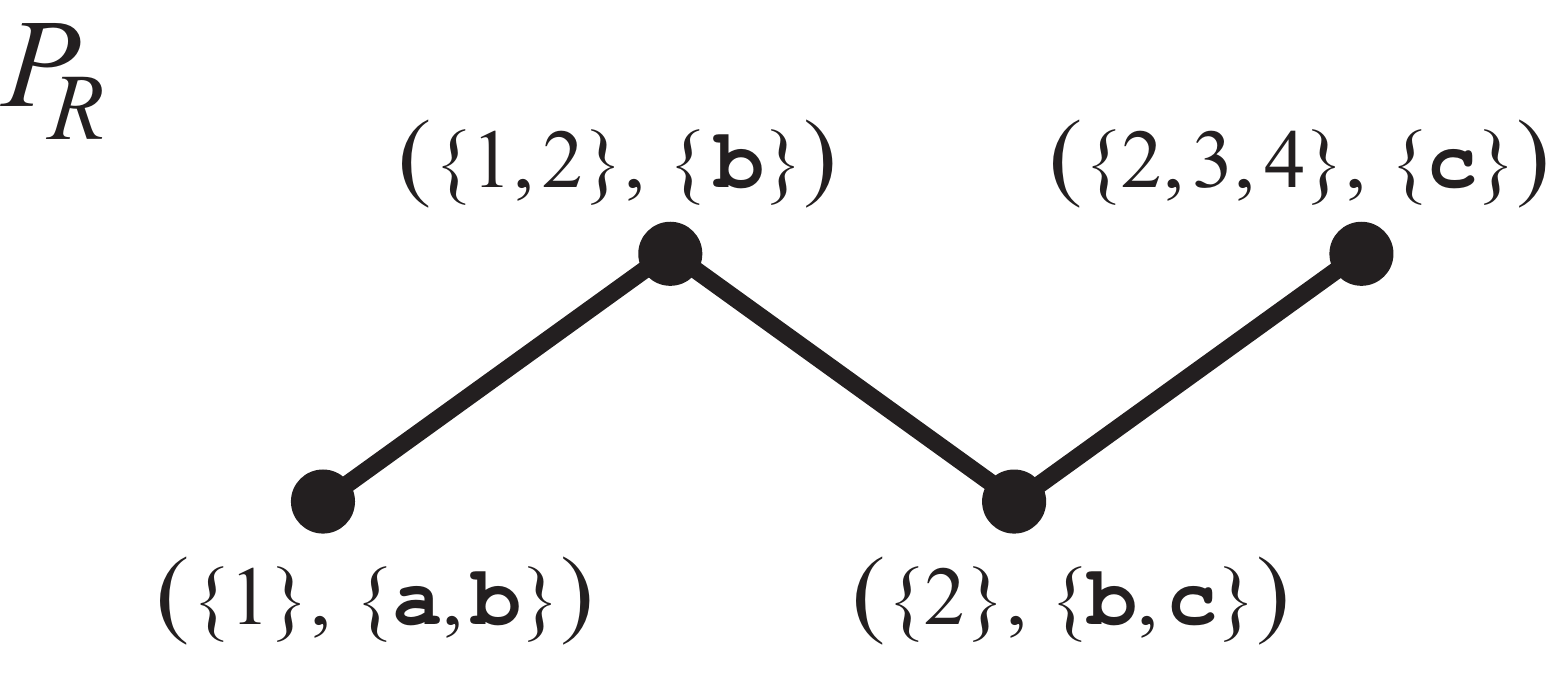}{width=2.8in}
\end{minipage}
\end{center}
\vspace*{-0.1in}
\caption[]{Relation $R$ serves as a model for the two examples of
  Sections~\ref{ToyExample} and \ref{payroll}.  The doubly-labeled poset
  $\PR$ describes the inferences facilitated by $R$.}
\label{IntroPR}
\end{figure}

\vspace*{-0.05in}

Figure~\ref{IntroPR} shows a relation $R\mskip2mu$ that serves as a
model for both the health example of Section~\ref{ToyExample} and the
payroll example of Section~\ref{payroll}.  The relation is identical
to those given earlier, but with abstract labels in place of both
individuals and attributes.  The figure also depicts a partially
ordered set (poset) $\PR$, designed to model the inferences discussed
previously.  We refer to that poset as the {\em doubly-labeled poset
associated with $R$}.  We next discuss the semantics of $\PR$.
Section~\ref{galois} discusses the construction of $\PR$.  The
underlying concepts are important throughout the report.

\paragraph{Semantics of the poset $\PR$:} \label{PRsemantics}

\vspace*{-0.075in}

\begin{itemize}
\item Each element in the poset consists of an ordered pair $(\sigma, \gamma)$,
  with $\emptyset \neq \sigma \subseteq \{1,2,3,4\}$ describing a set
  of individuals and $\emptyset \neq \gamma \subseteq \{\ta,\tb,\tc\}$
  describing a set of attributes.  We say that the poset element is
  {\em labeled with $\sigma\!$ and $\gamma$}.  The meaning of such a
  double-labeling (with respect to the information described by
  relation $R$) is:

\vspace*{-0.05in}

   \begin{enumerate}

      \item[(a)] All individuals in $\sigma$ have all attributes in
            $\gamma$.

\vspace*{1pt}

      \item[(b)] If (and only if) an individual has at least all the
            attributes in $\gamma$, then that individual must be in
            $\sigma$.  For example, we see that individual \#2, and
            only individual \#2, has both attributes $\tb$ and $\tc$
            in $R$.

      \item[(c)] If (and only if) an attribute is shared by at least
           all individuals in $\sigma$, then that attribute must be in
           $\gamma$.  For example, individual \#1 has both attributes
           $\ta$ and $\tb$, so $\PR$ cannot contain simply $(\{1\},
           \{\ta\})$, but must contain $(\{1\}, \{\ta, \tb\})$.
           \end{enumerate}

\item The partial order for $\PR$ is described by the edges in the
  figure.  There is an edge between two elements $(\sigma_1,
  \gamma_1)$ and $(\sigma_2, \gamma_2)$ of $\PR$ whenever the
  corresponding sets are subset comparable.  In particular,
  $(\sigma_1, \gamma_1) \leq (\sigma_2, \gamma_2)$ in $\PR$ precisely
  when $\sigma_1 \subseteq \sigma_2$ and $\gamma_1 \supseteq
  \gamma_2$.  [Observe that the comparability ($\subseteq$ versus
  $\supseteq$) is opposite for $\sigma$ versus $\gamma$.]

\end{itemize}

\vspace*{0.05in}

\noindent {\bf Using the poset $\PR$ for attribute inference:}

\vspace*{0.05in}

\noindent Suppose $\gamma$ is {\em any\,} nonempty subset of attributes
in $\{\ta,\tb,\tc\}$.  Then one of (i) or (ii) holds:

\vspace*{-0.05in}

   \begin{enumerate}

      \item[(i)] Perhaps no individual modeled by $R\hspt$ has all the
        attributes $\gamma$.  For example, no individual has
        attributes $\gamma=\{\ta, \tc\}$.  We would not expect to see
        $\gamma$ and so $\gamma$ does not appear in the poset $\PR$.

      \item[(ii)] Alternatively, $\gamma$ is a subset of at least one
        set of attributes that does appear in the poset.  In this
        case, one {\em may} be able to enlarge $\gamma$ nontrivially,
        resulting in privacy loss.

       For example, imagine we discover that a friend with attribute
       $\ta$ is modeled by the given relation (e.g., Bob, who {\sc
       smokes}, says he is part of the health study $H$).\\  Using
       $\gamma = \{\ta\}$, the poset then allows us to infer that Bob
       must also have attribute $\tb$ (that is, {\sc has\_cancer}).  \
       Why?  \ Because $\{\ta, \tb\}$ is a minimal set in $\PR$
       containing $\{\ta\}$.

       We can say yet more: The element labeled with $\{\ta,\tb\}$ is
       also labeled with $\{1\}$.  So now we have {\em de-anonymized\,}
       individual \#1 (identifying him to be Bob).

       Regardless of whether Bob ever actually talks to us, the poset
       tells us that individual \#1 {\em could}\, suffer privacy loss,
       and in fact, is uniquely identifiable in the context of
       relation $R$ without needing to reveal everything about
       himself.

    \end{enumerate}

\noindent Similar reasoning is possible for {\bf association
inference}, as we saw earlier.

\vspace*{0.1in}

\begin{figure}[h]
\vspace*{0.25in}
\begin{center}
\qquad\qquad\begin{minipage}{1.5in}{$\begin{array}{c|ccc}
R^\prime & \ta & \tb & \tc \\[2pt]\hline
1 & \one & \one & \\
2 &      & \one & \one \\
3 & \one &      & \one \\
4 &      &      & \one \\
\end{array}$}
\end{minipage}
\quad
\begin{minipage}{3.7in}
\ifig{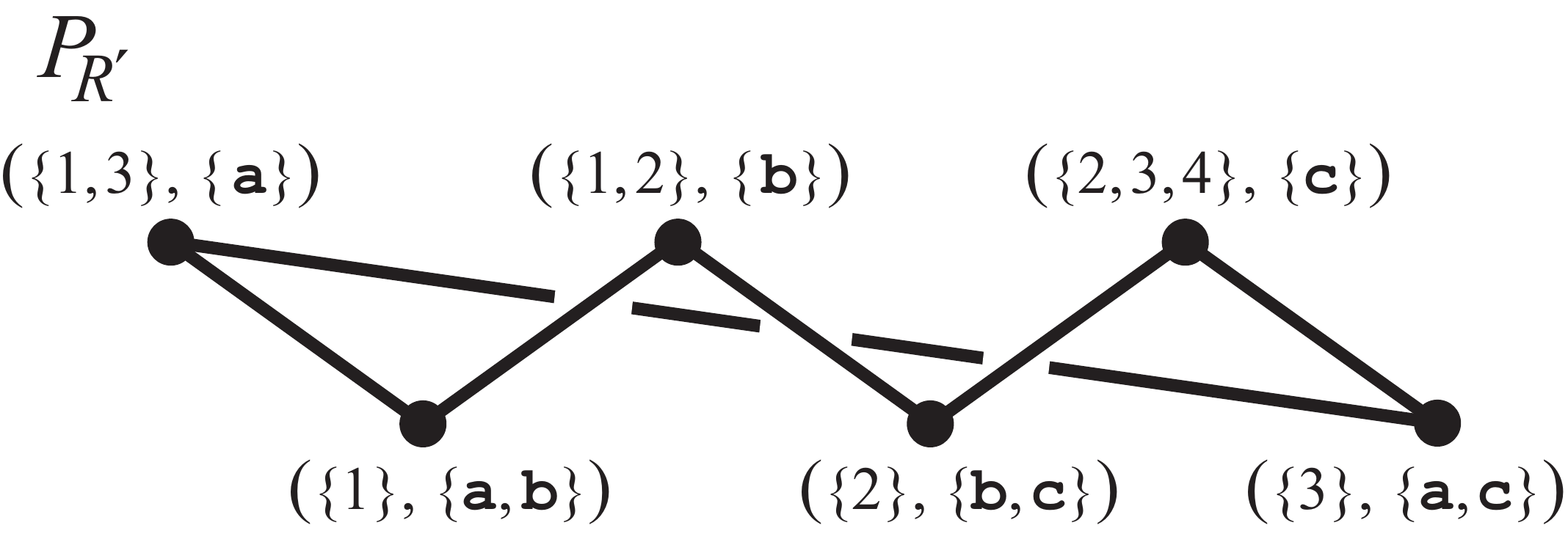}{width=3.5in}
\end{minipage}
\end{center}
\vspace*{-0.1in}
\caption[]{A relation $\Rp$, along with its doubly-labeled poset
  $\PRp$.  The relation preserves attribute privacy but allows a small
  amount of association inference: If ones sees individual \#4 in some
  context $\tc$, then one can infer that individuals \#2 and \#3 are
  also present in that same context, without needing to observe them
  directly.}
\label{IntroExDis}
\end{figure}

\paragraph{Disinformation Revisited:}  Figure~\ref{IntroExDis} shows
relation $\Rp$, constructed from $R$ by adding an entry of
disinformation, much as we constructed $H^\prime$ from $H$ earlier.
The figure also shows the corresponding doubly-labeled poset $\PRp$.
Observe that it is no longer possible to infer $\{\ta, \tb\}$ from
$\{\ta\}$, because $\{\ta\}$ now appears directly in the poset.  The
added entry $(3, \ta)$ in $\Rp$ has increased attribute privacy
compared to $R$.

There is, however, still some opportunity for making association
inferences.  For instance, knowing that individual \#4 (Julie,
earlier) works on an important secret project still allows the
inference that individuals \#2 and \#3 have valuable information.
That is because the minimal set containing $\{4\}$ in the poset is
$\{2,3,4\}$.  Notice that no such association inference is possible if
someone says that individual \#3 works on an important secret project,
though that would have been possible in the original relation $R$.

\paragraph{Comment:}  Artificial entries can potentially also produce
inferences of disinformation.  For instance, if, in our earlier
relation $H$, the entry $(1, \hbox{\sc has\_cancer})$ is artificial,
then inferring that Bob has cancer from his smoking, when in fact Bob
is healthy, would be disinformation.

\clearpage
\section{The Galois Connection for Modeling Privacy}
\markright{Dowker Complexes and the Galois Connection for Modeling Privacy}
\label{galois}

Section~\ref{privacyexamples} showed by example how a relation
determines a partially ordered set (poset) useful for modeling
privacy.  The elements in the poset are ordered pairs --- a set of
attributes and a set of individuals --- that are equivalent from the
relation's perspective.  Privacy loss occurs when an observer has data
(for example, background knowledge) that is not directly in the poset
but is a proper subset of some set of attributes or individuals in the
poset.  The observer may then infer some additional attributes or
individuals.  This section develops the connection between relations
and posets more precisely, continuing to use the earlier examples for
illustration.  See also Appendices~\ref{prelim} and~\ref{basictools}
for notation and additional material.

\subsection{Dowker Complexes}

\begin{definition}[Dowker Complexes]\label{basicdefs}\quad{Let
$\hspt{}X\!$ and $\hspt{}Y\!$ be finite discrete spaces and let $R$ be
a relation on $\XxY$.  This means $R$ is a set of ordered pairs $(x,y)$,
with $x\in X$ and $y\in Y$.  We frequently view/depict $R$ as a matrix
of $0$s and $1$s, or as a matrix of blank and nonblank entries, with
$X\!$ indexing rows and $Y\!$ indexing columns.

\vspace*{-0.05in}

\begin{itemize}
\item[(a)] We often refer to elements of $X\!$ as
\hspt\mydefem{individuals}{\hspt} and to elements of $\,Y\!$ as
\hspt\mydefem{attributes}.

\item[(b)] For each $x\in X$, let $Y_x = \setdef{y\in Y}{(x,y)\in R}$.
Then $Y_x$ consists of all attributes of individual $x$.  We may view
$Y_x$ as a \mydefem{row} of $R$.  \ We say that the \mydefem{row is
blank} if $Y_x=\,\emptyset$.

\item[(c)] For each $y\in Y$, let $X_y = \setdef{x\in X}{(x,y)\in R}$.
Then $X_y$ consists of all individuals who have attribute $y$.  We may
view $X_y$ as a \mydefem{column} of $R$.  \ The 
\mydefem{column is blank} if $X_y=\,\emptyset$.

\item[(d)] We next define two simplicial complexes $\dowy$ and $\dowx$
(with some special cases below):
\begin{eqnarray*}
   \dowy &\;=\;& \setdef{\gamma\subseteq Y}{\hbox{there exists $x\in
   X$ such that $(x,y)\in R\hspt$ for all $\hspt{}y\in\gamma$}},\\[4pt]
   \dowx &\;=\;& \setdef{\sigma\subseteq X}{\hbox{there exists $y\in
   Y$ such that $(x,y)\in R\hspt$ for all $\hspt{}x\in\sigma$}}.
\end{eqnarray*}

Special cases: \ If $X\!=\emptyset\,$ and/or $\,Y\!=\emptyset$, then
we say \,\mydefem{the relation is void}.  In this case, with some
exceptions discussed later (see Section~\ref{linkscond},
Section~\ref{leveraginglattices}, and
Appendix~\ref{linksandinference}), we let $\dowy$ and $\,\dowx$ each
be an instance of the void complex, containing no simplices.
Otherwise, with $X\!$ and $\mskip1.5mu{}Y\mskip-1.5mu$ both nonempty, each
of $\,\dowy$ and $\dowx$ contains at least the empty simplex $\,\emptyset$.

\vspace*{0.1in}

  We refer to $\,\dowy$ and $\dowx$ as \mydefem{Dowker complexes}, after
  the author of upcoming Theorem~\ref{dowker}.\\[2pt]
  We say that each complex is the \mydefem{Dowker dual} of the other,
  with respect to relation $R$.

\vspace*{0.1in}

  Interpretation: A nonempty set $\gamma$ of attributes is a simplex
  in $\dowy$ precisely when at least one individual has at least all
  the attributes in $\gamma$.  We refer to any such individual as a
  \mydefem{witness for $\gamma$}.

\vspace*{0.1in}

  Similarly, a nonempty set $\sigma$ of individuals is a simplex in
  $\dowx$ precisely when there is at least one attribute that is
  shared by at least all the individuals in $\sigma$.  We refer to any
  such attribute as a \mydefem{witness for $\sigma$}.

\end{itemize}
}
\end{definition}

Figure~\ref{IntroDOW} shows the Dowker complexes for the relation $R$
of Section~\ref{privacyposet}.

\begin{figure}[h]
\begin{center}
\begin{minipage}{1.5in}{$\begin{array}{c|ccc}
R & \ta & \tb & \tc \\[2pt]\hline
1 & \one & \one & \\
2 &      & \one & \one \\
3 &      &      & \one \\
4 &      &      & \one \\
\end{array}$}
\end{minipage}
\begin{minipage}{4in}
\ifig{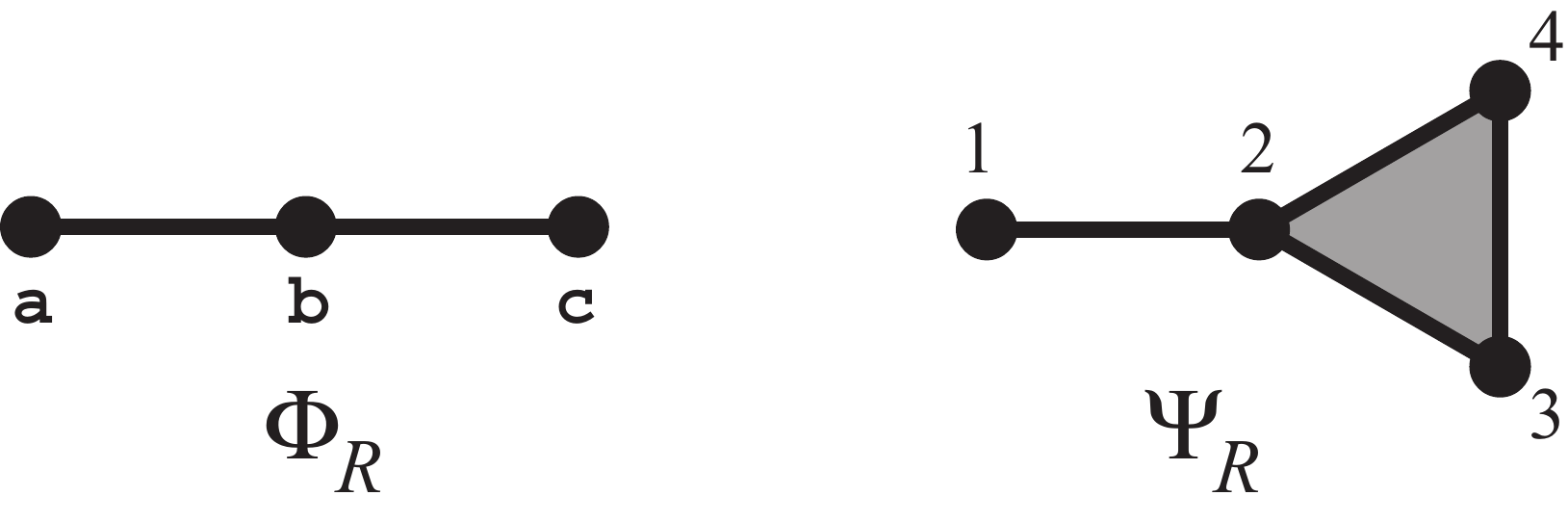}{scale=0.6}
\end{minipage}
\end{center}
\vspace*{-0.2in}
\caption[]{Dowker simplicial complexes $\dowy$ and $\dowx$ determined
           by relation $R$.}
\label{IntroDOW}
\end{figure}

\vst

Dowker's Theorem \cite{tpr:dowker, tpr:bjorner} says that the two
simplicial complexes $\dowy$ and $\dowx$ have the same homotopy type.
As we will see, the maps establishing that homotopy equivalence define
the doubly-labeled poset $\PR$ and describe how privacy may be lost.

\begin{theorem}[Dowker Duality \cite{tpr:dowker}]\label{dowker}
Suppose $R$ is a relation on $\XxY$.  Let $\dowy$ and $\dowx$ be
as in Definition~\ref{basicdefs}.  Then $\dowy$ and $\dowx$ are
homotopy equivalent.
\end{theorem}

Every nonvoid simplicial complex $\Sigma\,$ determines a partially
ordered set $\F(\Sigma)$ called the {\em face poset} of $\Sigma$.  The
elements of this poset are the {\em nonempty\,} simplices of $\Sigma$,
partially ordered by set inclusion.  \label{faceposetDef}
\ (Recall that {\tt{\char13}poset{\char13}} is short for
{\tt{\char13}partially ordered set{\char13}}.)

For the finite setting, the homotopy equivalence of Dowker's Theorem
may be seen by explicit formulas for maps between the face posets of
the two Dowker complexes.  These maps describe what is known as a {\em
Galois connection}.  [This construction also appears as a core tool
within the field of Formal Concept Analysis \cite{tpr:wille,tpr:ganterwille}.]
\ Here are the formulas:\label{phipsiDef}

\vspace*{-0.05in}

\qquad
\begin{minipage}{2.5in}
\begin{eqnarray*}
\phi_R &:& \Fdowx \rightarrow \Fdowy\\[4pt]
&& \sigma \mapsto \biginter_{x\in\sigma} Y_x \\
\end{eqnarray*}
\end{minipage}\label{mapformulae}
\begin{minipage}{2.5in}
\begin{eqnarray*}
\psi_R &:& \Fdowy \rightarrow \Fdowx\\[4pt]
&&\gamma \mapsto \biginter_{y\in\gamma} X_y \\
\end{eqnarray*}
\end{minipage}

\vspace*{0.0375in}

These two maps are inverse homotopy equivalences.  One sees this by
considering the maps $\clsy$ and $\clsx$.  These compositions turn out
to be what are called {\em closure operators} on the face posets
$\Fdowy$ and $\Fdowx$, respectively, implying that each is homotopic
to an identity map, thereby establishing the desired homotopy
equivalence. See Appendix~\ref{basictools} for detailed computations;
see the next subsection for interpretation.

\subsection{Inference from Closure Operators}
\markright{Inference from Closure Operators}

An order-preserving poset map $f : P \rightarrow P$ is said to be a
{\em closure operator} whenever $x \leq f(x)$ and $f(f(x))=f(x)$ for
all $x\in P$.  If $f$ is a closure operator, then it induces a
homotopy equivalence between $P$ and the image $f(P)$.  \ See
\cite{tpr:bjorner, tpr:wachs, tpr:quillenH, tpr:quillenK} for more
details.

One can think of a closure operator as ``pushing elements up'' in the
poset.  From a privacy perspective, ``pushing up'' amounts to
inference.  Specifically, $\sdiff{(\phi_R \circ
\psi_R)(\gamma)}{\gamma}$ consists of all additional attributes that
may be inferred from observing attributes $\gamma$, while
$\sdiff{(\psi_R \circ \phi_R)(\sigma)}{\sigma}$ consists of all
additional individuals that may be inferred from observing individuals
$\sigma$.

\vspace*{-0.15in}

\paragraph{Comment:} The formulas for $\phi_R$ and $\psi_R$ in
Section~\ref{mapformulae} extend to the empty simplex and to the
spaces $X$ and $Y$, suggesting ``inferences from nothing'': \ Observe
that $\psi_R(\emptyset) = X$, so $(\clsy)(\emptyset)$ consists of all
attributes that every individual in $X$ has.  If $(\clsy)(\emptyset)
\neq \emptyset$, then the attributes $(\phi_R \circ
\psi_R)(\emptyset)$ are inferable ``for free'' from $R$, that is,
without making any observations.  Similarly, $(\clsx)(\emptyset)$
consists of all individuals who have every attribute in $Y$.

Any poset $P$ defines a simplicial complex $\Delta(P)$ called the {\em
\hspc{}order complex\,} of $P$.  The simplices of $\Delta(P)$ are
given by the finite chains $\{p_ 0 < p_1 < \cdots < p_n\}$ in $P$.
Suppose we start with a nonvoid simplicial complex $\Sigma$, construct
its face poset $\F(\Sigma)$, and then construct the order complex
\label{ordercpxDef} $\Delta(\F(\Sigma))$.  The result is isomorphic to
the {\em first barycentric subdivision} of $\Sigma$ \cite{tpr:rotman,
tpr:wachs}.  A convenient visualization of the face posets $\Fdowy$
and $\Fdowx$ therefore is to draw the first barycentric subdivisions
of $\dowy$ and $\dowx$, respectively, as in Figure~\ref{IntroSD}.

\begin{figure}[h]
\vspace*{-0.1in}
\begin{center}
\ifig{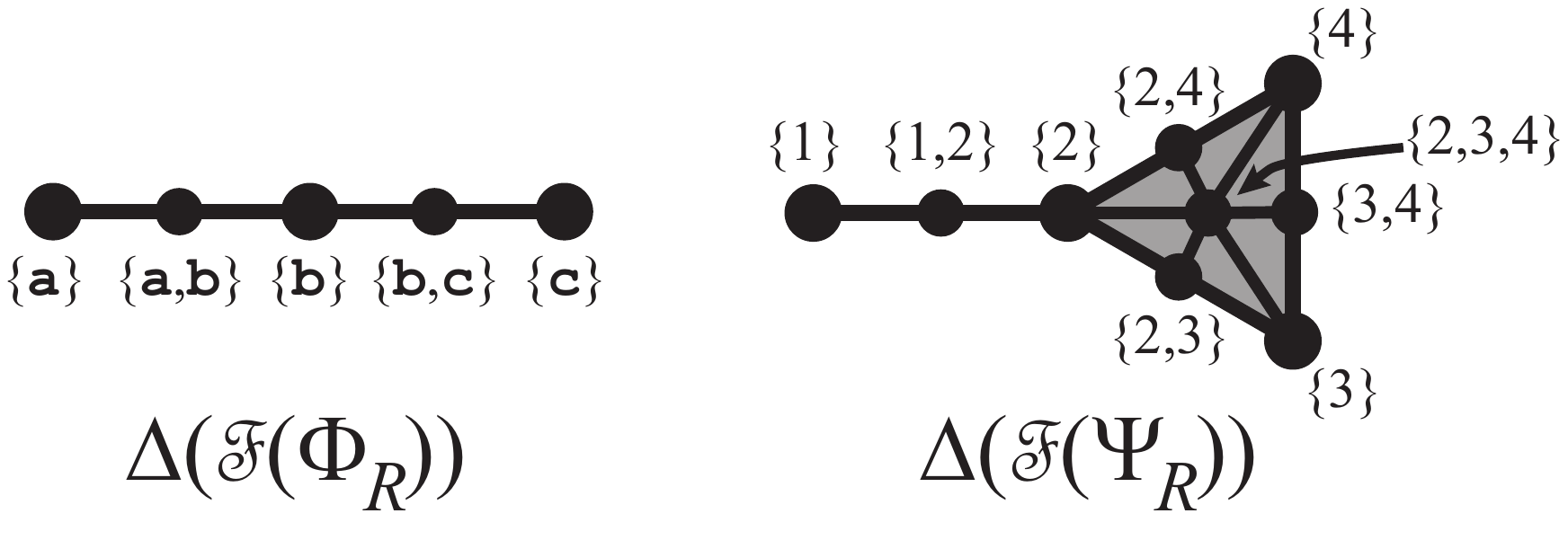}{scale=0.75}
\end{center}
\vspace*{-0.3in}
\caption[]{Order complexes of the face posets of the complexes $\dowy$
  and $\dowx$ shown in Figure~\ref{IntroDOW}.}
\label{IntroSD}
\end{figure}

Viewed in the order complexes, functions $\psi_R$ and $\phi_R$ are
easy to visualize.  They are fully determined by their actions on
vertices of the order complexes, as shown in Table \ref{IntroMapsR}.
(Bear in mind that each element of $\Fdowy$ represents a simplex in
$\dowy$ but is a vertex in $\Delta(\Fdowy)$.  Similarly, each element
of $\Fdowx$ represents a simplex in $\dowx$ but is a vertex in
$\Delta(\Fdowx)$.)

\begin{table}[h]
\begin{center}
\[\begin{array}{ccccc}
\gamma & \;\; & \psi_R(\gamma) & \;\; & (\clsy)(\gamma)\\[1.5pt]\hline
\{\ta\} & & \{1\} & & \{\ta,\tb\} \\[1.5pt]
\{\tb\} & & \{1,2\} & & \{\tb\} \\[1.5pt]
\{\tc\} & & \{2,3,4\} & & \{\tc\} \\[1.5pt]
\{\ta,\tb\} & & \{1\} & & \{\ta,\tb\} \\[1.5pt]
\{\tb,\tc\} & & \{2\} & & \{\tb,\tc\} \\[1pt]
\end{array}
\hspace*{0.6in}
\begin{array}{ccccc}
\sigma & \;\; & \phi_R(\sigma) & \;\; & (\clsx)(\sigma)\\[1.5pt]\hline
\{1\} & & \{\ta,\tb\} & & \{1\} \\[1.5pt]
\{2\} & & \{\tb,\tc\} & & \{2\} \\[1.5pt]
\{3\} & & \{\tc\} & & \{2,3,4\} \\[1.5pt]
\{4\} & & \{\tc\} & & \{2,3,4\} \\[1.5pt]
\{1,2\} & & \{\tb\} & & \{1,2\} \\[1.5pt]
\{2,3\} & & \{\tc\} & & \{2,3,4\} \\[1.5pt]
\{3,4\} & & \{\tc\} & & \{2,3,4\} \\[1.5pt]
\{2,4\} & & \{\tc\} & & \{2,3,4\} \\[1.5pt]
\{2,3,4\} & & \{\tc\} & & \{2,3,4\} \\[1pt]
\end{array}\]
\end{center}
\vspace*{-0.15in}
\caption{The maps $\psi_R$ and $\phi_R$, and their compositions, for
 relation $R$ of Figure~\ref{IntroDOW}.}
\label{IntroMapsR}
\end{table}

\vspace*{0.05in}

\noindent Using Table~\ref{IntroMapsR}, one can again see how privacy
loss might occur via $R$.

\vst

\noindent For instance, the map $\clsy$ gives rise to the closure
(i.e., a ``pushing up'')

\vspace*{-0.075in}

$$\{\ta\} \argmap{\psi_R} \{1\} \argmap{\phi_R} \{\ta,\tb\},$$

\noindent telling us how to infer unobserved attribute \tb\ from
observed attribute \ta\ (in the health study example of
Section~\ref{ToyExample}, Alice could infer that Bob {\sc has\_cancer}
from knowing that he {\sc smokes}).

\noindent Similarly, for the map $\clsx$,

\vspace*{-0.075in}

$$\{4\} \argmap{\phi_R} \{\tc\} \argmap{\psi_R} \{2,3,4\},$$

\noindent leading to association inference (in the payroll example
from Section~\ref{payroll}, the payroll clerk could infer Bob and
Mary's exposure to valuable information after learning of Julie's work
on an important project).

\vspace*{0.1in}

\noindent Figure~\ref{IntroHomotopy} indicates the homotopy
deformations produced by the maps $\clsy$ and $\clsx$, while
Figure~\ref{IntroClosure} shows the resulting image of each face poset.

\begin{figure}[h]
\vspace*{0.1in}
\begin{center}
\ifig{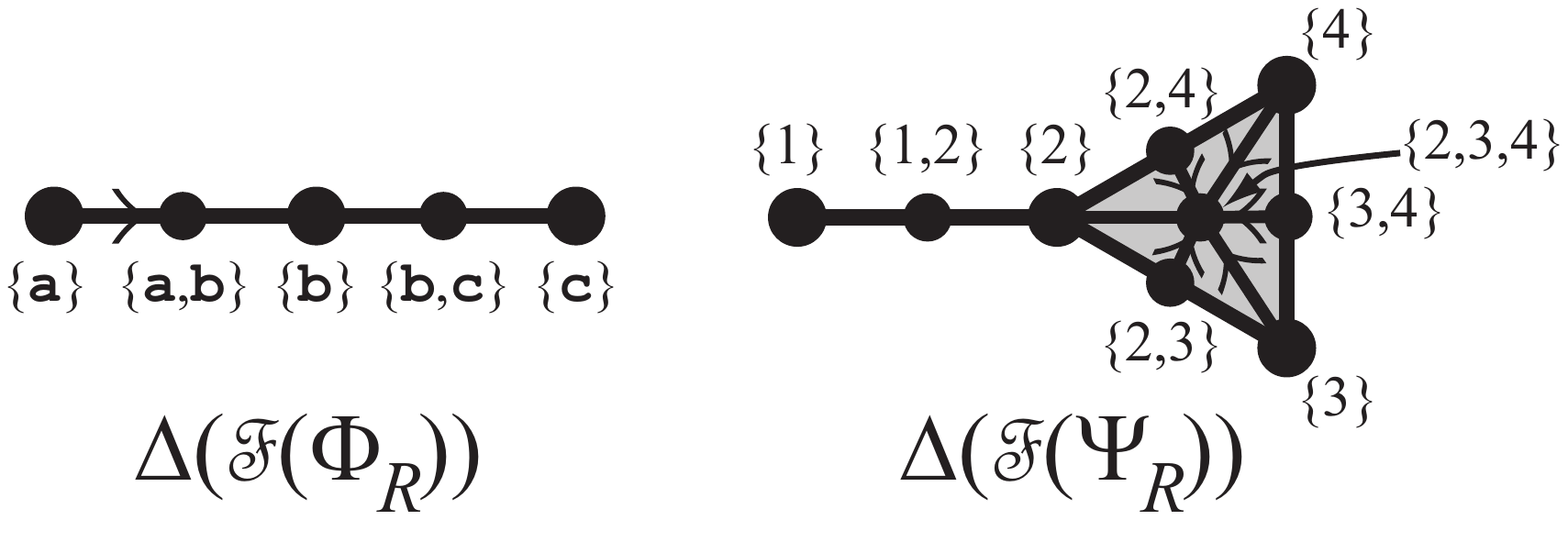}{scale=0.75}
\end{center}
\vspace*{-0.3in}
\caption[]{Closure operators $\clsy$ and $\clsx$ produce homotopy
  deformations, indicated by directed edges.  In $\Fdowy$, $\{\ta\}$
  closes up to $\{\ta,\tb\}$.  In $\Fdowx$, most of the subsets of
  $\{2,3,4\}$ close up to $\{2,3,4\}$.  The exception is subset
  $\{2\}$, which does not move.}
\label{IntroHomotopy}
\end{figure}

\begin{figure}[h]
\vspace*{0.2in}
\begin{center}
\ifig{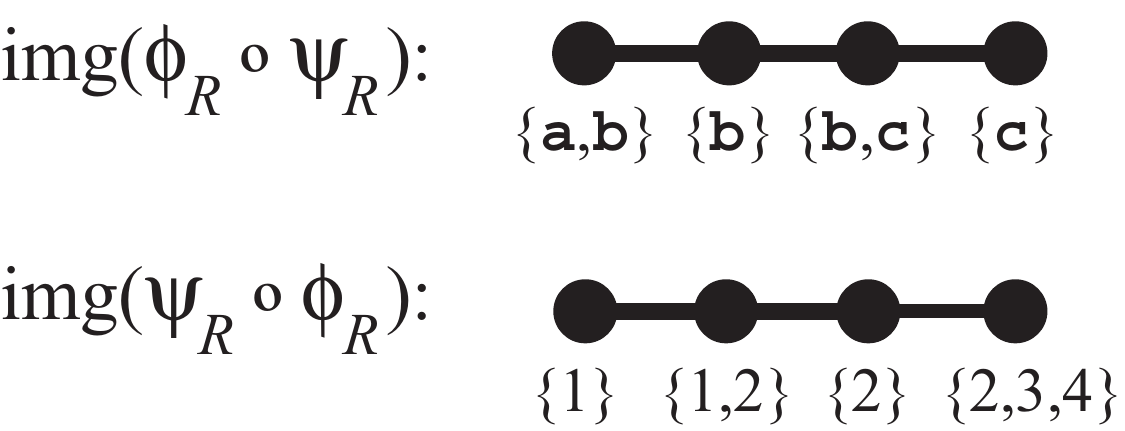}{scale=0.65}
\end{center}
\vspace*{-0.25in}
\caption[]{Images of the closure operators of
  Figure~\ref{IntroHomotopy}.}
\label{IntroClosure}
\end{figure}

Observe that these two images are isomorphic.  Matching up
corresponding elements produces the poset $\PR$ of Figure~\ref{IntroPR}.

\paragraph{Summary:} A relation $R$ produces two simplicial complexes,
$\dowy$ and $\dowx$, one modeling attributes shared by individuals,
the other modeling individuals with common attributes.  The complexes
are related by two maps, $\phi_R$ and $\psi_R$, that are homotopy
inverses.  The compositions of these maps describe the attribute and
association inferences possible via $R$, leveraging background
information someone may have.  These inferences are summarized by a
poset $\PR$ that pairs sets of individuals with sets of attributes.
We may describe $\PR$ as follows:

\vspace*{0.1in}

\begin{definition}[Doubly-Labeled Poset]
\ Let $R$ be a relation with nonvoid Dowker complexes.

\vst

The \mydefem{doubly-labeled poset $\PR$ associated with $R$} consists
of all ordered pairs of sets $(\sigma, \gamma)$ such that
$\emptyset \neq \sigma \in \dowx$, $\,\emptyset \neq \gamma \in \dowy$,
$\,\sigma = \psi_R(\gamma)$, and $\gamma = \phi_R(\sigma)$. $\phantom{\big|}$

\vsr \label{defPR}

\vst

The partial order on $\PR$ is defined by: $(\sigma_1, \gamma_1) \leq
(\sigma_2, \gamma_2)$ if and only if $\hspt\sigma_1 \subseteq \sigma_2$

(and/or, equivalently, $\gamma_1 \supseteq \gamma_2$).

\vspace*{0.1in}

See Appendix~\ref{PRspecialcases}, specifically
page~\pageref{PRspecialcases}, for some special cases.

\end{definition}

\vspace*{0.05in}

(This definition agrees with the intuition that $\PR$ is both the
image $(\clsx)(\Fdowx)$ and the image $(\clsy)(\Fdowy)$, by
Appendix~\ref{basictools}.)

\vspace*{0.2in}
\subsection{Attribute and Association Privacy}
\markright{Doubly-Labeled Poset and Definitions of Attribute and Association Privacy}
\label{AttribAssocPriv}

Here are formal definitions for the intuition developed via the previous examples:

\begin{definition}[Attribute Privacy]\label{defAttribPriv}
\ Let $R$ be a relation with nonvoid Dowker complexes.

\vso

We say that $R$ \mydefem{preserves attribute privacy} precisely when\\[3pt]
\hspace*{0.25in}$\phi_R \circ \psi_R$ is the identity operator on the poset \ $\Fdowy \union \{\emptyset\}$.
\end{definition}

\vspace*{0.1in}

\begin{definition}[Association Privacy]
\ Let $R$ be a relation with nonvoid Dowker complexes.

\vso

We say that $R$ \mydefem{preserves association privacy} precisely when\\[3pt]
\hspace*{0.25in}$\psi_R \circ \phi_R$ is the identity operator on the poset \ $\Fdowx \union \{\emptyset\}$.
\end{definition}

\paragraph{Comment:} \ For notational simplicity, we frequently say
simply that

\noindent\hspace{1in}\mydefem{$\clsy$ is the identity on $\dowy$}
\ and/or that \ 
\mydefem{$\clsx$ is the identity on $\dowx$}.

\vspace*{0.25in}
\subsection{Disinformation Example Re-Revisited}

\begin{figure}[h]
\begin{center}
\ifig{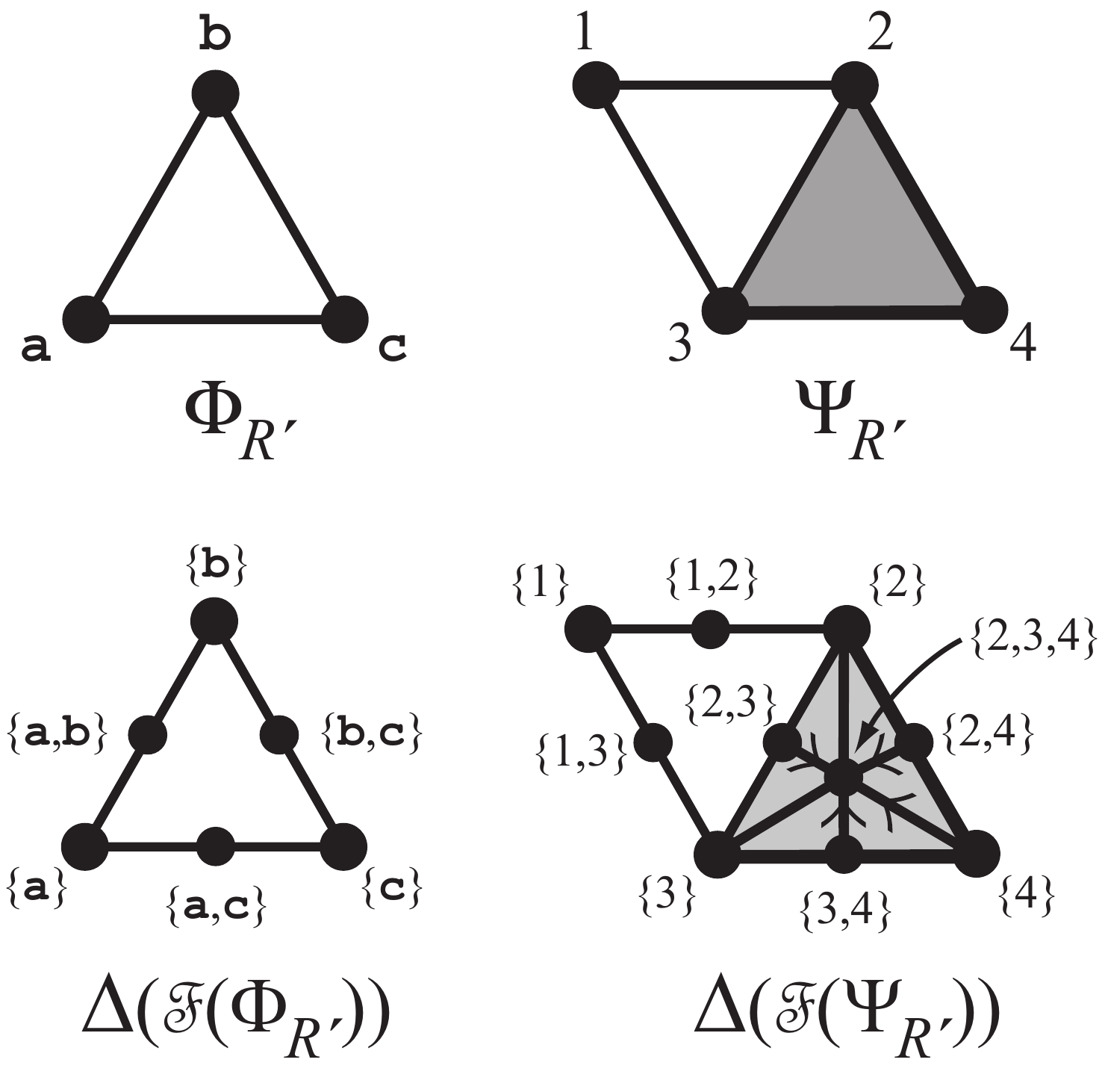}{scale=0.6}
\end{center}
\vspace*{-0.25in}
\caption[]{The Dowker complexes, as well as the order complexes of
  their face posets, for the relation $\Rp$ of Figure~\ref{IntroExDis}
  on page~\pageref{IntroExDis}.  The closure operator $\clsyp$ is the
  identity on $\Fdowyp$.  The closure operator $\clsxp$ on $\Fdowxp$
  closes many (but not all) subsets of $\{2,3,4\}$ up to $\{2,3,4\}$,
  as indicated by the directed arrows.  The result is a poset
  isomorphic to the poset $\PRp$ of Figure~\ref{IntroExDis}, drawn
  again slightly differently in Figure~\ref{PRptriangle}.  Also,
  $(\clsyp)(\emptyset) = \emptyset$.  Thus relation $\Rp$ preserves
  attribute privacy but not association privacy.}
\label{IntroExDisClosure}
\end{figure}

\begin{figure}[h]
\begin{center}
\ifig{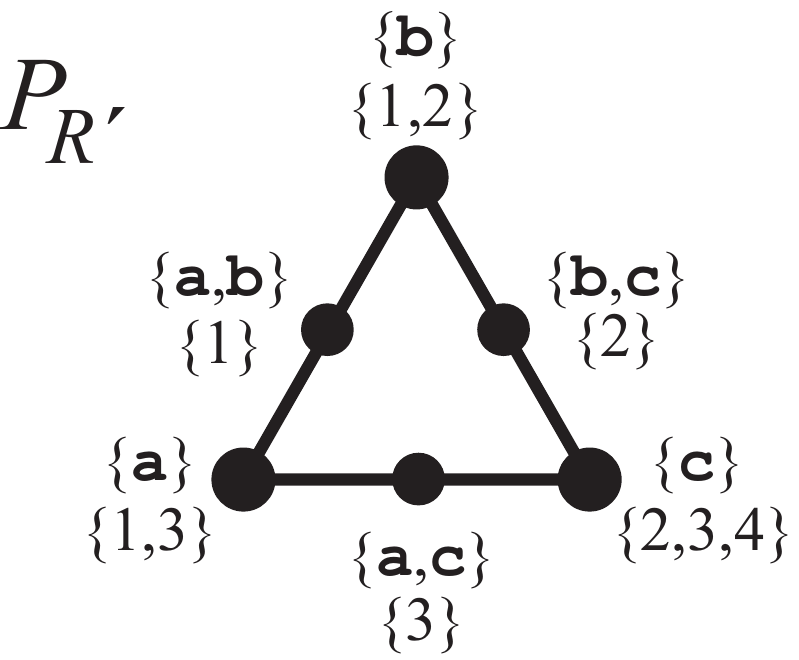}{scale=0.75}
\end{center}
\vspace*{-0.25in}
\caption[]{A flattened view of the doubly-labeled poset $\PRp$ from
  Figure~\ref{IntroExDis}.  Combined with
  Figure~\ref{IntroExDisClosure}, this perspective shows how $\PRp$
  arises as the images of $\Fdowyp$ and $\Fdowxp$ under the closure
  operators $\clsyp$ and $\clsxp$, respectively.  (The vertices drawn
  as bigger dots in the current figure were higher up in the poset of
  Figure~\ref{IntroExDis} than those drawn as smaller dots.)}
\label{PRptriangle}
\end{figure}

Recall the relation $R^\prime$ of Figure~\ref{IntroExDis} on
page~\pageref{IntroExDis}, which is relation $R$ of
Figure~\ref{IntroPR} but with an added entry of disinformation.
Figure~\ref{IntroExDisClosure} displays the resulting Dowker complexes
and the actions of the closure operators.  Figure~\ref{PRptriangle}
flattens out the poset $\PRp$ of Figure~\ref{IntroExDis}, so one sees
its triangle structure and how it is the image of the Dowker complexes
under the closure operators for $\Rp$.

\clearpage
\section{The Face Shape of Privacy}
\label{faceshape}

\begin{figure}[h]
\begin{center}
\begin{minipage}{1.5in}{$\begin{array}{c|ccc}
R & \ta & \tb & \tc \\[2pt]\hline
1 & \one & \one & \\
2 &      & \one & \one \\
3 &      &      & \one \\
4 &      &      & \one \\
\end{array}$}
\end{minipage}
\begin{minipage}{2in}
\ifig{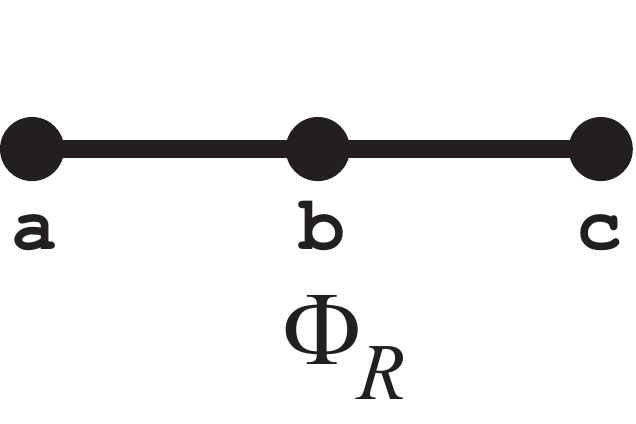}{scale=0.6}
\end{minipage}

\vspace*{0.2in}

\begin{minipage}{1.5in}{$\begin{array}{c|ccc}
R^\prime & \ta & \tb & \tc \\[2pt]\hline
1 & \one & \one & \\
2 &      & \one & \one \\
3 & \one &      & \one \\
4 &      &      & \one \\
\end{array}$}
\end{minipage}
\begin{minipage}{2in}
\quad\ifig{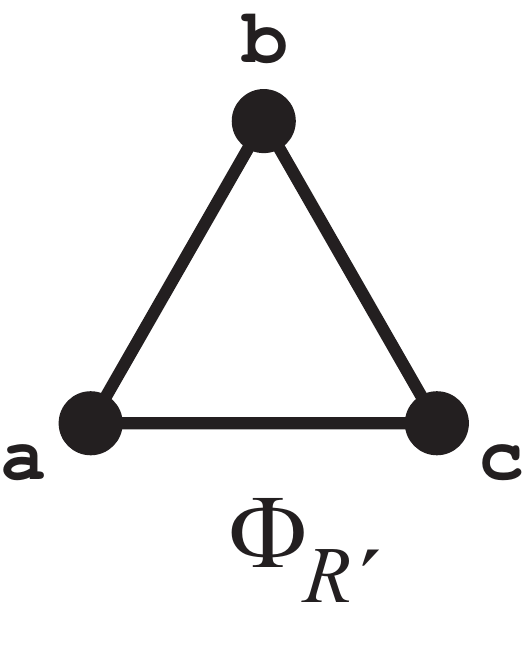}{scale=0.6}
\end{minipage}
\end{center}
\vspace*{-0.2in}
\caption[]{Relations $R$ and $\Rp$ of Section~\ref{privacyexamples},
 along with their attribute complexes $\dowy$ and $\dowyp$.}
\label{freefacefig}
\end{figure}

\subsection{Free Faces}
\markright{Free Faces and Privacy}
\label{freefaces}

Figure~\ref{freefacefig} recapitulates relation $R$ and $\Rp$ from the
previous two sections, along with their Dowker attribute complexes,
$\dowy$ and $\dowyp$, respectively.  Recall that in $R$ one could make
the inference $\ta \Rightarrow \tb$, but no such inference was
possible in $R^\prime$.

The structure of $\dowy$ suggests that the inference $\,\ta \Rightarrow
\tb\,$ {\em might}\, be possible in $R$.  In contrast, the structure of
$\dowyp$ makes clear that such an inference is {\em impossible} in
$\Rp$.  In particular, observe how vertex \ta\ has only one incident
edge in $\dowy$ but has two incident edges in $\dowyp$.  The fact that
there are two edges in $\dowyp$, with those edges being maximal
simplices, means, intuitively, that vertex \ta\ is being ``pulled'' in
two different inference directions, so one cannot conclude anything
additional from attribute \ta.  In contrast, in $\dowy$, vertex \ta\
is being ``pulled'' only toward \tb, so it is plausible that attribute
\ta\ might imply attribute \tb.

The underlying geometry is that of a free face.  A simplex $\gamma$ of
a simplicial complex $\Gamma$ is said to be a {\em free face} of
$\hspt\Gamma\hspt$ if it is a proper subset of exactly one maximal
simplex of $\Gamma$.  That is true for $\{\ta\}$ in $\dowy$ but not
for $\{\ta\}$ in $\dowyp$.

Of course, vertex $\{\tc\}$ also forms a free face in $\dowy$, yet one
cannot make any inferences upon observing just attribute \tc.  What is
going on?  The difference is that \tc\ is also an attribute of
individuals in $R\hspt$ who have {\em only\,} \tc\ as an attribute
(specifically, individuals \#3 and \#4).  Even though $\{\tc\}$ is
technically a free face of $\dowy$, it is not really free to move
under the closure operator $\clsy$, whereas $\{\ta\}$ is.

Observe that individuals \#2, \#3, and \#4 all have attribute \tc, but
only individual \#2 has additional attributes.  This means that
individuals \#3 and \#4 cannot ever be identified uniquely in the
context of relation $R$; they have effectively ``camouflaged''
themselves with individual \#2, as far as relation $R\hspt$ is concerned.
If one disallows or disregards such camouflage, then the idea of a
free face and privacy loss are equivalent.  The following definition
is useful:

\newcounter{DefUniqueID}
\setcounter{DefUniqueID}{\value{theorem}}
\begin{definition}[Unique Identifiability]
\ Let $R$ be a relation on $\XxY$ and suppose $x \in X$.\\
We say that $x$ is \,\mydefem{uniquely identifiable via relation $R$}\, when
$\,\psi_R(Y_x) = \{x\}$. \label{uniqueID}
\end{definition}

Suppose $R\hspt$ is a relation.  Appendix~\ref{uniqueident} proves
that if $\dowy$ has no free faces, then $R\hspc$ preserves attribute
privacy.  For the converse, Appendix~\ref{uniqueident} further proves
that if $R$ preserves attribute privacy \!{\em and} if every
individual is uniquely identifiable, then $\dowy$ has no free faces.
(Dual statements hold for association privacy.)

\subsection{Privacy versus Identifiability}
\label{conformism}

Section~\ref{freefaces} hinted at the difference between privacy and
identifiability.  In relation $I$ below (``I'' for ``individuality''
or ``identity''), every individual has exactly one attribute and that
attribute uniquely identifies the individual.  Relation $I$ {\em
preserves privacy} fully (assuming $n > 1$).  It is impossible to make
any attribute inferences.  If Bob reveals that he has attribute $y_7$,
then Alice cannot infer any additional attributes for Bob.  She now
knows that Bob is individual $x_7$ but cannot infer any additional
attributes.  He has himself revealed everything about himself that
there is to know, as far as relation $I$ is concerned.

$$\begin{array}{c|cccccc}
I       & y_1  & y_2  & \cdots & y_n  \\[2pt]\hline
x_1     & \one &      &        &      \\
x_2     &      & \one &        &      \\
\vdots  &      &      & \ddots &      \\
x_n     &      &      &        & \one \\
\end{array}$$

In contrast, all individuals in relation $C$ (for ``conformism'' or
``confusion'') have exactly the same set of attributes.  As a result,
there is {\em no privacy}: one can predict all the attributes of any
individual in the relation without making any observations.  On the
other hand, no individual is uniquely identifiable (assuming $n > 1$).
\label{relationC}

$$\begin{array}{c|cccccc}
C       & y_1    & y_2    & \cdots & y_n    \\[2pt]\hline
x_1     & \one   & \one   & \cdots & \one   \\
x_2     & \one   & \one   & \cdots & \one   \\
\vdots  & \vdots & \vdots & \ddots & \vdots \\
x_n     & \one   & \one   & \cdots & \one   \\
\end{array}$$

\paragraph{Homogeneity:} Relation $C$ exhibits a form of homogeneity
often sought by anonymization or other privacy techniques.  As we have
suggested before, the utility of relation $C$ is essentially zero,
unless one makes the entries stochastic, so that some utility is
encoded in the distribution.

\vspace*{0.05in}

The discussion of free faces in Section~\ref{freefaces} suggests an
alternative approach to homogeneity: one may preserve privacy and
retain utility by choosing the geometry of the relation appropriately,
for instance, so the space $\dowy$ exhibits sphere-like homogeneity.
There will be considerable discussion of the importance of spheres in
the rest of the report.

\subsection{Spheres and Privacy}
\markright{Spheres and Privacy}
\label{homogeneity}

The attribute complex $\dowyp$ of Figure~\ref{freefacefig} is equal to
a {\em boundary complex}, namely the boundary of the full simplex
consisting of the attributes $\{\ta, \tb, \tc\}$.  We will denote
boundary complexes by $\bndry{(V)}$, with $V$ some nonempty set.  The
simplices of $\bndry{(V)}$ are all proper subsets of $V$.  Boundary
complexes are homotopic to spheres, specifically $\bndry{(V)} \homot
\Snt$, with $n=\abs{V}$.  For $\dowyp$ of Figure~\ref{freefacefig}, we
have that $\dowyp = \bndry{(\{\ta, \tb, \tc\})} \homot \,\Sone$.  (In
English: The Dowker attribute complex is the boundary of a triangle,
so homotopic to a circle.)
\label{bndrycpxDef}

More generally, if for some relation $R$ on $\XxY$, $\dowy =
\bndry{(Y)}$, then $\dowy$ cannot have any free faces and so $R$
preserves attribute privacy.

\vspace*{-0.15in}

\paragraph{Privacy and Utility:}\ An important observation is that
boundary complexes exhibit homogeneity but still permit
identifiability.  If $\dowy = \bndry{(Y)}$, with $\abs{Y} > 1$, and if
no individual's attributes are a subset of another's attributes, then
one can and needs to specify $\abs{Y}-1$ attributes in order to
identify an individual.  The boundary structure ensures that one
cannot infer any attributes by specifying fewer than $\abs{Y}-1$
attributes, yet retains the ability to identify every individual.

Appendix~\ref{duncehat} gives an example of a contractible space that
preserves attribute privacy.  Observe, however, that the number of
attributes needed to identify an individual in that example is
considerably less than the total number of attributes in the space.
For a boundary complex, it is just one less.

\paragraph{Preserving Attribute and Association Privacy:}  A
consequence of these observations is that if one wishes to preserve
both attribute and association privacy with a connected relation, then
one requires both Dowker complexes to look like spheres.  More
specifically, either both Dowker complexes are linear cycles of the
same length or both are boundary complexes of the same dimension.  In
the latter case, the relation is isomorphic to a relation of the
following form, in which the diagonal $\{(x_i, y_i)\}$ is blank but
all other entries are present:

$$\begin{array}{c|cccccc}
R       & y_1    & y_2    & \cdots & \cdots & y_{n-1} & y_n    \\[2pt]\hline
x_1     &        &  \one  &  \one  & \cdots &  \one  & \one   \\
x_2     & \one   &        &  \one  & \cdots &  \one  & \one   \\
\vdots  & \one   &  \one  &        & \ddots  & \vdots & \one   \\
\vdots  & \vdots & \vdots & \ddots &        &  \one  & \vdots \\
x_{n-1} & \one   &  \one  & \cdots &  \one  &        & \one   \\
x_n     & \one   &  \one  &  \one  & \cdots &  \one  &        \\
\end{array}$$

See Appendix~\ref{preservingboth}, starting on
page~\pageref{preservingboth}, for further details.

\subsection{A Spherical Non-Boundary Relation that Preserves Attribute Privacy}

\begin{figure}[h]
\begin{center}
\hspace*{0.2in}\begin{minipage}{1.5in}{$\begin{array}{c|ccccc}
R & \ta & \tb & \tc & \td & \te \\[2pt]\hline
1 & \one & \one &      & \one &      \\
2 & \one &      & \one & \one &      \\
3 &      & \one & \one & \one &      \\
4 & \one & \one &      &      & \one \\
5 & \one &      & \one &      & \one \\
6 &      & \one & \one &      & \one \\
\end{array}$}
\end{minipage}
\hspace*{0.2in}
\begin{minipage}{4in}
\ifig{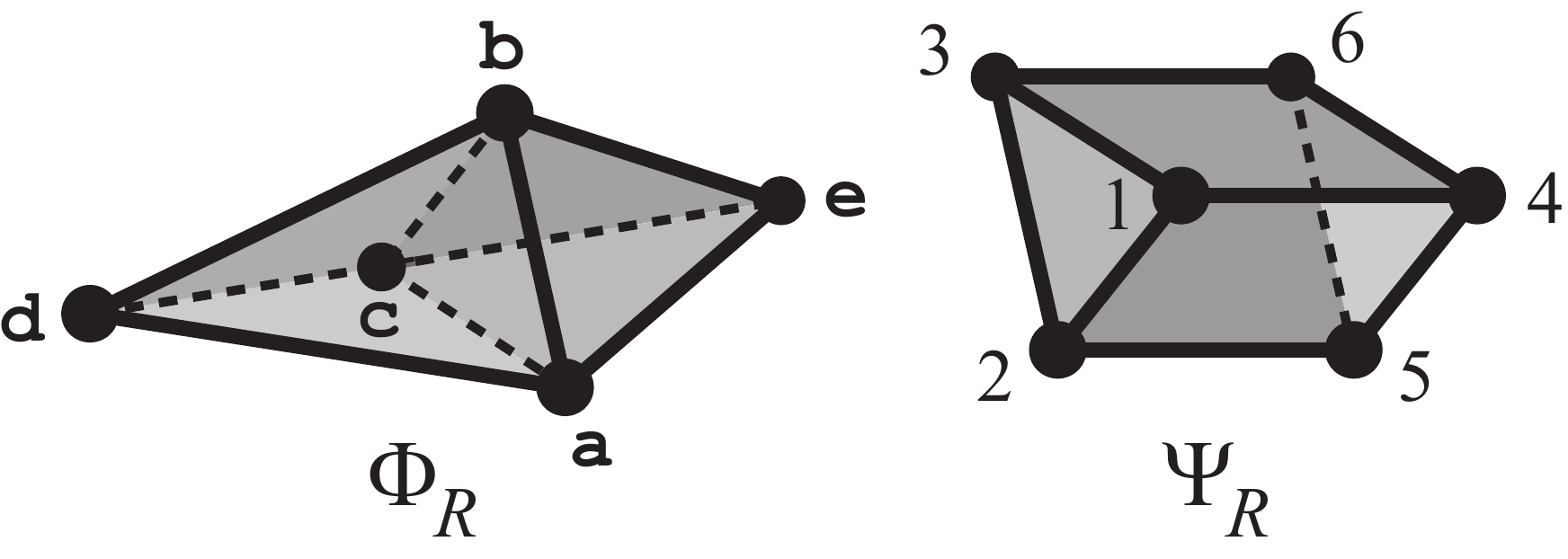}{scale=0.575}
\end{minipage}
\end{center}
\vspace*{-0.2in}
\caption[]{A relation $R$ and its Dowker complexes $\dowy$ and
  $\dowx$, each homotopic to the two-dimensional sphere $\Stwo$.
  (One may view $\dowy$ as two party hats glued together.  One may
  view $\dowx$ as a cylinder with a triangular cross-section and
  endcaps.  However, the quadrilaterals drawn for the cylinder
  portion of $\dowx$ are simply flattened sketches of what are
  actually solid tetrahedra.)}
\label{Ex20}
\end{figure}

Consider relation $R$ as in Figure~\ref{Ex20}.
Relation $R$ preserves attribute privacy, since $\dowy$ has no free
faces.  The relation does not preserve association privacy.  In
particular, the quadrilaterals drawn for $\dowx$ in the figure are
actually tetrahedra.  This means that the diagonals of the
quadrilaterals are free faces.  For instance, one would expect to infer
individuals \#1 and \#6 as additional unobserved associates if one
observes individuals \#3 and \#4.  Indeed, computing using the closure
operator $\clsx$, we see that:

\vspace*{-0.1in}

$$(\clsx)(\{3, 4\}) \;=\; \psi_R(\{\tb\}) \;=\; \{1,3,4,6\}.$$

\vspace*{0.05in}

Relation $R$ has another interesting feature.  Even though $\dowy$ is
not itself a boundary complex, it is the {\em simplicial join} (see
page~\pageref{joinAppdef}) of two boundary complexes:

\vspace*{-0.05in}

$$\dowy \;=\; \bndry{(\{\ta, \tb, \tc\})} * \bndry{(\{\td, \te\})}.$$

\vspace*{0.05in}

In fact, we can think of $R$ as $R_1 \union R_2$ and $\dowy$ as
$\dowyone \mskip-1.5mu * \mskip1.5mu \dowytwo$, with $R_1$ the
restriction of $R$ to the attributes $\{\ta, \tb, \tc\}$ and $R_2$ the
restriction of $R$ to the attributes $\{\td, \te\}$.  This join
structure of $\dowy$ means that we can view every individual in $R$ as
being described by two {\em independent} attribute spaces.  The
attribute space $\{\td, \te\}$ acts like a standard bit; every
individual has exactly one of these two attributes.  In contrast, the
attribute space $\{\ta, \tb, \tc\}$ is an ``any 2 of 3'' type of
descriptor.  Every individual has exactly two of these three
attributes.

\vso

Figure~\ref{Ex20decomp} shows the relations $R_1$ and $R_2$ along with
their Dowker attribute complexes.

\begin{figure}[h]
\vspace*{0.05in}
\begin{center}
\hspace*{0.4in}\begin{minipage}{1in}{$\begin{array}{c|ccc}
R_1 & \ta & \tb & \tc \\[2pt]\hline
1 & \one & \one &      \\
2 & \one &      & \one \\
3 &      & \one & \one \\
4 & \one & \one &      \\
5 & \one &      & \one \\
6 &      & \one & \one \\
\end{array}$}
\end{minipage}
\hspace*{0.3in}
\begin{minipage}{1in}{$\begin{array}{c|cc}
R_2 & \td & \te \\[2pt]\hline
1 & \one &      \\
2 & \one &      \\
3 & \one &      \\
4 &      & \one \\
5 &      & \one \\
6 &      & \one \\
\end{array}$}
\end{minipage}
\hspace*{0.5in}
\begin{minipage}{2.5in}
\ifig{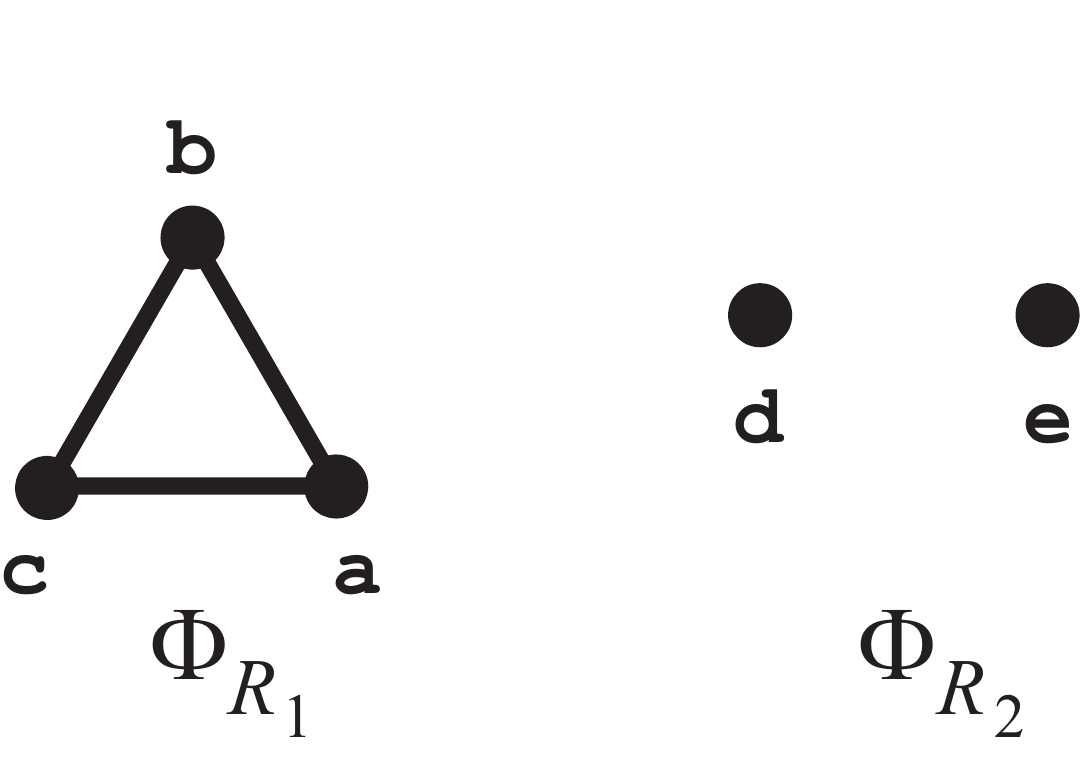}{scale=0.5}
\end{minipage}
\end{center}
\vspace*{-0.1in}
\caption[]{Relation $R$ of Figure~\ref{Ex20} decomposes into two
  disjoint relations $R_1$ and $R_2$ such that
  $\dowy = \dowyone \mskip-1.5mu * \mskip1.5mu \dowytwo$,
  with $\dowyone$ the boundary complex of a triangle and $\dowytwo$
  two isolated points.  This means every individual in $R$ has
  attributes that act like two independent coordinates: an ``any 2 of
  3'' component and a bit.}
\label{Ex20decomp}
\end{figure}

\clearpage
\section{Conditional Relations as Simplicial Links}
\markright{Conditional Relations as Simplicial Links}
\label{linkscond}

The decomposition of Figures~\ref{Ex20} and \ref{Ex20decomp} is
reminiscent of stochastic independence expressed as multiplication of
probabilities.  Similarly, there is a combinatorial analogue to the
notion of a {\em conditional probability distribution}.  It appears
as the {\em link\,} of a simplex in a simplicial complex.

Given a relation $R$, suppose we have observed attributes $\gamma$ for
some unknown individual.  The remaining possible combinations of
attributes we might yet observe are described by the simplicial
complex
$\lk(\dowy, \gamma) =
\setdef{\tau\in\dowy}{\tau\inter\gamma=\emptyset\;\hbox{and}\;\tau\union\gamma\in\dowy}$.
Interpretation: $\tau\inter\gamma=\emptyset$ means that $\tau$
consists of as yet unobserved attributes, while
$\tau\union\gamma\in\dowy$ means that there is some individual who has
the attributes $\tau$ in addition to the attributes $\gamma$ that we
have already observed.

\begin{figure}[h]
\begin{center}
\hspace*{0.4in}\begin{minipage}{1in}{$\begin{array}{c|ccc}
Q & \ta & \tb & \tc \\[2pt]\hline
1 & \one & \one &      \\
2 & \one &      & \one \\
3 &      & \one & \one \\
\end{array}$}
\end{minipage}
\hspace*{0.5in}
\begin{minipage}{2.5in}
\ifig{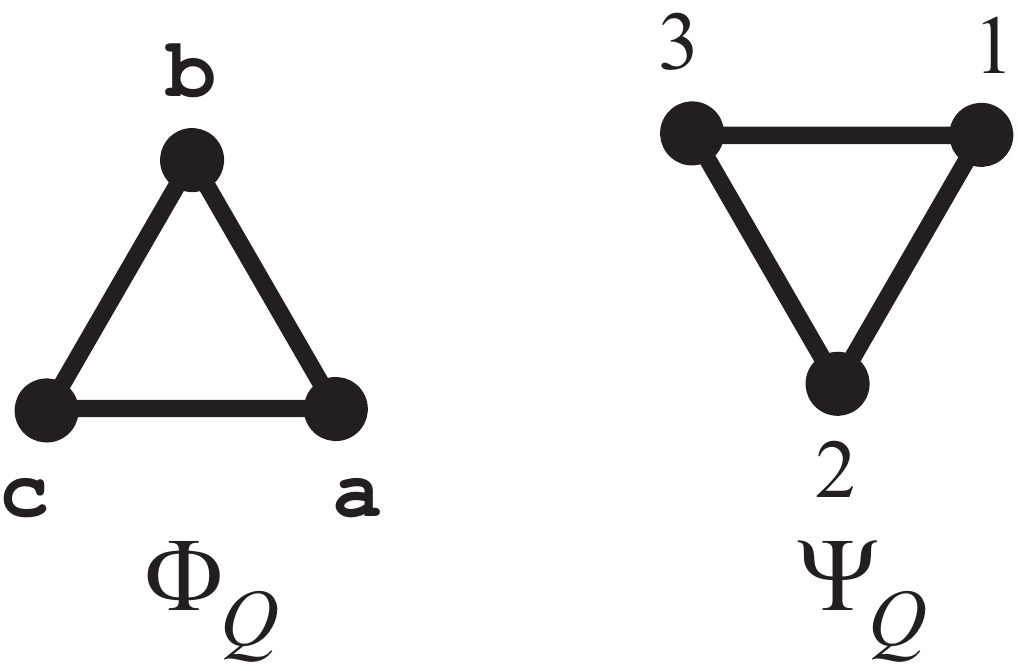}{scale=0.5}
\end{minipage}
\end{center}
\vspace*{-0.2in}
\caption[]{Relation $Q$ describes the conditional relation resulting
  from $R$ of Figure~\ref{Ex20} upon observing attribute \td.  Note
  that $\dowqy = \Lk(\dowy, \{\td\})$.}
\label{Ex20Qfromd}
\end{figure}

For instance, after observing attribute \td\ in relation $R$ of
Figure~\ref{Ex20}, we may conclude that we are observing one of the
individuals in $\{1, 2, 3\}$ and that the remaining attributes we
might yet observe are any two attributes drawn from $\{\ta, \tb,
\tc\}$.  We can express these conclusions as yet another relation,
namely the relation $Q$ of Figure~\ref{Ex20Qfromd}.  Relation $Q$
describes exactly which individuals could give rise to which
attributes, consistent with the prior observation of attribute \td.
{\bf Thus $\dowy$ plays a role much like a probability distribution,
while $\dowqy$ plays the role of a conditional distribution.}
For another example, suppose we have observed attribute \tb\ in $R$.
Then the resulting conditional relation $Q^\prime$ is as in
Figure~\ref{Ex20Qfromb}.

\begin{figure}[h]
\begin{center}
\hspace*{0.4in}\begin{minipage}{1in}{$\begin{array}{c|cccc}
Q^\prime & \ta  & \tc & \td & \te \\[2pt]\hline
1 & \one &      & \one &      \\
3 &      & \one & \one &      \\
4 & \one &      &      & \one \\
6 &      & \one &      & \one \\
\end{array}$}
\end{minipage}
\hspace*{0.7in}
\begin{minipage}{2.5in}
\ifig{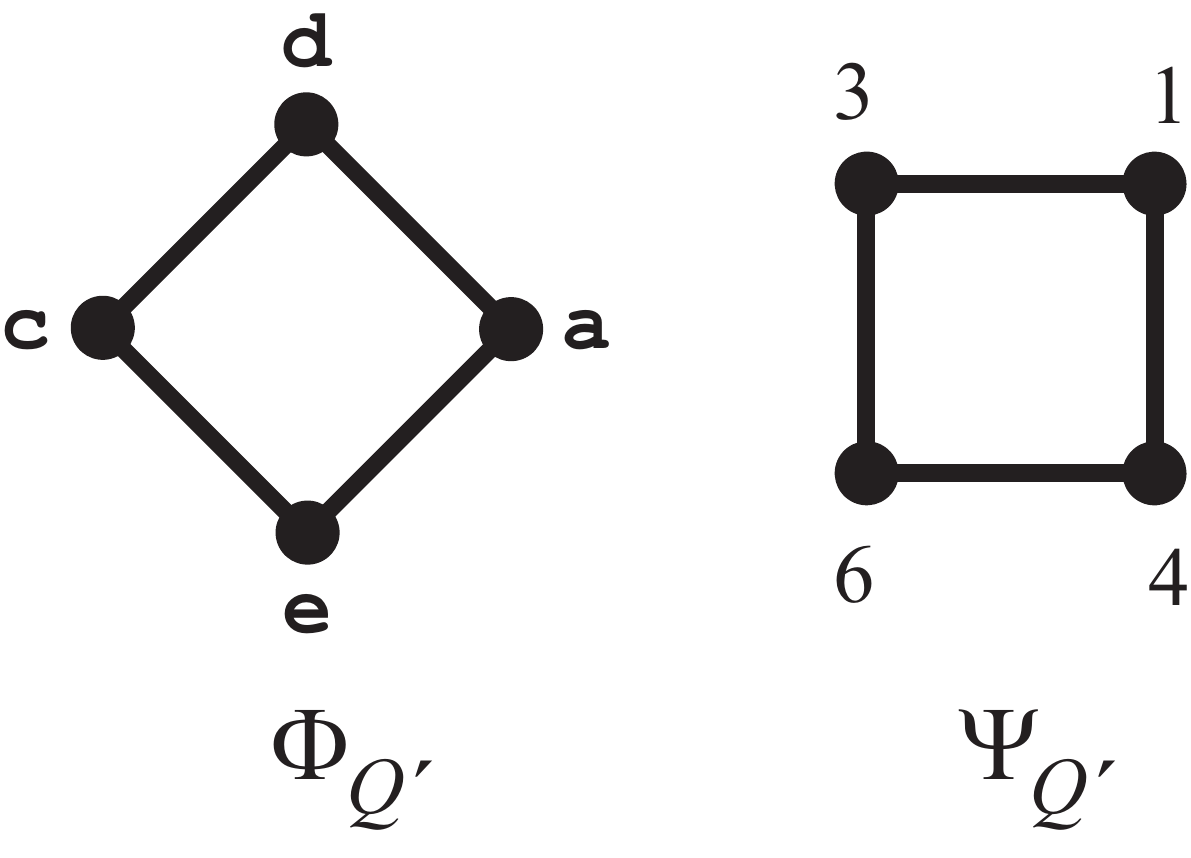}{scale=0.5}
\end{minipage}
\end{center}
\vspace*{-0.25in}
\caption[]{Relation $Q^\prime$ describes the conditional relation
  resulting from $R$ of Figure~\ref{Ex20} upon observing attribute
  \tb.  Here $\dowqpy = \Lk(\dowy, \{\tb\})$.  Observe that the
  attribute space for $Q^\prime$ now factors into two independent
  bits: $\{\ta, \tc\}$ constitutes one bit, $\{\td, \te\}$ the other.
  This factoring is {\em conditional}\hspt\ on having observed \tb.}
\label{Ex20Qfromb}
\end{figure}

\clearpage

The formal constructions of conditional relations appear below.
\ See also Appendix~\ref{linksanddeletions}.

\paragraph{Notation:}\  A symbol of the form $R|_W$ means ``restrict
$R$ to $W$''.  For instance, if $R$ is a relation on $\XxY$, and
if $A \subseteq X$ and $B \subseteq Y$, then \,$R|_{A \times B} \;=\; R
\,\inter (A \times B)$.

\vspace*{0.1in}

\begin{definition}[Conditional Attribute Relations]\label{linkgamma}
\ Let $R$ be a nonvoid relation on $\XxY$ and suppose $\gamma \subseteq Y$.
\ The following \mydefem{relation $Q$ models $\lk(\dowy, \gamma)$}:

$$Q \;=\; R\,|_{\sigma \times \tY}, \quad \hbox{with} \quad
            \sigma = \psi_R(\gamma) \quad \hbox{and} \quad
	    \tY = \bigunion_{x \in \sigma}\sdiff{Y_x}{\gamma}.\phantom{00000}$$

The Dowker complexes are defined in the standard way, except for this
special case:

\vst

If $\,\tY = \emptyset$ and $\sigma\neq\emptyset$, then we
    let $\dowqy$ and $\dowqx$ be instances of the empty complex $\{\emptyset\}$.

\end{definition}

\paragraph{Observe:} \  $\lk(\dowy,\gamma) = \dowqy$ \ (a proof
appears in Appendix~\ref{linksanddeletions}, on
page~\pageref{linksanddeletions}).

\vspace*{-0.05in}  

\paragraph{Comment:}\ If $\gamma\not\in\dowy$, then $\sigma=\emptyset$
and $Q$ is void, and so $\dowqy$ is an instance of the void complex,
consistent with the standard definition of $\lk(\dowy, \gamma)$ being
void in this situation.
\ (See page~\pageref{complexesApp} in Appendix~\ref{complexesApp} for
the definitions of {\em void simplicial complex\,} and {\em empty
simplicial complex}, and page~\pageref{relationsApp} in
Appendix~\ref{relationsApp} for the definition of {\em void
relation}.)

\vspace*{0.3in}

There is a dual construction for links of individuals $\sigma$ in the
Dowker complex modeling associations:

\vspace*{-0.05in}

\begin{definition}[Conditional Association Relations]\label{linksigma}
\ Let $R$ be a nonvoid relation on $\XxY$ and suppose $\sigma \subseteq X$.
\ The following \mydefem{relation $Q$ models $\lk(\dowx, \sigma)$}:

$$Q \;=\; R\,|_{\tX \times \gamma}, \quad \hbox{with} \quad
            \gamma = \phi_R(\sigma) \quad \hbox{and} \quad
	    \tX = \bigunion_{y \in \gamma}\sdiff{X_y}{\sigma}.\phantom{00000}$$

The Dowker complexes are defined in the standard way, except for this
special case:

\vst

If $\,\tX = \emptyset$ and $\gamma\neq\emptyset$, then we
    let $\dowqx$ and $\dowqy$ be instances of the empty complex $\{\emptyset\}$.

\end{definition}

\paragraph{Observe:} \ $\lk(\dowx,\sigma) = \dowqx$.

\vspace*{0.2in}

As we will see in Section~\ref{sphereprivacy}, the complex
$\lk(\dowx,\{x\})$ is useful for characterizing individual $x$'s
attribute privacy.  If that seems surprising, observe that
$\lk(\dowx,\{x\})$ describes other individuals in $R$ who share
attributes with $x$, with simplices modeling the extent of
commonalities.  These commonalities, or lack thereof, determine
whether in $\dowqy$, and thus back in $\dowy$, there are attributes of
$x$ that are ``free to move'' under the closure operators.

\clearpage
\section{Privacy Characterization via Boundary Complexes}
\markright{Privacy Characterization via Boundary Complexes}
\label{sphereprivacy}

\begin{figure}[h]
\begin{center}
\hspace*{0.3in}\begin{minipage}{1.25in}{$\begin{array}{c|ccccc}
Q & \tb  & \tc  & \td \\[2pt]\hline
1 & \one &      & \one \\
2 &      & \one & \one \\
4 & \one &      &      \\
5 &      & \one &      \\
6 & \one & \one &      \\
\end{array}$}
\end{minipage}
\hspace*{0.2in}
\begin{minipage}{4in}
\ifig{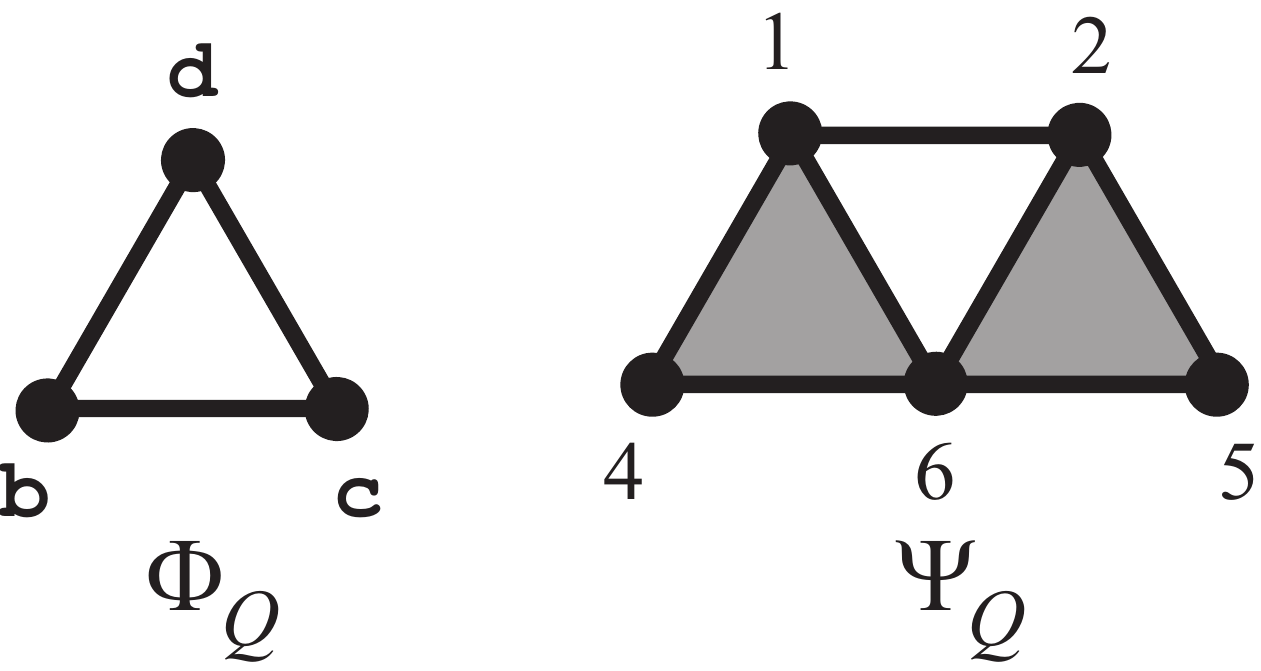}{scale=0.575}
\end{minipage}
\end{center}
\vspace*{-0.2in}
\caption[]{With $R$ as in Figure~\ref{Ex20}, relation $Q$ describes
  the conditional relation corresponding to $\lk(\dowx, \{3\})$.  Also
  shown are the Dowker complexes of $Q$.  By design, $\dowqx =
  \lk(\dowx, \{3\})$.  Observe that $\dowqy$ is the boundary complex
  $\bndry{(\{\tb,\tc,\td\})}$, with $\{\tb,\tc,\td\}$ being all of
  individual \#3's attributes in relation $R$.  That boundary
  condition characterizes attribute privacy for an identifiable
  individual.  \ Here, it means that individual \#3 has full attribute
  privacy.}
\label{Ex20lkthree}
\end{figure}

We observed earlier that relation $R$ of Figure~\ref{Ex20} preserves
attribute privacy.  We came to that conclusion after observing that
$\dowy$ has no free faces.  In fact, one can focus on the privacy of
any identifiable individual rather than look at the whole relation.
Let us pick one such individual, say \#3, and look at the conditional
relation $Q$ that models the link $\lk(\dowx, \{3\})$, as shown in
Figure~\ref{Ex20lkthree}.
\ (Observe that individual \#3 is indeed uniquely identifiable via $R$.)

\vso

Individual \#3 has attributes $\{\tb,\tc,\td\}$ in $R$.  The attribute
complex $\dowqy$ for $Q$ is the boundary complex on exactly this set.
Interpretation: for any nonempty proper subset of individual \#3's
attributes, some \hspc{\em other\hspt} individual in $R\hspt$ has at
least those attributes but not all of individual \#3's attributes.
Consequently, there is a different such individual for each proper
subset of $\{\tb,\tc,\td\}$ that is missing exactly one of \#3's
attributes.  That diversity of individuals ensures individual \#3's
attribute privacy.

\vst

The previous example suggests the following characterization: {\bf An
identifiable individual has full attribute privacy precisely when the
attribute complex of the individual's link is the boundary complex of
the individual's attributes}.

Observe that this characterization is local to the individual; it does
not depend on other individuals having privacy.  \ We now formalize
this intuition.  \ Proofs appear in Appendix~\ref{privacyspheres}.

\vspace*{0.1in}

First, a definition to make precise the notion of individual privacy:

\vspace*{0.05in}

\newcounter{IndivPrivacyDef}
\setcounter{IndivPrivacyDef}{\value{theorem}}
\begin{definition}[Individual Privacy]\label{individualprivacy}
\ Let $R$ be a relation on $\XxY$ and suppose $x \in X$.

We say that \mydefem{$R$ preserves attribute privacy for $x$} whenever
$(\clsy)(\gamma) = \gamma$ for all $\gamma \subseteq Y_x$.

\vspace*{0.05in}

Informally, we may also say that \mydefem{\,individual $x$ has full
attribute privacy}.
\end{definition}

\vspace*{0.1in}

Recall also Definitions~\ref{defAttribPriv} and \ref{uniqueID}, from
pages~\pageref{defAttribPriv} and \pageref{uniqueID}, respectively,
formalizing the notions of (attribute) privacy preservation and unique
identifiability.  And recall the semantics of $\PR$, for instance from
Definition~\ref{defPR} on page~\pageref{defPR}.

\clearpage

\noindent Here is the characterization of individual attribute privacy formalized:

\vspace*{0.025in}

\newcounter{localprivThm}
\setcounter{localprivThm}{\value{theorem}}
\begin{theorem}[Individual Attribute Privacy]\label{privacysingle}
\ Let $R$ be a relation on $\XxY$, with $\abs{X} > 1$.
Suppose $x\in X$ is uniquely identifiable via $R$.
\ Let $Q$ be the relation modeling $\lk(\dowx, x)$.\\
Then the following three conditions are equivalent:

\vspace*{0.125in}

\hspace*{1in}\begin{minipage}{4in}
\begin{enumerate}
\addtolength{\itemsep}{-1pt}
\item[(a)] $R$ preserves attribute privacy for $x$.
\item[(b)] $\Lk(\dowx, x) \;\homot\;\, \Skt$, with $k = \abs{Y_x}$.
\item[(c)] $\dowqy \;=\; \partial(Y_x)$.
\end{enumerate}
\end{minipage}

\end{theorem}

\vspace*{0.2in}

The previous theorem generalizes to sets of individuals for sets that
are ``stable'' under the closure operators, i.e., that appear as the
``set of individuals component'' in an element of $\PR$:

\newcounter{multprivThm}
\setcounter{multprivThm}{\value{theorem}}
\begin{theorem}[Group Attribute Privacy]\label{privacymultiple}
\ Let $R$ be a relation on $\XxY$.\\
Suppose $(\sigma, \gamma) \in \PR$, with $\sigma \neq X$.
\ Let $Q$ be the relation modeling $\lk(\dowx, \sigma)$.\\
Then the following three conditions are equivalent:

\vspace*{0.125in}

\hspace*{0.4in}\begin{minipage}{4in}
\begin{enumerate}
\addtolength{\itemsep}{-1pt}
\item[(a)] $(\clsy)(\gamma^\prime) \;=\; \gamma^\prime$,
   for every subset $\gamma^\prime$ of $\gamma$.
\item[(b)] $\Lk(\dowx, \sigma) \;\homot\; \Skt$, with $k = \abs{\gamma}$.
\item[(c)] $\dowqy \;=\; \partial(\gamma)$.
\end{enumerate}
\end{minipage}

\end{theorem}

\vspace*{0.25in}

\noindent The following lemma relates interpretation and inference in
a link to the encompassing relation:

\newcounter{InterpLocalOperators}
\setcounter{InterpLocalOperators}{\value{theorem}}
\begin{lemma}[Interpreting Local Operators]\label{interplocal}
\ Let $R$ be a relation on $\XxY$.

\vst

Suppose $(\sigma, \gamma) \in \PR$, with $\sigma \neq X$.

Let $Q$ be the relation on $\tX \times \gamma$ that models
$\lk(\dowx, \sigma)$ and suppose $\tX \neq \emptyset$.$\phantom{\Big|}$

\vspace*{0.05in}

Then, for every $\gamma^\prime \subseteq \gamma$:

\vspace*{-0.15in}

\hspace*{1.7in}\begin{minipage}{4in}
\begin{enumerate}

\item[(i)] If $\,\gamma^\prime \not\in\dowqy$,
              then $\psi_R(\gamma^\prime) = \sigma$.

\item[(ii)] If $\,\gamma^\prime \in\dowqy$,
              then $\psi_R(\gamma^\prime) \supsetneq \sigma$.

\vspace*{0.05in}

   Moreover, in this case:

For $\,\gamma^\prime=\emptyset$, \ $(\clsqy)(\emptyset) \supseteq
(\clsy)(\emptyset)$.

\vspace*{0.05in}

If $\,\gamma^\prime\neq\emptyset$, \,then
$(\clsqy)(\gamma^\prime) = (\clsy)(\gamma^\prime)$.

\end{enumerate}
\end{minipage}

\end{lemma}

\vspace*{0.15in}

The lemma says that observations of attributes consistent in $Q$ have
as interpretation more individuals in $R$ than just the individuals
$\sigma$.  However, if ever those observations become inconsistent in
$Q$, then one has identified $\sigma$ in $R$.  Here ``inconsistent in
$Q$'' means that the observed attributes are legitimate attributes for
$Q$ but do not constitute a simplex of $\dowqy$.  (Note: Such observed
attributes necessarily constitute a simplex of $\dowy$ since they are
a subset of $\gamma \in \dowy$).

Moreover, attribute inferences are identical in $Q$ and $R$ for
nonempty simplices of $\dowqy$.

\clearpage
\section{The Meaning of Holes in Relations}
\markright{The Meaning of Holes in Relations}
\label{holemeaning}

We have seen how spheres characterize privacy.  More generally, when
working with topological spaces, holes are significant.  One wonders
what topological holes mean for relations.

\begin{itemize}

\item Some holes arise as a consequence of exclusion between
  attributes, as we saw in the decomposition of Figures~\ref{Ex20} and
  \ref{Ex20decomp}.

  Sticking with binary exclusions, suppose a group of individuals are
  described by $k$ bits.  One can model those individuals via a
  relation containing $2k$ binary attributes (two such attributes per
  bit, one for each possible bit value).  Every individual has exactly
  $k$ of those $2k$ attributes.  If all possible $2^k$ combinations of
  bit values are represented by individuals in the relation, then the
  two Dowker complexes are both homotopic to $\Sko$, the sphere of
  dimension $k-1$.  In fact, $\dowy$ is the simplicial join of $k$
  copies of $\Szero$, while $\dowx$ is visualizable as a hollow
  hypercube in $k$ dimensions, in which solid ($k-1$)-dimensional
  subcubes represent $(2^{k-1}-1)$-dimensional simplices (flattened,
  when $k \geq 3$).  Figures~\ref{onebit1}, \ref{twobit2}, and
  \ref{threebit3} depict the cases $k=1$, 2, and 3, respectively.

  In short, $k$ bits means a hole of dimension $k$$-$$1$, {\em if\,}
  all possible individuals are actually present in the relation.

\vspace*{0.05in}

  (The lack of an expected hole may mean that the capacity of a
  relation has not been exhausted, hinting at possible inference.  See
  Appendix~\ref{insufficientrep}.)

\begin{figure}[h]
\begin{center}
\vspace*{0.25in}
\qquad
\begin{minipage}{1in}{$\begin{array}{c|cccc}
\hbox{\large $S$} & \ta  & \neg \ta  \\[2pt]\hline
              1   & \one &           \\
              2   &      & \one      \\
\end{array}$}
\end{minipage}
\hspace*{0.5in}
\begin{minipage}{2in}
\ifig{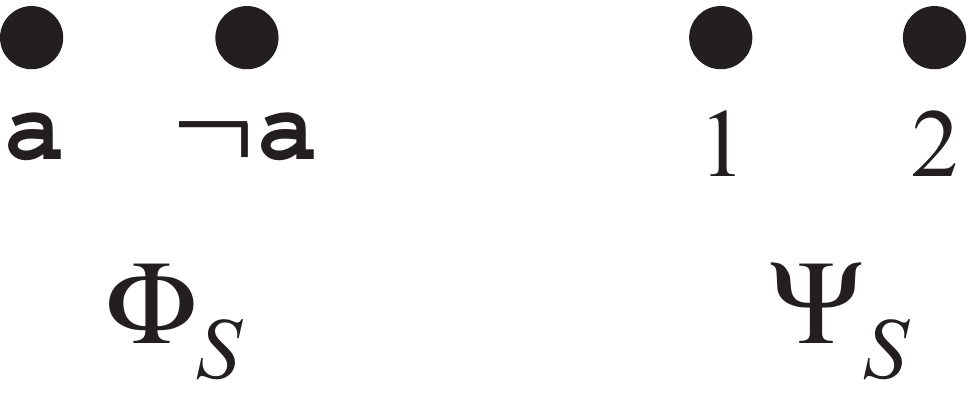}{scale=0.4}
\end{minipage}
\end{center}
\vspace*{-0.2in}
\caption[]{Relation $S$ describes two individuals in terms of a single
  attribute and its negation.  The topology of the Dowker complexes is
  $\Szero$.}
\label{onebit1}
\end{figure}

\begin{figure}[h]
\begin{center}
\qquad\qquad
\begin{minipage}{1in}{$\begin{array}{c|cccc}
\hbox{\large $Q$} & \ta  & \neg\ta & \tb  & \neg\tb  \\[2pt]\hline
              1   & \one &         & \one &           \\
              2   & \one &         &      & \one      \\
              3   &      & \one    & \one &           \\
              4   &      & \one    &      & \one      \\
\end{array}$}
\end{minipage}
\hspace*{1in}
\begin{minipage}{3in}
\ifig{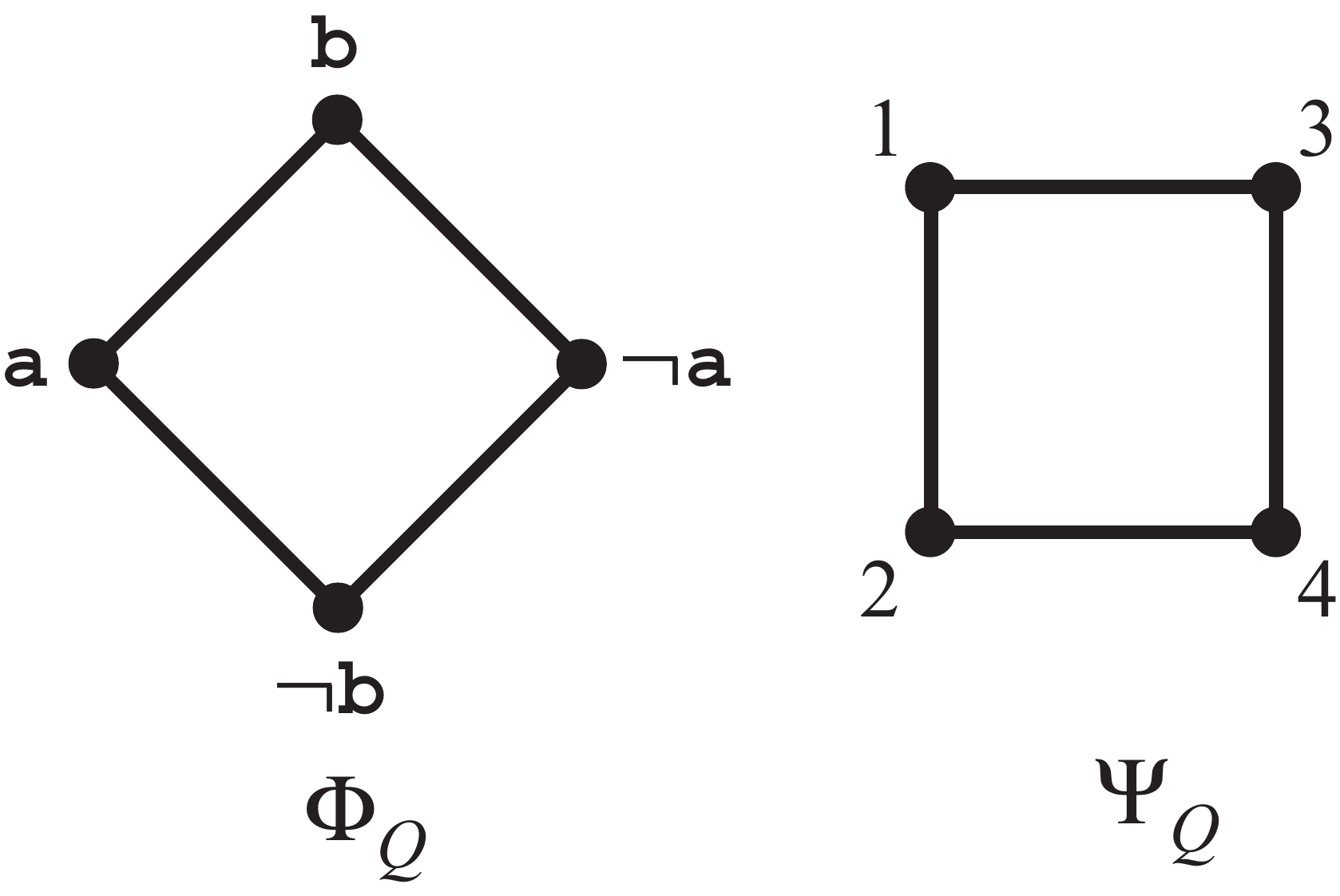}{scale=0.4}
\end{minipage}
\end{center}
\vspace*{-0.2in}
\caption[]{Relation $Q$ describes four individuals in terms of two
  attributes and their negations.  The topology of the Dowker
  complexes is $\Sone$.}
\label{twobit2}
\end{figure}

\begin{figure}[h]
\begin{center}
\quad
\begin{minipage}{1.5in}{$\begin{array}{c|cccccc}
\hbox{\large $R$} & \ta  & \neg\ta & \tb  & \neg\tb & \tc  & \neg\tc \\[2pt]\hline
              1   & \one &         & \one &         & \one &         \\
              2   & \one &         & \one &         &      & \one    \\
              3   & \one &         &      & \one    & \one &         \\
              4   & \one &         &      & \one    &      & \one    \\
              5   &      & \one    & \one &         & \one &         \\
              6   &      & \one    & \one &         &      & \one    \\
              7   &      & \one    &      & \one    & \one &         \\
              8   &      & \one    &      & \one    &      & \one    \\
\end{array}$}
\end{minipage}
\hspace*{0.75in}
\begin{minipage}{3in}
\ifig{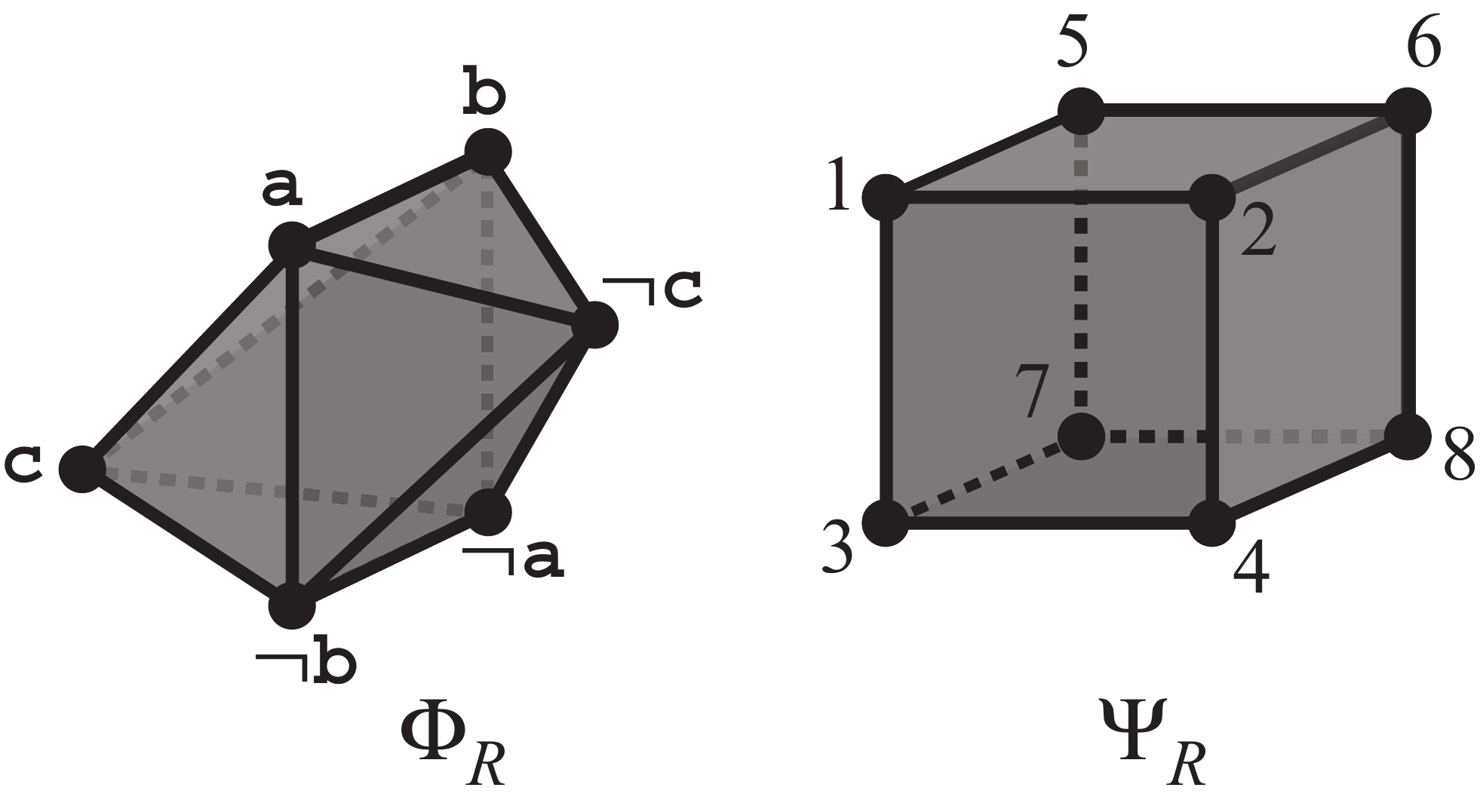}{scale=0.4}
\end{minipage}
\end{center}
\vspace*{-0.2in}
\caption[]{Relation $R$ describes eight individuals in terms of three
  attributes and their negations.  The topology of the Dowker
  complexes is $\Stwo$.  The cube faces are actually tetrahedra,
  flattened to parallelograms in the drawing.}
\label{threebit3}
\end{figure}

\vspace*{0.1in}

\item Suppose $\Sigma$ is a simplicial complex with underlying vertex
  set $X$.  A {\em minimal nonface} of $\Sigma\hspc$ is a subset of $X$
  that is not itself a simplex but all of whose proper subsets are
  simplices in $\Sigma$.  A minimal nonface may or may not be a
  topological hole.  Regardless, a minimal nonface of size two or
  greater in a Dowker complex suggests restricting the relation to
  equal-numbered attributes and individuals for whom there is both
  attribute and association privacy, within the restricted relation.
  This observation dovetails with the following results (here we
  assume that each relation has no blank rows or columns):

\vspace*{-0.075in}

\begin{itemize}
\item A relation with more attributes than individuals cannot fully
  preserve attribute privacy.

\item A relation with more individuals than attributes cannot fully
 preserve association privacy.

\item A relation that preserves both attribute and association privacy
  must have the same number of attributes and individuals.  Moreover,
  if the relation is connected, then both Dowker complexes are either
  linear cycles of the same length or boundary complexes of full
  simplices of the same dimension, as we indicated previously.
\end{itemize}

\vspace*{-0.05in}

  See Appendices~\ref{linksandinference} and \ref{privacyspheres}\,
  for further details and proofs.

\vspace*{0.1in}

\item Minimal nonfaces can have other context-dependent meanings.  For
  instance, in a certain authorship relation, knowing that {\em each
  pair\,} of three individuals has written a paper together appears to
  be a good predictor that {\em all three\,} individuals will
  co-author a paper together \cite{tpr:missedcollab}.  This
  observation suggests the following: if one sees that such an
  authorship hole does {\em not}\, fill over time, then one likely can
  infer some kind of obstruction, perhaps an incompatibility in the
  group as a whole, or the death of an author, for instance.

\vspace*{0.1in}

\item When designing relations or anonymizing relations, these results
  suggest transformations that create ``bubbly spaces'' of some sort,
  in order to retain identifiability but also reduce unwanted
  inference.  Section~\ref{coordchanges} and
  Appendix~\ref{disinformation} discuss examples.

\vspace*{0.1in}

\item Whatever topological holes there are in $\dowy$ and $\dowx$ must
  also show up in the poset $\PR$, since that poset is formed by
  homotopy equivalences from $\dowy$ and $\dowx$.  Interestingly,
  whereas one thinks of $\dowy$ and $\dowx$ simply as spaces, one sees
  a partial order on $\PR$.  Something can move, ``up'' or ``down''.
  The elements of $\PR$ are inference-stable, by design.  So, what is
  this possible motion? \ It is a dynamic process that describes how
  information acquisition changes interpretation.  For instance, as an
  individual reveals information about him- or herself, an observer
  can attempt to identify the individual, by finding interpretations
  in $\PR$ of the information revealed.  As the individual reveals
  additional information, the observer's interpretation moves downward
  in $\PR$, narrowing the set of individuals.

  Topological holes in the spaces $\dowy$ and $\dowx$ (and thus $\PR$)
  constrain how that interpretation moves downward in $\PR$.  The
  greater a hole's dimension, the further a downward path has to move
  before identifying an individual.  One can think of holes in a
  relation much like boulders in a stream.  Eventually, the current of
  information sweeps past the hole, but it is forced to divert its
  motion, covering more distance.  Moreover, there may be many paths
  around the hole, much like a leaf in a stream may divert around a
  boulder in different directions.  The individual can force a
  particular path by choosing to reveal attributes in a particular
  order.

  Much of the rest of the report explores the implications of this
  stream analogy.  The analogy merges with the realization that
  privacy is a dynamic process, certain to flow toward identification
  when attributes are static or persistent, yet subject to channeling
  (perhaps even turbulence in more fluid settings than those discussed
  in this report).  See, in particular,
  Section~\ref{leveraginglattices} onward.

\end{itemize}

\clearpage
\section{Change-of-Attribute Transformations}
\markright{Change-of-Attribute Transformations}
\label{coordchanges}

Free faces and holes in the Dowker complex $\dowy$ can sometimes
suggest changes in attributes that preserve desired information but
reduce inference.  Consider the hypothetical ``ice-cream cone''
relation $C$ of Figure~\ref{icecream} and the corresponding complexes
shown in Figure \ref{icecreamcomplexes}.  The relation describes four
individuals in terms of the two-flavor two-scoop ice-cream cones each
individual enjoys at a particular ice-cream parlor.

\begin{figure}[h]
\begin{center}
\qquad\qquad\quad
\begin{minipage}{2.5in}{$\begin{array}{l|cccccc}
\hbox{\large $\;\;C$} & \gc  & \gs  & \cs  & \cv  & \sv  & \gv  \\[2pt]\hline
      \hbox{Bob}   & \one & \one & \one &      &      &      \\
      \hbox{Alice} &      &      & \one & \one & \one &      \\
      \hbox{David} &      & \one &      &      & \one & \one \\
      \hbox{Cindy} & \one &      &      & \one &      & \one \\
\end{array}$}
\end{minipage}
\qquad\qquad\quad
\begin{minipage}{1.5in}{$\begin{array}{lcl}
& & \\
\ig & = & \hbox{ginger}\\
\ic & = & \hbox{chocolate}\\
\is & = & \hbox{strawberry}\\
\iv & = & \hbox{vanilla}\\
\end{array}$}
\end{minipage}
\end{center}
\vspace*{-0.15in}
\caption[]{Four individuals and their preferences for ice-cream cones
  containing two scoops, with different flavors (each letter
  represents a flavor, as indicated).  See Figure
  \ref{icecreamcomplexes} for the Dowker complexes.}
\label{icecream}
\end{figure}

\begin{figure}[h]
\begin{center}
\ifig{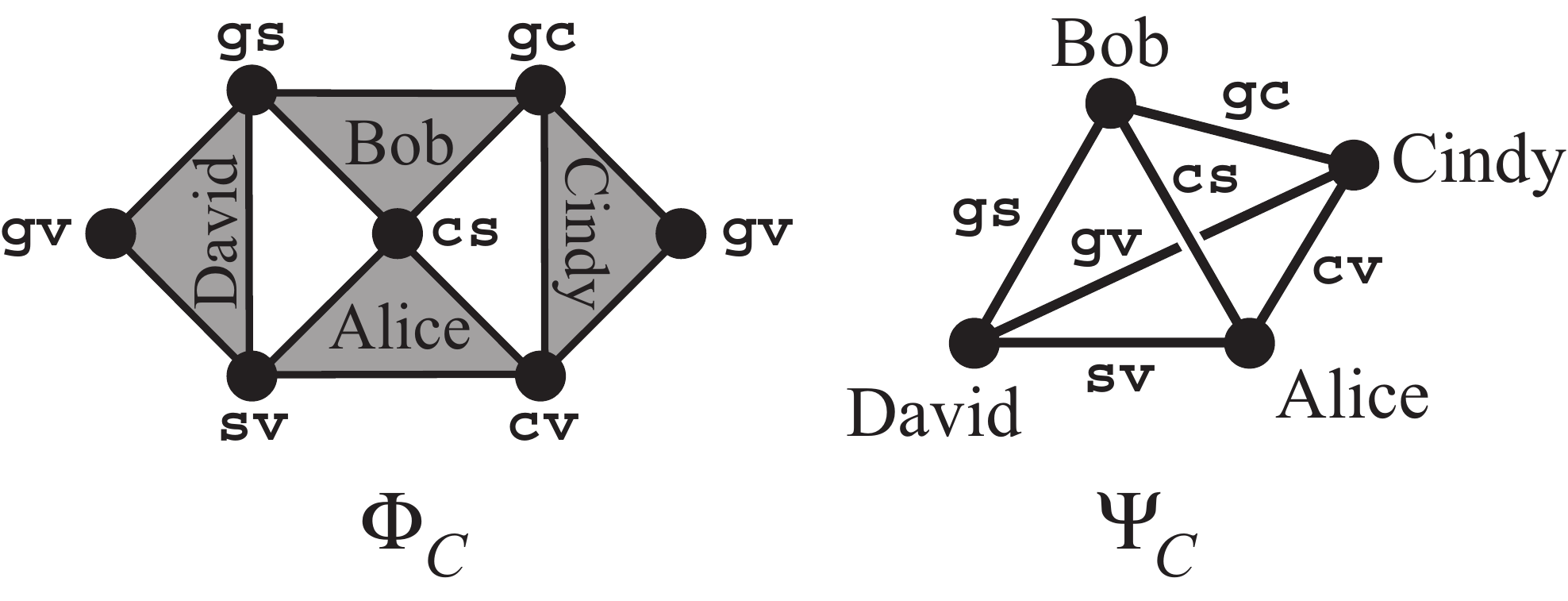}{scale=0.65}
\end{center}
\vspace*{-0.35in}
\caption[]{The Dowker complexes for the relation of Figure
  \ref{icecream}.  $\dowcy$ is a complex whose vertices are ice-cream
  cones (two flavors).  (For visualization purposes, the complex is
  flattened, with the leftmost and rightmost vertices really
  representing the same ice-cream cone.)  Each maximal simplex is a
  triangle, labeled with the individual who enjoys the three types of
  cones comprising the triangle.  $\dowcx$ is a complex whose vertices
  are individuals.  Each maximal simplex is an edge, representing a
  two-flavor two-scoop ice-cream cone that each of two individuals
  enjoys; the edge is labeled with the cone flavors.  The homotopy
  type of each complex is $\Sone \join \Sone \join \Sone$.}
\label{icecreamcomplexes}
\end{figure}

Relation $C$ is a typical ``2-implies-3'' relation: Any two different
ice-cream cones uniquely identify an individual, thereby implying a
third ice-cream cone, as can be seen from either Dowker complex: In
$\dowcy$, every edge is a free face of its encompassing triangle.
Moreover, the edge is not itself generated by any
individual.\footnote{We say that an individual $x$ of a relation
$R\mskip1mu$ {\em generates\,} the simplex $Y_x \in \dowy$.
Similarly, an attribute $y$ generates the simplex $X_y \in \dowx$.
\ Individuals generate triangles in $\dowcy$.
\ Ice-cream cones generate edges in $\dowcx$.}
The closure operator $\phi_C \circ \psi_C$ must therefore map every
edge to a triangle.  Dually, in $\dowcx$, any two edges intersecting
at a vertex imply the third edge incident on that vertex.

This type of relation models, in the small, inferences such as those
reported in \cite{tpr:sweeneykanon, tpr:netflix}.  For instance,
\cite{tpr:sweeneykanon} reported that zip code, gender, and birth date
were likely sufficient in 1990 to identify $87\%$ of individuals in
the U.S.  That is nearly a ``3-implies-all'' type of relation.
Similarly, \cite{tpr:netflix} reported that 8 movie ratings and dates
were enough to uniquely identify $99\%$ of viewers in the Netflix
Prize dataset.  That is essentially an ``8-implies-all'' type of
relation.

Let us focus for a moment on Bob's neighborhood.  That relation, let
us call it $B$, and its complexes are depicted in Figure~\ref{Bob}.
(The relation models $\st(\dowcx, \{\hbox{Bob}\})$; see
Appendix~\ref{subcomplexes}.)

\begin{figure}[h]
\begin{center}
\begin{minipage}{1.5in}{$\begin{array}{l|ccc}
\hbox{\large $\;\;B$} & \gc  & \gs  & \cs   \\[2pt]\hline
      \hbox{Bob}   & \one & \one & \one \\
      \hbox{Alice} &      &      & \one \\
      \hbox{David} &      & \one &      \\
      \hbox{Cindy} & \one &      &      \\
\end{array}$}
\end{minipage}
\qquad
\begin{minipage}{3.5in}
\ifig{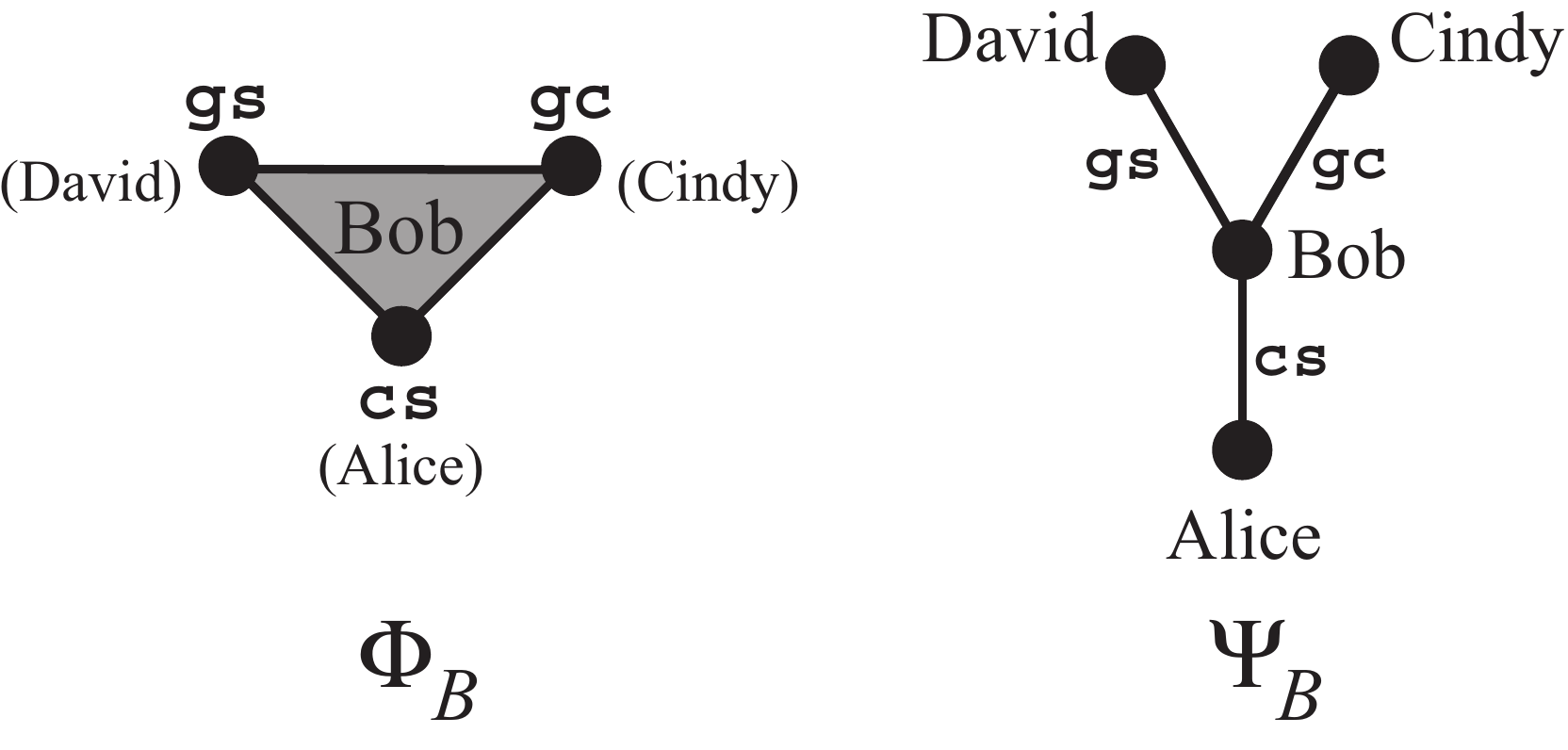}{scale=0.5}
\end{minipage}
\end{center}
\vspace*{-0.25in}
\caption[]{Relation $B$ models Bob's neighborhood in the ice-cream
  relation of Figure~\ref{icecream}.
  Each maximal simplex is labeled with its generator.  Generators of
  nonmaximal simplices are indicated in parentheses.}
\label{Bob}
\end{figure}

As in $C$, seeing someone eat one ice-cream cone is not enough to
identify anyone in $B$ uniquely.  Seeing someone (in this case, Bob)
eat two {\em different} types of ice-cream cones is sufficient to
infer the third type of ice-cream cone that individual prefers.  How
might we prevent this?  We observe that the vertices of $\dowby$ are
themselves generated by individuals while the edges are not.
Homotopically, therefore, we want to expand the vertices of $\dowby$
into edges, and contract the edges of $\dowby$ into vertices.  One
possible way to accomplish this is the take logical {\sc or}s of the
existing attributes.
With $\oplus$ meaning Boolean {\sc or}, we define:

\vspace*{-0.1in}

$$\alpha \;=\; \gc \;\oplus\; \gs, \qquad
  \beta  \;=\; \gc \;\oplus\; \cs, \qquad
  \gamma \;=\; \gs \;\oplus\; \cs. $$

Then relation $B$ becomes $B^\prime$ as in Figure~\ref{Bobp}.  The
result is that the free faces of $\dowbpy$ now are generated by other
individuals, so even though they are free, the closure operator does
not move them.  In fact, the closure operator $\phi_{B^\prime} \circ
\psi_{B^\prime}$ is the identity on ${\F(\Phi_{B^\prime}) \union
\{\emptyset\}}$, meaning that no attribute inference is possible in
$B^\prime$.

\begin{figure}[h]
\begin{center}
\qquad
\begin{minipage}{1.5in}{$\begin{array}{l|ccc}
\hbox{\large $\;B^\prime$} & \alpha  & \beta  & \gamma   \\[2pt]\hline
      \hbox{Bob}   & \one & \one & \one \\
      \hbox{Alice} &      & \one & \one \\
      \hbox{David} & \one &      & \one \\
      \hbox{Cindy} & \one & \one &      \\
\end{array}$}
\end{minipage}
\qquad\quad
\begin{minipage}{3.5in}
\ifig{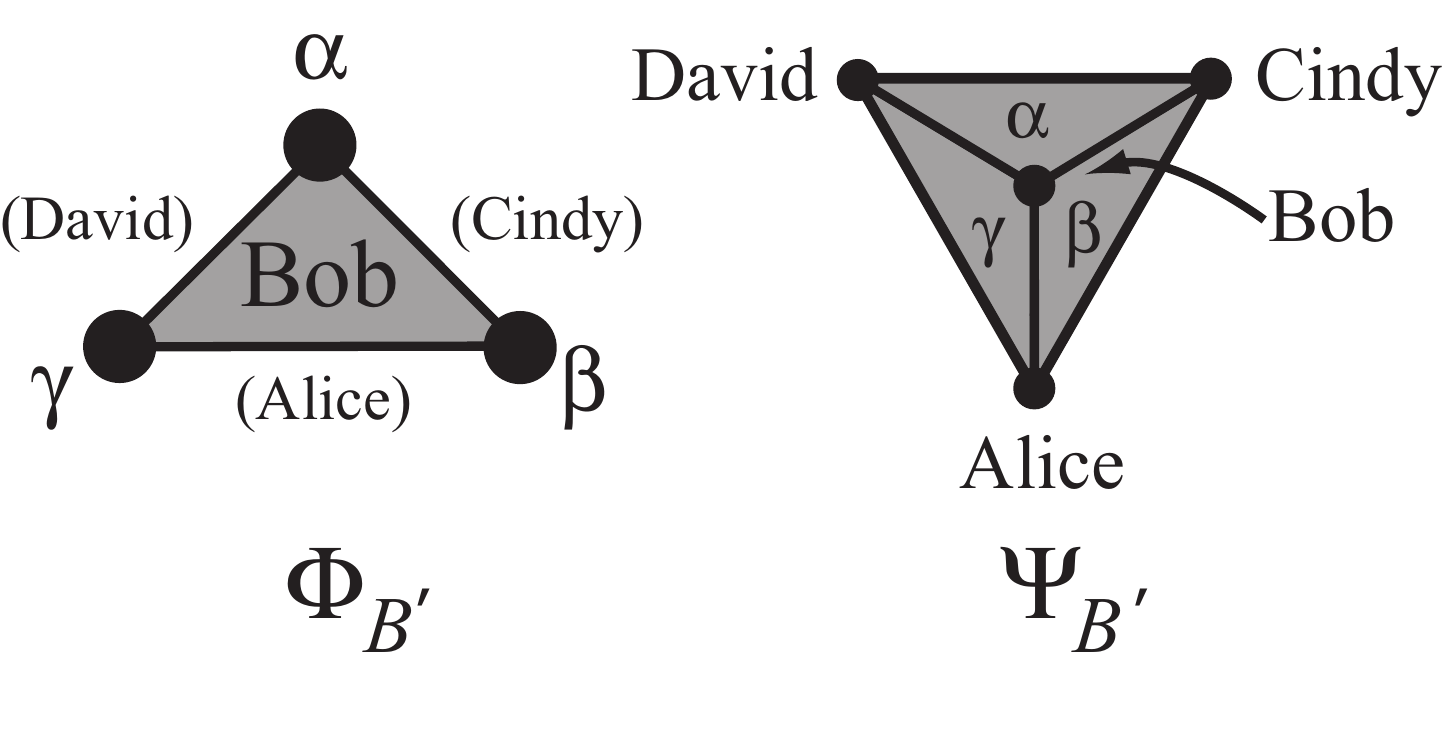}{scale=0.5}
\end{minipage}
\end{center}
\vspace*{-0.35in}
\caption[]{Relation $B^\prime$ represents relation $B$ of Figure
  \ref{Bob}, now with a coordinate transformation for the attributes.
  Simplices are again labeled by generators.}
\label{Bobp}
\end{figure}

\clearpage

Now imagine performing similar operations for all four individuals of
relation $C$ from Figure~\ref{icecream}.  One winds up constructing
four logical {\sc or}s:

\vspace*{-0.2in}

$$\gc \;\oplus\; \gs \;\oplus\; \gv, \qquad
  \gc \;\oplus\; \cs \;\oplus\; \cv, \qquad
  \gs \;\oplus\; \cs \;\oplus\; \sv, \qquad
  \cv \;\oplus\; \sv \;\oplus\; \gv.$$

\vspace*{0.1in}

\noindent Two observations:

\begin{enumerate}
\item Each {\sc or} describes three ice-cream cones that form a hole
  in the complex $\dowcy$ of Fig.~\ref{icecreamcomplexes}.

\item Each such hole may be interpreted as a single flavor, namely the
flavor in common to the three ice-cream cones appearing in the {\sc
or}.  For instance, ``ginger'' (abbreviated as $\ig$) is the common
flavor for the {\sc or}\, $\gc \oplus \gs \oplus \gv$.

\end{enumerate}

In order to describe the resulting relation, it is perhaps easiest to
express those four new coordinates themselves via a relation $S$ that
describes the scoops present in an ice-cream cone:

$$\begin{array}{c|cccc}
\hbox{\large $S$} & \ig  & \ic  & \is  & \iv   \\[2pt]\hline
             \gc  & \one & \one &      &       \\
             \gs  & \one &      & \one &       \\
             \cs  &      & \one & \one &       \\
             \cv  &      & \one &      & \one  \\
             \sv  &      &      & \one & \one  \\
             \gv  & \one &      &      & \one  \\
\end{array}$$

\vspace*{0.1in}

Finally, to perform the coordinate-transformation, one simply
multiplies Boolean matrices, with addition being Boolean {\sc or} and
multiplication being Boolean {\sc and}:  $ F = CS $.  The relation $F$
and its complexes appear in Figure~\ref{flavors}.  

\begin{figure}[h]
\begin{center}
\qquad
\begin{minipage}{1.5in}{$\begin{array}{l|cccc}
\hbox{\large $\;\;F$} & \ig  & \ic  & \is  & \iv  \\[2pt]\hline
     \hbox{Bob}   & \one & \one & \one &      \\
     \hbox{Alice} &      & \one & \one & \one \\
     \hbox{David} & \one &      & \one & \one \\
     \hbox{Cindy} & \one & \one &      & \one \\
\end{array}$}
\end{minipage}
\qquad\quad
\begin{minipage}{3.5in}
\ifig{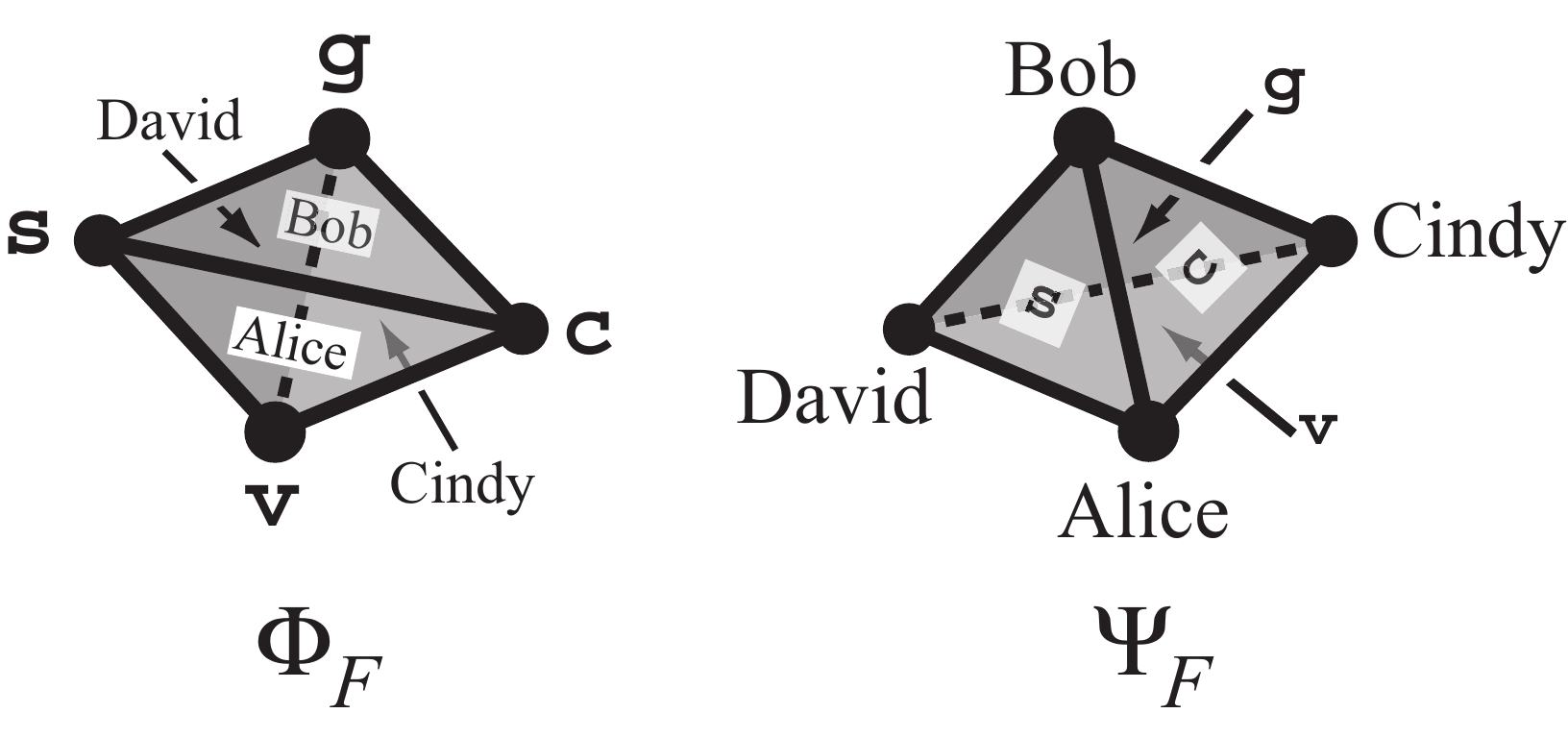}{scale=0.5}
\end{minipage}
\end{center}
\vspace*{-0.25in}
\caption[]{Relation $F$ describes the ice-cream flavors each
  individual prefers.  $\dowfy$ is the boundary complex of a
  tetrahedron, with flavors as vertices.  $\dowfx$ is the Dowker dual
  of $\dowfy$, with respect to relation $F$.  Consequently, $\dowfx$
  also is the boundary complex of a tetrahedron, now with the roles of
  flavors and individuals interchanged.  \,For both $\dowfy$ and
  $\dowfx$, each maximal simplex is a triangle, labeled with its
  generator.}
\label{flavors}
\end{figure}

Relation $F$ represents a description of the four individuals'
preferences in terms of flavors not cones.  The resulting complexes
$\dowfy$ and $\dowfx$ are now boundary complexes of full simplices,
each homeomorphic to $\Stwo$.  These complexes have no free faces, so
no inference is possible.  Observe further that $\dowfy$ is homotopic
to what one obtains from $\dowcy$ by filling the $\Sone$-holes.
Indeed, this idea implicitly motivated our construction, as a way to
remove free faces.  Similarly, $\dowfx$ is isomorphic to what one
obtains from $\dowcx$ by filling its $\Sone$-holes.

One should ask how this approach might generalize.  The answer is
mixed.  The idea of removing free faces is central.  There are many
ways to accomplish that, with relational composition being but one
method.  One issue with logical {\sc or}s is that it is very easy to
obtain an {\sc or} that is always {\sc True}, at which point the
resulting attribute is of little use.

Even with more general transformations, there remains the issue of
whether the new attributes are grounded in what is actually
observable.  In the ice-cream example, it was fortunate that cones
decomposed naturally into flavors.  It is at least plausible that
someone might merely observe the flavors a customer prefers, not the
combinations of flavors as cones.  If, however, only cones can be
observed, then one is forced to deal with relation $C$ as given.

\clearpage
\section{Leveraging Lattices for Privacy Preservation}
\label{leveraginglattices}

This section examines more carefully the lattice structure of a
relation's poset, leading to the idea of {\em informative attribute
release sequences}.  Such a sequence consists of attributes that an
individual releases in a particular order, so as to prevent inference
of any attributes yet to be released via the sequence.  The length of
the lattice representing the individual's link relation then describes
the extent to which that individual can defer identification.
Homology provides lower bounds on that length.

\subsection{Attribute Release Order}

Relation $G$ of Figure~\ref{travelguides} describes hypothetical
co-authorships among five authors in producing travel guides for five
European cities.  Each collaboration consists of three authors working
together on one of the five travel guides.

\begin{figure}[h]
\begin{center}
\qquad
\begin{minipage}{2in}{$\begin{array}{cl|ccccc}
\multicolumn{2}{c|}{\raisebox{0.2in}{\hbox{\LARGE $G$}}}%
                & \rcityA & \rcityB & \rcityC & \rcityD & \rcityE \\[1pt]\hline
1 & (\hauthorA) & \one    & \one    &         &         & \one    \\
2 & (\hauthorB) & \one    & \one    & \one    &         &         \\
3 & (\hauthorC) &         & \one    & \one    & \one    &         \\
4 & (\hauthorD) &         &         & \one    & \one    & \one    \\
5 & (\hauthorE) & \one    &         &         & \one    & \one    \\
\end{array}$}
\end{minipage}
\hspace*{0.5in}
\begin{minipage}{3in}
\ifig{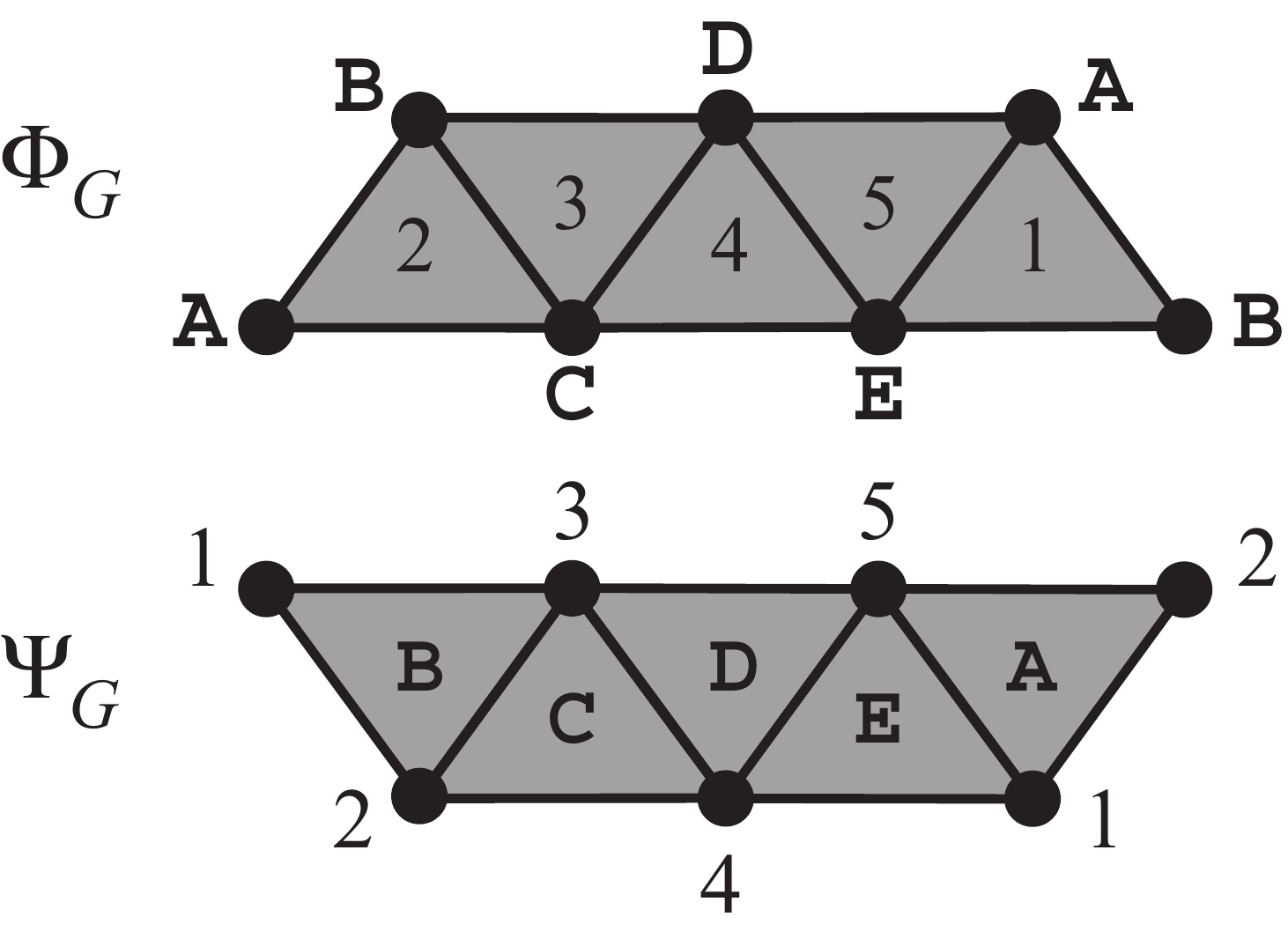}{scale=0.4}
\end{minipage}
\end{center}
\vspace*{-0.1in}
\caption[]{A relation $G$ describing co-authorship of travel guides.
  The Dowker complexes are dual triangulations of the M\"obius strip,
  with $\Sone$ homotopy type.  (Notes: Integers indicate authors,
  letters indicate cities via first letter abbreviations.  Some
  vertices and edges appear twice for ease of viewing.  Each maximal
  simplex is labeled with its generating author or city.)}
\label{travelguides}
\end{figure}

Suppose in casual conversation a person mentions that he/she worked on
producing a travel guide for \cityB.  In the context of relation $G$,
that information means the author is one of $\{\authorA, \authorB,
\authorC\}$.  If the author further mentions working on the travel guide
for \cityD, then that identifies the author uniquely as \authorC.
Equivalently, the listener can infer that the author also helped write
the travel guide for \cityC.  (This form of inference was a source of
privacy problems for the Netflix Prize \cite{tpr:netflix}.)

\begin{figure}[h]
\hspace*{0.1in}\ifig{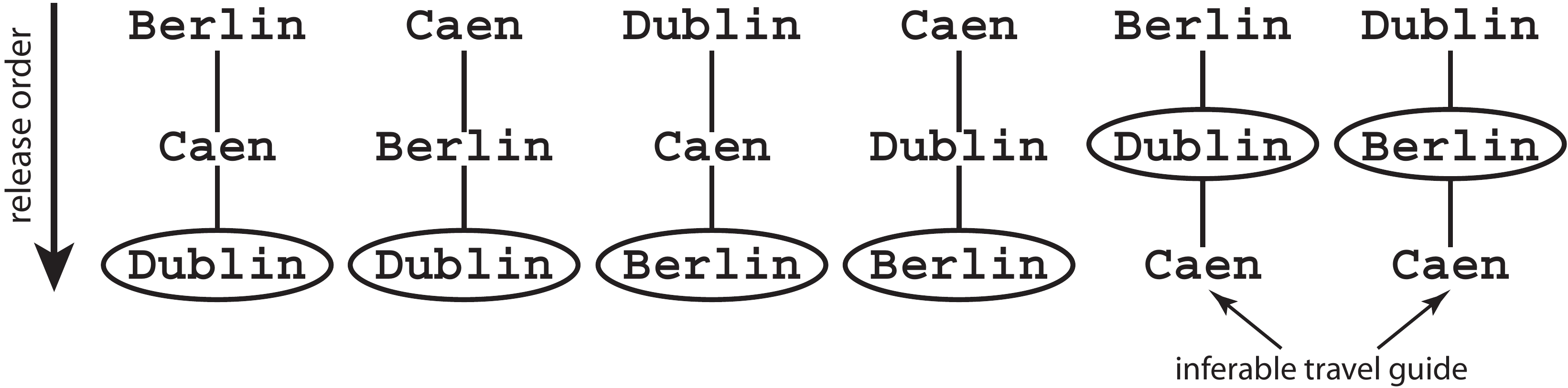}{scale=0.45}
\vspace*{-0.1in}
\caption[]{This figure shows the six possible sequential ways in which
  author \#3 (\authorC) of Figure~\ref{travelguides} can mention the
  cities for which she co-authored travel guides.  The point at which
  her identity becomes known in any such release sequence is circled.
  If \authorC\ does not mention \cityC, one can infer, via relation
  $G$ of Figure~\ref{travelguides}, that she co-authored a travel
  guide for that city as soon as she mentions the other two cities,
  \cityB\ and \cityD, in either order.}
\label{clairerelease}
\end{figure}

\authorC\ was a co-author on three travel guides, for \cityB, \cityC,
and \cityD.  Now consider the different possible sequential ways in
which \authorC\ might reveal which books she helped co-author, along
with the points at which her identity becomes known (see
Figure~\ref{clairerelease}).

Of the six possible ways, four do not uniquely identify \authorC\
until she has revealed all three books that she co-authored.  However,
two of the possible six release sequences do allow a listener to
identify the author and infer an additional book that she co-authored.

This example shows how inference may be a dynamic process.  While a
consumer of data may wish to identify \authorC\ with as little
information as possible, the author herself may wish to delay that
identification for as long as possible (perhaps for reasons of public
mystery in selling books).  In the example, the {\em minimal length}
of an {\em identifying attribute release sequence} is two, while the
{\em maximal length} is three.  If \authorC\ can control how
information is released, then she can choose to reveal what might
otherwise be inferred, namely that she co-authored a travel guide to
\cityC, thereby delaying her identification.

Finally, we observe that the order of attributes released may or may
not matter.  In the travel guide example, \authorC\ should mention
\cityC\ before the end of her disclosures (if she wants to delay her
identification), but the order of cities mentioned is otherwise
irrelevant.  The topology of the doubly-labeled poset $\PG$ encodes
this order (in)dependence, as we will see shortly.  Indeed, much of
the remainder of this report examines the connection between the
topology of a relation's doubly-labeled poset and the length of
attribute release sequences.

\subsection{Inferences on a Lattice}
\markright{Inferences via the Galois Lattice}

\noindent The doubly-labeled poset of a relation produces a lattice
\cite{tpr:wille}, as follows:

\begin{definition}[Galois Lattice]\label{galoislattice}
Let $R$ be a relation on $\XxY$, with both $\hspt{}X\!$ and
$\hspt{}Y\!$ nonempty.
\quad Let $\PR$ be the associated doubly-labeled poset.

\vst

(Recall from Definition~\ref{defPR} on page~\pageref{defPR} that an
element of $\PR$ is an ordered pair $(\sigma, \gamma)$, with\\[1pt]
\hspace*{0.5in}
$\emptyset \neq \sigma = \psi_R(\gamma) \in \dowx$ and\,
$\emptyset \neq \gamma = \phi_R(\sigma) \in \dowy$.

\vst

\noindent\hspace*{0.3in}We previously defined a partial order on
$\PR$ by $(\sigma_1, \gamma_1) \leq (\sigma_2, \gamma_2)$ iff
$\sigma_1 \subseteq \sigma_2$ (iff $\gamma_1 \supseteq \gamma_2$).)

\vst

$\PR$ may already contain a unique bottom element of the form
$(\sigma, Y)$, with $\sigma$ those individuals in $X\!$ who have all
the attributes in $Y$.  If not, we adjoin $(\hspace*{0.01in}\emptyset,
Y)$ to the bottom of $\PR$.

\vst

$\PR$ may already contain a unique top element of the form $(X, \gamma)$,
with $\gamma$ those attributes in $Y\!$ that every individual in $X\!$
has.  If not, we adjoin $(X, \emptyset)$ to the top of $\PR$.

\vspace*{0.05in}

We refer to the resulting poset as the \mydefem{Galois lattice $\PRplus$}.
It has lattice operations $\join$ and $\meet$:

\vspace*{-0.2in}

\begin{eqnarray*}
(\sigma_1, \gamma_1) \;\join\; (\sigma_2, \gamma_2)   &\;=\;&
  \big((\psi_R \circ \phi_R)(\sigma_1 \union \sigma_2),\;\;
                               \gamma_1 \inter \gamma_2\big),\\[2pt]
(\sigma_1, \gamma_1) \;\meet\; (\sigma_2, \gamma_2) &\;=\;&
  \big(\sigma_1 \inter \sigma_2,\;\;
         (\phi_R \circ \psi_R)(\gamma_1 \union \gamma_2)\big).
\end{eqnarray*}

\vspace*{-0.025in}

We sometimes refer to the bottom element of $\PRplus$ by
\underline{$\zeroR$} and to the top element by \underline{$\oneR$}.

\end{definition}

\begin{figure}[h]
\begin{center}
\ifig{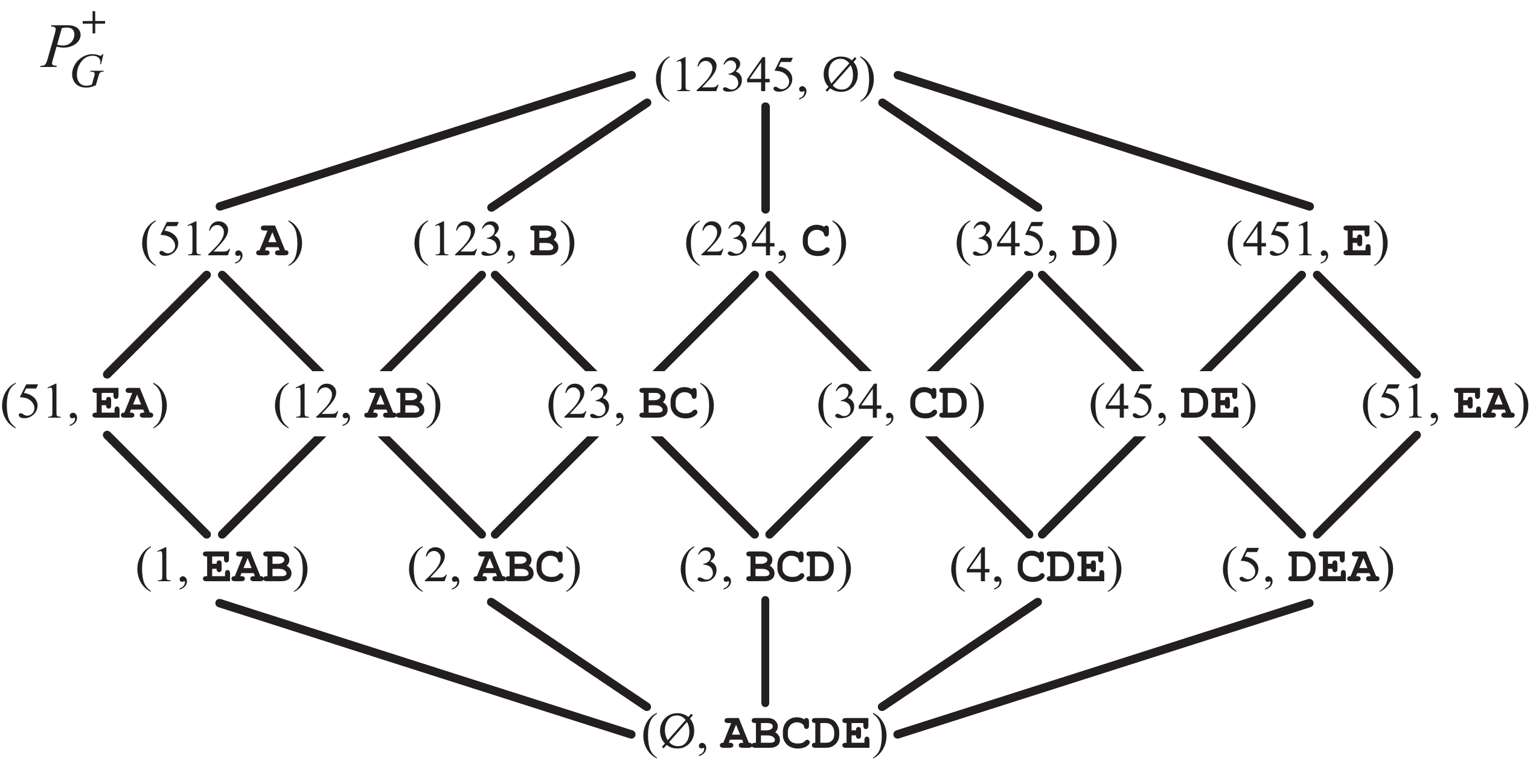}{scale=0.5}
\end{center}
\vspace*{-0.25in}
\caption[]{The lattice $\PGplus$ for the travel guide relation of
  Figure~\ref{travelguides}.  Each element is an ordered pair of sets
  $(\sigma, \gamma)$ such that $\sigma = \psi_G(\gamma)$ and $\gamma =
  \phi_G(\sigma)$.  (We have elided commas and braces in sets, for
  ease of viewing.)  The lattice operations model inferences possible
  from observations.  For instance, $(123, \cB) \wedge (345, \cD)
  \;=\; (3, \cB\cC\cD)$, meaning that observation of attributes $\cB$
  and $\cD$ permits inference of additional attribute $\cC$ and
  identification of author \#3.  (In Figure~\ref{travelguides},
  attribute $\cC$ is the travel guide for \cityC\ and author \#3 is
  \authorC.)  The lattice wraps around, with element $(51, \cE\cA)$
  duplicated for ease of viewing. If one removes the top and bottom
  elements, the remaining poset $\PG$ has $\Sone$ homotopy type, just
  like the M\"obius strip.}
\label{authorlattice}
\end{figure}

\noindent Figure~\ref{authorlattice} shows the lattice $\PGplus$ for
the travel guide relation of Figure~\ref{travelguides}.  Observe how
the lattice encodes attribute and association inferences (or lack
thereof) via its lattice operations.

\paragraph{Special Cases:}  \ 
It can happen that the lattice consists of a single element.  For
example, with relation $C$ as on page~\pageref{relationC}, $\PCplus =
\PC = \{(X, Y)\}$.  In particular, $\zeroC = \oneC$.

\vst

Definition~\ref{galoislattice} ignores the situation in which $R$ is
void.  One possibility is to leave $\PR$ undefined and let
$\PRplus=\emptyset$.  \ See page~\pageref{PRspecialcases} in
Appendix~\ref{PRspecialcases} for additional comments.

\subsection{Preserving Attribute Privacy for Sets of Individuals}
\markright{Preserving Attribute Privacy for Sets of Individuals}

Theorem~\ref{privacysingle} on page~\pageref{privacysingle} described
the conditions under which an individual has full attribute privacy.
For such an individual, the order in which that individual (or anyone)
releases the individual's attributes is irrelevant.  Any order is
fine.  Only once all attributes have been released, can an observer
uniquely identify the individual.  Theorem~\ref{privacymultiple}
described a similar result for certain sets of individuals, including
sets of individuals with whom a given individual is confusable after
only some of his/her attributes have been released.

\begin{figure}[h]
\vspace*{-0.05in}
\begin{center}
\begin{minipage}{1.5in}{$\begin{array}{cl|ccc}
\multicolumn{2}{c|}{\raisebox{0.2in}{\hbox{\LARGE $C$}}}%
                 & \rcityB & \rcityC & \rcityD \\[1pt]\hline
1 & (\hauthorA)  & \one    &         &         \\
2 & (\hauthorB)  & \one    & \one    &         \\
4 & (\hauthorD)  &         & \one    & \one    \\
5 & (\hauthorE)  &         &         & \one    \\
\end{array}$}
\end{minipage}
\hspace*{0.3in}
\begin{minipage}{3.75in}
\ifig{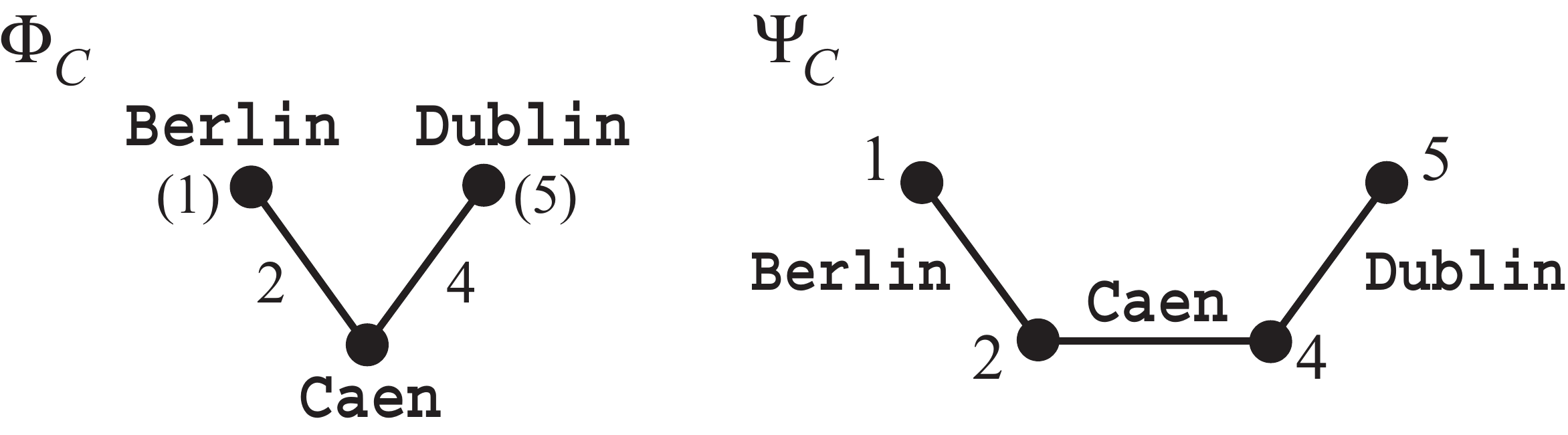}{scale=0.425}
\end{minipage}
\end{center}
\vspace*{-0.175in}
\caption[]{Relation $C$ describes $\Lk(\dowGx, 3)$, the link of
  \authorC\ in the relation of Figure~\ref{travelguides}.  (Each
  maximal simplex in any one complex is labeled with its generating
  attribute or individual from the other complex.  Generators of
  nonmaximal simplices are indicated in parentheses.)}
\label{clairelink}
\end{figure}

Consider $\Lk(\dowGx, 3)$, modeled by relation $C$ as in
Figure~\ref{clairelink}.  This relation describes the authors with
whom \authorC\ has collaborated, via their co-authored books.  The
Dowker complexes are contractible, so by either
Theorem~\ref{privacysingle} or Theorem~\ref{privacymultiple}, we know
that some attribute inference is possible involving \authorC.
Lemma~\ref{interplocal} on page~\pageref{interplocal} tells us to look
for a proper subset of $\{\hcityB, \hcityC, \hcityD\}$ that is {\em
not\,} a simplex of $\dowcy$.  As is apparent from
Figure~\ref{clairelink}, the set $\{\hcityB, \hcityD\}$ satisfies
these conditions, consistent with our earlier observations.
Alternatively, looking at $\PCplus$ in Figure~\ref{clairelattice}, we
see that $(12, \cB) \wedge (45, \cD)
\;=\; (\emptyset, \cB\cC\cD)$, allowing us to draw the same conclusion.
Consequently, \authorC\ should be sure to mention her travel guide for
\cityC\ early on, not leave it for last, if she wants to delay
identification.

\begin{figure}[t]
\begin{center}
\ifig{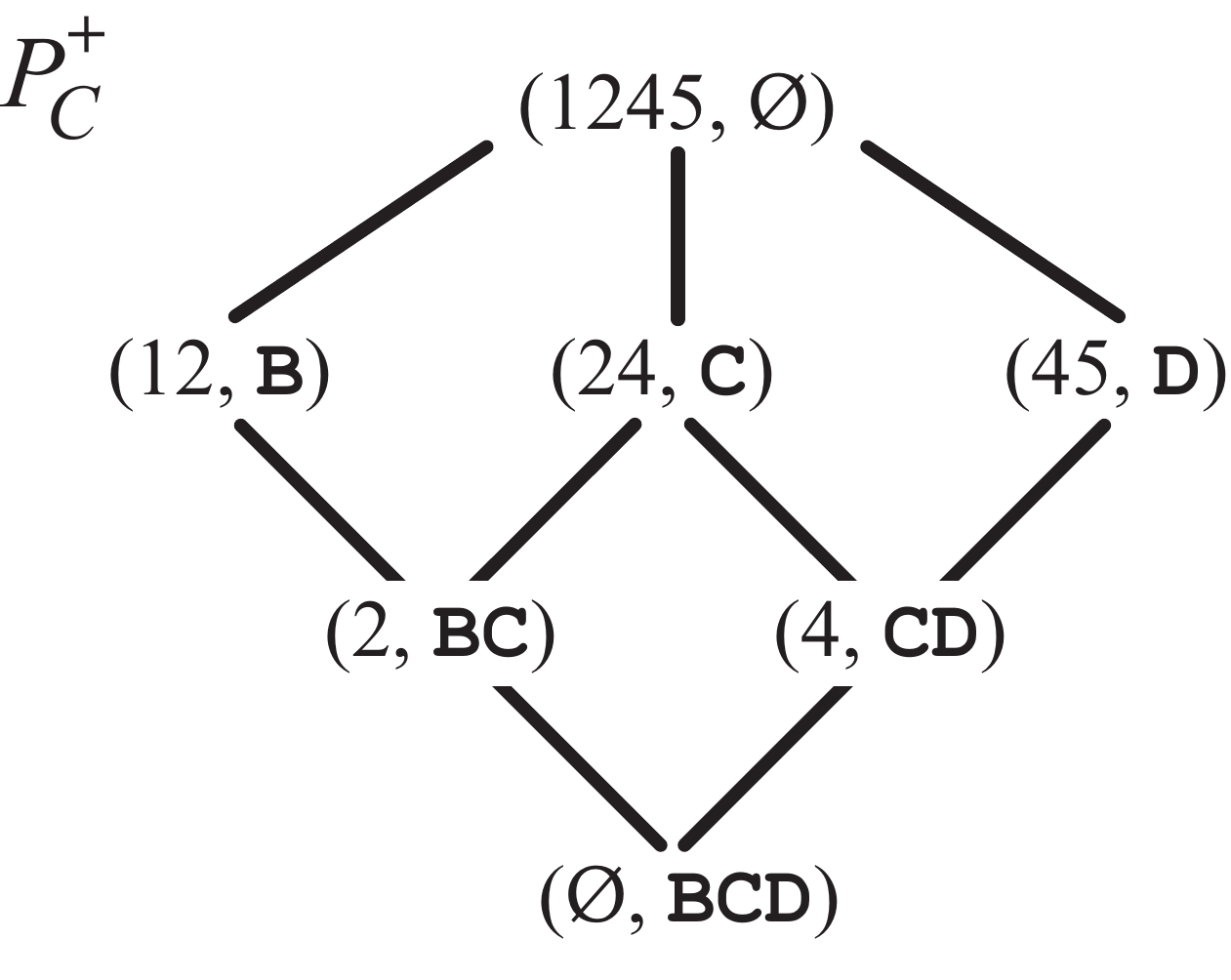}{scale=0.5}
\end{center}
\vspace*{-0.3in}
\caption[]{The lattice $\PCplus$ for the link of \authorC, as given in
  Figure~\ref{clairelink}.  (Here authors appear as integers and city
  names appear as first letter abbreviations.)  Observe that $\PCplus$
  may be viewed as a sublattice of $\PGplus$, consisting of all
  elements that include individual \#3 there, but with that individual
  removed here.  (See Figure~\ref{authorlattice} for $\PGplus$.)}
\label{clairelattice}
\end{figure}

Now let us take this reasoning one step further.  Consider an element
of $\PGplus$ corresponding to some state just prior to identification
of \authorC, for instance $(23, \cB\cC)$.  This element corresponds to
both of the first two release sequences of Figure~\ref{clairerelease}:
\authorC\ has mentioned her work regarding the travel guides for
\cityB\ and \cityC, but has not yet mentioned \cityD.  Thus there is
still some ambiguity as to her identity (it is either author \#2 or
author \#3).  In terms of Theorem~\ref{privacymultiple} on
page~\pageref{privacymultiple}, $\sigma = \{2, 3\}$, $\gamma =
\{\hcityB, \hcityC\}$, and $k=2$.

\begin{figure}[h]
\begin{center}
\qquad
\begin{minipage}{1in}{$\begin{array}{cl|cc}
\multicolumn{2}{c|}{\raisebox{0.2in}{\hbox{\LARGE $Q$}}}%
                 & \rcityB & \rcityC \\[1pt]\hline
1 & (\hauthorA)  & \one    &         \\
4 & (\hauthorD)  &         & \one    \\
\end{array}$}
\end{minipage}
\hspace*{1in}
\begin{minipage}{3.75in}
\ifig{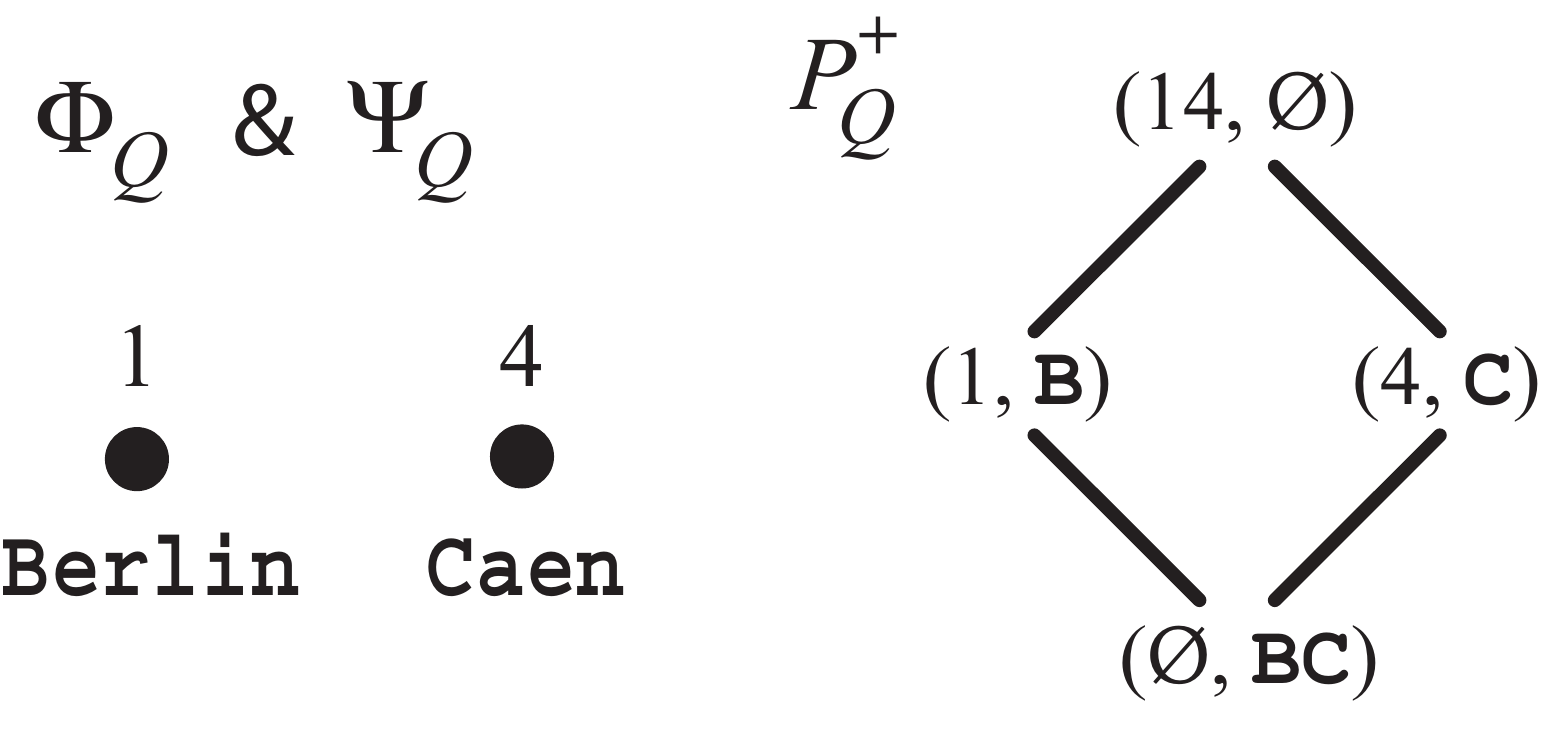}{scale=0.45}
\end{minipage}
\end{center}
\vspace*{-0.25in}
\caption[]{Relation $Q$ describes $\Lk(\dowGx, \{2, 3\})$, the
  combined link of authors \#2 and \#3 (\authorB\ and \authorC) in the
  relation of Figure~\ref{travelguides}.  These two authors have
  together collaborated with each of authors \#1 and \#4 (\authorA\
  and \authorD) but have not both together collaborated with author
  \#5 (\authorE).  The two Dowker complexes are each instances of
  $\Szero$, so essentially the same.  The corresponding lattice
  $\PQplus$ is also very simple.}
\label{BenClaireLink}
\end{figure}

\vst

Figure~\ref{BenClaireLink} shows the relation describing $\Lk(\dowGx,
\{2, 3\})$.  The Dowker complexes have $\Szero$ homotopy type, thus
satisfying the topological conditions of Theorem~\ref{privacymultiple}.
Consequently, there is no attribute inference possible in the
encompassing relation $G$ based on attributes that appear in the link
relation $Q$.  That means the order in which \authorC\ releases the
two attributes \cityB\ and \cityC\ is immaterial.  This conclusion
is consistent with the conclusion one draws upon explicitly
enumerating all release sequences, as in Figure~\ref{clairerelease}.

\subsection{Informative Attribute Release Sequences}
\markright{Informative Attribute Release Sequences, Isotropy, and Minimal Identification}
\label{iarsnarrative}

This subsection defines more precisely the idea of controlled
information release.  These definitions will help us better understand
topological holes in a relation's Dowker complexes.  Subsequently,
Section~\ref{experiments} will explore these insights with data from
the world wide web.

\begin{definition}[Attribute Release Sequence]\label{iars}
Let $R$ be a relation on $\XxY$, with both $\hspt{}X\!$ and
$Y\!$ nonempty.
\quad An \mydefem{attribute release sequence for $R$} is a nonempty
set of attributes from $Y$ released in a particular sequential order:

\vspace*{-0.15in}

$$y_1, y_2, \ldots, y_k, \quad\hbox{with $k \geq 1$}.$$

We say that the sequence has \mydefem{length} $k$.

\vspace*{0.05in}

We say that an attribute release sequence is \mydefem{informative} if

\vspace*{-0.1in}

$$y_i \not\in (\clsy)(\{y_1, \ldots, y_{i-1}\}), \quad
  \hbox{for all $1 \leq i \leq k$}.$$

\vspace*{0.03in}

(Note: for $i=1$, the requirement states that
  $y_1\not\in(\clsy)(\emptyset)=\phi_R(X)$.)

\vspace*{0.05in}

(We sometimes use the abbreviation {\tt{\char13}iars{\char13}} to mean
either  
   {\tt{\char13}informative attribute release sequence{\char13}}
or
   {\tt{\char13}informative attribute release sequences{\char13}}.)
\end{definition}

\vspace*{-0.15in}

\paragraph{Interpretation:}  When $i=1$, the argument to $\phi_R \circ
  \psi_R$ is the empty set, so the condition requires that $y_1
  \not\in \phi_R(X)$.  In other words, $y_1$ may not be any attribute
  that is shared by all individuals in $X$.  Any such attribute could
  be inferred ``for free'' in the context of relation $R$, and thus
  would not be informative.  Thereafter, the condition requires that
  any attribute to be released not be inferable from those already
  released.

\vspace*{0.1in}

\noindent We are interested in understanding the extent to which order
of release matters:

\begin{definition}[Isotropy]\label{isotropydef}
Let $R$ be a relation on $\XxY$, with both $\hspt{}X\!$ and
$\hspt{}Y\!$ nonempty.

Suppose $\emptyset \neq \gamma \subseteq Y$.

We say that $\gamma\!$ is \mydefem{isotropic} if every possible
ordering of all the elements in $\gamma$ forms an informative attribute
release sequence for $R$.
\end{definition}

\noindent We are interested in the minimal and maximal lengths of
informative attribute release sequences:

\begin{definition}[Identification and Minimal Identification]\label{DefMinIdent}
Let $R$ be a relation on $\XxY$, with both $\hspt{}X\!$ and
$\hspt{}Y\!$ nonempty.

We say that a set of attributes $\gamma \subseteq Y$
\mydefem{identifies} a set of individuals $\sigma \subseteq X$ in $R$
when $\psi_R(\gamma) = \sigma$.  \quad (We sometimes alternatively say
that $\gamma$ \mydefem{localizes to} $\sigma$ in $R$.)

\vspace*{0.1in}

We say that $\gamma$ is \mydefem{minimally identifying (for $\sigma$)}
if both the following conditions hold:

\vspace*{0.075in}

\hspace*{0.4in}\begin{minipage}{3.5in}
\begin{enumerate}
\addtolength{\itemsep}{-2pt}
\item[(i)] $\psi_R(\gamma) = \sigma$.
\item[(ii)] $\psi_R(\gamma^\prime) \supsetneq \sigma$ for every
  $\gamma^\prime \subsetneq \gamma$.
\end{enumerate}
\end{minipage}

\end{definition}

\vspace*{0.1in}

\begin{definition}[Identification Lengths]\label{DefIdentLengths}
Let $R$ be a relation on $\XxY$, with both $\hspt{}X\!$ and $\hspt{}Y$
nonempty. \ Suppose $(\sigma, \gamma) \in \PR$.
\ Define the \mydefem{fast} and \mydefem{slow attribute release lengths}
for $\sigma$ as:

\vspace*{0.075in}

\hspace*{0.7in}\begin{minipage}{5in}
$\rfast(\sigma) =
 \min{\setdefbV{\,\abs{\chi}}{\hbox{$\chi\subseteq\gamma$ and $\psi_R(\chi)=\sigma$}}}$.

\vspace*{0.075in}

$\rslow(\sigma) = \max{\setdefbV{\,k}{\hbox{$y_1, \ldots,
  y_k$ is an iars for $R$ and $\psi_R\big(\{y_1, \ldots, y_k\}\big) = \sigma$}}}$.

\end{minipage}
\end{definition}

\vspace*{0.1in}

An argument similar to that in Appendix~\ref{nphardness} shows that
the following problem is \np-complete: \ Given $R$, $\sigma$, and $k$,
is there some minimally identifying $\gamma$ for $\sigma$ with
$\abs{\gamma} \leq k$?

\subsection{Isotropy, Minimal Identification, and Spheres}

There is no requirement in Definition~\ref{iars} that an informative
attribute release sequence be a simplex in $\dowy$.  (Indeed, when
working with links of individuals, it can be useful to create
informative attribute release sequences that are not simplices in the
link, thereby identifying the given individuals in the encompassing
relation, as per Lemma~\ref{interplocal} on
page~\pageref{interplocal}.)  However, it is always the case that any
inconsistency arises only with the last attribute released:

\begin{lemma}[Almost a Simplex]
Let $R$ be a relation on $\XxY$, with both $\hspt{}X\!$ and
$\hspt{}Y\!$ nonempty.

\vspace*{0.02in}

Suppose $\{y_1, \ldots, y_k\}$ is an informative attribute release
sequence for $R$.

\vspace*{0.02in}

Then $\{y_1, \ldots, y_{k-1}\} \in \dowy$.
\end{lemma}

\vspace*{-0.1in}

\begin{proof}
If $\{y_1, \ldots, y_{k-1}\} \not\in \dowy$, then $(\phi_R \circ
\psi_R)(\{y_1, \ldots, y_{k-1}\}) = \phi_R(\emptyset) = Y$.  Since $y_k
\in Y$, this contradicts the requirement of Definition~\ref{iars}.
\end{proof}

\vspace*{0.05in}

Consequently, a nonempty set of attributes $\gamma \subseteq Y$, with
$\gamma\not\in\dowy$, is isotropic if and only if it is a minimal
nonface of $\dowy$. \ We can view such an isotropic $\gamma$ as
minimally identifying for $\emptyset$.

\vspace*{0.1in}

When a nonempty set of attributes $\gamma$ {\em is\hspt} a simplex in
$\dowy$, then being isotropic is again equivalent to being minimally
identifying, now for some nonempty set of individuals $\sigma$.  \
Moreover, topologically, we can again characterize this isotropy as a
sphere, appearing via a restricted link:

\clearpage

\begin{definition}[Restricted Link]\label{restrictedlink}
Let $R$ be a relation on $\XxY$, with both $\hspt{}X\!$ and
$\hspt{}Y\!$ nonempty.

\vspace*{0.02in}

Suppose $\sigma \in \dowx$ and $\gamma \subseteq \phi_R(\sigma)$.

\vspace*{0.02in}

Define relation $Q(\sigma, \gamma)$ as follows:

\vspace*{-0.1in}

$$Q(\sigma, \gamma) \;=\; R\,|_{\tX \times \gamma}, \quad \hbox{with} \quad
            \tX = \bigunion_{y \in \gamma}\sdiff{X_y}{\sigma}.$$

\vspace*{-0.05in}

The Dowker complexes are defined in the standard way, except for these special cases:

\vspace*{0.05in}

\hspace*{0.2in}\begin{minipage}{6in}

If $\sigma=X$, we let $\dqx$ and $\dqy$ be instances of the void complex $\emptyset$.

If $\sigma \neq X$ but $\tX = \emptyset$, we let $\dqx$ and $\dqy$
be instances of the empty complex $\{\emptyset\}$.
\end{minipage}

\vspace*{0.1in}

We say that $Q(\sigma, \gamma)$ \mydefem{models the link of $\sigma$
restricted to $\gamma$}.

\end{definition}

\begin{figure}[t]
\begin{center}
\begin{minipage}{1.5in}{$\begin{array}{cl|cc}
\multicolumn{2}{c|}{\raisebox{0.2in}{\hbox{\LARGE $\Qp$}}}%
                 & \rcityB & \rcityD \\[1pt]\hline
1 & (\hauthorA)  & \one    &         \\
2 & (\hauthorB)  & \one    &         \\
4 & (\hauthorD)  &         & \one    \\
5 & (\hauthorE)  &         & \one    \\
\end{array}$}
\end{minipage}
\hspace*{0.3in}
\begin{minipage}{3.75in}
\ifig{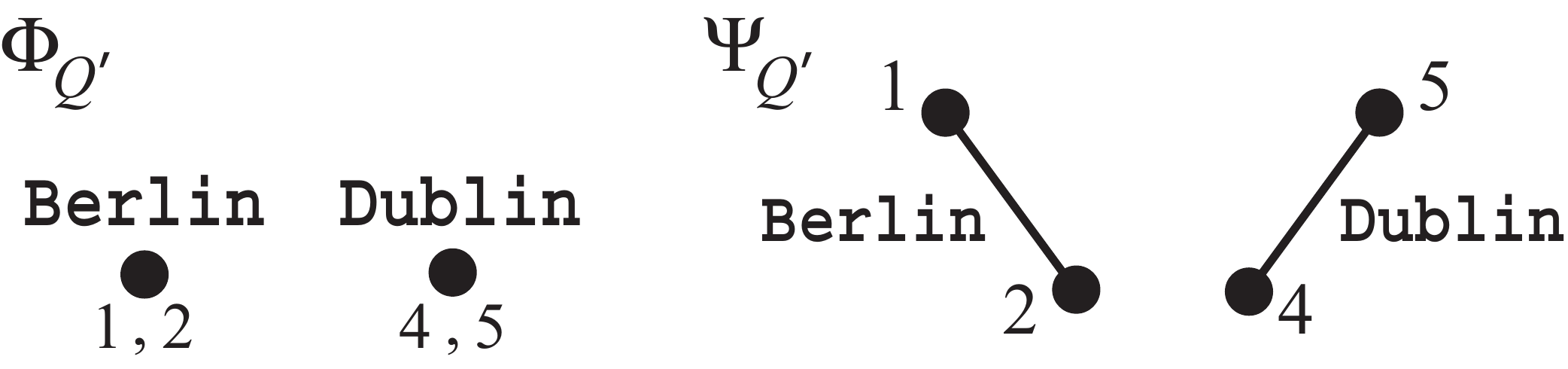}{scale=0.425}
\end{minipage}
\end{center}
\vspace*{-0.15in}
\caption[]{Relation $\Qp=Q(\sigma,\gamma)$, for the book authorship
  example of Figure~\ref{travelguides}, with $\sigma = \{3\}$ and
  $\gamma = \{\hcityB, \hcityD\}$.  Relation $\Qp$ describes the link
  of author \#3 (\authorC) restricted to the attribute set $\{\hcityB,
  \hcityD\}$.  See Figure~\ref{clairelink} for the whole link relation.
  (Each maximal simplex in the Dowker complexes is again labeled with
  its generating individuals or attribute.)}
\label{clairelinkrestricted}
\end{figure}

\vspace*{-0.1in}

\paragraph{Comments:}  Although the previous definition looks similar
to that for $\lk(\dowx, \sigma)$ on page~\pageref{linksigma}, there
are some differences:
(a) Here, we require that $\sigma$ be a simplex in $\dowx$.
(b) Here, we do {\em not}\hspt\ assume $\gamma = \phi_R(\sigma)$,
merely $\gamma \subseteq \phi_R(\sigma)$.
(c) When $\sigma=X\in\dowx$, the current definition creates void
complexes, whereas Definition~\ref{linksigma} on
page~\pageref{linksigma} creates empty complexes.
(d) Finally, when $\sigma\neq{X}$ but $\gamma=\emptyset$, the current
definition creates empty complexes rather than void complexes.
\\[2pt]
Interpretation: When $\sigma\in\dowx$ and $\sigma\neq{X}$,
$\,Q(\sigma, \gamma)$ models those simplices of $\Lk(\dowx, \sigma)$
that are witnessed by attributes in $\gamma$, plus the empty simplex.

\vspace*{0.05in}

\newcounter{IsotropyThm}
\setcounter{IsotropyThm}{\value{theorem}}
\begin{theorem}[Isotropy = Minimal Identification = Sphere]\label{isotropyThm}
Let $R$ be a relation and suppose $\emptyset\neq\gamma\in\dowy$.
\ Let $\sigma = \psi_R(\gamma)$.
Then the following four conditions are equivalent:

\vspace*{0.1in}

\hspace*{0.4in}\begin{minipage}{3.5in}
\begin{enumerate}
\addtolength{\itemsep}{-1pt}
\item[(a)] $\gamma$ is isotropic.
\item[(b)] $\gamma$ is minimally identifying (for $\sigma$).
\item[(c)] $\dqx \;\homot\; \Skt$, with $k = \abs{\gamma}$.
\item[(d)] $\dqy \;=\; \partial(\gamma)$.
\end{enumerate}
\end{minipage}
\end{theorem}

\vspace*{0.05in}

\noindent See Appendix~\ref{isotropyApp} for a proof.

\vspace*{-0.1in}

\paragraph{Collaboration Example Revisited:}  To illustrate
Theorem~\ref{isotropyThm}, consider again the example of
Figure~\ref{travelguides}.  Recall that together the travel guides for
\cityB\ and \cityD\ identify \authorC.  Indeed, $\{\hcityB, \hcityD\}$
is a minimally identifying set of books for \authorC.  It is
isotropic, as Figure~\ref{clairerelease} shows.
Figure~\ref{clairelinkrestricted} depicts the link of \authorC\
restricted to $\{\hcityB, \hcityD\}$, modeled by relation $Q^\prime$.
Observe that $\dowqpy = \partial(\{\hcityB, \hcityD\})$ and that
$\dowqpx\homot\Szero$, as the theorem asserts.

\subsection{Poset Lengths and Information Release}
\markright{Poset Lengths and Information Release}
\label{informativerelease}

We have seen how minimal identification appears topologically via
spheres.  Spheres are isotropic so perhaps it is not surprising that
they encode isotropic attribute release sequences.  We cannot
therefore expect a spherical characterization for the problem of
finding a maximally long informative attribute release sequence.
Instead, we find an answer in the combinatorial structure of the
doubly-labeled poset $\PR$ and its lattice $\PRplus$.  We summarize
the key results below.  For proofs, see Appendix~\ref{posetchains}.

\newcounter{ChainToIarsLemma}
\setcounter{ChainToIarsLemma}{\value{theorem}}
\begin{lemma}[Informative Attributes from Maximal Chains]\label{chaintoiars}
Let $R$ be a relation on $\XxY$, with both $\hspt{}X\!$ and
$\hspt{}Y\!$ nonempty.
Suppose $\{(\sigma_k, \gamma_k) < \cdots < (\sigma_1, \gamma_1) <
(\sigma_0, \gamma_0)\}$, with $k\geq 1$, is a maximal chain in
$\PRplus$.

\vso

Define $y_1, \ldots, y_k$ by selecting some $y_i \in
\gamma_i\setminus\gamma_{i-1}$, for each $i=1, \ldots, k$.

\vso

Then $y_1, \ldots, y_k$ is an informative attribute release sequence for
$R$.

\vso

Moreover, $(\clsy)(\{y_1, \ldots, y_i\}) = \gamma_i$, for each $i = 0, 1,
\ldots, k$.
\end{lemma}

(Notes:
\ (a) For a maximal chain in $\PRplus$, $\gamma_k = Y$ and $\sigma_0 = X$.
\ (b) The hypothesis $k\geq 1$ excludes any relation $R$ for which
$\zeroR = \oneR$.)

\vspace*{0.1in}

Lemma~\ref{chaintoiars} implies that every nontrivial maximal chain in
the doubly-labeled poset associated with a relation gives rise to an
informative attribute release sequence that tracks the chain.
\ A partial converse holds as well:

\newcounter{IarsToChainLemma}
\setcounter{IarsToChainLemma}{\value{theorem}}
\begin{lemma}[Chains from Informative Attributes]\label{iarstochain}
Let $R$ be a relation on $\XxY$, with both $\hspt{}X$ and $\hspt{}Y\!$
nonempty.
Suppose $y_1, \ldots, y_k$ is an informative attribute release sequence
for $R$, with $k \geq 1$.

\vso

Let $\gamma_i = (\clsy)(\{y_1, \ldots, y_i\})$ and $\sigma_i =
\psi_R(\gamma_i)$, for $i=1, \ldots, k$.

\vso

Then $\{(\sigma_k, \gamma_k) < \cdots < (\sigma_1, \gamma_1) < (X,
\gamma_0)\}$ is a (not necessarily maximal) chain in $\PRplus$, with
$\gamma_0 = \phi_R(X)$.
\end{lemma}

Consequently, one can obtain all informative attribute release
sequences as subsequences of those constructed from maximal chains in
$\PRplus$.

\vspace*{-0.1in}

\paragraph{Comment about ``length'':}\ The {\em length} $\ell(P)$ of a
poset $P$ is defined to be one less than the number of elements
comprising a longest chain in the poset \cite{tpr:wachs}.
The {\em length} of an informative attribute release sequence $y_1,
\ldots, y_k$ is $k$.
These definitions match much like the dimension of a simplex is one
less than the number of its elements. \ Consequently, one obtains:

\vspace*{0.1in}

\begin{corollary}[Maximal Length]\label{MaxLength}
The maximum length of an informative attribute release sequence for a
nonvoid relation $R$ is $\ell(\PRplus)$.
\hspt(If $R$ has no iars, then the maximum length is $0$.)
\end{corollary}

\begin{corollary}[Maximal Identification Length]\label{MaxIdent}
Suppose $R$ is a relation such that no attribute is shared by all
individuals.  For any $(\sigma, \gamma) \in \PR$, $\rslow(\sigma) =
\ell(P_{Q(\sigma,\gamma)}) + 2$.
\end{corollary}

\paragraph{Collaboration Example Re-Revisited:}  Returning again to
the travel guide example, observe in Figure~\ref{authorlattice} that
$\ell(\PGplus)=4$.  This tells us, by Corollary~\ref{MaxLength}, that
a longest informative attribute release sequence for relation $G$
contains four attributes.  Indeed, we can pick three attributes to
identify an individual, and then a fourth to form an inconsistency.
How do we know that we can choose three attributes informatively to
identify an individual?  See, for example, $\Lk(\dowx, \hauthorC)$ in
Figure~\ref{clairelink}, with associated lattice $\PCplus$ in
Figure~\ref{clairelattice}.  In this case, $\ell(\PC) + 2 =
\ell(\PCplus) = 3$.  Moreover, by the construction of
Lemma~\ref{chaintoiars}, one can read off four different such
informative sequences, namely the first four sequences appearing in
Figure~\ref{clairerelease}.

We thus see that $\rslow(\{\hauthorC\}) = 3$, and as we have seen
previously, $\rfast(\{\hauthorC\}) = 2$.  In other words, if \authorC\
has control over how to release information, she can draw out
identification for three books, while the fastest anyone can identify
her is via two books.

\vspace*{0.1in}

\begin{figure}[h]
\begin{center}
\quad
\begin{minipage}{1.5in}{$\begin{array}{c|cccc}
\hbox{\large $T$} & \ta  & \tb  & \tc  & \td  \\[2pt]\hline
              1   & \one & \one & \one &      \\
              2   &      & \one & \one & \one \\
              3   & \one &      & \one & \one \\
              4   & \one & \one &      & \one \\
\end{array}$}
\end{minipage}
\hspace*{0.75in}
\begin{minipage}{3in}
\ifig{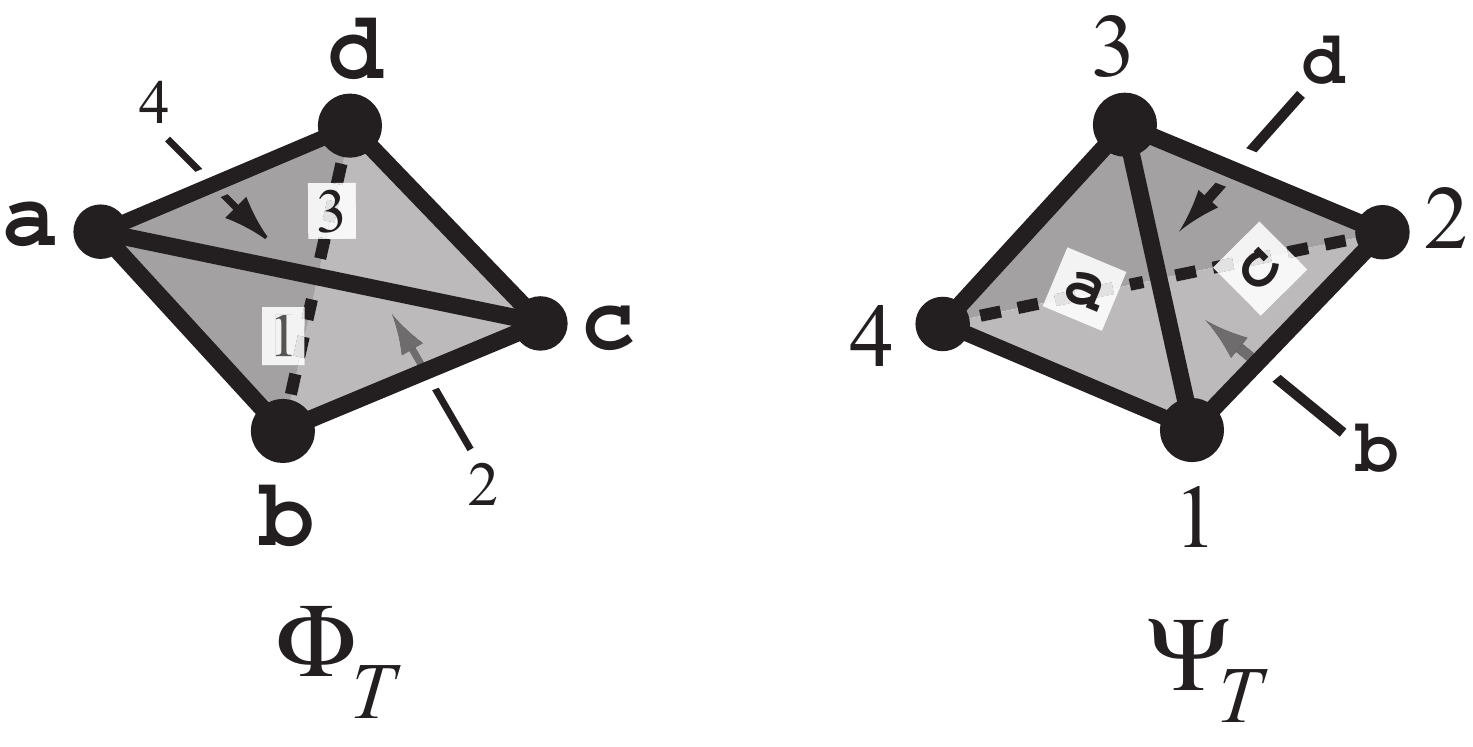}{scale=0.5}
\end{minipage}
\end{center}
\vspace*{-0.25in}
\caption[]{Relation $T$ describes four individuals with four
  attributes, with Dowker complexes that are boundary complexes of
  tetrahedra, meaning they have homotopy type $\Stwo$.}
\label{tetrahedra2}
\end{figure}

In contrast, consider the tetrahedral relation of
Figure~\ref{tetrahedra2}.  The Dowker complexes are boundary
complexes, so we know that no attribute or association inference is
possible.  This is evident from the lattice $\PTplus$ depicted in
Figure~\ref{tetralattice2} as well.  It has length 4, just as did the
travel guide lattice, but the inference structure is now different.
For any $(\sigma, \gamma) \in \PT$, with $Q=Q(\sigma, \gamma)$
modeling $\Lk(\dowTx, \sigma)$ on attributes $\gamma$, we see that
$\dowqy = \partial(\gamma)$ and thus that $\ell(\PQplus) =
\ell(\PQ) + 2 = \abs{\gamma}$.  This tells us, by
Theorem~\ref{isotropyThm} and Corollary~\ref{MaxIdent}, that
$\rfast(\sigma) = \rslow(\sigma) = \abs{\gamma}$, as one would expect
in an inference-free world.
For a specific instance, Figure~\ref{tetralink} depicts
$Q=Q(\{3\},\{\ta,\tc,\td\})$ along with $Q$'s Dowker complexes and
the lattice $\PQplus$.

\begin{figure}[h]
\begin{center}
\ifig{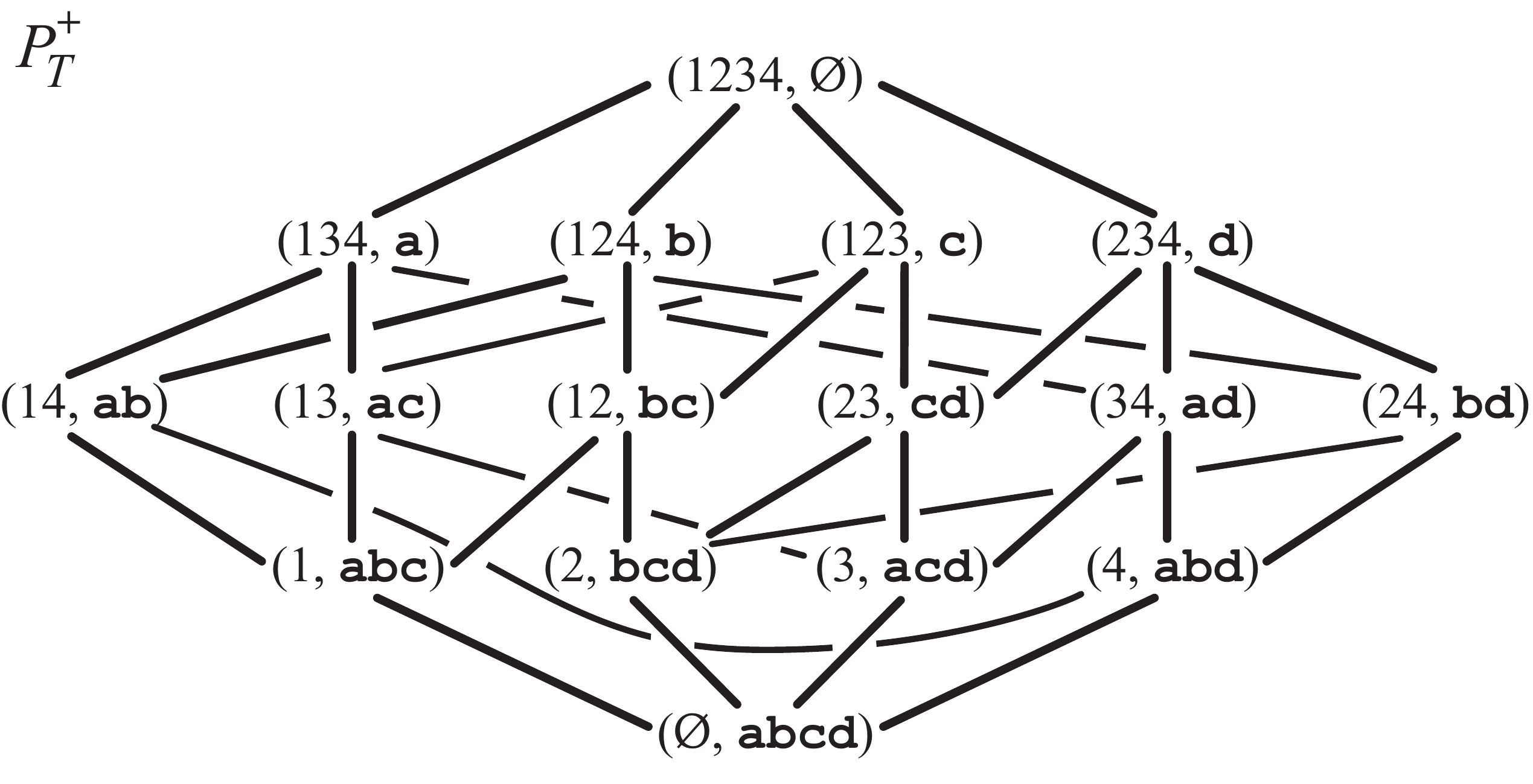}{scale=0.5}
\end{center}
\vspace*{-0.2in}
\caption[]{The lattice $\PTplus$ for the tetrahedral relation of
  Figure~\ref{tetrahedra2}.  Each element is an ordered pair of sets
  $(\sigma, \gamma)$ such that $\sigma = \psi_T(\gamma)$ and $\gamma =
  \phi_T(\sigma)$.  (We have elided commas and braces in sets, for
  ease of viewing.)  This lattice is isomorphic to the Boolean lattice
  on four atoms, consistent with the fact that $T$ preserves both
  attribute and association privacy.  If one removes the top and
  bottom elements, the remaining poset $\PT$ has $\Stwo$ homotopy
  type.}
\label{tetralattice2}
\end{figure}

\subsection{Hidden Holes}
\markright{Hidden Holes}
\label{hiddenholes}

We saw via Theorem~\ref{isotropyThm} that whenever a nonempty set of
attributes $\gamma$ minimally identifies some set of individuals
$\sigma$, then the link of $\sigma$, restricted to those simplices
that are witnessed by attributes in $\gamma$, defines a sphere in both
Dowker complexes.  It is a topological hole.

All sets of individuals that are identifiable in some way, in other
words, that appear in the doubly-labeled poset $\PR$ of a relation,
must be minimally identifiable in some way.  That suggests there must
be holes everywhere in a relation's Dowker complexes, and yet we do
not see many holes.  What is going on?

The answer is that the restricted link construction $Q(\sigma,
\gamma)$ focuses on a particular subrelation, thereby
exposing/highlighting a potential hole.  The hole could in fact be
hidden, that is, filled-in by the encompassing relation.  For
instance, we saw that relation $Q$ of Figure~\ref{tetralink} defines
an $\Sone$ hole.  If $Q$ happened to be a subrelation of relation $R$
as in Figure~\ref{tetralinkfilled}, then $Q$ would not appear as a
hole when viewed in $R$, merely a boundary.

Notice that the lattice $\PRplus$ is isomorphic to the lattice
$\PQplus$.  The difference is that for every lattice element $(\sigma,
\gamma)$, the set of individuals $\sigma$ includes $3$ in $\PRplus$
but not in $\PQplus$.  Consequently, the bottom element
$(3,\ta\tc\td)$ of $\PRplus$ is actually an element of the poset
$\PR$, meaning $\Delta(\PR)$ is a cone, hence contractible.  In
contrast, the poset $\PQ$ does not contain the bottom element
$(\emptyset,\ta\tc\td)$ of $\PQplus$ and so $\Delta(\PQ)$ has $\Sone$
homotopy type.

\paragraph{Aside:} Why not always focus on a relation's lattice rather
than its doubly-labeled poset?  Because the lattice is always
contractible.  Any informative topology lies in the poset.  See
\cite{tpr:wachs}.

\paragraph{Conclusion:} Even though $R$ is contractible, it offers the
same choices for informative attribute release sequences as does $Q$.
More generally, the analysis of this subsection suggests that one look
for potential holes in {\em subrelations\,} of a given relation.
Looking at links is one way to focus on subrelations.  Removing
individuals or attributes that represent cone apexes is another, as we
just saw.  More generally, any simplicial cycle may define a useful
hole even though the hole appears to be filled-in.  So long as one can
remove any coboundary of that cycle, {\em by restricting the relation
to a subrelation without destroying the cycle}, the cycle is
informational.  In particular, it offers opportunities for informative
attribute release sequences, as the next subsection makes precise.

\begin{figure}[h]
\vspace*{-0.15in}
\begin{center}
\begin{minipage}{1in}{$\begin{array}{c|ccc}
\hbox{\large $Q$} & \ta  & \tc  & \td  \\[2pt]\hline
              1   & \one & \one &      \\
              2   &      & \one & \one \\
              4   & \one &      & \one \\
\end{array}$}
\end{minipage}
\hspace*{0.3in}
\begin{minipage}{2in}
\ifig{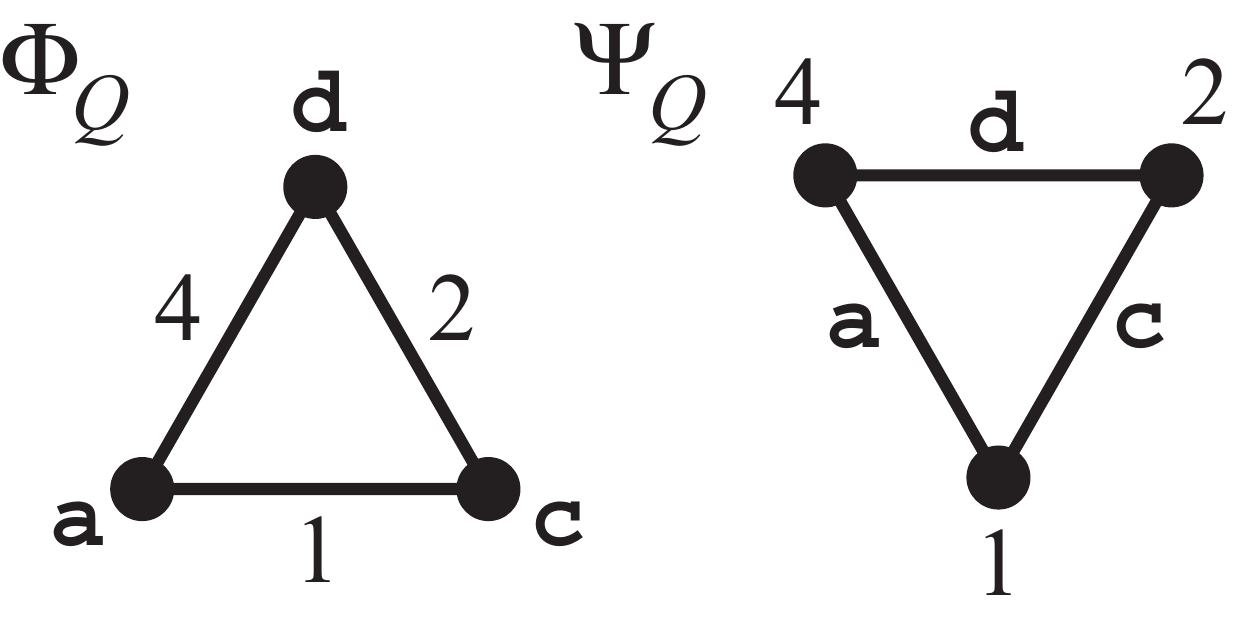}{scale=0.4}
\end{minipage}
\begin{minipage}{2in}
\hspace*{0.3in}
\ifig{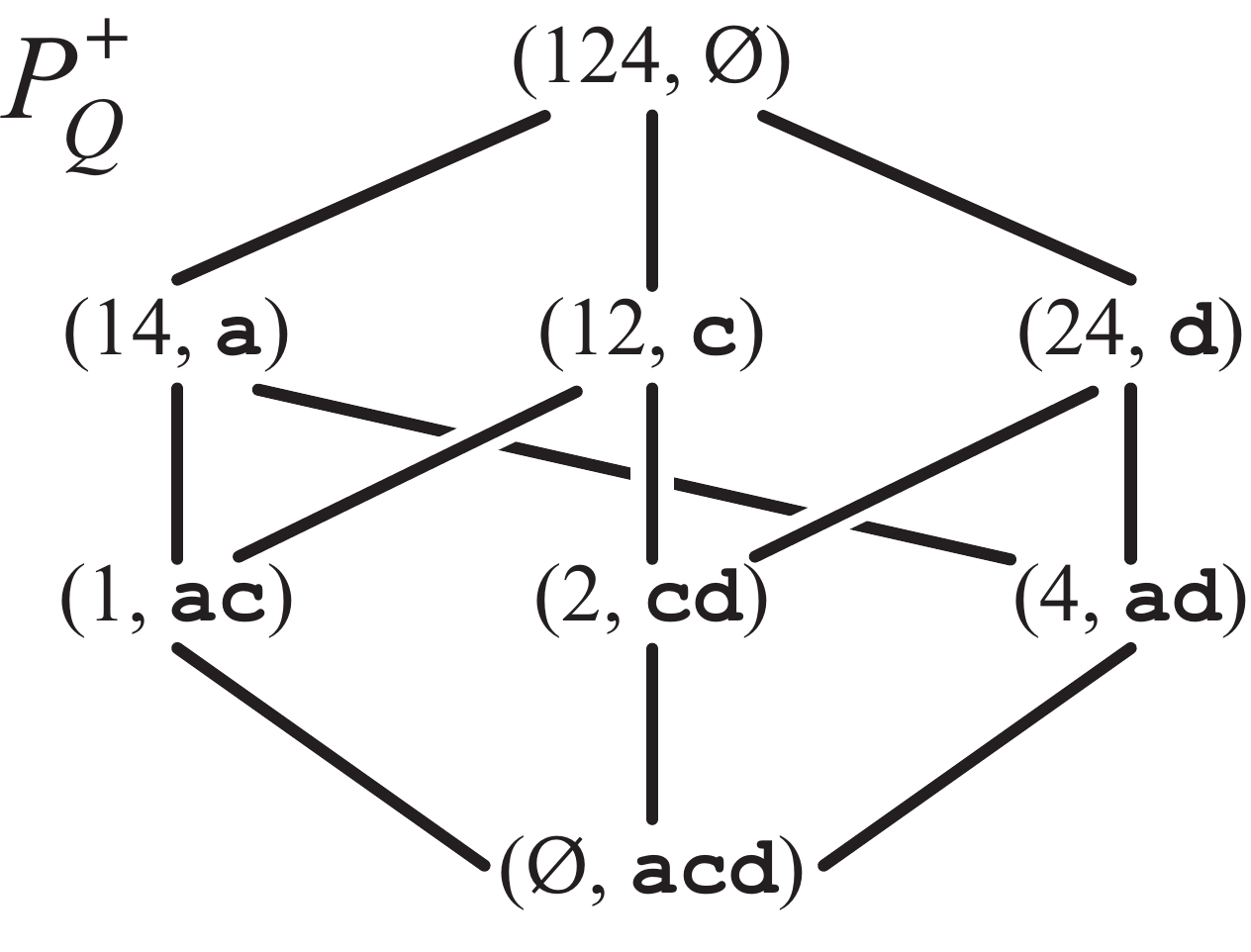}{scale=0.375}
\end{minipage}
\end{center}
\vspace*{-0.225in}
\caption[]{Relation $Q$ models $\Lk(\dowTx, 3)$, with $T$ as in Figure~\ref{tetrahedra2}.}
\label{tetralink}
\end{figure}

\begin{figure}[h]
\vspace*{0.2in}
\begin{center}
\begin{minipage}{1in}{$\begin{array}{c|ccc}
\hbox{\large $R$} & \ta  & \tc  & \td  \\[2pt]\hline
              1   & \one & \one &      \\
              2   &      & \one & \one \\
              3   & \one & \one & \one \\
	      4   & \one &      & \one \\
\end{array}$}
\end{minipage}
\hspace*{0.3in}
\begin{minipage}{2in}
\ifig{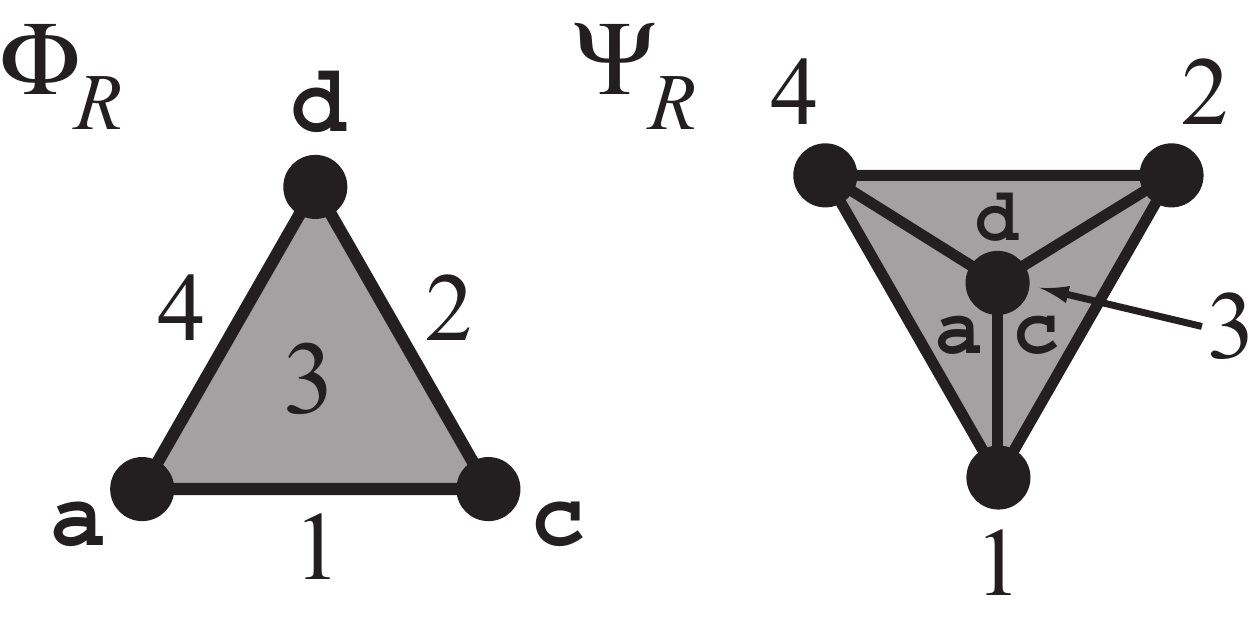}{scale=0.4}
\end{minipage}
\begin{minipage}{2in}
\hspace*{0.3in}
\ifig{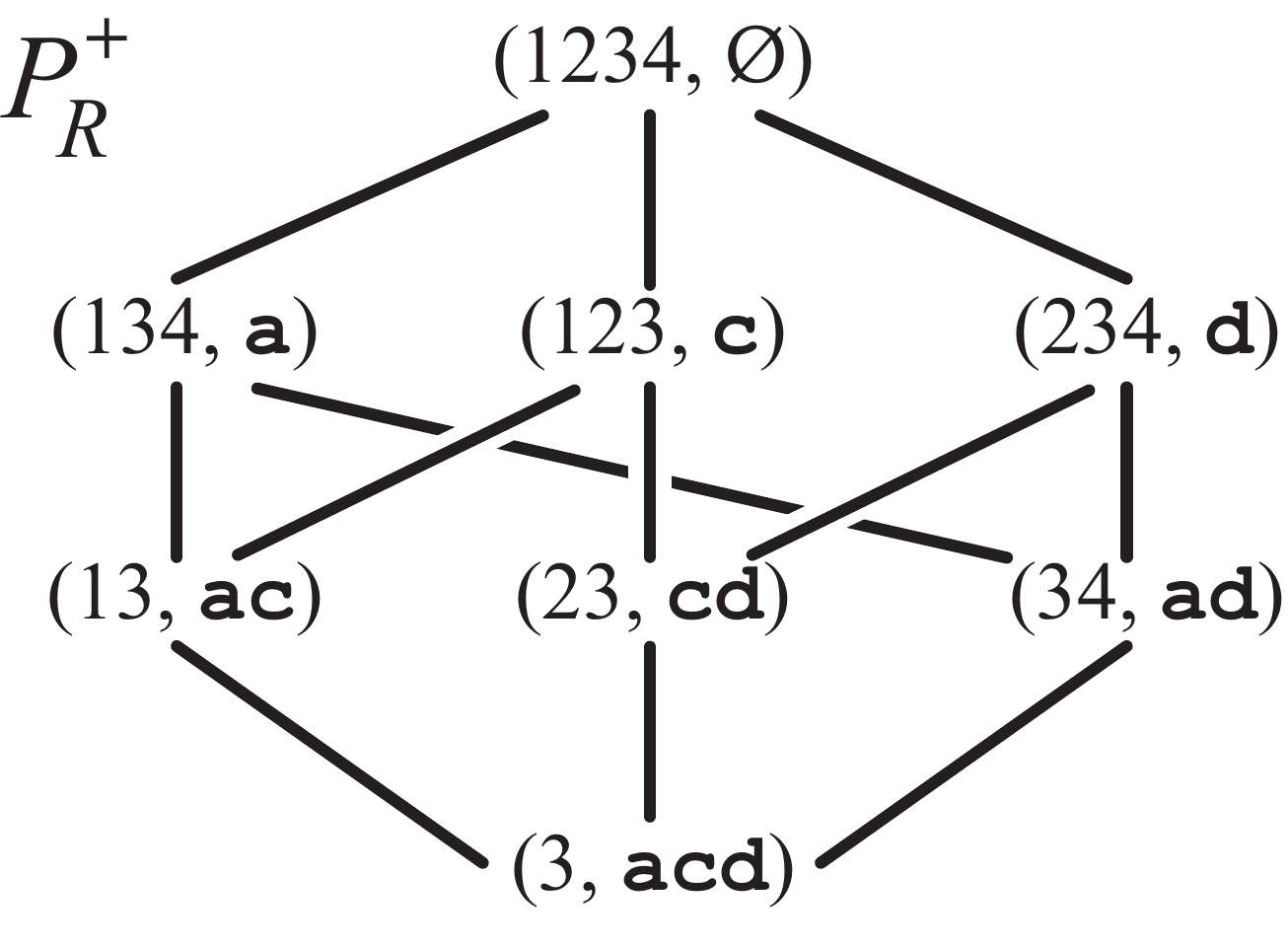}{scale=0.375}
\end{minipage}
\end{center}
\vspace*{-0.225in}
\caption[]{Relation $R$ fills in the hole of relation $Q$ from
  Figure~\ref{tetralink}.  It is still true that $Q$ models a link,
  namely $\Lk(\dowx, 3)$.  Relations $R$ and $Q$ have the same lattice
  structure, but the bottom element of $\PRplus$ defines the set of
  individuals $\{3\}$, whereas the bottom element of $\PQplus$ defines
  the empty set.  Thus relation $R\hspt$ defines a contractible poset
  for $\PR$, whereas relation $Q$ defines an $\Sone$ hole for $\PQ$.}
\label{tetralinkfilled}
\end{figure}

\subsection{Bubbles are Lower Bounds for Privacy}
\markright{Bubbles are Lower Bounds for Privacy}
\label{bubbles}

We have seen minimal identifiability characterized by holes, via
Theorem~\ref{isotropyThm}.  The previous subsections make clear that
the topological characterization of $\rslow$ is not so direct.  In
this subsection we establish a sufficient condition.  We will see that
holes provide lower bounds for $\rslow$.  We will focus on a relation
and its links, but these results apply more generally to any hidden
holes made visible by focusing on subrelations, as suggested in the
previous subsection.

\vspace*{0.1in}

\noindent The connection between a relation's poset $\PR$ and its
lattice $\PRplus$ suggests the following:

\begin{definition}[Almost a Join-Based Lattice]\label{almostlattice}
  \, Let $P$ be a finite poset.  We say that $P$ is \,\mydefem{almost}
  \mydefem{a join-based lattice} if adjoining a new topmost element
  $\topone$ means $P \union \{\topone\}$ is a join semi-lattice.
\end{definition}

\paragraph{Comments:}  (a) We adjoin a new $\topone$ even if $P$
  already has a unique top (i.e., maximal) element.
  \ (b) Since $P$ is finite, if $P$ is almost a join-based lattice,
  then if we adjoin both a new topmost element $\topone$ and a new
  bottommost element $\botzero$, the result will be a lattice.  See
  also \cite{tpr:wachs}.

\vspace*{0.1in}

\noindent This definition leads to the following result (for a proof, see
  Appendix~\ref{manylongchains}):

\newcounter{manychainThm}
\setcounter{manychainThm}{\value{theorem}}
\begin{theorem}[Many Maximal Chains]\label{manychains}
Let $P$ be almost a join-based lattice.
Suppose $P$ has reduced integral homology in dimension $k \geq 0$,
that is, $\widetilde{H}_k(\Delta(P); \Z) \neq 0$.

Then there are at least $(k+2)!$ maximal chains in $P$ of length at
least $k$.
\end{theorem}

\vspace*{-0.15in}

\paragraph{Interpretation:}  The theorem says that a homology hole
acts like a spherical hole, from the perspective of producing
informative attribute release sequences.  Consider again the
tetrahedral relation of Figure~\ref{tetrahedra2}.  The Dowker
complexes form two-dimensional spherical holes, so $k=2$ and $(k+2)! =
24$.  The poset $\PT$ is the proper part of the lattice shown in
Figure~\ref{tetralattice2}, that is, all the elements except the
topmost and bottommost.  There are indeed 24 different chains of
length $2$, i.e., containing three elements, in $\PT$.

These chains represent the 24 different ways in which one might start
at a vertex of one of the Dowker complexes, walk from that vertex to
the middle of an incident edge, then walk from the middle of that edge
to the centroid of an encompassing triangle.  For instance: the walk
from the vertex $\{\ta\}$ to the edge $\{\ta, \tc\}$ to the triangle
$\{\ta, \tc, \td\}$ in $\dowTy$.  One can think of this walk as
sequential acquisition of attribute information about an individual in
a particular order.  The order may perhaps be determined by chance or
perhaps by an individual purposefully releasing information in a
particular order.  Once (and only once) one has arrived at the
centroid of the triangle, one has identified the individual uniquely
(in this case, as individual \#3).

With that observation, we finally see how the global geometry/topology
of the Dowker complexes, as encoded in their doubly-labeled poset,
affects inference, beyond the local simplicial collapses of the
closure operators.  We will presently formalize this insight via two
corollaries to Theorem~\ref{manychains}.

\begin{figure}[h]
\begin{center}
\begin{minipage}{1in}{$\begin{array}{c|ccc}
\hbox{\large $R$} & \ta  & \tb  & \tc  \\[2pt]\hline
              1   & \one & \one & \one \\
              2   & \one & \one & \one \\
              3   & \one & \one & \one \\
\end{array}$}
\end{minipage}
\hspace*{0.3in}
\begin{minipage}{3.75in}
\hspace*{0.2in}
\ifig{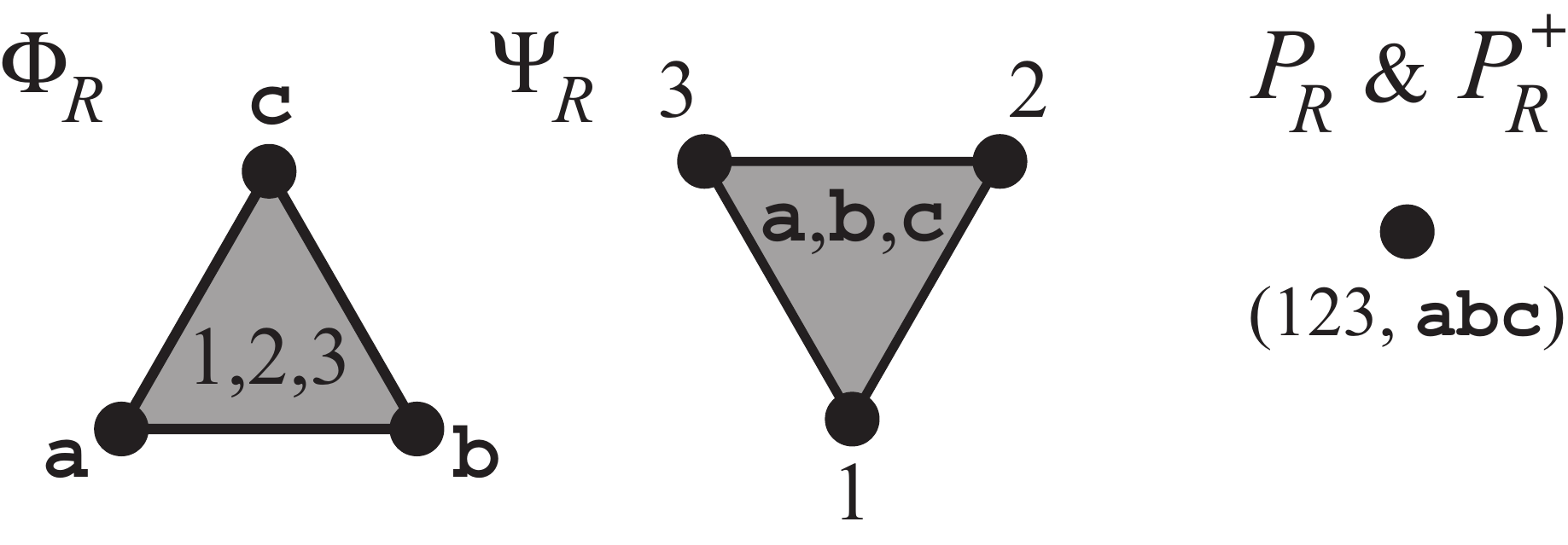}{scale=0.5}
\end{minipage}
\end{center}
\vspace*{-0.25in}
\caption[]{Relation $R$ describes three individuals all of whom have
  the exact same three attributes.  The Dowker complexes are both
  triangles, but the poset $\PR$ is a single point.  This single point
  captures the indistinguishability of the individuals and the
  attributes.  In fact, $\PRplus=\PR$, meaning one can infer
  everything from nothing (in the context of relation $R$).}
\label{trivialthree}
\end{figure}

We caution that the dimension of a simplex in a Dowker complex is not
meaningful in and of itself, since the simplex may collapse under the
closure operators.  (Consider the example of
Figure~\ref{trivialthree}, in which the Dowker complexes are fully
filled-in triangles, but the doubly-labeled poset is a single point.)
Instead, the length of chains in a relation's poset is significant.
Holes prevent these chains from being short, summarized as follows
(proofs appear in Appendix~\ref{manylongchains}):

\newcounter{holeinferCor}
\setcounter{holeinferCor}{\value{theorem}}
\begin{corollary}[Holes Reduce Inference]\label{holesreduceinference}
Let $R$ be a nonvoid relation.  Suppose $\PR$ has reduced integral
homology in dimension $k \geq 0$.  Then there are at least $(k+2)!$
maximal chains in $\PR$ of length at least $k$.
\end{corollary}

\newcounter{holerecCor}
\setcounter{holerecCor}{\value{theorem}}
\begin{corollary}[Holes Defer Recognition]\label{holesdeferrecog}
Let $R$ be a nonvoid relation and let $(\sigma, \gamma) \in \PR$.

\vst

Define $Q = Q(\sigma,\gamma)$ as per Definition~\ref{restrictedlink}
and recall Definition~\ref{DefIdentLengths}, from
pages~\pageref{DefIdentLengths}--\pageref{restrictedlink}.

\vst

Suppose $\PQ$ is well-defined and has reduced integral homology in
dimension $k \geq 0$.

\vst

Then there are at least $(k+2)!$ distinct informative attribute
release sequences $y_1, \ldots, y_\ell$ for $R$, each with $\ell \geq
k+2$, such that $\psi_R(\{y_1, \ldots, y_\ell\}) = \sigma$.
Consequently, $\rslow(\sigma) \geq k+2$.

\end{corollary}

\paragraph{Comment:}  Since $(\sigma, \gamma) \in \PR$ and by the
assumptions about $\PQ$, relation $Q(\sigma,\gamma)$ models the link
$\Lk(\dowx, \sigma)$.

\paragraph{Terminology:} \ Here and elsewhere, the term
{\tt{\char13}distinct{\char13}} means {\tt{\char13}different{\char13}}
or {\tt{\char13}distinguishable{\char13}}, as determined by the given
context.  For instance, the two sequences \ta, \tb, \tc\ \ and \ \ta,
\tc, \tb\ \ are distinct sequences even though the underlying set is
$\{\ta, \tb, \tc\}$ in both cases.

\paragraph{Collaboration Example Once Again:}  The Dowker complexes
for the travel guide example of Figure~\ref{travelguides} have $\Sone$
homotopy type, meaning $\PG$ has homology in dimension $k=1$.
Corollary~\ref{holesreduceinference} therefore says that there are at
least 6 maximal informative attribute release sequences in $\PG$.
Being maximal, each such sequence must identify some author, since
each author is uniquely identifiable via relation $G$.  In fact, we
saw that there were 4 different maximal informative attribute release
sequences for identifying any one author.  Since there are 5 authors,
$\PG$ actually contains at least 20 distinct maximal informative
attribute release sequences.  Indeed, one can readily see, via
$\PGplus$ in Figure~\ref{authorlattice} on
page~\pageref{authorlattice}, that $\PG$ contains exactly 20 maximal
informative attribute release sequences.

\vspace*{0.05in}

Can we find these 20 sequences via our corollaries?  Not by looking at
individual authors, since, as we saw via Figure~\ref{clairelink}, the
link of any one author is contractible, meaning that
Corollary~\ref{holesdeferrecog} does not help us directly.

\vspace*{0.05in}

There is more to be said, however: The proof of
Theorem~\ref{manychains} actually establishes that, for certain
representatives of a homology class, the maximal elements in the
support of that representative each give rise to $(k+1)!$ many chains.
In the collaboration example, by choosing the homology generator
appropriately, this implies that for each author there are at least
two informative attribute release sequences for identifying the
author.  That gives us 10 sequences overall for relation $G$.  To find
20, we would likely want to examine links of pairs of co-authors.
There are 10 such links, 5 of which\footnote{For the curious reader:
Each of the remaining 5 links is a singleton.  For instance,
$\Lk(\dowGx, \{2,4\})$ is the simplicial complex consisting of the
single vertex $\{3\}$.  It is generated in the corresponding link
relation by attribute \cityC.  These comments are another way of
saying that the only author who has co-authored a book together with
both \authorB\ and \authorD\ is \authorC, producing the travel guide
for \cityC.  Observe as well that the pair of co-authors $\{\authorB,
\authorD\}$ does {\em not}\, appear as the $\sigma$ component of an
element $(\sigma, \gamma)$ in $\PG$ or $\PGplus$.  This means one
cannot identify just the pair of co-authors $\{\authorB, \authorD\}$,
but invariably infers the full triple $\{\authorB, \authorC,
\authorD\}$ of collaborators, given relation $G$.}
look similar to the one in Figure~\ref{BenClaireLink} on
page~\pageref{BenClaireLink}.  Each of those is an instance of
$\Szero$, meaning each has two different iars for identifying the pair
of co-authors.  That therefore gives us 10 iars for identifying
certain pairs of co-authors, and thus 20 iars for identifying
individual authors (each author participates in two of the
identifiable pairs).

Corollary~\ref{holesdeferrecog} further allows us to conclude that the
maximal length of an informative attribute release sequence for
identifying an identifiable pair of co-authors is at least two.
Consequently, the maximal length of an informative attribute release
sequence that identifies a given individual author must be at least
(and thus exactly) three.

\clearpage
\section{Experiments}
\markright{Experiments}
\label{experiments}

An individual may wish to reveal information about himself/herself
while delaying full identification.  We saw in Section~\ref{bubbles}
that homology provides a lower bound on the number and length of such
informative attribute release sequences.  The lower bound need not be
tight.  In order to explore these results experimentally, we examined
two datasets of different character:

\vspace*{0.05in}

\begin{description}

\item[Medals:] We obtained this dataset in August 2014 from

  \hspace*{0.4in} {\tt http://www.tableausoftware.com/public/community/sample-data-sets}.

  The dataset contained information about athletes who participated in
  the Olympics during the years 2000--2012.  \ The attribute fields
  that we considered were:

  \hspace*{0.15in} {\tt Age, Country, Year, Sport, Gold Medals, Silver Medals, Bronze Medals}

  (The last three fields counted the number of medals won by an athlete.)

  Every athlete therefore had exactly 7 attributes, with each
  attribute taking on one of a finite discrete set of pairwise
  exclusive values.  We represented these 7 dimensions of multivalent
  attributes as a collection of 223 binary attributes.

  There were 8613 individuals (we regarded the same physical person in
  different years as distinct athletes), who partitioned into 6955
  equivalence classes (for team sports, athletes often were
  indistinguishable).

  The result was a binary relation $M$ with 6955 rows and 223 columns.

\vspace*{0.15in}

\item[Jazz:] We assembled this relation in June 2015 by examining the
  website

  \hspace*{0.4in} {\tt http://www.redhotjazz.com}.

  The website contained information about jazz musicians and bands,
  mainly from the early to late-mid 20th century.

  We assembled a relation $J$ whose rows were indexed by musicians and
  whose columns were indexed by bands, with $(m,b) \in J$ meaning that
  musician $m$ played in band $b$.

  The result was a binary relation $J$ with 4896 rows and 990 columns.

  Caution: \ We were somewhat but not particularly careful in
  determining whether similar names constituted different spellings of
  the same musician's actual name.  For some bands, the website listed
  one or more bandmembers as ``unknown''.  We ignored those
  bandmembers.  We ignored bands for whom we could not determine any
  bandmembers.  Since our goal was to examine and compare homology and
  informative attribute release sequences, merely constructing a
  relation was sufficient for our purposes.  However, it is unlikely
  that the resulting relation satisfied the assumption of relational
  completeness stated on page~\pageref{relationalcompleteness},
  relative to data obtainable from other sources.

\vspace*{0.1in}

  We encountered the jazz website because it was the source of data
  for a paper on collaboration networks \cite{tpr:gleiserdanon} that
  explored the dual nature of individuals and attributes.  The paper
  constructed two graphs, one with musicians as vertices and bands as
  edges, the other with those roles reversed, then analyzed the
  information each representation highlighted.  \ We may view those
  graphs as the 1-skeleta of our Dowker complexes.
\end{description}

\subsection{Compare and Contrast}
\label{CompareAndContrast}

We review some key differences between the two relations $M$ and $J$.

\vspace*{-0.1in}

\begin{description}

\item[Identifiability:] The original 8613 individuals in the Olympic
  Medals dataset were not all uniquely identifiable.  For some
  athletes, even knowing an athlete's full set of 7 attributes left
  ambiguity as to the athlete's identity.  This was true for 2810 of
  the athletes.  Fortuitously, an athlete's ambiguity was fully
  symmetric, meaning that one could in fact partition the set of all
  athletes into equivalence classes.  This symmetry was likely due to
  the fact that some competitions involved teams, with team members
  indistinguishable from each other.  Each equivalence class then
  formed a uniquely identifiable ``individual'' in relation $M$.

  For the Jazz relation, 863 of the 4896 musicians were uniquely
  identifiable, but 4033 were not.  Unfortunately, this time the
  ambiguity was not fully symmetric.  One could again partition the
  4033 individuals into 1022 equivalence classes based on having
  identical rows in $J$.  However, some rows remained subsets of other
  rows, giving a directionality to the ambiguity.   For this reason,
  we did not pass to equivalence classes.

\item[Attribute Size:] In the medals relation $M$, every individual
  had exactly 7 binary attributes, describing one value for each of
  the 7 possible fields: \ {\tt Age, Country, Year, Sport, Gold
  Medals, Silver Medals, Bronze Medals}.  Consequently, there were
  also always exactly 7 binary attributes in each relation modeling
  the link of an athlete in $\dowMx$.

  In the Jazz dataset, there was no structural bound to the number of
  bands in which a musician might have played, so a musician's
  attributes could be many.  The largest number of bands in which any
  one musician played was in fact 44.  The average was a little over 2
  and the median 1.  Dually, the largest band had 288 musicians, with
  an average of 10.4 and a median of 7.

\item[Link Size:] For $M$, the number of other athletes in any given
  athlete's link was always close to the entire set of possible
  athletes.  With only 7 attribute fields and few distinct values, any
  two athletes shared almost certainly some attribute value (for
  instance, winning zero gold medals).

  In contrast, for the 767 musicians in $J$ for whom we computed links
  (described further in Section~\ref{Jprime}), the number of other
  musicians in any given musician's link was relatively small.  The
  average was 55.3, the median 37, with a maximum of 301.  With
  musicians generally playing in few bands, each collaborated
  artistically on average with only a few score fellow musicians of
  the 4895 musicians in the database.

\end{description}

\subsection{Homology Computations}

For each of the link relations discussed below, we computed homology
of the Dowker complex $\dowqy$, with relation $Q$ modeling the
link.\footnote{Formally, the link is equal to $\dowqx$.  By Dowker's
Theorem, $\dowqx$ and $\dowqy$ have the same homology.} Since our goal
was to find lower bounds for informative attribute release sequences,
we modified $\dowqy$ slightly, as suggested by
Section~\ref{hiddenholes}.  Specifically, whenever $\dowqy$ was a cone
with more than one maximal simplex, we removed all its cone apexes.

\vst

Comment: \ The homology lower bound results of
Section~\ref{leveraginglattices} and Appendix~\ref{manylongchains} do
not depend directly on the chain coefficients being integers (of
course, the actual homology observed may depend on the type of
coefficients).  We therefore computed homology with $\Z_2$
coefficients, using the {\tt Perseus} software previously written at
the University of Pennsylvania.  We downloaded an executable version
in 2014 from {\tt http://www.sas.upenn.edu/{$\sim$}vnanda/perseus/}.

\subsection{Homology and Release Sequences in the Olympic Medals Dataset}
\markright{Homology and Release Sequences in the Olympic Medals Dataset}
\label{olympic_homology}

\paragraph{Overall Homology:}

  A collection of $k$ attributes, each taking on one of a finite
  discrete set of pairwise exclusive values, produces Dowker complexes
  with homotopy types that are wedge sums of $\Sko$s, assuming all
  possible combinations of attributes are represented by individuals.

  Consequently, with every individual having exactly 7 attributes, one
  might expect to see some homology in dimension 6.  But of course,
  not every combination is possible.  For instance, no one athlete is
  going to simultaneously win the gold, silver, and bronze medals in
  the same event.
  {\em From this perspective, real-world constraints show up as
  absence of potential homology.}
  In fact, relation $M$ had the Betti numbers described in
  Table~\ref{medalsbetti}, computed using $\Z_2$ coefficients:

\begin{table}[h]
\begin{center}
{$\begin{array}{c|ccccc}
         d   &  0 & 1 &  2 &  3  &  4  \\[1pt]\hline
   \beta_d   &  1 & 0 & 23 & 757 & 503 \\
\end{array}$}
\end{center}
\vspace*{-0.2in}
\caption[]{Betti numbers for the topology of the Olympic Medals
         relation $M$.}
\label{medalsbetti}
\end{table}

\vspace*{-0.1in}

The table does suggest that there could be quite a few informative
attribute release sequences of length at least 5 for identifying
athletes ($\beta_4 \neq 0$ in $\PM$ implies length 5 iars).

\vspace*{-0.1in}

\paragraph{Link Homology:} We computed the link of each athlete in
$M$ (or more precisely, of each equivalence class), and determined
homology for the resulting relation, with the modifications mentioned
before.  Specifically, we removed all cone apexes from an athlete's
Dowker complex $\dowqy$ (assuming it contained more than one maximal
simplex) before computing homology, with $Q$ being the link relation.
Of the 6955 links, 3822 contained attribute cone apexes in $\dowqy$.

Table~\ref{medalslinkhomol} summarizes the results.  One may conclude
more strongly now that (at least) 2198 athletes could find (at least)
120 different ways of releasing (at least) 5 of their 7 attributes
without identifying themselves uniquely prior to having released all 5
attributes ($\beta_3 \neq 0$ in $\PQ$ minimally implies $5!$ many iars
of length 5 for relation $M$, by Corollary~\ref{holesdeferrecog} on
page~\pageref{holesdeferrecog}).

\begin{table}[h]
\begin{center}
\vspace*{0.05in}
{$\begin{array}{c|ccccc}
      d           &  0  &   1  &   2  &   3  &  4  \\[1pt]\hline
 \Mlkcard         & 229 & 1355 & 2773 & 2198 & 57  \\[0.5pt]
 \Mlkmax{\beta_d} &  2  &   4  &   7  &   4  &  2 \\
\end{array}$}
\end{center}
\vspace*{-0.2in}
\caption[]{Histogram indexed by dimension $d$, describing athletes
  whose links $\Lk(\Psi_M, \hbox{athlete})$ had reduced homology in
  dimension $d$ (after removal of attribute cone apexes from the dual
  complexes), for the 6955 athletes in the Olympic Medals relation
  $M$.  \ (525 of the 6955 links had no reduced homology; they do not
  appear in the histogram.)  \ Also shown are the maximum Betti
  numbers seen in each dimension, with the maximum taken over all
  possible athletes.}
\label{medalslinkhomol}
\end{table}

\vspace*{-0.1in}

\paragraph{Informative Attribute Release Sequences:}  We computed
  a maximal length informative attribute release sequence for each
  link relation.  One can find such a sequence by searching for a
  least-cost path from $\oneQ$ to $\zeroQ$ in $\PQplus$, picking
  attributes along the way as per the construction of
  Lemma~\ref{chaintoiars} on page~\pageref{chaintoiars}, with cost
  being the number of attributes inferred as one traverses the path.
  Here $Q$ is again the link relation.  Of the 6955 athletes, 6229
  actually had a maximal informative attribute release sequence of
  length 7.  Each such athlete could order his/her attributes in such
  a way that his/her identity would not become fully known until s/he
  had released all 7 attributes.  Of the remaining athletes, 719 had a
  maximal informative attribute release sequence of length 6, and 7
  had a maximal length of 5.

  Of course, Corollary~\ref{holesdeferrecog} on
  page~\pageref{holesdeferrecog} makes a stronger claim, suggesting possible
  permutability of some attributes.  Consequently, we computed for each
  link relation all possible isotropic sets of attributes (see again
  Definition~\ref{isotropydef} on page~\pageref{iars}, now with
  $Q$ in place of $R$).  Table~\ref{medalslinkiars} summarizes the
  results:

\begin{table}[h]
\begin{center}
{$\begin{array}{c|ccccc}
 \abs{\kappa} &   2  &   3  &   4  &   5  &  6  \\[2pt]\hline
 \Mlkcard               & 6955 & 6955 & 6955 & 5568 & 171 \\[1pt]
 \Mlkmax{\abs{\{\kappa\}}} &  21  &  35  &  35  &  21  &  5  \\
\end{array}$}
\end{center}
\vspace*{-0.175in}
\caption[]{Histogram indexed by size $\abs{\kappa}$, describing
  athletes whose link relations contained isotropic attribute sets
  $\kappa$.  An athlete could have several distinct (possibly
  overlapping) such sets for any given size.  Also shown therefore are
  the maximum numbers of such sets, with the maximum taken over all
  possible athletes.  For example: 171 athletes had at least one
  isotropic set of size $6$, and the maximum number of isotropic sets
  of size $6$ any one athlete had was 5.}
\label{medalslinkiars}
\end{table}

\begin{figure}[h]
\begin{center}
\ifig{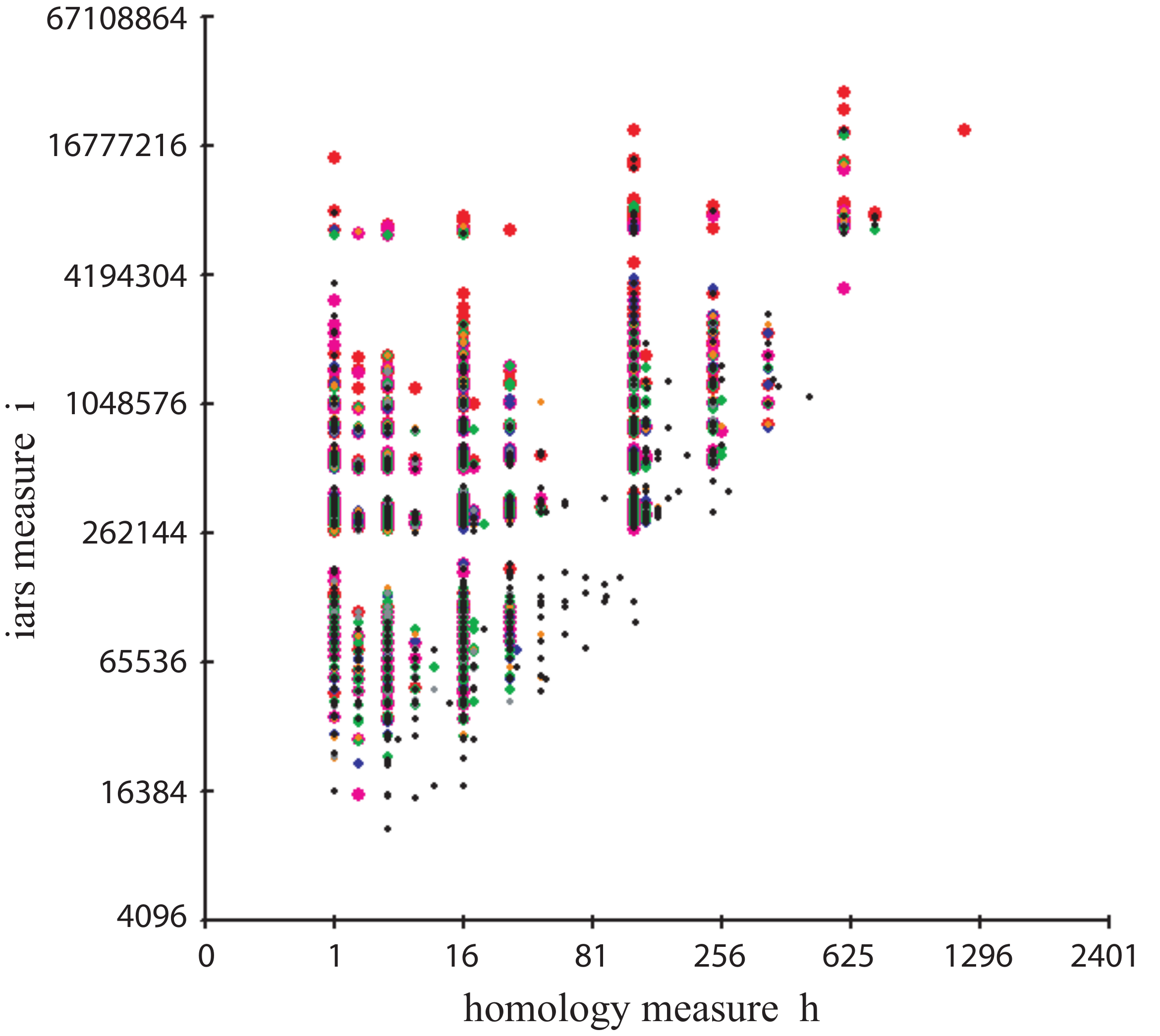}{scale=0.65}
\end{center}
\vspace*{-0.15in}
\caption[]{Scatterplot describing each athlete's link in the medals
  relation $M$.  The scatterplot shows for each link a point $(h, i)$,
  with $h$ a measure of the link's homology (after removal of
  attribute cone apexes) and $i$ a measure of how many significant
  informative attribute release sequences exist for the link relation.
  The scatterplot suggests that the homology measure $h$ serves as a
  loose lower bound for the iars measure $i$.  \ See also
  Corollary~\ref{holesdeferrecog} on
  page~\pageref{holesdeferrecog}.\\[3pt]
  (The colors and radii indicate the numbers of athletes in the links.
   \ The color ordering and size boundaries are:\\[2pt]
\hspace*{0.5in} {\sc black}--$_{6821}$--{\sc silver}--$_{6831}$--{\sc orange}--$_{6851}$--{\sc green}--$_{6859}$--{\sc blue}--$_{6865}$--{\sc magenta}--$_{6872}$--{\sc red}.\\[4pt]
  In this figure, the boundaries between colors were chosen so that
  each bucket would hold roughly 1000 links.  As one can see, the
  number of athletes in a link was generally large.)}
\label{medalsscatter}
\end{figure}

\vspace*{-0.15in}

\paragraph{Scatterplot:}  Finally, we computed for each link a pair of
  numbers $(h, i)$, with $h$ representing a measure of link homology
  and $i$ representing a measure of informative attribute release
  sequences for the link relation.  The resulting scatterplot appears
  in Figure~\ref{medalsscatter}.

The exact formulas for $h$ and $i$ are not that significant, but we
mention them here for completeness.  To obtain a measure of homology,
we assembled for each link a vector with the Betti numbers computed
earlier: $(\beta_0, \beta_1, \beta_2, \beta_3, \beta_4)$.  We
determined maximum values for each component (as given in
Table~\ref{medalslinkhomol}).  We could then think of any such vector
as defining, in reverse order, a varying-radix numeral.  We converted
that numeral to an integer.  For example, a link with Betti vector
$(2,3,6,0,0)$ would have $h$ value ${2 + 3{\cdot}(2+1) +
6{\cdot}(4+1){\cdot}(2+1)} = {101}$.  A link relation that remains
contractible after removal of attribute cone apexes would have $h$
value $1$.  In order to graph the scatterplot nicely, we scaled the
$h$-axis by taking a fourth root.

We computed a link's $i$ value similarly, now from the following
vector of data: $(\ell_{\rm max}, c_2, c_3, c_4, c_5, c_6)$.  Here
$\ell_{\rm max}$ is the largest $\ell$ in an informative attribute
release sequence $y_1, \ldots, y_\ell$ for the link relation, while
$c_k$ is the number of different isotropic attribute sets $\kappa$ in
the link relation such that $\abs{\kappa} = k$.  We scaled the
$i$-axis by taking a logarithm.

\clearpage
\subsection{Homology and Release Sequences in the Jazz Dataset}
\markright{Homology and Release Sequences in the Jazz Dataset}
\label{jazz_homology}

\paragraph{Overall Homology:}  

Given the large number of bands in which some musicians played, and
given memory constraints of our machines, we did not compute homology
for the whole Jazz relation $J$.  Instead, we computed homology for
restricted relations consisting of musicians who played in fewer than
20 bands.  This covered 4856 of the 4896 musicians in the overall
relation.  Since we did not see any homology above dimension 2 after
considering several of these restricted cases, we used the 4-skeleton
(all simplices of dimension 4 or less)
of $\Phi_{\!J}$ as proxy for the topology of the whole relation $J$
and computed its homology.  Table~\ref{jazzbetti} summarizes the
results.  Using a graph algorithm, we verified that the whole relation
$J$ did indeed have 107 components, as indicated by $\beta_0$ for the
4-skeleton $\Phi^{(4)}_{\!J}$.  Given the low dimension of homology
for the restricted relations, conceivably even $J$ might not tell us
much about the length of informative attribute release sequences for
the various musicians, suggesting we look at links.

\begin{table}[h]
\begin{center}
{$\begin{array}{cc|c|ccccc}
  b & \Sigma &   m  & \beta_0 & \beta_1 & \beta_2 & \beta_3 & \beta_4 \\[3pt]\hline
 14 & \Phi_{\!J|b} & 4819 & 111 & 613  &    20    &     0   & 0 \\[2pt]
 15 & \Phi_{\!J|b} & 4831 & 111 & 613  &    32    &     0   & 0 \\[2pt]
 16 & \Phi_{\!J|b} & 4838 & 111 & 605  &    42    &     0   & 0 \\[2pt]
 17 & \Phi_{\!J|b} & 4848 & 110 & 603  &    58    &     0   & 0 \\[2pt]
 18 & \Phi_{\!J|b} & 4851 & 110 & 603  &    65    &     0   & 0 \\[2pt]
 19 & \Phi_{\!J|b} & 4856 & 109 & 596  &    75    &     0   & 0 \\[5pt]
\infty & \Phi^{(4)}_{\!J} & 4896 & 107 & 550 & 93 &    10   & - \\[4pt]
 15 & \Phi_{\!J^\prime}   &  767 &  18 & 595 & 32 &     0   & 0 \\
\end{array}$}
\end{center}
\vspace*{-0.15in}
\caption[]{Betti numbers for subcomplexes $\Sigma$ of $\Phi_{\!J}$,
  with $J$ being the Jazz relation.  The first six rows correspond to
  restrictions of $J$ to musicians who played in at most $b$ bands.
  For each row, $m$ indicates the number of musicians in the relation.
  The penultimate row describes the 4-skeleton of $\Phi_{\!J}$.  The
  last row refers to a relation $J^\prime$ described further in the
  text.}
\label{jazzbetti}
\end{table}

\vspace*{-0.2in}

\paragraph{Link Homology:} We computed the link of some of the
musicians in $J$, and determined homology for the resulting relations
(again after removal of attribute cone apexes, when appropriate).
\label{Jprime} Table~\ref{jazzlinkhomol} summarizes the results.
Given the inability to uniquely identify some musicians even knowing all
their bands (as described in Section~\ref{CompareAndContrast}) and the
difficulty of computing homology when musicians played in many bands, we
computed links only for a subset of the musicians.  We required each
musician to be uniquely identifiable, to have played in at most 15
bands, and to have a nontrivial link.  There were 767 such musicians.
Betti numbers for the relation $J^\prime$ representing the restriction
of $J$ to these 767 musicians also appear in Table~\ref{jazzbetti}.
(Note, however, that we computed the complete link $\Lk(\Psi_{\!J},
\hbox{musician})$ for each of the 767 musicians, not merely
$\lk(\Psi_{\!J^\prime}, \hbox{musician})$.)  \ We removed attribute cone
apexes from the link relation for 106 of these 767 musicians.

\begin{table}[h]
\begin{center}
{$\begin{array}{c|cccc}
    d             &  0  &  1  &  2 &  3  \\[2pt]\hline
 \Jlkcard         & 604 & 145 & 20 &  1  \\[1pt]
 \Jlkmax{\beta_d} &  7  &  6  &  3 &  1  \\
\end{array}$}
\end{center}
\vspace*{-0.2in}
\caption[]{Histogram indexed by dimension $d$, describing musicians
  whose links $\Lk(\Psi_{\!J}, \hbox{musician})$ had reduced homology
  in dimension $d$ (after removal of attribute cone apexes from the
  dual complexes), for the 767 musicians who were uniquely identifiable
  in $J$, played in at most 15 bands, and had nontrivial link.  \ (52
  of the 767 links had no reduced homology; they do not appear in the
  histogram.)  \ Also shown are the maximum Betti numbers seen in each
  dimension, with the maximum taken over the 767 possible musicians.
  For $d=0$, this means that 604 of the 767 musicians had
  collaborations with other musicians that split into pairwise
  disjoint groups. The maximum number of such components for any one
  musician was 7.}
\label{jazzlinkhomol}
\end{table}

These results suggest that the relationships to other musicians do
indeed $\!${\em not} have many high-dimensional holes in them.
Recall, by Corollary~\ref{holesdeferrecog} on
page~\pageref{holesdeferrecog}, one can assert the existence of at
least $(k+2)!$ distinct informative attribute release sequences of
length at least $k+2$ for any musician with a $k$-dimensional hole.
For almost all musicians this lower bound means 2 sequences of length
2, for some it means 6 sequences of length 3, for a few it means 24
sequences of length 4, and for one musician it means 120 sequences of
length 5.  These implications are roughly consistent with the data for
informative attribute release sequences described next, though, as
expected for the theoretical reasons discussed earlier, they
constitute lower bounds.

\vspace*{-0.1in}

\paragraph{Informative Attribute Release Sequences:}  We computed
a maximal length informative attribute release sequence for each link
relation.  Table~\ref{jazzlinkiarsMaxlen} summarizes the results.  We
mention in passing: Any attribute release sequence that was
informative for a musician's link relation was also informative for
the encompassing relation $J$ (by Lemma~\ref{interplocal}(ii) on
page~\pageref{interplocal}).  For a few musicians, the maximal
sequence found within the link relation $Q$ could be further extended
in the encompassing relation $J$, with a prefix of one attribute,
namely an attribute shared by all members of the link, yet remain
informative and identifying within $J$.  This occurred for the 17
musicians whose maximum sequence length $\ell$ in the musician's link
was 1.

We also computed for each link relation all possible isotropic sets of
attributes.  Table~\ref{jazzlinkiarsPerm} summarizes those results.

\begin{table}[h]
\vspace*{-0.05in}
\begin{center}
{$\begin{array}{c|ccccccccccc}
 \ell      &  1 &  2  &  3  &  4  &  5 &  6 &  7 &  8 &  9 & 10 & 11 \\[2pt]\hline
 \Jlkcard  & 17 & 248 & 218 & 125 & 72 & 35 & 23 & 15 & 11 &  2 &  1 \\
\end{array}$}
\end{center}
\vspace*{-0.2in}
\caption[]{Histogram of musicians, indexed by length $\ell$ of a
  longest informative attribute release sequence for the musician's
  link relation, for the 767 musicians described in the text.}
\label{jazzlinkiarsMaxlen}
\end{table}

\begin{table}[h]
\vspace*{-0.15in}
\begin{center}
{$\begin{array}{c|cccc}
 \abs{\kappa}   &  2  &  3  &  4 &  5  \\\hline
 \Jlkcard           & 750 & 219 & 49 &  3  \\[1pt]
 \Jlkmax{\abs{\{\kappa\}}} & 105 & 202 & 40 &  2  \\
\end{array}$}
\end{center}
\vspace*{-0.2in}
\caption[]{Histogram indexed by size $\abs{\kappa}$, describing
  musicians whose link relations contained isotropic attribute sets
  $\kappa$.  Also shown are the maximum numbers of such sets, with the
  maximum taken over the 767 possible musicians described in the
  text.}
\label{jazzlinkiarsPerm}
\end{table}

\begin{figure}[h]
\begin{center}
\ifig{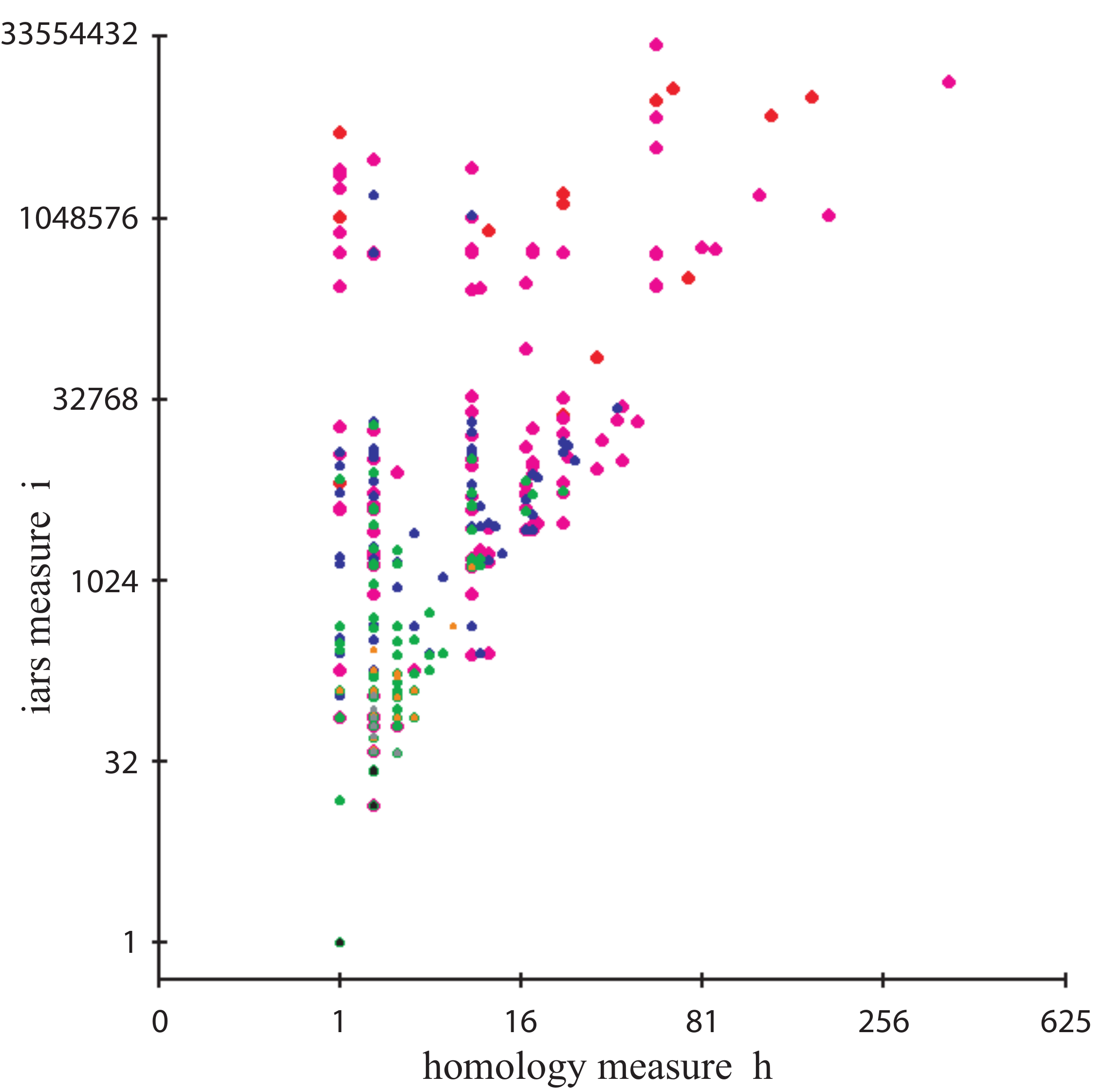}{scale=0.65}
\end{center}
\vspace*{-0.15in}
\caption[]{Scatterplot describing the links computed for 767 of the
  musicians in the Jazz relation $J$.  The scatterplot shows for each
  link a point $(h, i)$, with $h$ a measure of the link's homology
  (after removal of attribute cone apexes) and $i$ a measure of the
  link's informative attribute release sequences.\\[2pt]
  (The colors and radii indicate the numbers of musicians in the links.
  Link sizes were fairly small.
  \ The color ordering and size boundaries are:\\[2pt]
\hspace*{0.5in} {\sc black}--$_{5}$--{\sc silver}--$_{10}$--{\sc orange}--$_{20}$--{\sc green}--$_{50}$--{\sc blue}--$_{100}$--{\sc magenta}--$_{200}$--{\sc red}.\\[4pt]
  In this figure, the buckets could hold noticeably varying numbers of links.)}

\label{jazzscatter}
\end{figure}

\vspace*{-0.2in}

\paragraph{Scatterplot:}  We computed for each link a pair of
  numbers $(h, i)$, with $h$ representing a measure of homology and
  $i$ representing a measure of the link's informative attribute
  release sequences, much as for the medals relation $M$ of
  Section~\ref{olympic_homology}.  Figure~\ref{jazzscatter} depicts
  the scatterplot.

\clearpage
\section{Inference in Sequence Lattices}
\markright{Inference in Sequence Lattices}
\label{sequencelattices}

We have seen how a relation gives rise to a lattice via the Galois
connection, as per Definition~\ref{galoislattice} on
page~\pageref{galoislattice}.  The lattice structure describes the
ways in which privacy may be preserved or lost.  Consequently, when
thinking about privacy, perhaps one can also start with lattices that
do not necessarily arise initially from relations.

This section will look at inferences from sequences of observations.
The next section examines strategy obfuscation in planning with
uncertainty.

\vspace*{0.02in}

We should mention some equivalences: Lattices are particular kinds of
partially ordered sets (posets).  Posets and simplicial complexes are
topologically identical; one can move back and forth between these
representations while preserving homeomorphism type (see
\cite{tpr:wachs} and Appendix~\ref{prelim}).  Furthermore, one may
describe a finite simplicial complex by a relation in several
different ways that preserve homotopy type, including ways in which
one of the two resulting Dowker complexes is identical to the original
simplicial complex.  For instance, maximal simplices can play the role
of individuals and vertices can play the role of attributes.  In
short, one has three different categories of structures with which to
think about privacy: relations, simplicial complexes, and lattices.
One may start with any one representation and build the other two from
that.

\subsection{Sequence Lattices for Dynamic Attribute Observations}

\begin{figure}[h]
\begin{center}
\ifig{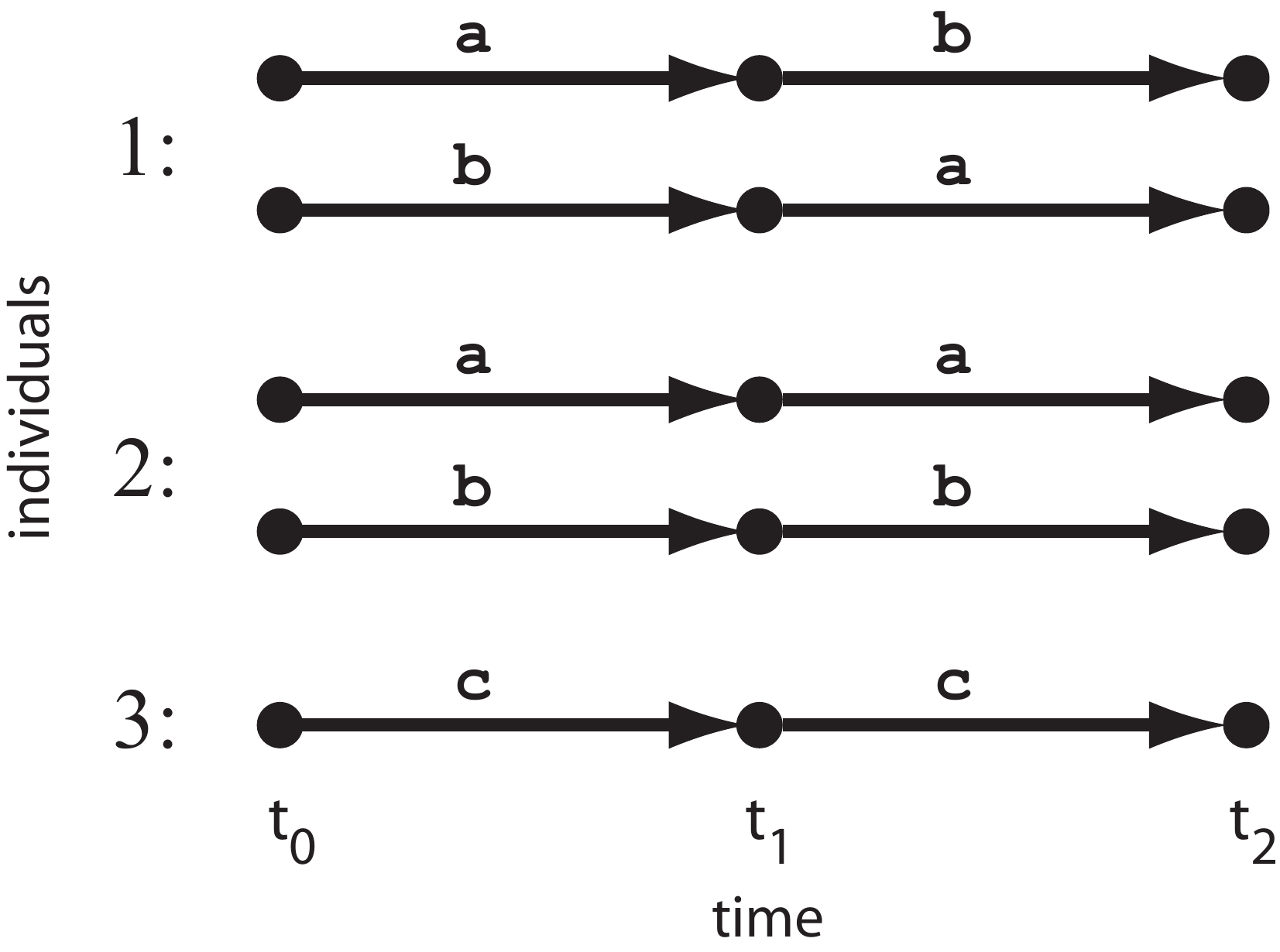}{width=3.25in}
\end{center}
\vspace*{-0.25in}
\caption[]{Three types of individuals and the attributes each might
  reveal in two successive time intervals.}
\label{dfaprocess}
\end{figure}

Consider the dynamic process of Figure~\ref{dfaprocess}.  The process
models observations of individuals who reveal attributes over
successive time steps.  There are three possible individuals (or more
generally, types of individuals).  The first individual emits
attributes ``\ta'' and ``\tb'' alternatingly at successive times, but
one does not know which of those attributes one might see first.  The
second individual always emits the same attribute, either ``\ta'' or
``\tb'', but one does not know {\em a priori} which it is.  The third
individual always emits the same attribute ``\tc''.

\begin{figure}[h]
\vspace*{-0.05in}
\begin{center}
\begin{minipage}{1.5in}{$\begin{array}{c|ccc}
\hbox{\Large $S$} & \ta & \tb & \tc \\[2pt]\hline
1 & \one & \one &      \\
2 & \one & \one &      \\
3 &      &      & \one \\
\end{array}$}
\end{minipage}
\hspace*{0.25in}
\begin{minipage}{1.5in}{$\begin{array}{c|ccccc}
\hbox{\Large $T$} & \ta\ta & \tb\tb & \ta\tb & \tb\ta & \tc\tc \\[2pt]\hline
1 &      &      & \one & \one &      \\
2 & \one & \one &      &      &      \\
3 &      &      &      &      & \one \\
\end{array}$}
\end{minipage}
\end{center}
\vspace*{-0.05in}
\caption[]{Relation $S$ describes individuals and single attributes,
  while $T$ describes individuals and sequences of two attributes.}
\label{dfarelations}
\end{figure}

A relation for these (types of) individuals that models the
individuals in terms of single attributes appears as relation $S$ in
Figure~\ref{dfarelations}.  Individual \#3 is distinguishable from the
other two individuals, but the relation provides no means for
distinguishing those two individuals from each other.  The relation is
homogeneous with regard to single attributes for individuals \#1 and
\#2.  Of course, we can see from the dynamic process of
Figure~\ref{dfaprocess}, that distinguishing information appears via
sequences of two attributes.  Relation $T$ of
Figure~\ref{dfarelations} models such sequences.  Now all three
individuals are uniquely identifiable.  Should one wish to model
inferences based on both one and two observations, one could use the
relation $S \union T$.

\begin{figure}[h]
\begin{center}
\vspace*{-0.05in}
\ifig{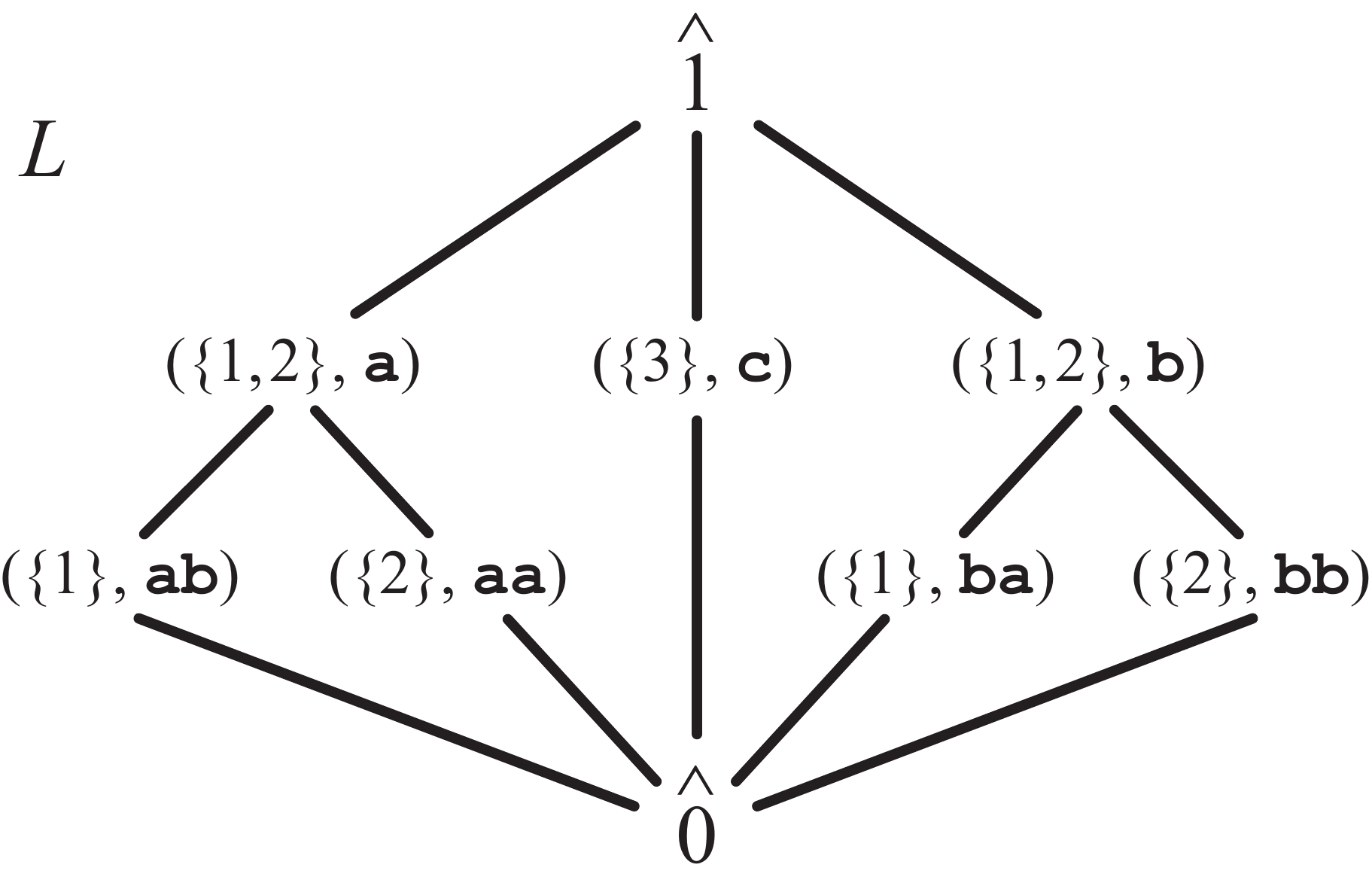}{height=2.5in}
\end{center}
\vspace*{-0.25in}
\caption[]{Lattice representing the dynamic process of Figure~\ref{dfaprocess}.}
\label{dfalattice}
\end{figure}

That jump from single to double attributes is useful, but where does
it come from intrinsically?  After all, without additional knowledge,
we might simply consider infinitely long sequences, even though those
would not add anything in this example.  In fact, the dynamic process
of Figure~\ref{dfaprocess} gives us the information.  It is itself
basically a decision tree that amounts to the lattice of
Figure~\ref{dfalattice}.  In that figure, we have annotated each
internal node of the lattice with an ordered pair, consisting of a set
of individuals and either a single attribute or a sequence of two
attributes. This lattice differs from previous ones in this report in
that a set of individuals (or attributes, more generally) is no longer
constrained to appear in at most one node of the lattice.  By allowing
multiple nodes, we enhance our ability to encode state in the lattice.
For example, observing attribute ``\ta'' carries different meaning
depending on whether one has already seen attribute ``\ta'' or
attribute ``\tb'' or no attribute at all.  Also: While we could have
included $(\{3\}, \tc\tc)$ in the lattice, we did not need that
element.

In the lattice of Figure~\ref{dfalattice} it is tempting to merge the
two identifying nodes for individual \#1 into one node and to merge
the two identifying nodes for individual \#2 into one node.  There is
apparently no harm in doing so, in that the decision process would
still be correct.  However, the resulting structure would no longer be
a lattice but merely a poset.  That may or may not be desirable in a
given application.  For instance, using homology to estimate lower
bounds for how long one can delay identification suggests using almost
a join-based lattice, if one wishes to fulfill the hypotheses of
Theorem~\ref{manychains} on page~\pageref{manychains}.

If we did want to merge nodes as just described, while maintaining a
lattice, then we would perhaps also merge the two nodes containing the
set $\{1, 2\}$, giving us the lattice of Figure~\ref{dfalatticealt}.
This lattice is similar to the lattice $P^{+}_{S \union T}$ that one
would construct from the relation $S \union T$, except that it does
not include singleton attributes in the nodes identifying individuals
\#1 and \#2 and it does not include the sequence ``\tc\tc'' in the node
identifying individual \#3.

\begin{figure}[h]
\begin{center}
\vspace*{-0.1in}
\ifig{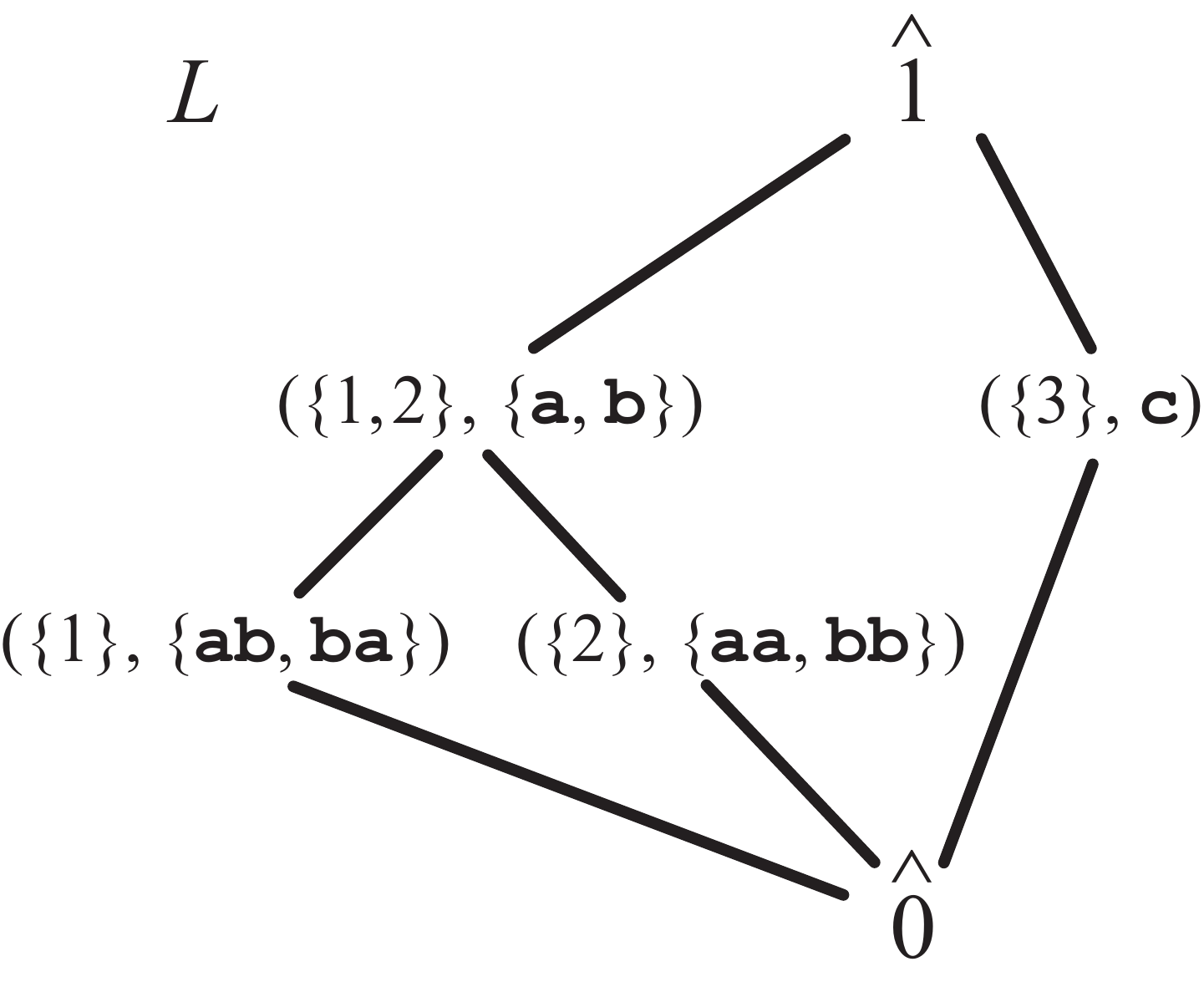}{height=2.5in}
\end{center}
\vspace*{-0.25in}
\caption[]{Modified lattice of Figure~\ref{dfalattice}, after merging some
nodes.}
\label{dfalatticealt}
\end{figure}

Regardless, the lattices of Figures~\ref{dfalattice} and
\ref{dfalatticealt} encode the inferences possible for the dynamic
process of Figure~\ref{dfaprocess}. In particular, if we observe
either attribute ``\ta'' or attribute ``\tb'', then we know the set of
possible individuals is $\{1, 2\}$; we have excluded individual \#3.
Moreover, if we observe any two-attribute sequence, with attributes
drawn from $\{\ta, \tb\}$, then we can identify the observed
individual uniquely as either \#1 or \#2.  Thus the required sequences
come directly from the dynamic process, not requiring an explicit
intermediate representation as a relation.  (One might argue, however,
that a relation is implicit in our reasoning.)

\subsection{Lattices of Stochastic Observations}

The dynamic sequence perspective incorporates repeated randomized
response within the lattice framework.  Instead of arising via a
(non)deterministic process as in Figure~\ref{dfaprocess}, the
attributes ``\ta'' and ``\tb'' for two (types of) individuals could
flow from a stochastic process.  One obtains an infinite lattice
determined by increasingly longer sequences of observations.
Depending on the confidence intervals one wishes to set, one obtains
stochastic decision regions such as those sketched in
Figure~\ref{rrlattice}, with a central region of ambiguity, bounded by
regions of exclusion, for identifying individuals.

\begin{figure}[h]
\begin{center}
\vspace*{-0.2in}
\ifig{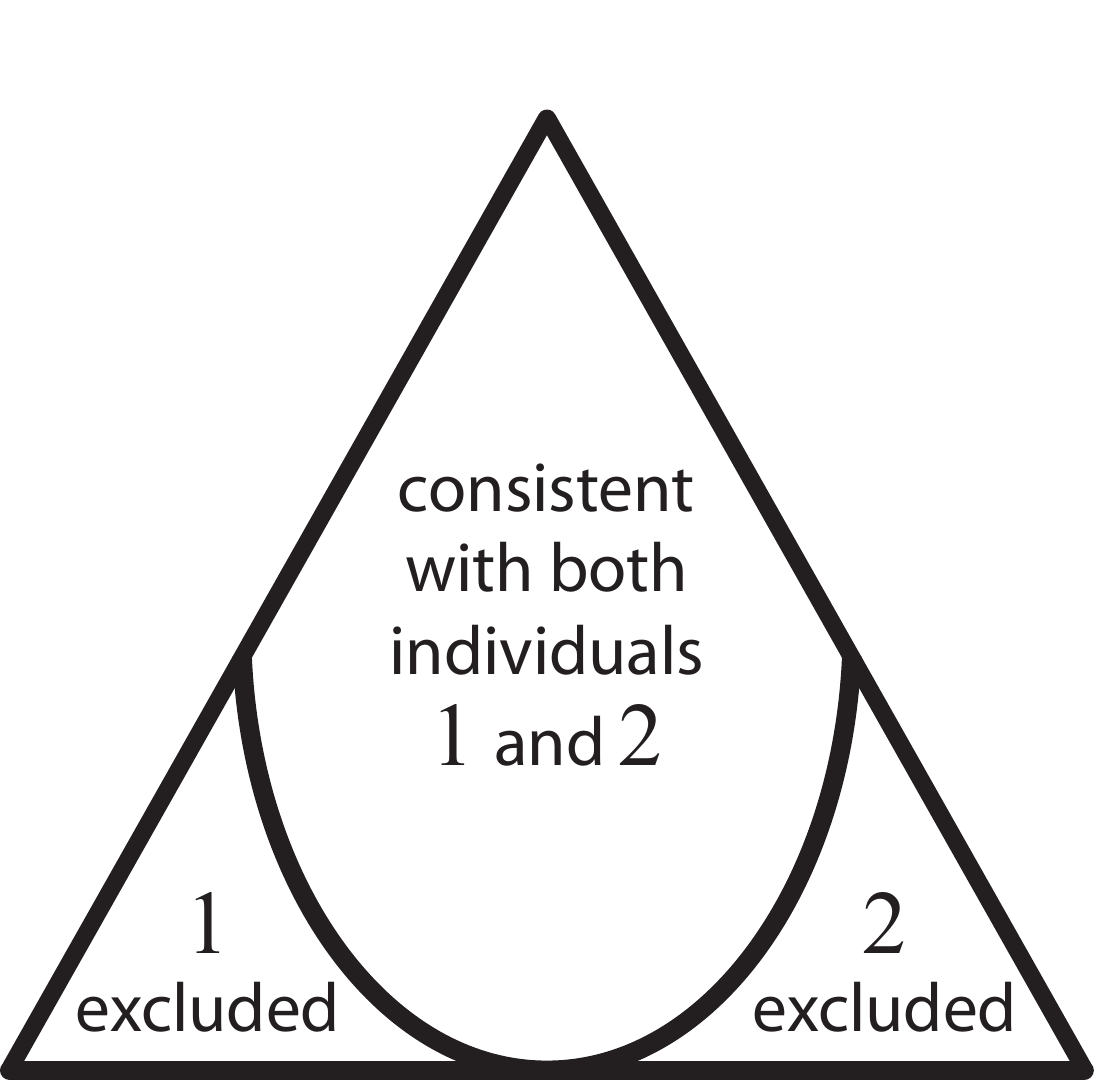}{height=2in}
\end{center}
\vspace*{-0.2in}
\caption[]{Sketch of an inference lattice for sequences of
  randomized response queries.}
\label{rrlattice}
\end{figure}

\subsection{General Inference Lattices}
\markright{General Inference Lattices}

Lattices are useful tools for inference.  Rather than work with
completely arbitrary lattices, we give here a definition that makes
explicit the existence of two underlying structures over which we wish
to perform inferences.  However, we no longer assume a pair of
underlying discrete spaces $X$ and $Y$ for individuals and attributes,
but instead posit posets $P$ and $Q$.  The connection to our earlier
relational perspective is that $P$ would be the powerset of $X$ and
$Q$ the powerset of $Y$.  By allowing potentially different posets $P$
and $Q$ for a given lattice $L$, one can in some instances obtain
different ``views'' of that lattice, thereby increasing flexibility in
the interpretation process.  For instance, $Q$ might consist of all
sequences up to a specified length or it might consist of {\em sets}
of such sequences.

\vspace*{0.1in}

\begin{definition}[Inference Lattice]\label{inferencelattice}
\ Let $P$ and $Q$ be finite posets.

\vst

An \underline{{\em inference lattice $L$ with respect to $P$ and $Q$}}
  is a bounded lattice whose proper part $\Lprop$ consists of ordered
  pairs $(p,q)$, with $p\in P$ and $q\in Q$, satisfying the following
  conditions:

\vspace*{0.1in}

\noindent For all $(p_1, q_1)$ and $(p_2, q_2)$ in $\Lprop$:

\vspace*{-0.05in}

\begin{enumerate}
\item[(i)] $(p_1,q_1) \Llte (p_2,q_2)$ if and only if $p_1 \Plte p_2$ and
  $q_1 \Qgte q_2$;
\item[(ii)] $(p_1,q_1) \Lor (p_2,q_2)$ is either $\oneL$ or a pair
  $(p,q)\in\Lprop$ such that $p$ is an upper bound for both $p_1$ and
  $p_2$ in $P$ and $q$ is a lower bound for both $q_1$ and $q_2$ in $Q$;
\item[(iii)] $(p_1,q_1) \Land (p_2,q_2)$ is either $\zeroL$ or a pair
  $(p,q)\in\Lprop$ such that $p$ is a lower bound for both $p_1$ and
  $p_2$ in $P$ and $q$ is an upper bound for both $q_1$ and $q_2$ in
  $Q$.
\end{enumerate}
(Note that $\,\,\zeroL \Llt (p, q) \Llt \oneL\,$ for every
$(p,q)\in\Lprop$, when $\Lprop \neq \emptyset$.\\[2pt]
Also, be aware that $\Lprop$ need {\em not}, and generally will not,
contain all possible pairs $(p,q)$ of $\,\PxQ$.)

\end{definition}

\paragraph{Inference Protocol:}  Suppose we have observed some
$q\in Q$.  \ How should we interpret that observation in terms of the
lattice $L$?  \ \ Here is a possible protocol:

\noindent (In terms of our earlier relational model, one may view this protocol
as inferring sets of individuals from sets of attributes.)

\begin{itemize}

\item Let $\Gamma = \setdef{(p^{\prime}, q^{\prime}) \in \Lprop}{q
    \;\Qlte\; q^{\prime}}$.

\item If $\Gamma = \emptyset$, then we view $q$ as inconsistent,
  producing interpretation $\zeroL \in L$.

\item Otherwise, let $\Gamma_{\rm max}$ consist of all the maximal
  elements of $\Gamma$ (maximal with respect to the partial order on
  $L$).  We view $q$ as implying this set of elements in $L$.  One can
  project each of those elements onto its $P$ coordinate, if that is
  useful.

\end{itemize}

\noindent There is a dual protocol for interpreting an observation $p\in P$:

\noindent (In terms of our earlier relational model, one may view this
  protocol as inferring sets of attributes from sets of individuals.)

\begin{itemize}

\item Let $\Sigma = \setdef{(p^{\prime}, q^{\prime}) \in \Lprop}{p
    \;\Plte\; p^{\prime}}$.

\item If $\Sigma = \emptyset$, then we view $p$ as inconsistent,
  producing interpretation $\oneL \in L$.

\item Otherwise, let $\Sigma_{\rm min}$ consist of all the minimal
  elements of $\Sigma$ (minimal with respect to the partial order on
  $L$).  We view $p$ as implying this set of elements in $L$.  Again,
  one can project each of those elements onto its $Q$ coordinate, if
  that is useful.

\end{itemize}

\vspace*{-0.2in}

\paragraph{Comments:} \ (1) In our previous relational setting, the
structure of Galois lattices ensured that, for nonempty observations,
each of $\Gamma_{\rm max}$ and $\Sigma_{\rm min}$ never contained more
than one element.  That need not be true for general inference
lattices. \ (2) One may augment the previous protocols, so as to
regard some element(s) of $Q$ much like the empty attribute simplex,
giving interpretation $\oneL \in L$.  Similarly, some element(s) of
$P$ might have interpretation $\zeroL \in L$.

\paragraph{Example:}\ Consider Figure~\ref{dfaPQ}.  Poset $P$ models
subsets drawn from the set of two individuals $\{1, 2\}$, while poset
$Q$ models sequential observations of ``\ta'' and ``\tb'', of lengths
one and two, as in our earlier example of Figure~\ref{dfaprocess}.
(For presentational simplicity, $P$ and $Q$ ignore individual \#3 and
attribute ``\tc'', instead focusing on individuals $\{1, 2\}$ and
attributes $\{\ta, \tb\}$.)  Let lattice $L$ be as in
Figure~\ref{dfalattice}.  \ Assume the interpretation of $\emptyset\in
P$ is $\zeroL$ in $L$, and that of $\,\botzero\in Q$ is $\oneL$.

\begin{figure}[h]
\begin{center}
\vspace*{0.1in}
\ifig{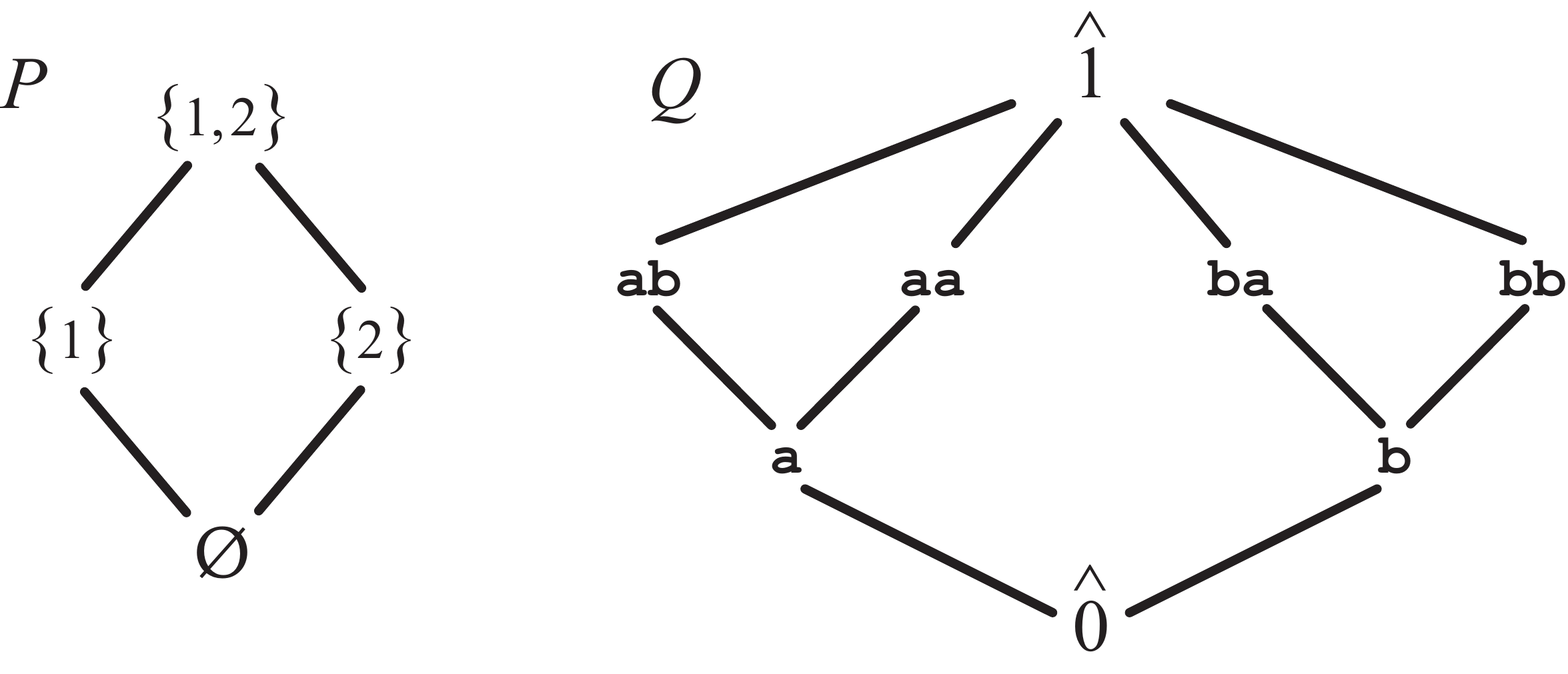}{width=4.75in}
\end{center}
\vspace*{-0.3in}
\caption[]{Poset $P$ models some sets of individuals; poset $Q$
  models some sequences of attributes.}
\label{dfaPQ}
\end{figure}

\paragraph{Observing an attribute:}\ Suppose we have observed
attribute ``\tb'', i.e., $q=\tb$.  What can we infer from $q$ in $P$
via $L$?  \quad Let us follow the protocol given earlier:

\vspace*{-0.2in}

\begin{itemize}
\item The subposet of $Q$ consisting of elements $q^{\prime}$
  greater than or equal to $q$ is:
\begin{minipage}{0.8in}
$$\begin{diagram}[size=0.125in]
       &         & \topone &         &        \\ 
       & \ruLine &     & \luLine &        \\ 
\tb\ta &         &     &         & \tb\tb \\
       & \luLine &     & \ruLine &        \\ 
       &         & \tb &         &        \\ 
\end{diagram}$$
\end{minipage}.

\item Consequently, $\Gamma$ is the following subposet of $L$:
\hspace*{0.1in}\begin{minipage}{2in}
$$\begin{diagram}[size=0.2in]
                   &         & (\{1, 2\}, \tb) &         &        \\ 
                   & \ruLine &                 & \luLine &        \\ 
(\{1\}, \tb\ta) &         &                 &         & (\{2\}, \tb\tb) \\
\end{diagram}$$
\end{minipage}.

\vspace*{0.25in}

\item There is one maximal element in $\Gamma$,
  so $\Gamma_{\rm max} = \{(\{1, 2\}, \tb)\}$.

\end{itemize}

Projecting onto the $P$ component tells us how to interpret $q$: The
observation ``\tb'' must have come from either individual \#1 or individual
\#2, as one would hope.  (This conclusion would hold as well if $P$
had modeled individual \#3 and if $Q$ had modeled attribute ``\tc''.)

\paragraph{Observing an individual:}\ Suppose we have observed
individual \#1, i.e., $p=\{1\}$.  What can we infer from $p$ in $Q$ via
$L$?  \quad Again, let us follow the inference protocol given earlier:

\vspace*{-0.2in}

\begin{itemize}
\item The subposet of $P$ consisting of elements $p^{\prime}$
  greater than or equal to $p$ is:
\hspace*{0.1in}\begin{minipage}{0.75in}
$$\begin{diagram}[size=0.2in]
\{1,2\}\\
\vLine\\
\{1\}\\
\end{diagram}$$
\end{minipage}.

\item Consequently, $\Sigma$ is the following subposet of $L$:
\hspace*{0.1in}\begin{minipage}{2in}
$$\begin{diagram}[size=0.2in]
(\{1,2\}, \ta)   & \qquad  &  (\{1,2\},\tb)  \\
   \vLine        &         &     \vLine      \\
(\{1\}, \ta\tb)  &         &  (\{1\},\tb\ta) \\
\end{diagram}$$
\end{minipage}.

\vspace*{0.25in}

\item The minimal elements of $\Sigma$ give us
      $\Sigma_{\rm min} = \{(\{1\}, \ta\tb),\; (\{1\}, \tb\ta)\}$.
\end{itemize}

Projecting onto the $Q$ component tells us how to interpret $p$: The
individual observed can or did reveal one of the two-attribute sequences
``\ta\tb'' or ``\tb\ta''.

\begin{figure}[t]
\begin{center}
\vspace*{-0.1in}
\ifig{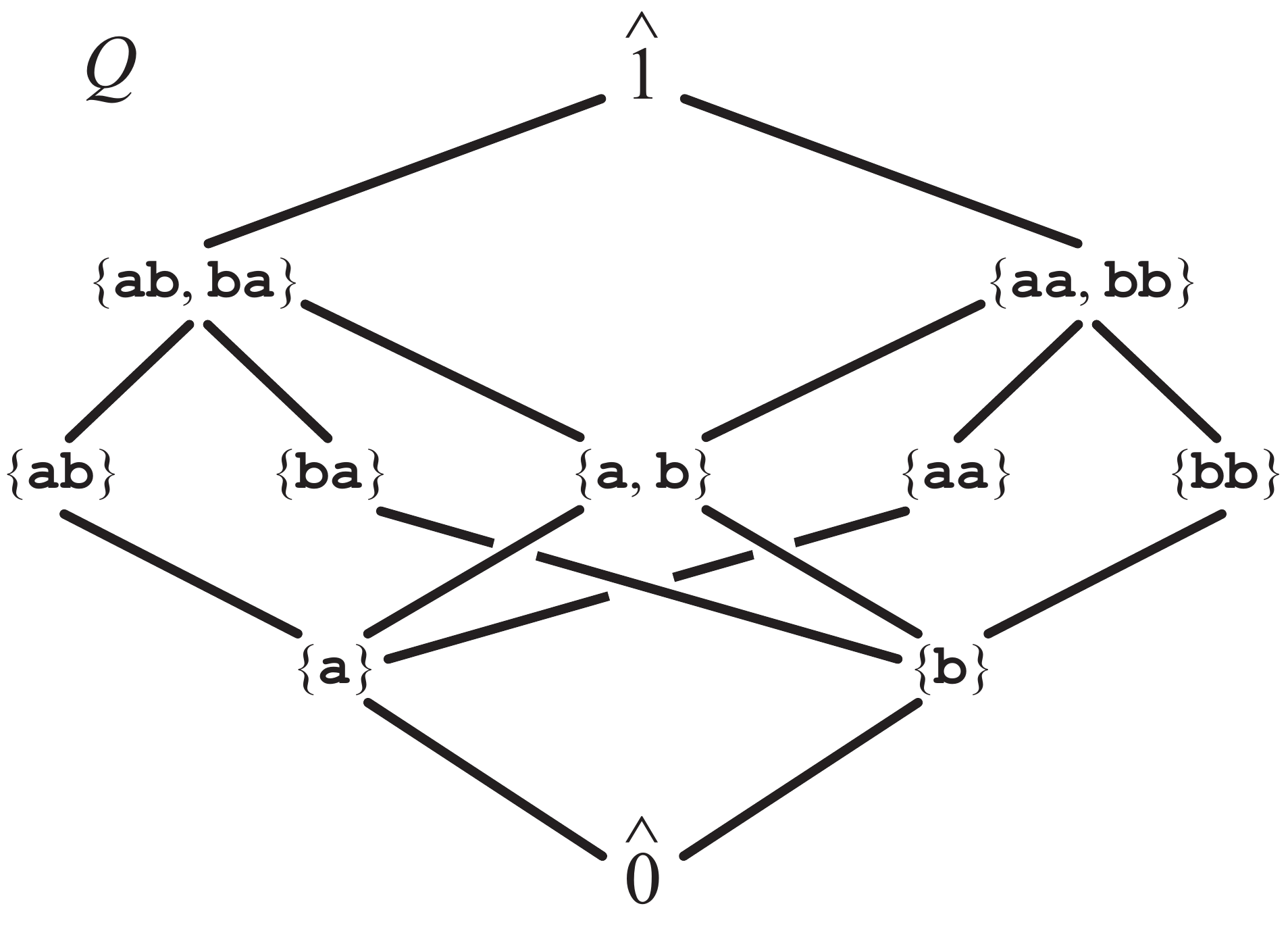}{height=2.75in}
\end{center}
\vspace*{-0.25in}
\caption[]{Poset $Q$ modeling {\em sets} of attribute sequences, for
  inferences in the lattice of Figure~\ref{dfalatticealt}.}
\label{dfaQalt}
\end{figure}

\paragraph{Comment:}\ The poset $Q$ of Figure~\ref{dfaPQ} would not
be relevant for inferences in the lattice of
Figure~\ref{dfalatticealt}, since that lattice now models attribute
observations involving ``\ta'' and/or ``\tb'' as sets of sequences
rather than merely as sequences.  We would instead probably want $Q$
to be something like the poset of Figure~\ref{dfaQalt}.  So even
though $L$ has become simpler than in Figure~\ref{dfalattice}, $Q$ has
become more complicated.  On the other hand, the new $(L,P,Q)$ triple
means that one can infer $(\{1, 2\}, \{\ta, \tb\})$ from the
observation ``\tb''.  As before, that says the observation ``\tb''
must have come from individual \#1 or \#2, but it also says directly
that the individual could alternatively have produced attribute
``\ta''.  In summary, by altering the triple $(L,P,Q)$, one changes
the possible inferences.

\paragraph{Aside:}\ The poset $Q$ of Figure~\ref{dfaQalt} is a
conveniently chosen finite subposet of a particular infinite poset
modeling sets of sequences.  In that model, each set is required to be
finite and {\em prefix-free}, meaning that if two distinct sequences
appear in an element of $\Qprop$, neither may be a prefix of the
other.  The partial order on $\Qprop$ is defined by: $q_1 \,\leq_Q\,
q_2$ precisely when every sequence in $q_1$ is a prefix of (possibly
equal to) some sequence in $q_2$.
\ \ (Notation: $\Qprop$ is the proper part of $Q$, that is, $\Qprop =
\sdiff{Q}{\{\botzero, \topone\}}$, and $\botzero < q < \topone$ for
every $q\in\Qprop$.)

\clearpage
\section{Lattices for Strategy Obfuscation}
\label{obfuscation}

In Section~\ref{sequencelattices}, we saw sublattices of powerset
lattices, those being prototypical examples of Boolean lattices.  A
related example is given by {\em strategy complexes}
\cite{tpr:strategies, tpr:plans}, which may be viewed as lattices of
(stochastic) partial orders formed from potentially nondeterministic
or stochastic transitions in a graph.  The basic elements in such a
lattice are {\em strategies} for attaining various goals.  Our work on
privacy now raises the question of strategy obfuscation: How can
someone reveal the {\em actions} of a strategy in a fashion that
delays identification of the strategy?

\subsection{Strategies for Nondeterministic Graphs}
\markright{Strategies for Nondeterministic Graphs}
\label{stratcomplex}

\begin{figure}[h]
\begin{center}
\vspace*{-0.1in}
\hspace*{-0.6in}\ifig{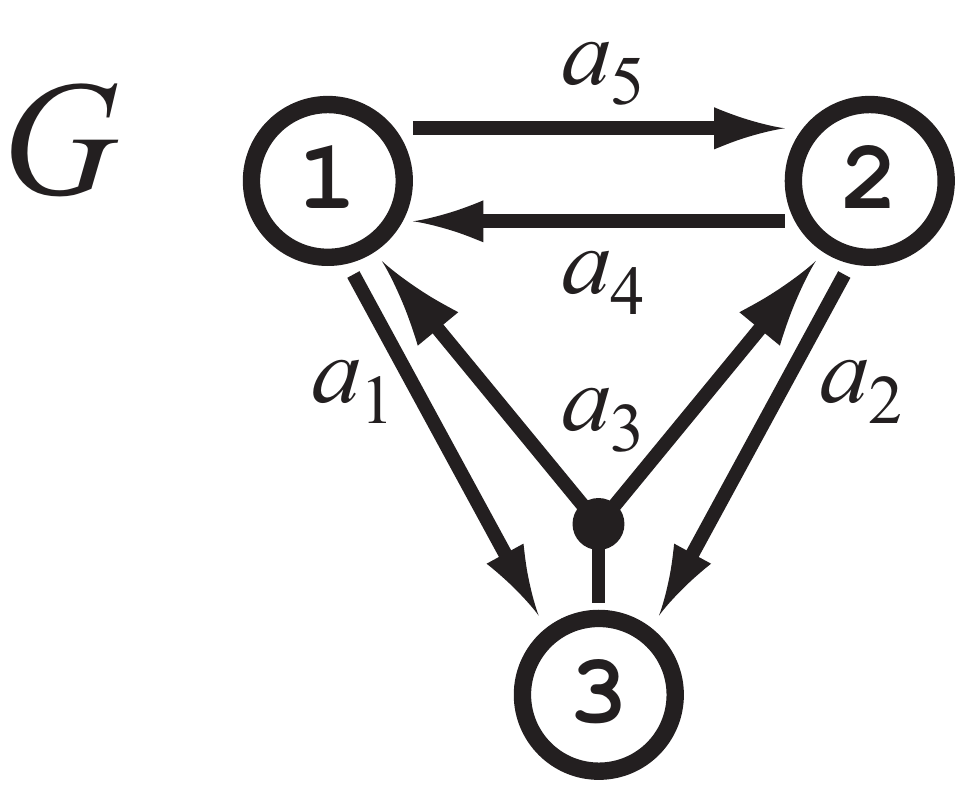}{height=1.8in}
\end{center}
\vspace*{-0.25in}
\caption[]{A graph $G$ with three states, four deterministic actions,
  and one nondeterministic action ($a_3$).}
\label{graph214}
\end{figure}

For a very simple example, consider the graph of
Figure~\ref{graph214}.  We might think of this graph as modeling some
kind of dynamic system, for instance, a person driving between three
shopping malls or a robot moving among clutter in a warehouse or an
intruder in a server network.

There are three states in the graph, along with five actions.  Each
action has a {\em source\,} state and one or more {\em target\,}
states, indicated in the figure by arrows.  An action may be {\em
executed\,} when the system is at the source state of the action,
causing the system to move from the action's source state to one of
its target states.

Four of the actions, $\{a_1, a_2, a_4, a_5\}$, are standard
deterministic directed edges, leading for certain from one state to
another.  The remaining action, $a_3$, is nondeterministic.
Nondeterminism of $a_3$ means that if the system is at state \tthree\
and executes action $a_3$, then the precise outcome is uncertain: The
system might move either to state \tone\ or to state \ttwo.
Nondeterminism is potentially adversarial: The precise target state
attained is unpredictable and could vary nonstochastically on
different executions of the action, perhaps determined by an adversary
outside the graph.  One may generalize this idea to include stochastic
actions along with deterministic and nondeterministic actions, thus
modeling adversarial combinations of Markov chains
\cite{tpr:strategies, tpr:plans}.

In the nondeterministic setting, a {\em strategy\hspt} is a set of
actions whose underlying directed edge set contains no directed
cycles.  The semantics of a strategy are: If the system is at the
source state of an action in the strategy, then the system executes
that action.  If the strategy contains multiple actions with that same
source state, then the actual action executed is again determined
nondeterministically.  For instance, in the example, if actions $a_1$
and $a_5$ both appear in a strategy, then the strategy is indifferent
as to whether the system will transition to state \ttwo\ or to state
\tthree\ from state \tone.  One or the other will occur.  If a
strategy does not contain any action with a given source state, then
the system will stop moving if it is ever in that state.

\vso

The lattice operations for strategies are set union and set
intersection, with one proviso: Suppose $\sigma_1$ and $\sigma_2$ are
two strategies.  Each strategy is a set of actions with no directed
cycles in its underlying directed edge set.  If the union of the two
strategies, $\sigma_1 \union\mskip1mu \sigma_2$, contains a directed
cycle in its underlying directed edge set, then the lattice operation
becomes $\sigma_1 \join \sigma_2 = \topone$, with $\topone$ the top
element of the lattice.  That top element represents cyclicity.  The
bottom element $\botzero$ of the lattice is equivalent to the empty
strategy $\emptyset$, amounting to no motion.

\begin{figure}[h]
\begin{center}
\vspace*{-0.1in}
\hspace*{-0.6in}\ifig{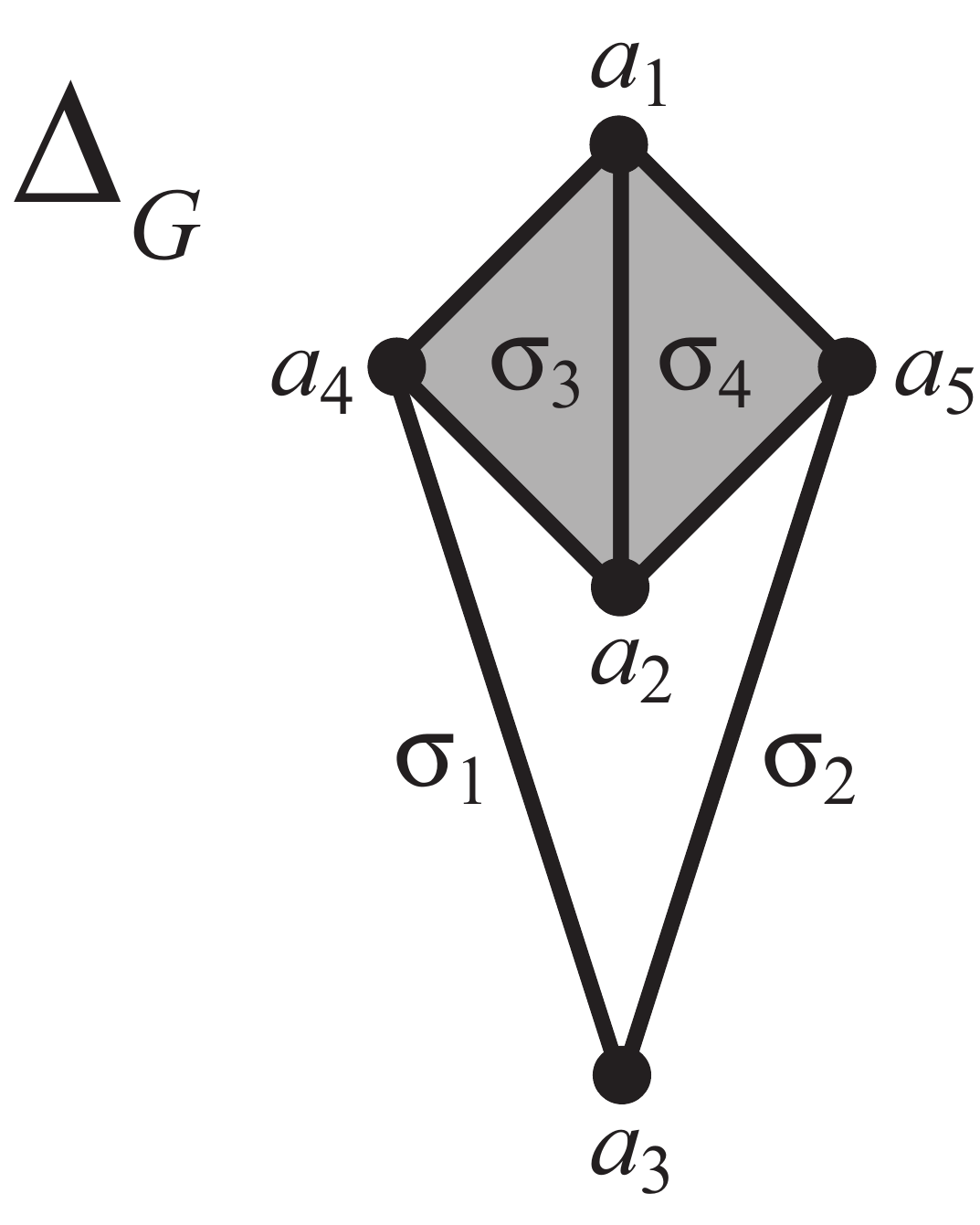}{height=2.5in}
\end{center}
\vspace*{-0.25in}
\caption[]{The strategy complex for the graph of
  Figure~\ref{graph214}.  Each vertex represents an action, as
  indicated by the labels.  Each maximal simplex also has a label, for
  the purposes of Figure~\ref{Arel214}.}
\label{strat214}
\end{figure}

Rather than draw a lattice of strategies $L$, it is more convenient to
draw an equivalent simplicial complex whose vertices are the (acyclic)
actions of the graph.  This simplicial complex is denoted by $\DG$ and
is called the {\em strategy complex}\, of $G$.
\label{stratcomplexDescrip}
The connection is that the proper part of the lattice is the face
poset of the simplicial complex, that is $\sdiff{L}{\{\botzero,
\topone\}} = \F(\DG)$.  Figure~\ref{strat214} shows the strategy
complex for the graph of Figure~\ref{graph214}.  The constituent
simplices of the strategy complex are strategies, that is, all sets of
actions whose underlying directed edge sets are acyclic.

\vspace*{0.05in}

Now that we have a simplicial complex, we can form a relation, whose
``individuals'' are all maximal strategies of the complex and whose
``attributes'' are the underlying actions, as shown in
Figure~\ref{Arel214}.  The figure also shows each maximal strategy's
{\em goal}, that is, the state at which the strategy would stop
moving.  \ (In general, a strategy, even a maximal strategy, may have a
multi-state goal set, but in this example the goals of all maximal
strategies are singleton states.)
\ Of course, a system could employ nonmaximal strategies, but for
identifiability purposes it is natural to consider maximal strategies.

\begin{figure}[h]
\begin{center}
\begin{minipage}{1.8in}{$\begin{array}{c|ccccc}
\hbox{\Large $A$} & a_1 & a_2 & a_3 & a_4 & a_5 \\[2pt]\hline
\sigma_1 &      &      & \one & \one &      \\
\sigma_2 &      &      & \one &      & \one \\
\sigma_3 & \one & \one &      & \one &      \\
\sigma_4 & \one & \one &      &      & \one \\
\end{array}$}
\end{minipage}
\quad
\begin{minipage}{0.5in}{$\begin{array}{c}
\hbox{Goal} \\[2pt]\hline
1 \\
2 \\
3 \\
3 \\
\end{array}$}
\end{minipage}
\end{center}
\vspace*{-0.15in}
\caption[]{Relation $A$ describes the strategy complex of
  Figure~\ref{strat214} in terms of its maximal simplices and their
  constituent actions.  The rightmost column shows each maximal
  strategy's goal, i.e., that state at which motion ceases.}
\label{Arel214}
\end{figure}

\vspace*{0.05in}

We make the following observations:

\vspace*{-0.05in}

\begin{itemize}
\item There is at least one strategy for attaining each state in the
  graph, meaning it is possible to move from every state to every
  other state, despite uncertainty in the outcome of one of the
  actions.  Such graphs are called {\em fully controllable} in
  \cite{tpr:strategies, tpr:plans}, and have properties similar to
  those of strongly connected directed graphs.

\item Each maximal strategy contains {\em \,two\,} informative
  attribute (i.e., action) release sequences, with each sequence
  consisting of two actions that together identify the strategy.  For
  instance, for $\sigma_1$, one could reveal actions $a_3$ and $a_4$
  in either order, identifying $\sigma_1$ only after revealing both
  actions.  For $\sigma_3$, one could reveal actions $a_1$ and $a_4$
  in either order, now identifying $\sigma_3$ only after revealing
  both actions.

\item Some actions reveal the goal even though they do not identify
  the maximal strategy.  In particular, actions $a_1$ and $a_2$ each
  individually reveal the goal to be 3.  (The two actions are in fact
  equivalent in $A$, in that either one implies the other.)  For
  instance, if one knows that $a_1$ is in a maximal strategy $\sigma$,
  then one knows that the strategy cannot also contain $a_3$, as
  adding $a_3$ would create a directed cycle in the underlying
  directed edge set.  Action $a_2$ must therefore also be in the
  strategy, since the strategy is maximal.  Consequently, the goal is
  state \tthree\ and $\sigma$ is either $\sigma_3$ or $\sigma_4$.  The
  difference between these two maximal strategies is a choice between
  $a_4$ and $a_5$.  That choice does not affect the final goal, but
  could affect intermediate motions and the time to reach the goal.  A
  rough analogy is knowing that a car on a freeway must continue on
  the freeway until at least the next exit but has a choice between
  lanes enroute.

\item Each maximal strategy contains at least {\em one\,} informative
  action release sequence consisting of two actions that do not reveal
  the goal until the second action has been released.  For instance,
  for $\sigma_3$, one could first release $a_4$, leaving open the
  possibility of either state \tone\ or state \tthree\ being the goal,
  then subsequently release either $a_1$ or $a_2$.

\end{itemize}

The rest of this section and Appendix~\ref{obfuscatingstrategies}
explore these observations more generally.

\subsection{Connecting the Topologies of Strategy Complexes and Privacy}
\markright{Connecting the Topologies of Strategy Complexes and Privacy}

{\bf Notation:}

\vspace*{-0.1in}

\begin{itemize}
\item $G=(V,\frakA)$ denotes a graph with states $V\!$ and actions
  $\frakA$.  An action may be deterministic, nondeterministic, or
  stochastic. (For simplicity, we assume here that $V\!\neq\emptyset$
  and $\frakA\neq\emptyset$.)

\label{GandDGdefsApp}

\item $\DG$ denotes the strategy complex of $G$; it includes the empty
  strategy $\emptyset$.

\end{itemize}

\begin{lemma}\label{dowkerstrats}
Let $G=(V,\frakA)$ be a graph as above and $\,\maxDG$ the set of
maximal simplices of $\DG$.

Define relation $A$ on $\maxDG \times \frakA$ by
$A = \setdef{(\sigma, a)}{a \in \sigma \in \maxDG}.$\quad
Then $\dowAy = \DG$.  In other words,
the Dowker complex over the set of actions is the same as
the graph's strategy complex.
\end{lemma}

\vspace*{-0.07in}

(The lemma holds more generally for simplicial complexes.
 The proof is nearly definitional.)

\vspace*{0.025in}

(The ``$A$'' stands for ``Action'' and we refer to relation $A$ as
$G$'s {\em action relation}.)

\vspace*{0.1in}

One of the fundamental results from \cite{tpr:strategies, tpr:plans}
is that a graph is fully controllable if and only if its strategy
complex is homotopic to a sphere of dimension two less than the number
of states in the graph:
\quad (Recall that ``$\homot$'' denotes a homotopy equivalence.)

\begin{theorem}
\label{controllability}
A graph $G=(V,\frakA)$ is fully controllable if and only if $\DG \homot
\Snt$, with $n=\abs{V}$.
\end{theorem}

Now recall our fundamental privacy result,
Corollary~\ref{holesreduceinference} from
page~\pageref{holesreduceinference}.  That corollary, along with
Theorem~\ref{controllability}, tells us that if a graph $G=(V,\frakA)$
is fully controllable, then the poset $\PA$, formed from relation $A$
of Lemma~\ref{dowkerstrats}, must contain at least $n!$ maximal
chains, each consisting of at least $n-1$ elements, with $n=\abs{V}$
(recall that the number of elements in a chain is one more than its
length).

We actually want a stronger result, speaking to individual strategies
and we can get that by looking into the details of the proof of
Theorem~\ref{manychains}.  The proof is an induction that recursively
considers links, giving us the following (see
Appendices~\ref{manylongchains} and \ref{obfuscatingstrategies}):

\newcounter{stratIDdelay}
\setcounter{stratIDdelay}{\value{theorem}}
\begin{theorem}[Delaying Strategy Identification]
\label{stratdelay}
Let $G=(V,\frakA)$ be a fully controllable graph, with $n=\abs{V} > 1$.
Let $A$ be the relation constructed as in Lemma~\ref{dowkerstrats} and
let $\PA$ be its associated doubly-labeled poset.  Then:

For each $v\in V$, there exists a maximal strategy $\sigma_v \in \DG$
for attaining singleton goal state $v$ such that $\PA$ contains at least
$(n-1)!$ distinct maximal chains for identifying $\sigma_v$, with each
chain consisting of at least $n-1$ elements.
\end{theorem}

{\bf Clarifying Observation:}\ Each maximal chain for identifying
$\sigma_v$ specifies, via the construction of Lemma~\ref{chaintoiars}
on page~\pageref{chaintoiars}, at least $n-1$ actions and an order for
releasing them, such that no action is implied by those previously
released.  In particular, the sequence of actions does not identify
$\sigma_v$ until all actions have been released.

\vspace*{0.05in}

{\bf Comments:}\ Theorem~\ref{stratdelay} does {\em not\hspt} assert that
{\em every}\, maximal strategy in $\DG$ has $(n-1)!$ many ``long''
identifying chains, merely that, for every possible singleton goal $v$,
there is {\em some}\, strategy for attaining $v$ with $(n-1)!$ many
``long'' identifying chains.  It is not hard to construct examples for
which some maximal strategy has fewer than $(n-1)!$ identifying chains
(see Section~\ref{multiexample}).  \ This fact raises further
questions ($G$ is assumed fully controllable throughout):

\begin{itemize}

\vspace*{-0.05in}

\label{cycletheorems}

\item Given an arbitrary maximal strategy $\sigma_v$ for attaining a
  singleton goal state $v$, can we find at least {\em one}\, chain in
  $\PA$ that identifies $\sigma_v$ but requires release of at least
  $n-1$ actions before doing so?  The answer in general is ``no'', but
  ``yes'' for certain kinds of graphs.

  One can construct counterexamples, in which the strategy $\sigma_v$
  is always inferable before $n-1$ of its actions have been revealed,
  regardless of the order in which one reveals the actions.
  Appendix~\ref{shortiars} describes one such example, containing a
  mix of stochastic and nondeterministic actions.
  Nonetheless, even for such mixed graphs one can describe situations
  in which the answer is ``yes''.  This occurs for instance when the
  graph contains a Hamiltonian cycle consisting of directed edges that
  arise from deterministic or stochastic actions (see
  Lemma~\ref{hamiltonianhelp} on page~\pageref{hamiltonianhelp}, in
  Appendix~\ref{hamiltoniangraph}).
  Leveraging that insight, one can prove that, for \spc{\em pure\,}
  nondeterministic or stochastic graphs (defined on
  page~\pageref{puregraphs} in Appendix~\ref{puregraphs}), every
  maximal strategy (even one with a multi-state goal) has an
  informative action release sequence of length at least $n$$-$$1$.

\vspace*{-0.425in}

\paragraph{}$\phantom{0}$

\item \label{singletongoaldelay} Given a singleton goal state $v$, can
  we find at least one maximal strategy $\tau_v$ and at least one
  chain in $\PA$ that eventually identifies $\tau_v$, but does not
  reveal the goal $v$ before releasing at least $n-1$ actions?  The
  answer to this question is ``yes''.  The proof operates by
  repeatedly creating quotient graphs.  In forming a quotient graph,
  the proof regards as equivalent a certain set of states that are
  connected by a cycle of directed edges, with each edge coming from
  some deterministic or stochastic action.  For instance, in the graph
  of Figure~\ref{graph214}, the proof would regard states
  \tone\ and \ttwo\ as equivalent.  The resulting quotient graph would
  then consist of two states with deterministic actions between them,
  since action $a_3$ becomes a deterministic transition in the
  quotient graph.  Inductively, one therefore sees that an entity can
  hide its true goal until at least two actions in the original graph
  $G$ have been revealed.  \ (See Appendix~\ref{delaygoalrecog} for
  further details.)

\label{cycletheoremsend}

\end{itemize}

\vspace*{-0.15in}

\paragraph{\hspace*{0.25in}A comment/caution regarding the availability of many chains:}\ The
\label{chaininferencecaution}
$(n-1)!$ chains mentioned above may come from all possible
permutations of the same underling set of $\mskip1.5mu{}n-1$ actions.
Alternatively, these $(n-1)!$ chains may involve creative sequencing
of more than $n-1$ actions.  The precise makeup of the chains depends
on the underlying homology generators.  However, even if the chains
are merely reordering the same $n-1$ actions, there is good reason to
take advantage of that capability, rather than pick one particular
sequence via a deterministic algorithm.  The reason is that knowledge
of how an algorithm releases actions may leak information to an
adversary.  Such leakage may be understood as changing the effective
relation.  For instance, despite thinking one is working with relation
$A$, a particular release protocol may simply be focusing on some
proper subset of $A$ or some proper subset of the poset $\PA$,
possibly resulting in very different inference characteristics.  A
good release strategy may be to choose randomly from among the
$(n-1)!$ possible chains.  In that way, one is taking good advantage
of the spherical homogeneity suggested by homology.

\subsection{Example: Multi-State Goals and Multi-Strategy Singleton Goals}
\markright{Example: Multi-State Goals and Multi-Strategy Singleton Goals}
\label{multiexample}

\begin{figure}[h]
\begin{center}
\vspace*{-0.1in}
\ifig{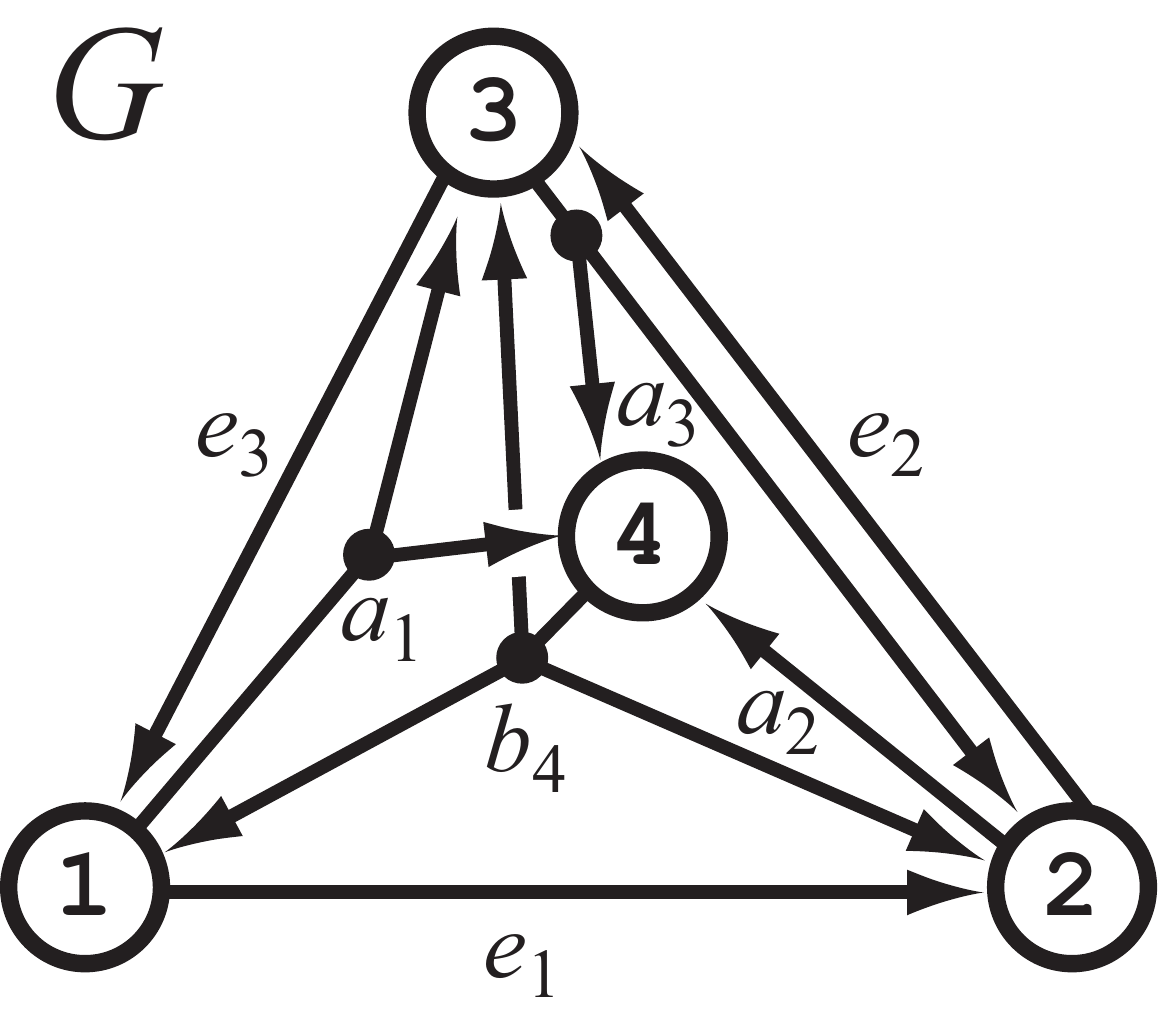}{height=2in}
\end{center}
\vspace*{-0.3in}
\caption[]{A graph with four states $\{1,2,3,4\}$, four deterministic
  actions $\{e_1, e_2, e_3, a_2\}$, and three nondeterministic actions
  $\{a_1, a_3, b_4\}$.}
\label{graph202}
\end{figure}

Figure~\ref{graph202} shows a fully controllable nondeterministic
graph on four states.  The graph contains four deterministic actions
and three nondeterministic actions.

\vspace*{0.025in}

Three of the deterministic actions form a directed cycle: $1
\xrightarrow{e_1} 2 \xrightarrow{e_2} 3 \xrightarrow{e_3} 1$.  The
remaining deterministic action, $a_2$, moves from state 2 to state 4.

Actions $a_1$, $a_3$, and $b_4$ are nondeterministic.  Action $a_1$
moves nondeterministically from state 1 to either state 3 or state 4,
while action $a_3$ moves nondeterministically from state 3 to either
state 2 or state 4.  Finally, action $b_4$ moves nondeterministically
back from state 4 to any of the other three states.

\vspace*{0.025in}

Figure~\ref{Arel202} shows relation $A$ for the graph of
Figure~\ref{graph202}, with $A$ as defined in Lemma~\ref{dowkerstrats}
on page~\pageref{dowkerstrats}.  As indicated in the figure, some
maximal strategies have two-state goals.  In addition, two of the
maximal strategies converge to the same singleton goal, namely state
$4$.

\begin{figure}[h]
\begin{center}
\begin{minipage}{2.5in}{$\begin{array}{c|ccccccc}
       A & e_1  & e_2  & e_3  & a_1  & a_2  & a_3  & b_4 \\[2pt]\hline
\sigma_1 &      & \one & \one &      &      &      & \one \\[2pt]
\sigma_2 & \one &      & \one &      &      &      & \one \\[2pt]
\sigma_3 & \one & \one &      &      &      &      & \one \\[2pt]
\sigma_4 & \one &      & \one &      & \one & \one &      \\[2pt]
\sigma_5 & \one &      &      & \one & \one & \one &      \\[2pt]
\sigma_{14} &      & \one & \one &      & \one &      &      \\[2pt]
\sigma_{34} & \one & \one &      & \one & \one &      &      \\[2pt]
\end{array}$}
\end{minipage}
\quad
\begin{minipage}{0.5in}{$\begin{array}{c}
\hbox{Goal} \\[2pt]\hline
1 \\[2pt]
2 \\[2pt]
3 \\[2pt]
4 \\[2pt]
4 \\[2pt]
\{1,4\} \\[2pt]
\{3,4\} \\[2pt]
\end{array}$}
\end{minipage}
\end{center}
\vspace*{-0.15in}
\caption[]{Relation $A$ describes the strategy complex for the graph
  of Figure~\ref{graph202} in terms of its maximal strategies and
  their constituent actions.  The rightmost column further shows each
  maximal strategy's goal.  Observe that some strategies converge to
  multi-state goals.}
\label{Arel202}
\end{figure}

\begin{figure}[h]
\begin{center}
\begin{minipage}{2.5in}{$\begin{array}{c|cccc}
          Q & e_1  & a_1  & a_2  & a_3  \\[2pt]\hline
\sigma_2    & \one &      &      &      \\[2pt]
\sigma_3    & \one &      &      &      \\[2pt]
\sigma_4    & \one &      & \one & \one \\[2pt]
\sigma_{14} &      &      & \one &      \\[2pt]
\sigma_{34} & \one & \one & \one &      \\[2pt]
\end{array}$}
\end{minipage}
\quad
\begin{minipage}{0.5in}{$\begin{array}{c}
\hbox{Goal} \\[2pt]\hline
2 \\[2pt]
3 \\[2pt]
4 \\[2pt]
\{1,4\} \\[2pt]
\{3,4\} \\[2pt]
\end{array}$}
\end{minipage}
\end{center}
\vspace*{-0.1in}
\caption[]{Relation $Q$ models $\lk(\dowAx, \sigma_5)$, with $A$ as in
  Figure~\ref{Arel202}.}
\label{Qrel202}
\end{figure}

\begin{figure}[h]
\vspace*{0.2in}
\begin{center}
\ifig{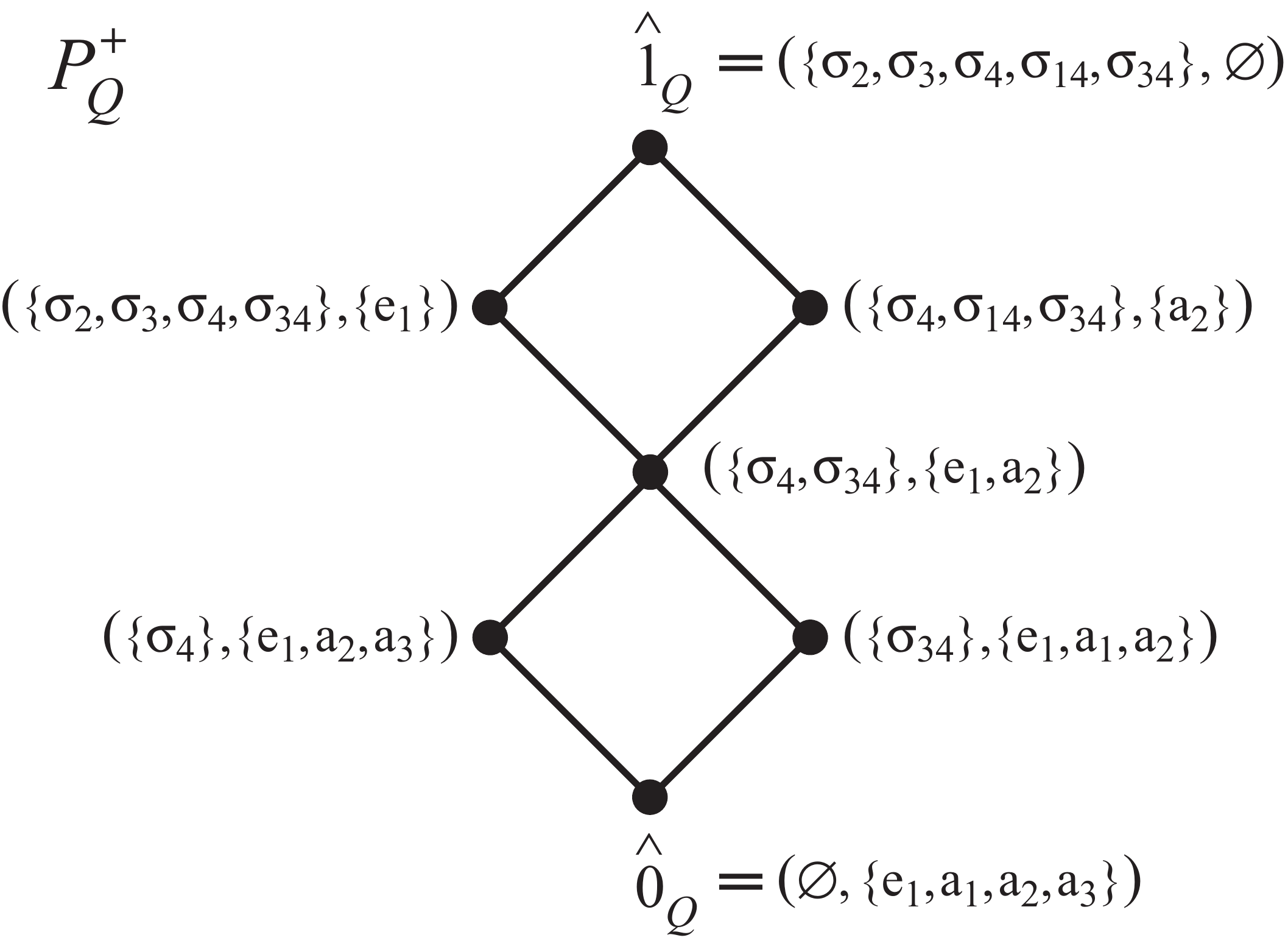}{scale=0.475}
\end{center}
\vspace*{-0.2in}
\caption[]{The Galois lattice $\PQplus$ has length 4, with $Q$ as in
  Figure~\ref{Qrel202}.}
\label{Qlattice202}
\end{figure}

Strategies $\sigma_1$, $\sigma_2$, $\sigma_3$, and $\sigma_4$ each
have (at least) six different informative action release sequences of
length (at least) 3, as guaranteed to exist in $\DG$ by
Theorem~\ref{stratdelay}.  Strategy $\sigma_5$ also has informative
action release sequences of length at least 3 (in fact, length 4), but
it only has four such sequences, consistent with the comments on
page~\pageref{cycletheorems}.  We can see this by constructing the
Galois lattice $\PQplus$, with $Q$ modeling $\lk(\dowAx, \sigma_5)$.
Figures~\ref{Qrel202} and \ref{Qlattice202} depict $Q$ and $\PQplus$,
respectively.  There are indeed four downward paths of length 4 from
$\oneQ$ to $\zeroQ$.

In order to delay identification of $\sigma_5$ as long as possible,
the lattice $\PQplus$ further tells us that one should reveal action
$a_2$ either right away or right after first revealing action $e_1$.
However, as soon as one has revealed action $a_2$, an observer knows
that the goal is either state 4 or a two-state set containing state 4.
If the observer has adversarial control over the outcome of
nondeterministic actions in $G$, then as soon as one has revealed two
of $\sigma_5$'s actions, the observer-adversary can lie in wait for
the system at state 4.

Of course, as indicated on page~\pageref{singletongoaldelay}, one can
delay goal recognition for 3 steps if one is free to choose any
strategy for that goal.  For goal state 4, one should choose strategy
$\sigma_4$ rather than $\sigma_5$, and reveal actions $e_1$, $e_3$,
$a_2$, either in that order or in the order $e_3$, $e_1$, $a_2$.

\vspace*{-0.05in}

\subsection{Randomization}
\markright{Randomization}

Suppose $G=(V,\frakA)$ is a fully controllable graph with $\abs{V} =
\abs{\frakA} = n > 1$.  These conditions imply that there is exactly
one action at each state and that the maximal simplices of $\DG$
consist of all subsets of $\frakA$ of size $n-1$.  Consequently, $\DG$
is a boundary complex, specifically $\DG=\bndry{(\frakA)}$.  As we
have seen, such complexes preserve attribute privacy, meaning it is
impossible to infer any additional actions from actions already
revealed.  The existence of $(n-1)!$ different informative action
release sequences of length $n-1$ for any given maximal strategy here
simply means that one may reveal the actions of that strategy in any
order.

The technical complications discussed previously arise when there are
multiple actions at some or all states of $V$.  One may circumvent
such complications by creating a single ``effective action'' at each
state, for instance by choosing stochastically among the given actions
available at a state when the system is in that state.  The precise
probabilities are not significant from the perspective of
combinatorial strategy obfuscation, so long as the probability of
choosing any given original action is greater than 0, and all such
probabilities sum to 1.  (Of course, the actual probabilities will
affect the expected time to attain the goal.)

A related issue concerns execution order versus release order.  The
privacy results in this report assume that a system can control the
order in which it reveals actions.  If instead actions are revealed as
they are executed, then an observer may be able to infer the
underlying strategy more quickly than desired.  In order to obfuscate
the strategy, the system may need to be willing to ignore early
arrival at the goal and instead continue moving.  The precise criteria
determining whether the system stops or continues could be stochastic,
or could reflect a protocol determined by the desired action release
sequence.

\clearpage
\section{Relations as a Category}
\label{category}

We have discussed disinformation, obfuscation, and other manipulation
of relations.  The goal of such transformations has been to preserve
privacy by removing or hiding free faces.  We have not yet discussed
such transformations formally.  For instance, the coordinate
transformations of Section~\ref{coordchanges} raise the question:

\begin{center}
How should one think about maps between relations?
\end{center}

\subsection{Relationship-Preserving Morphisms}
\markright{Relationship-Preserving Morphisms}

Traditionally, relations are themselves morphisms between sets (with
functions a special case).  In thinking about privacy, it is useful to
define a category in which relations are the objects.
We have some choices in defining morphisms for this category.  Bearing
in mind our Dowker constructions (see again Definition~\ref{basicdefs}
on page~\pageref{basicdefs}), we adopt the following standard definition:

\paragraph{Notation:}\ (1) We frequently will be working with two
relations: $R$ is a relation on $\XR \times \YR$ and $Q$ is a relation
on $\XQ \times \YQ$ (the superscripts are just indices to indicate the
underlying relation).  In order to distinguish rows and columns
between the two relations, we will also use notation of the form
$\XRy$, $\YRx$, $\XQy$, and $\YQx$. \quad (2) By a {\em set map\,} we
mean a function between two sets.

\vspace*{0.05in}

\newcounter{morphismDef}
\setcounter{morphismDef}{\value{theorem}}
\begin{definition}[Morphism]\label{morphism}
Let $R$ be a relation on $\XR \times \YR$ and let $Q$ be a relation on
$\XQ \times \YQ$.
\ A \mydefem{morphism of relations} \ $f : R \rightarrow Q$ is a pair
of set maps:

\vspace*{-0.2in}

\begin{eqnarray*}
\fX &:& \XR \rightarrow \XQ\\[2pt]
\fY &:& \YR \rightarrow \YQ
\end{eqnarray*}

\noindent such that $\big(\fX(x), \,\fY(y)\big) \in Q$ whenever $(x,y) \in R$.
\end{definition}

\noindent In other words, a morphism of relations maps individuals to
individuals and attributes to attributes in a way that preserves
relationships.

\vspace*{0.1in}

\noindent The following lemma follows from the definitions (a proof
appears in Appendix~\ref{morphismpropertiesApp}):

\newcounter{SimplicialMorphisms}
\setcounter{SimplicialMorphisms}{\value{theorem}}
\begin{lemma}[Induced Simplicial Maps]\label{simpmaps}
A morphism $f : R \rightarrow Q$ between nonvoid relations induces
simplicial maps between the Dowker complexes:

\vspace*{-0.2in}

\begin{eqnarray*}
\fX &:& \dowx \rightarrow \dowqx\\[2pt]
\fY &:& \dowy \rightarrow \dowqy
\end{eqnarray*}
\end{lemma}

\paragraph{Notational comment: } The symbols $f_X$ and $f_Y$ are
overloaded intentionally.  The simplicial map $f_X$ is precisely the
set map $f_X$ applied to the vertices of any simplex: \quad If
$\sigma=\{x_0, \ldots, x_k\} \in \dowx$, then $f_X(\sigma) =
\{f_X(x_0), \ldots, f_X(x_k)\} \in \dowqx$.  Similarly for $f_Y$.

\vspace*{0.1in}

\noindent {\bf Connectivity Implication:} \ Intuitively, one cannot
partition the individuals of a connected relation into two or more
pairwise disjoint classes without misclassifying some individuals or
ignoring some relationships.  A graph connectivity argument provides a
possible proof.  Lemma~\ref{simpmaps} provides another, with
additional insight.  \ Let us look at some examples:

\paragraph{Two Bits onto One:}
Consider again the relations $S$ and $Q$ of Figures \ref{onebit1} and
\ref{twobit2}, respectively, on page~\pageref{onebit1}.
Relation $S$ models a one-bit relation --- an attribute and its
negation.  Relation $Q$ models a two-bit relation --- two attributes
and their negations.  The Dowker complexes for $S$ have $\Szero$
homotopy type, while those for $Q$ have $\Sone$ homotopy type.  We can
think of $S$ as a classification, splitting individuals into those
that have some attribute $\ta$ and those that do not.

By Lemma~\ref{simpmaps}, a morphism $f : Q \rightarrow S$ induces
simplicial (hence continuous) maps between the corresponding Dowker
complexes of $S$ and $Q$.  Since $\Sone$ is connected but $\Szero$ is
not, there is no surjective continuous function from $\Sone$ to
$\Szero$.  Consequently, no morphism $f : Q \rightarrow S$ can truly
be a classification: $f_Y$ can map all four attributes $\{\ta,
\neg\ta, \tb, \neg\tb\}$ of $Q$ to the single attribute $\ta$ or all
four attributes to $\neg\ta$, but $f_Y$ cannot map to both $\ta$ and
$\neg\ta$.

\begin{figure}[h]
\begin{center}
\qquad\qquad
\begin{minipage}{1in}{$\begin{array}{c|cccc}
\hbox{\large $\;Q^\prime$} & \ta  & \neg\ta  \\[2pt]\hline
              1          & \one &          \\
              2          & \one &          \\
              3          &      & \one     \\
              4          &      & \one     \\
\end{array}$}
\end{minipage}
\hspace*{1in}
\begin{minipage}{3in}
\ifig{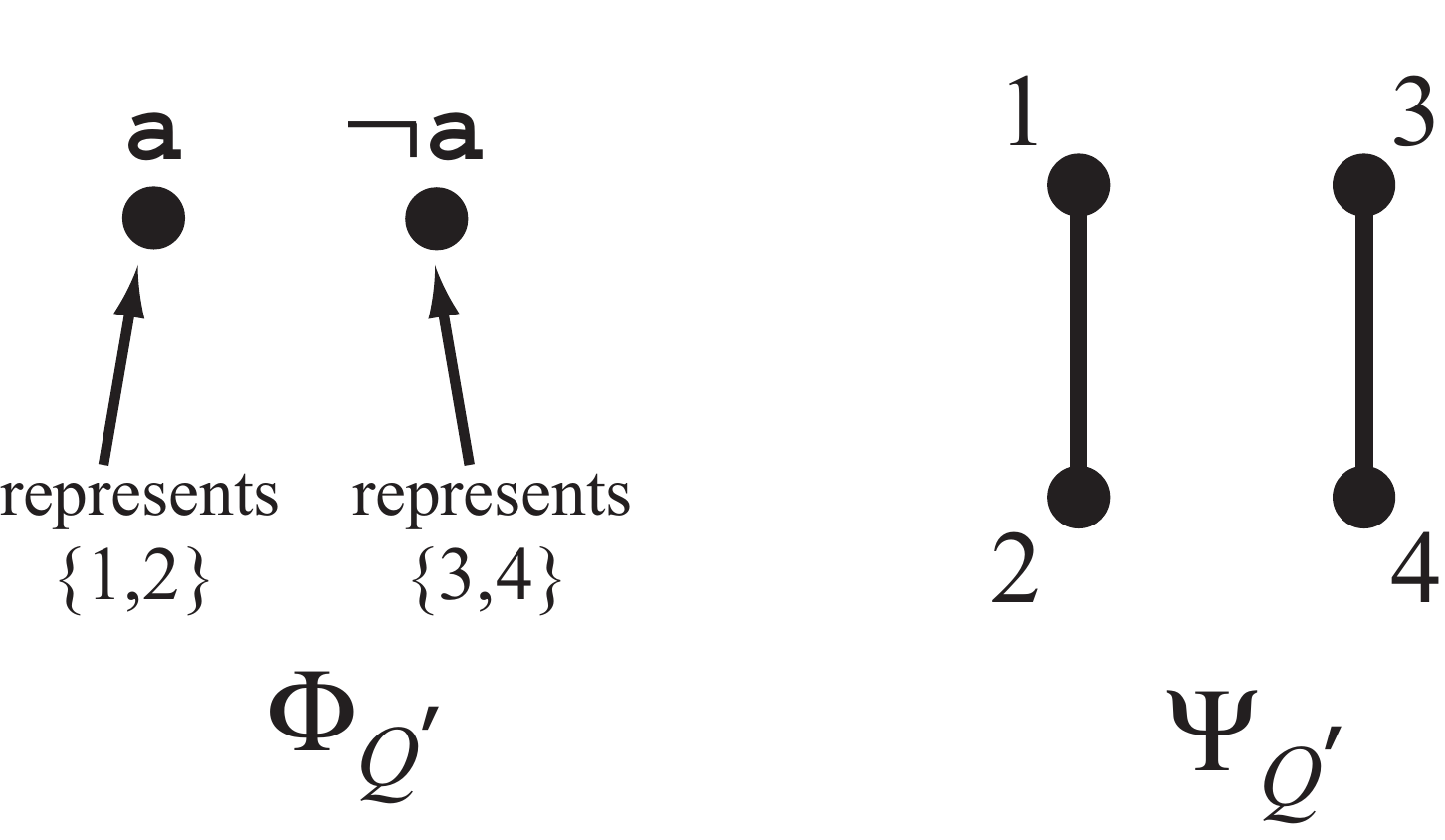}{scale=0.4}
\end{minipage}
\end{center}
\vspace*{-0.2in}
\caption[]{Relation $Q^\prime$ obtained from relation $Q$ of
  Fig.~\ref{twobit2} by discarding attributes $\tb$ and $\neg\tb$.}
\label{twobitcut}
\end{figure}

This impossibility may at first seem paradoxical.  After all, one can
simply cut relation $Q$ down the middle and throw away the columns
involving attributes $\tb$ and $\neg\tb$, as shown in Figure
\ref{twobitcut}.  After that, a surjective morphism $f^\prime :
Q^\prime \rightarrow S$ is immediate.  Indeed, that is possible.
However, in so doing, one has discarded some relationships, perhaps
purposefully, perhaps accidentally.  In particular, the relationship
between individuals \#1 and \#3 of $Q$ via attribute $\tb$ is lost, as
is the relationship between individuals \#2 and \#4 via attribute
$\neg\tb$.  This reasoning simply underscores the fact that morphisms
of relations preserve relationships.  Lack of continuity in a function
therefore is a sign that one is discarding some relationships.
Whether such discard is desirable depends on one's goals in a
particular application.

\paragraph{Three Bits onto Two:}
Recall as well Figure~\ref{threebit3} on page~\pageref{threebit3},
which depicts a three-bit relation $R$ --- three attributes and their
negations, capable of distinguishing between eight individuals.  The
homotopy type of the Dowker complexes is $\Stwo$.  With $Q$ as above,
the following question arises naturally when trying to reduce
complexity of data yet preserve information:

\begin{center}
Does there exist a surjective morphism $f : R \rightarrow Q$ ?
\end{center}

Unlike the previous example, there do exist continuous maps from
$\Stwo$ onto $\Sone$, so perhaps one can find a surjective morphism $f
: R \rightarrow Q$.  \ In fact, one can not.  Intuitively, the issue
is that the two-dimensional relationships of $R$ try to fill the
one-dimensional hole of relation $Q$.
\ Here is a simplex-based argument:

\vspace*{-0.05in}

\begin{itemize}

\addtolength{\itemsep}{-2pt}

\item Suppose surjective $f : R \rightarrow Q$ exists.  As will be
  discussed later (see page \pageref{morphismprop}), this means the
  component functions $f_X : \XR \rightarrow \XQ$ and $f_Y :
  \YR \rightarrow \YQ$ are surjective as set maps.

\item One may therefore assume without loss of generality that
      $f_Y(\ta) = \ta$ and $f_Y(\tb) = \tb$.

\item The triangles $\{\ta, \tb, \tc\}$ and $\{\ta, \tb, \neg\tc\}$
  are both simplices in $\dowy$.  The maximal simplices of $\dowqy$
  are edges.

\item By Lemma~\ref{simpmaps}, this means that $f_Y(\tc)$ and
  $f_Y(\neg\tc)$ are both elements of $\{\ta, \tb\}$ in $\dowqy$.

\item Again by surjectivity, we therefore see that $\{f_Y(\neg\ta),
  f_Y(\neg\tb)\} = \{\neg\ta, \neg\tb\}$.

\item Another triangle-versus-edge argument then says that $f_Y(\tc)$
  and $f_Y(\neg\tc)$ are both elements of $\{\neg\ta, \neg\tb\}$,
  giving us a contradiction.

\end{itemize}

Of course, as in constructing $Q^\prime$ of Figure~\ref{twobitcut}, if
we are willing to tolerate discontinuities, we could discard one
attribute and its negation to obtain $Q$ from $R$.  As before,
discontinuity means losing awareness of some relationship(s).  For
instance, if we omit attribute $\tc$, we would become unaware in $Q$
of the relationship that exists in $R$ among the set of individuals
$\{1,3,5,7\}$.

\subsection{Privacy-Establishing Morphisms}
\markright{Privacy-Establishing Morphisms}

\begin{figure}[h]
\begin{center}
\quad
\begin{minipage}{1.5in}{$\begin{array}{c|ccccc}
\hbox{\large $M$} & \ta  & \tb  & \tc  & \td  & \te  \\[2pt]\hline
              1   & \one & \one &      &      & \one \\
              2   & \one & \one & \one &      &      \\
              3   &      & \one & \one & \one &      \\
              4   &      &      & \one & \one & \one \\
              5   & \one &      &      & \one & \one \\
\end{array}$}
\end{minipage}
\hspace*{0.75in}
\begin{minipage}{3in}
\ifig{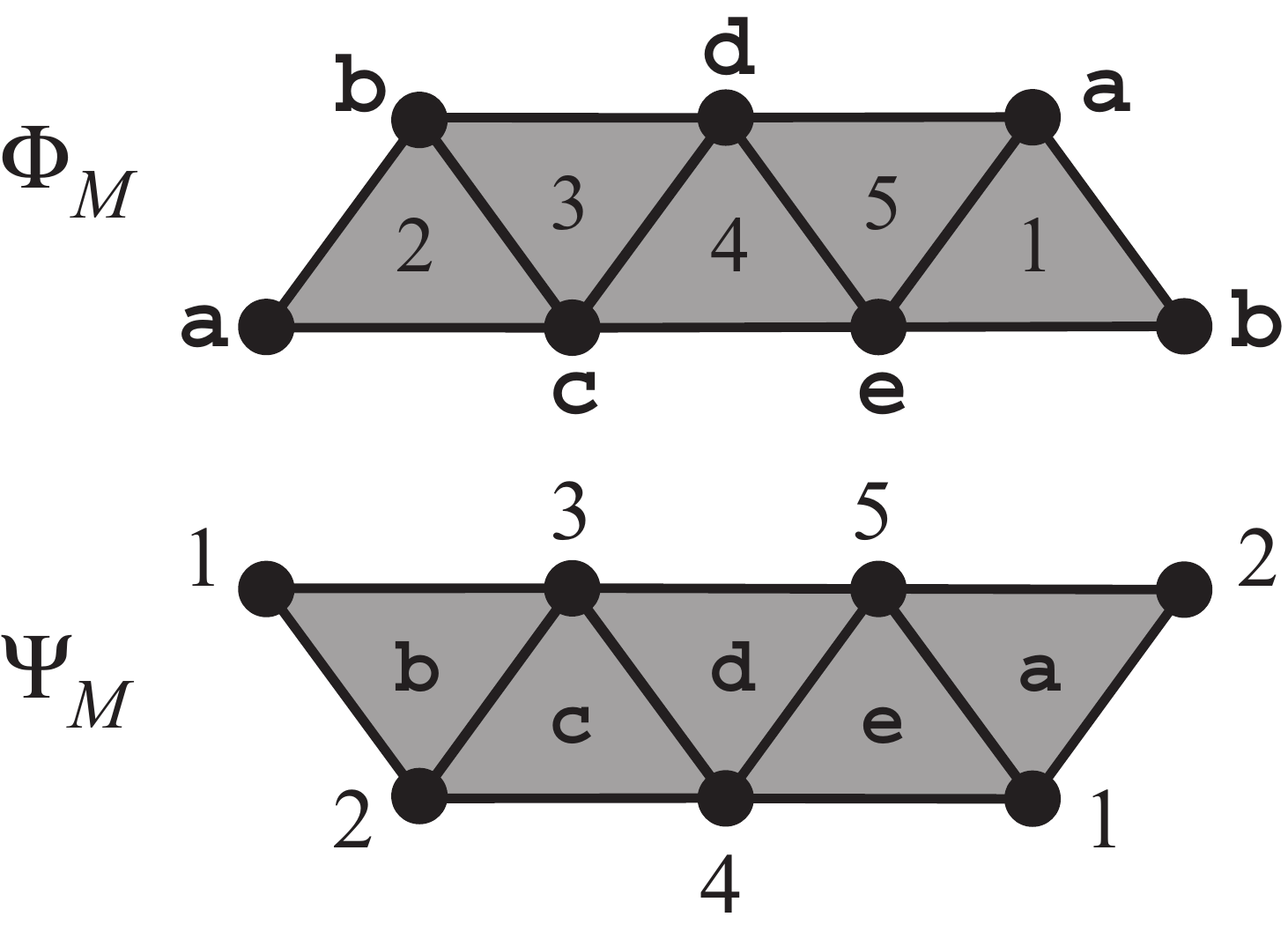}{scale=0.4}
\end{minipage}
\end{center}
\vspace*{-0.1in}
\caption[]{Relation $M$ is isomorphic to relation $G$ of
  Figure~\ref{travelguides} on page~\pageref{travelguides}, now
  without the author-book semantics.  The Dowker complexes are dual
  triangulations of the M\"obius strip, with $\Sone$ homotopy type.}
\label{moebius}
\end{figure}

Relations involve two spaces.  Looking at just $\dowy$ or just $\dowx$
may hide some interesting properties.  For instance, consider the
M\"obius strip relation $M$ of Figure~\ref{moebius}.  We encountered
this relation previously, in Section~\ref{leveraginglattices}.

We might wish to remove some of the inferences discussed in
Section~\ref{leveraginglattices} by reshaping the underlying relation
without discarding any relationships.  Doing so leads to the following
question:

\vspace*{-0.05in}

\begin{center}
\begin{minipage}{4.2in}
\label{surjqueststart}
Does there exist a surjective morphism $f : M \rightarrow T$, with $T$
a relation that preserves both attribute and association privacy ?
\end{minipage}
\end{center}

\vspace*{0.05in}

If such a morphism $f$ exists, then, as we mentioned in
Sections~\ref{faceshape} and \ref{holemeaning}, relation $T$ must have
the topology of either a linear cycle or a spherical boundary complex.
It turns out that the answer to the question above is ``yes'', with
$T$ being a relation whose Dowker complexes are boundaries of
tetrahedra (see Figure~\ref{tetrahedra2} on
page~\pageref{tetrahedra2}).

This construction is not immediately obvious from the complexes
$\dowMy$ and $\dowMx$.  Although those simplicial complexes are
2-dimensional, suggesting that their triangles can be wrapped around a
tetrahedron, doing so actually collapses two of the five triangles to
edges.  Indeed, the component functions for one such surjective
morphism $f : M \rightarrow T$ are:

\vspace*{-0.1in}

\begin{center}
\begin{minipage}{2.5in}{%
\begin{eqnarray*}
f_X \;:\;  X^M   & \rightarrow & X^T  \\[2pt]
           1     & \mapsto     & 4    \\[-1pt]
           2     & \mapsto     & 1    \\[-1pt]
           3     & \mapsto     & 2    \\[-1pt]
           4     & \mapsto     & 3    \\[-1pt]
           5     & \mapsto     & 4    \\[-1pt]
\end{eqnarray*}}
\end{minipage}
\begin{minipage}{2.5in}{%
\begin{eqnarray*}
f_Y \;:\; Y^M   & \rightarrow & Y^T  \\[2pt]
          \ta   & \mapsto     & \ta  \\[-1pt]
          \tb   & \mapsto     & \tb  \\[-1pt]
          \tc   & \mapsto     & \tc  \\[-1pt]
          \td   & \mapsto     & \td  \\[-1pt]
          \te   & \mapsto     & \ta  \\
\end{eqnarray*}}
\end{minipage}
\end{center}

\vspace*{-0.35in}

\noindent The induced simplicial maps act on the five maximal
simplices of $\dowMx$ and $\dowMy$ as follows:

\vspace*{-0.15in}

\begin{center}
\label{MTproj}
\begin{minipage}{2.5in}{%
\begin{eqnarray*}
f_X \;:\; \dowMx & \rightarrow & \dowTx    \\[2pt]
      \{1,2,3\} & \mapsto     & \{1,2,4\} \\[-1pt]
      \{2,3,4\} & \mapsto     & \{1,2,3\} \\[-1pt]
      \{3,4,5\} & \mapsto     & \{2,3,4\} \\[-1pt]
      \{1,4,5\} & \mapsto     & \{3,4\}   \\[-1pt]
      \{1,2,5\} & \mapsto     & \{1,4\}   \\[-1pt]
\end{eqnarray*}}
\end{minipage}
\begin{minipage}{2.5in}{%
\begin{eqnarray*}
f_Y \;:\; \dowMy        & \rightarrow & \dowTy             \\[2pt]
     \{\ta, \tb, \tc\} & \mapsto     & \{\ta, \tb, \tc\}  \\[-1pt]
     \{\tb, \tc, \td\} & \mapsto     & \{\tb, \tc, \td\}  \\[-1pt]
     \{\tc, \td, \te\} & \mapsto     & \{\ta, \tc, \td\}  \\[-1pt]
     \{\ta, \td, \te\} & \mapsto     & \{\ta, \td\}       \\[-1pt]
     \{\ta, \tb, \te\} & \mapsto     & \{\ta, \tb\}       \\[-1pt]
\end{eqnarray*}}
\end{minipage}
\end{center}

\vspace*{-0.3in}

Even though $f_X$ and $f_Y$ are surjective as set maps on the vertices
of the Dowker complexes, they are {\em not\,} surjective as simplicial
maps on the complexes themselves.  Each only covers three of the four
triangles comprising the tetrahedron in its codomain.  At first glance
it may therefore seem that the morphism $f: M \rightarrow T$ resulting
from $f_X$ and $f_Y$ does not achieve the desired privacy
preservation.  A closer look, however, reveals that $f$ is actually
surjective as a map of relations: it maps the elements of $M\!$ {\em
onto} the elements of $T$.  Therefore, it does represent a
transformation that achieves privacy preservation.

In order to understand this paradox, imagine again that $M$ represents
an authorship database.  Think of the maps $f_X$ and $f_Y$ as quotient
maps, in this case equating authors 1 and 5 and books $\ta$ and $\te$.
The equivalencing of authors might constitute a recognition of
pseudonyms.  The equivalencing of books might represent a
generalization from titles to genres.  Such changes of resolution,
carefully chosen, perhaps based on external structure, can preserve
relationships while reducing recognition and inference granularity.
\label{surjqueststop}

\subsection{Summary of Morphism Properties}

Definition \ref{morphism} defines a morphism of relations $f : R
\rightarrow Q$ in terms of underlying set functions $f_X : \XR
\rightarrow \XQ$ and $f_Y : \YR \rightarrow \YQ$.  These set functions
further induce simplicial maps $f_X : \dowx \rightarrow \dowqx$ and
$f_Y : \dowy \rightarrow \dowqy$.  The previous subsections spoke of
surjectivity in varying contexts.  Similarly, one could speak of maps
as being injective in varying contexts.  Finally, one also speaks of
morphisms as being epimorphisms and monomorphisms.  This subsection
summarizes how these properties relate for the various maps.  \ See
Appendix~\ref{morphismpropertiesApp} for proofs.

\vspace*{0.1in}

First, some definitional context and reminders:

\begin{itemize}

\item A morphism of relations $f : R \rightarrow Q$ is also a set map
  between the set of pairs comprising $R$ and the set of pairs
  comprising $Q$.  Specifically, $f(x,y) = (f_X(x), f_Y(y))$ for all
  $(x,y)\in R$.

  One may speak of $f$ as being {\em surjective} and/or {\em
  injective}, meaning as a set map.

\item We say that two morphisms of relations $g,h : R \rightarrow Q$
  are {\em equal}, written $g = h$, when they are equal as set maps of
  ordered pairs, meaning $g(x,y)=h(x,y)$ for all $(x,y)\in R$.

  (Note: If $R$ contains blank rows and/or columns, then $g=h$ is
  possible even though $g_X \neq h_X$ and/or $g_Y \neq h_Y$, as set
  maps.  This will not cause us problems; one could pass to
  equivalence classes in Definition~\ref{morphism}.)

\item The functions $f_X : \XR \rightarrow \XQ$ and $f_Y : \YR
\rightarrow \YQ$ are set maps.  One may speak of them as being
surjective and/or injective.

\item One may also ask whether the induced simplicial maps $f_X :
  \dowx \rightarrow \dowqx$ and $f_Y : \dowy \rightarrow \dowqy$ are
  surjective and/or injective as maps between simplicial complexes
  viewed as sets.

\item Suppose $f : R \rightarrow Q$ is a morphism of relations.
 Recall from category theory that $f$ is an {\em epimorphism} if, for
 any pair of morphisms $g, h : Q \rightarrow S$, \ $g \circ f = h
 \circ f$ implies $g = h$.

  Recall further that a morphism $f : R \rightarrow Q$ is a {\em
  monomorphism} if, for any pair of morphisms $g, h : S \rightarrow
  R$, $f \circ g = f \circ h$ implies $g = h$.

\end{itemize}

\newcounter{morphismproperties}
\setcounter{morphismproperties}{\value{theorem}}
\begin{lemma}[Morphism Properties]\label{morphismprop}
Assume the notation from above and that all relevant relations are nonvoid.
Let $f : R \rightarrow Q$ be a morphism of relations (as per
Definition~\ref{morphism}).  Then:

\begin{itemize}

\item[(i)] $\fX$ {\em and} $\fY$ are injective set maps $\implies$ $f$ is
            injective $\iff$ $f$ is a monomorphism.

  \vspace*{0.05in}

\item [(ii)] $f$ surjective $\implies$ $f$ epimorphism $\iff$
  $\fX$ {\em and} $\fY$ are surjective set maps.

  \vspace*{0.05in}

  (Additional conditions for that last $\iff$: The $\Longrightarrow$
  direction assumes that $Q$ has no blank rows or columns, while the
  $\Longleftarrow$ direction assumes that $R$ has no blank rows or
  columns.)

\end{itemize}

\vspace*{-0.05in}

\noindent The two uni-directional implications $\implies$ above are strict.

\vspace*{0.075in}

\begin{itemize}

\item[(iii)] If $\fX : \dowx \rightarrow \dowqx$ is surjective and $Q$
  has no blank rows, then $\fX : \XR \rightarrow \XQ$ is surjective.

  Similarly for $\fY$, now assuming that $Q$ has no blank columns.

The converses need not hold.  Indeed, $f$ itself can be surjective but
the maps of simplicial complexes need not be (as we saw with the maps
of page~\pageref{MTproj}).

\item[(iv)] If $\fX : \XR \rightarrow \XQ$ is injective, then $\fX :
  \dowx \rightarrow \dowqx$ is injective.   The converse holds if $R$
  has no blank rows.

  Similarly for $\fY$, now assuming that $R$ has no blank columns for
  the converse.

\end{itemize}

\end{lemma}

\subsection{G-Morphisms}
\markright{G-Morphisms}
\label{gmorphismsection}

Since a relation $R$ defines a poset $\PR$, rather than merely create
morphisms from set maps between individuals and attributes as in
Definition~\ref{morphism}, we may broaden the definition by
considering maps between posets:

\vspace*{0.15in}

\begin{definition}[G-Morphism]\label{Gmorphism}
Let $R$ and $Q$ be nonvoid relations.\\
\hspace*{0.5in}A \,\mydefem{G-morphism}\ \ $f : R \rightarrow Q$ \ is
any poset map \ $f : \PR \rightarrow \PQ$.
\end{definition}

\paragraph{Comments: } The ``G'' stands for ``Galois''.  We might
have insisted that a G-morphism $R \rightarrow Q$ be a lattice
morphism $\PRplus \rightarrow \PQplus$ rather than merely a poset map
$\PR \rightarrow \PQ$, but that might be too restrictive.  Instead, as
subsequent lemmas will describe, we view a G-morphism as providing
homotopy flexibility.  In particular, a morphism between relations as
per Definition~\ref{morphism} induces two homotopic G-morphisms.  The
lattice structure of the codomain is relevant, in that it allows one
to fill in elements not directly in the image of any one \hbox{G-morphism},
as will become apparent in Theorem \ref{latsurj}.

\begin{figure}[h]
\begin{center}
{\LARGE
$$
\begin{CD}
@. \hspace*{-0.275in}\Fdowx  @>{\hspace*{0.3in}f_X\hspace*{0.3in}}>> \hspace*{0.05in}\Fdowqx\\[3pt]
@V\phi_RVV \hspace*{-0.1in}@AA\psi_RA           \hspace*{-0.2in}@V\phi_QVV \hspace*{-0.32in}@AA\psi_QA\\
@. \hspace*{-0.275in}\Fdowy  @>{\hspace*{0.3in}f_Y\hspace*{0.3in}}>> \hspace*{0.05in}\Fdowqy\\[3pt]
\end{CD}
$$
}
\end{center}
\vspace*{-0.1in}
\caption[]{Diagram showing the poset maps $f_X$ and $f_Y$ induced by a
morphism $f : R \rightarrow Q$, along with the homotopy equivalences
between each relation's face posets.  (The diagram need not be
commutative, but is almost so; see Lemma~\ref{containment}.)}
\label{morphismdiagram}
\end{figure}

\vspace*{0.1in}

Recall that a morphism $f : R \rightarrow Q$ as per
Definition~\ref{morphism} is built from two set maps $f_X$ and $f_Y$
and that these set maps induce simplicial maps between the Dowker
complexes, as per Lemma~\ref{simpmaps}.  We may therefore further
regard $f_X$ and $f_Y$ as order-preserving poset maps between the face
posets of the Dowker complexes: $f_X : \Fdowx \rightarrow \Fdowqx$ and
$f_Y : \Fdowy \rightarrow \Fdowqy$.  Consequently, we have a diagram
of maps as in Figure~\ref{morphismdiagram}.  The diagram need not be
commutative, but the following containments hold:

\newcounter{containmentlemma}
\setcounter{containmentlemma}{\value{theorem}}
\begin{lemma}[Witness Containment]\label{containment}
Let $f : R \rightarrow Q$ be a morphism of nonvoid relations.  Then:

\vspace*{0.1in}

\qquad
\begin{minipage}{3.7in}
\begin{itemize}

\item[(a)] $(\fY \circ \phi_R)(\sigma) \;\subseteq\; (\phi_Q \circ \fX)(\sigma)$,
           \ for every $\sigma \in \dowx$,

\item[(b)] $(\fX \circ \psi_R)(\gamma) \;\subseteq\; (\psi_Q \circ \fY)(\gamma)$,
           \ for every $\gamma \in \dowy$.

\end{itemize}
\end{minipage}
\end{lemma}

\vspace*{0.15in}

(See Appendix~\ref{gmorphisms} for a proof of the previous lemma and
its upcoming corollaries.)

\clearpage

\noindent As a corollary, we see that the diagram of
Figure~\ref{morphismdiagram} describes two pairs of homotopic maps:

\newcounter{homotopicfacemaps}
\setcounter{homotopicfacemaps}{\value{theorem}}
\begin{corollary}[Homotopic Face Maps]\label{facehomotopy}
Let $f : R \rightarrow Q$ be a morphism of nonvoid relations.  Then:

\qquad
\begin{minipage}{4.75in}
\begin{itemize}

\item[(a)] $\fX$ and $\,\psi_Q \circ \fY \circ \phi_R$ are homotopic
  poset maps $\,\Fdowx \rightarrow \Fdowqx$,

\item[(b)] $\fY$ and $\,\phi_Q \circ \fX \circ \psi_R$ are homotopic
  poset maps $\,\Fdowy \rightarrow \Fdowqy$.

\end{itemize}
\end{minipage}
\end{corollary}

\vspace*{0.1in}

The images of the compositions that appear in
Corollary~\ref{facehomotopy} may be regarded as lying in $\PQ$.  We may
further restrict the domain of these maps to be $\PR$, giving us the
following G-morphisms:

\begin{definition}[Induced G-Morphisms]\label{inducedgmorphisms}
A morphism of nonvoid relations $f : R \rightarrow Q$ induces two
G-morphisms $R \rightarrow Q$, defined by the following poset maps
$\PR \rightarrow \PQ$:

\vspace*{-0.05in}

$$
\fXg \;=\; (\psi_Q \circ f_Y \circ \phi_R) |_{\PR}
\qquad\qquad
\fYg \;=\; (\phi_Q \circ f_X \circ \psi_R) |_{\PR}.
$$
\end{definition}
(The ``$g$'' superscript stands for ``Galois'' while the vertical bar
``$|$'' means ``restricted to''.  See also Appendix~\ref{gmorphisms}.)

\vspace*{0.1in}

\newcounter{homotopicposetmaps}
\setcounter{homotopicposetmaps}{\value{theorem}}
\begin{corollary}[Homotopic G-Morphisms]\label{posethomotopy}
Let $f : R \rightarrow Q$ be a morphism of nonvoid relations.  The
induced G-morphisms given by the poset maps $\fXg, \fYg : \PR
\rightarrow \PQ$ are homotopic.
\end{corollary}

\vspace*{0.1in}

The proof of Corollary~\ref{posethomotopy} on
page~\pageref{posethomotopyAppPage} says that we may view the underlying
maps $f_X$ and $f_Y$ of a morphism $f$ as mapping any inference-closed
set (viewed either as a set of individuals or as a set of attributes)
from the domain of $f$ to an interval (in the poset sense) of
inference-closed sets in the codomain of $f$.

\vspace*{0.15in}

\begin{figure}[h]
\begin{center}
\qquad
\begin{minipage}{1in}{$\begin{array}{c|c}
\hbox{\large $R$} & \ta  \\[2pt]\hline
              1   & \one \\
\end{array}$}
\end{minipage}
\hspace*{0.15in}
\begin{minipage}{1in}{$\begin{array}{c|cc}
\hbox{\large $Q$} & \ta  & \tb  \\[2pt]\hline
              1   & \one & \one \\
              2   & \one &      \\
\end{array}$}
\end{minipage}
\hspace*{0.5in}
\begin{minipage}{2.5in}
\ifig{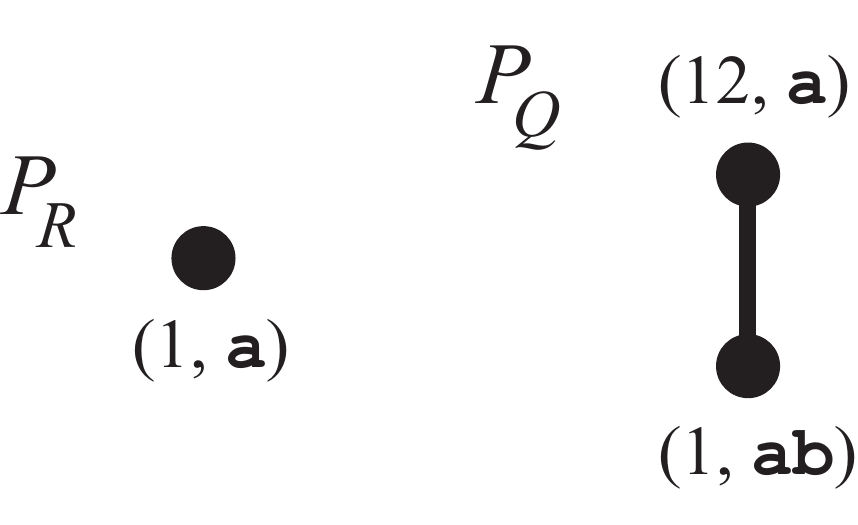}{scale=0.475}
\end{minipage}
\end{center}
\vspace*{-0.2in}
\caption[]{Relation $R$ is a subrelation of $Q$.  How should one embed
  $\PR$ into $\PQ$?  There are two possible embeddings, related by a
  homotopy.}
\label{Ex58}
\end{figure}

For a simple example, see Figure~\ref{Ex58}.  One may regard relation
$R$ as a subrelation of $Q$, then define $f : R \rightarrow Q$ to be
inclusion.  For instance, maybe $R$ and $Q$ represent individuals \#1
and \#2 at two parties $\ta$ and $\tb$, with $R$ representing known
parties and party-attendees at some time and $Q$ representing an
update of that information at a later time.  Observe that:

\vspace*{-0.1in}

\begin{eqnarray*}
\fXg((1,\ta)) \!&=&\! (\psi_Q \circ f_Y \circ \phi_R)(\{1\}) \;=\; (\psi_Q \circ f_Y)(\{\ta\})
   \;=\; \psi_Q(\{\ta\}) \;=\; \{1, 2\} \;\;\hbox{``$=$''}\;\;(12, \ta), \\[4pt]
\fYg((1,\ta)) \!&=&\! (\phi_Q \circ f_X \circ \psi_R)(\{\ta\}) \;=\; (\phi_Q \circ f_X)(\{1\})
   \;=\; \phi_Q(\{1\}) \;=\; \{\ta, \tb\} \;\;\hbox{``$=$''}\;\;(1, \ta\tb). \\
\end{eqnarray*}

\vspace*{-0.1in}

\noindent The last equality in each row indicates how to view the image
element on the left of the ``$=$'' as an element of the poset $\PQ$.

\vspace*{0.1in}

Both $\fXg$ and $\fYg$ tell us how to update inference-closed sets
from $\PR$ into inference-closed sets within $\PQ$:

\begin{itemize}

\item The map $\fXg$ updates associations while holding
observed attributes fixed.  In this example, based on initial
information (relation $R$), we know that person \#1 attended party
$\ta$.  Once we update that information (relation $Q$) we can conclude
that person \#2 also attended a party at which person \#1 was present.

\item Similarly, the map $\fYg$ updates attributes while
holding observed individuals fixed.  In this example, updated
information allows us to conclude that person \#1 attended not only
party $\ta$ but also party $\tb$.

\end{itemize}

In general, for any fixed element of $\PR$, the two maps may give
different results, but those results are comparable in $\PQ$.  Here $f$
was inclusion, so we could speak of holding attributes or individuals
``fixed''.  More generally, ``fixed'' is replaced by whatever $f$ does.

\subsection{Surjectivity Revisited}
\markright{Surjectivity of Morphisms}
\label{surjectivity}

A paradox: We saw on page
\pageref{surjqueststop} a surjective morphism $f$, from the M\"obius
strip relation of Figure~\ref{moebius} to the tetrahedral relation of
Figure~\ref{tetrahedra2}, whose induced simplicial maps $f_X : \dowMx
\rightarrow \dowTx$ and $f_Y : \dowMy \rightarrow \dowTy$ were not
surjective.   This raises some questions:

\vspace*{0.1in}

\hspace{0.25in}\begin{minipage}{5.3in}
\begin{enumerate}

\addtolength{\itemsep}{-2pt}

\item Are the induced poset maps $\fXg, \fYg : \PM \rightarrow \PT$
  surjective?

\item If not, how can one speak of a surjective morphism?

\end{enumerate}
\end{minipage}

\vspace*{0.1in}

(Note that $\PMplus$ is isomorphic to $\PGplus$ as shown in
Figure~\ref{authorlattice} on page~\pageref{authorlattice}.  A rendering
would be identical, except for lowercase letters in place of uppercase
ones.  The lattice $\PTplus$ appears in Figure~\ref{tetralattice2} on
page~\pageref{tetralattice2}.)

\vspace*{0.1in}

The answer to Question 1 is that the two poset maps are {\em not\,}
surjective.  Observe in Table~\ref{MTposetmaps} on
page~\pageref{MTposetmaps}, for instance, that the image of $\fXg$
does not include $(4, \ta\tb\td)$. Similarly, the image of $\fYg$ does
not include $(134, \ta)$.

These missing elements {\em are} in the image of both maps {\em
together}, viewed as a pair of homotopic maps, as per
Corollary~\ref{posethomotopy}.  Unfortunately, that explanation is not a
full answer to Question 2.  For instance, neither map's image includes
the element $(13, \ta\tc)$ of $\PT$, nor does that element appear in any
interval $[\fYg(p), \fXg(p)]$ as $p$ varies throughout $\PM$.

\vspace*{0.1in}

To answer question 2, the lattice structure of $\PT$ is useful.  In the
example, the image of $\fXg$ includes all elements of $\PT$ that
correspond to maximal simplices of $\dowTx$.  Similarly, the image of
$\fYg$ includes all elements of $\PT$ that correspond to maximal
simplices of $\dowTy$.  Intuitively, we therefore expect that the
lattice operations (which correspond to intersection in either $\dowTx$
or $\dowTy$) will generate all the elements of $\PT$.  In that sense,
the surjectivity of $f$ appears as surjectivity of each of $\fXg$ and
$\fYg$, once one {\em completes} their images under lattice operations.

\begin{table}[t]
\begin{center}
\fbox{$\begin{array}{ccccc}
           p        & \quad &       \fXg(p)  & \quad &      \fYg(p)         \\[3pt]\hline
  \pelem{12}{\ta\tb}   & & \pelem{14}{\ta\tb}   & & \pelem{14}{\ta\tb}    \\[2pt]
  \pelem{2}{\ta\tb\tc} & & \pelem{1}{\ta\tb\tc} & & \pelem{1}{\ta\tb\tc}  \\[2pt]
  \pelem{123}{\tb}     & & \pelem{124}{\tb}     & & \pelem{124}{\tb}      \\[2pt]
  \pelem{23}{\tb\tc}   & & \pelem{12}{\tb\tc}   & & \pelem{12}{\tb\tc}    \\[2pt]
  \pelem{3}{\tb\tc\td} & & \pelem{2}{\tb\tc\td} & & \pelem{2}{\tb\tc\td}  \\[2pt]
  \pelem{234}{\tc}     & & \pelem{123}{\tc}     & & \pelem{123}{\tc}      \\[2pt]
  \pelem{34}{\tc\td}   & & \pelem{23}{\tc\td}   & & \pelem{23}{\tc\td}    \\[2pt]
  \pelem{4}{\tc\td\te} & & \pelem{3}{\ta\tc\td} & & \pelem{3}{\ta\tc\td}  \\[2pt]
  \pelem{345}{\td}     & & \pelem{234}{\td}     & & \pelem{234}{\td}      \\[2pt]
  \pelem{45}{\td\te}   & & \pelem{34}{\ta\td}   & & \pelem{34}{\ta\td}    \\[2pt]
  \pelem{5}{\ta\td\te} & & \pelem{34}{\ta\td}   & & \pelem{4}{\ta\tb\td}  \\[2pt]
  \pelem{145}{\te}     & & \pelem{134}{\ta}     & & \pelem{34}{\ta\td}    \\[2pt]
  \pelem{15}{\ta\te}   & & \pelem{134}{\ta}     & & \pelem{4}{\ta\tb\td}  \\[2pt]
  \pelem{1}{\ta\tb\te} & & \pelem{14}{\ta\tb}   & & \pelem{4}{\ta\tb\td}  \\[2pt]
  \pelem{125}{\ta}     & & \pelem{134}{\ta}     & & \pelem{14}{\ta\tb}    \\[2pt]
\end{array}$}
\end{center}
\vspace*{-0.2in}
\caption[]{Each $p$ is of the form $(\sigma, \gamma) \in \PM$.  The
   elements $\fXg(p)$ and $\fYg(p)$ lie in $\PT$.  See also Figures
   \ref{authorlattice} and \ref{tetralattice2}, on
   pages~\pageref{authorlattice} and \pageref{tetralattice2},
   respectively. \ (As in those figures, the table elides commas and
   braces from set notation.  \ For Figure~\ref{authorlattice},
   recall that $M$ and $G$ are isomorphic relations.)}
\label{MTposetmaps}
\end{table}

\vspace*{0.25in}

\noindent The following theorem summarizes the intuition of the previous pages:

\vspace*{0.05in}

\newcounter{LatticeGenerators}
\setcounter{LatticeGenerators}{\value{theorem}}
\begin{theorem}[Lattice Surjectivity]\label{latsurj}
Let $R$ and $Q$ be nonvoid relations with no blank rows or columns.
Suppose $f : R \rightarrow Q$ is a surjective morphism (in the sense
of Definition~\ref{morphism}).
 \ For any $q \in \PQ$:

\vspace*{-0.05in}

$$ q \;=\; \bigwedge_j \bigvee_i q_{ji}, \quad \hbox{with each
  $q_{ji}$ in the image of $\fXg : \PR \rightarrow \PQ$,}$$

\vspace*{-0.05in}

$$ q \;=\; \bigvee_k \bigwedge_\ell q^{\prime}_{k\ell}, \quad \hbox{with each
  $q^\prime_{k\ell}$ in the image of $\fYg : \PR \rightarrow \PQ$.}$$

\vspace*{0.05in}

(Here, $\bigvee$ and $\bigwedge$ are the lattice operations of $\PQplus$.)

\end{theorem}

\vspace*{0.2in}

See Appendix~\ref{latticegen} for a proof.

\clearpage
\section{Future Thoughts}
\markright{Future Thoughts}
\label{future}

Throughout this report, one senses the inevitability of privacy loss,
that the topology of relations necessarily converts attribute
information into revealing gradient flow.  Gradient flow appears both
in the collapse of free faces \cite{tpr:forman-guide} and in the
lattice structure of information acquisition: Collapse of free faces
infers unobserved attributes from observed attributes.  The meet
operation of a relation's Galois lattice propels observed attributes
into downward motion, toward minima of identification. \ Still,
alternatives exist.

\subsection{Relaxing Assumptions}

Gradient flow is a natural consequence of the assumptions stated in
Section~\ref{assumptions}.  Let us discuss briefly how to relax those
assumptions, while leaving detailed explorations for the future.

\begin{enumerate}
\item We could drop the assumption of relational completeness.  We
might then observe a set of attributes $\gamma$ inconsistent with the
given relation $R$, meaning $\gamma\not\in\dowy$.  One possibility is
that some individual in $R$ has attributes $\gamma$, but the relation
does not capture this fact.  There is another possibility, that
$\gamma$ represents the attributes of some individual external to $R$.
For instance, recall that in Lemma~\ref{interplocal} a set of
attributes inconsistent with a link's relation identifies the linking
set of individuals.  Deciding between these two scenarios (new
attributes for given individuals versus wholly new individuals)
requires additional information, not so unlike the decisions faced in
mapping unknown environments.

\item We could drop the assumption of observational monotonicity.  We
   do wish to retain the ability to observe attributes asynchronously.
   However, we might be able to place algebraic structure on some
   attributes, so that certain newly observed attributes can cancel
   previously observed ones.  Spending a dollar versus earning a
   dollar for instance.  Such an algebraic structure would then permit
   upward motion in a relation's Galois lattice.  (This is not always
   possible, e.g., if a relation encodes history by time-indexing.)

\item We could drop the assumption of observational accuracy.
  Attributes frequently are measured by noisy sensors, whether based
  on physical instruments or errorful databases.  Existing privacy
  work has frequently assumed noise intrinsically (e.g.,
  identification in \cite{tpr:netflix} was successful despite database
  errors).  This report ignored noise in order to focus on the
  combinatorial structure of privacy.  Presently, we will sketch a
  possible noise model, in which a sensor reports attributes
  stochastically.  A sensor is a physical device and an interpretation
  algorithm, producing observed attributes $\gamma(t)$ as functions of
  time.  Thus, as $t$ varies, $\gamma(t)$ may move either up or down
  in a relation's Galois lattice, not just down.

\vso

  A caution: Moving from a purely combinatorial system to a stochastic
  system need not turn gradient flow into harmonic flow.  Reasonable
  but noisy sensors create a (stochastic) gradient flow, by the
  Central Limit Theorem.  Rather, a noisy sensor model in the
  observation of attributes facilitates the connection to other
  privacy work.  For instance, one view of Differential Privacy
  \cite{tpr:dworkcacm} is that it injects noise into a sensor, with
  the noise magnitude chosen in part as a function of the time
  interval allotted for observations, thus preventing gradient flow
  from reaching a minimum.  Moreover, adversarial control over
  $\gamma(t)$, perhaps by sensor disinformation, may be able to create
  more general flows.

\end{enumerate}

\subsection{Sensing Attributes Stochastically}
\markright{Combinatorial Model for Stochastic Sensing}

We briefly explore a model for stochastic sensors within combinatorial
relations, via a simple example.  Consider relation $R$ of
Figure~\ref{S1Complexes}.  This relation produces Dowker complexes
with $\Sone$ homotopy type, as indicated in the figure.  The
relation's Galois lattice $\PRplus$ appears in Figure~\ref{S1Lattice}.

\begin{figure}[h]
\begin{center}
\vspace*{-0.08in}
\begin{minipage}{1.5in}{$\begin{array}{c|ccc}
R & \ta & \tb & \tc \\[2pt]\hline
1 & \one & \one & \\
2 &      & \one & \one \\
3 & \one &      & \one \\
\end{array}$}
\vspace*{0.2in}
\end{minipage}
\begin{minipage}{4in}
\ifig{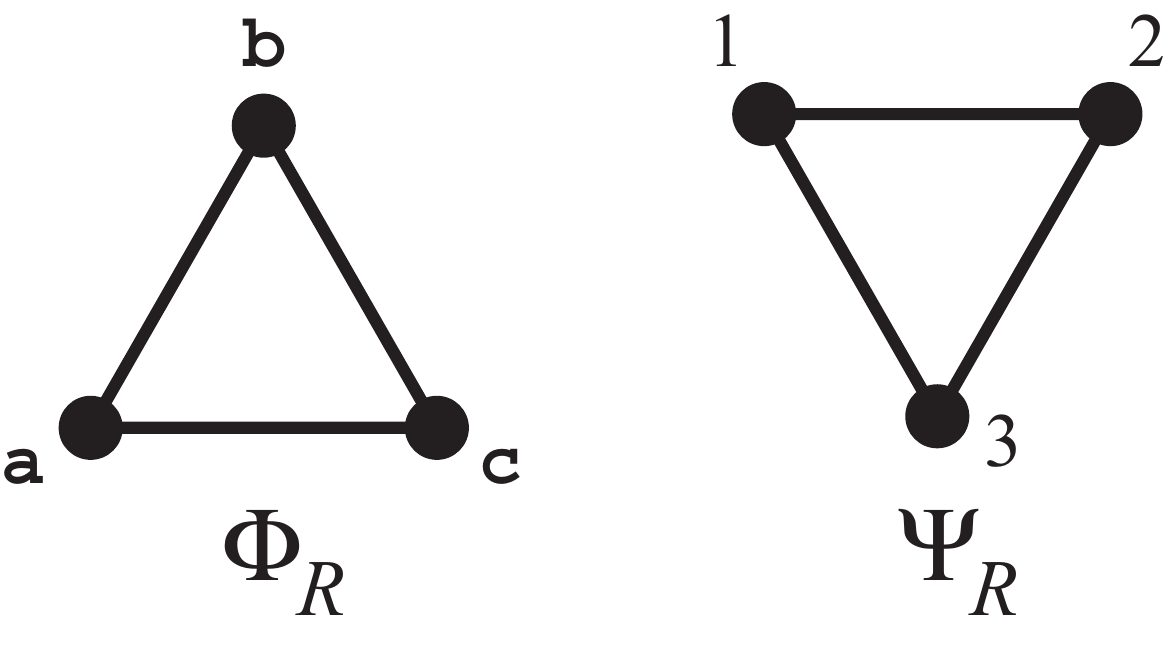}{scale=0.6}
\end{minipage}
\end{center}
\vspace*{-0.3in}
\caption[]{A relation whose Dowker complexes are dual triangulations
  of the circle.  See Figure~\ref{S1Lattice} for the associated Galois
  lattice.}
\label{S1Complexes}
\end{figure}

\begin{figure}[h]
\begin{center}
\vspace*{-0.08in}
\ifig{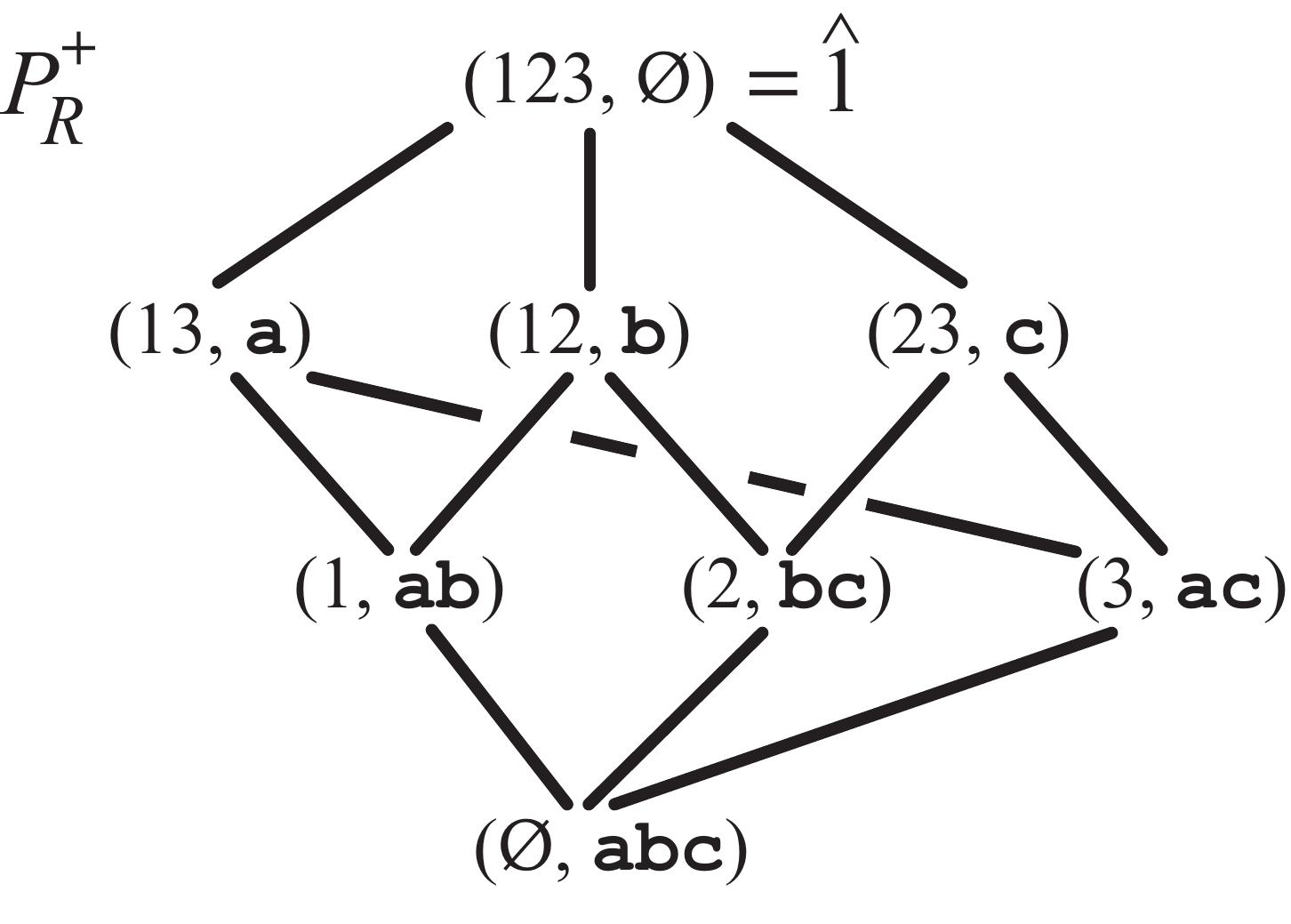}{scale=0.55}
\end{center}
\vspace*{-0.3in}
\caption[]{The lattice $\PRplus$ for relation $R$ of
  Figure~\ref{S1Complexes}.  (We have elided commas and braces in
  sets.)  --- For later reference: The poset $\PR \union \{\topone\}$
  consists of all elements in $\PRplus$ except for the bottom element,
  $(\emptyset, \ta\tb\tc)$.}
\label{S1Lattice}
\end{figure}

Relation $R$'s space of attributes is $Y = \{\ta, \tb, \tc\}$.
Suppose that a sensor reports these attributes by observing the world
and performing some computation.  The report may be inaccurate.  Such
inaccuracy could be either adversarial or stochastic.  We focus here
on the stochastic case, and on one particular model: The sensor
computes three probabilities, $p_\ta$, $p_\tb$, and $p_\tc$, with
$p_y$ being the probability that the sensor's observation came from
actual attribute $y\in Y$.  These three probabilities constitute a
point $\vp = (p_\ta, p_\tb, p_\tc)$ in the full simplex $\{\ta, \tb,
\tc\}$, with the point's barycentric coordinates being the three
probabilities.  Subsequently, the sensor reports an attribute by
interpreting $\vp$, perhaps by maximum probability.  In order to
reduce false positives, the sensor sets a confidence threshold below
which it interprets $\vp$ as too ambiguous.  This thresholding carves
the simplex $\{\ta, \tb, \tc\}$ into four regions, one for each
attribute, plus a {\em zone of indecision}.  A sketch of this process
appears in Figure~\ref{SensorProbabilities}.

\begin{figure}[t]
\begin{center}
\ifig{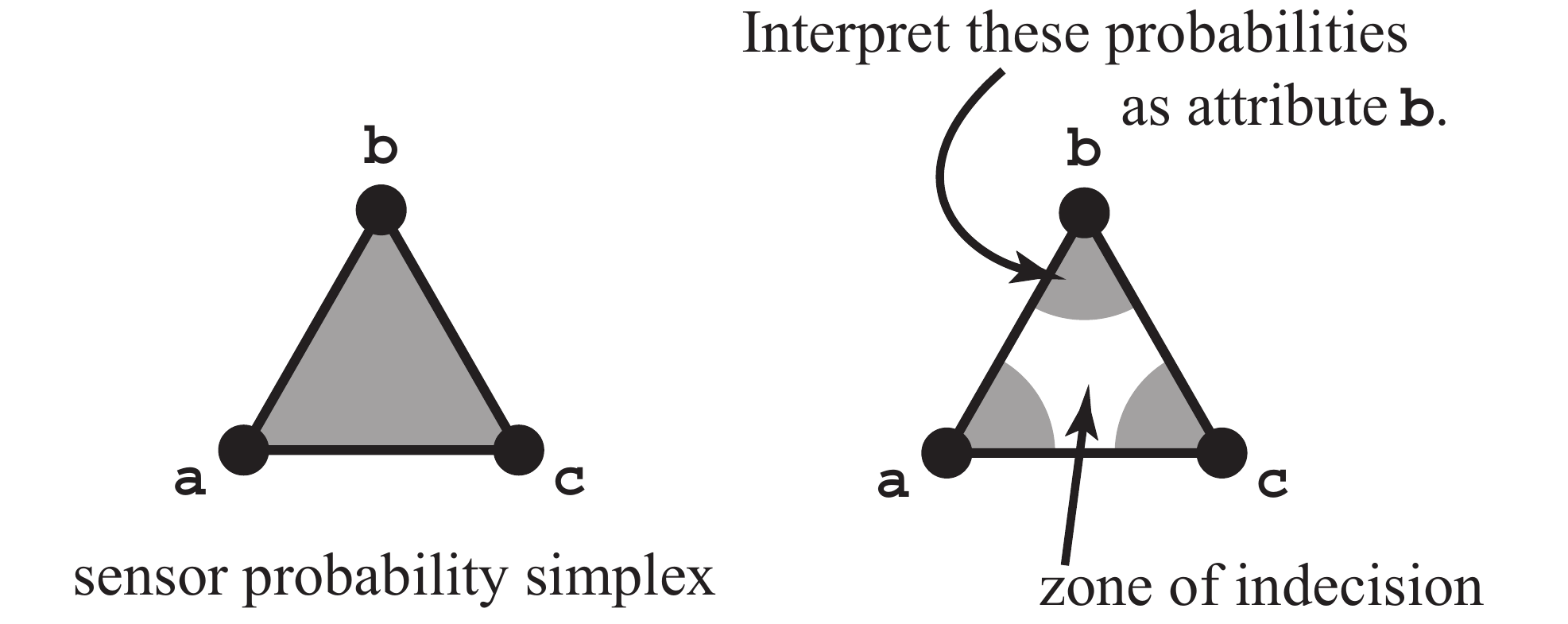}{scale=0.55}
\end{center}
\vspace*{-0.25in}
\caption[]{A sensor computes probabilities over a
  set of attributes $\{\ta, \tb, \tc\}$.
  Left panel: A simplex whose vertices are those attributes models the
  possible distributions, with barycentric coordinates being
  probabilities.
  Right panel: An interpretation of probabilities as attributes, along
  with a zone of indecision, partitions the simplex into four
  regions.}
\label{SensorProbabilities}
\end{figure}

Individuals in relation $R$ of Figure~\ref{S1Complexes} have two
attributes.  One could model two observations as a probability
distribution over $\YxY$.  However, since the first observation
may lead to an indecision, it is convenient to model the second
observation as conditional on the first having produced a particular
attribute.  The second observation may: (i) repeat the just seen
attribute, (ii) report an as yet unseen attribute, or (iii) fail to
report anything.  The net effect for (i) and (iii) is observation of
one attribute.  Consequently, the space of two observations may be
interpreted as the first barycentric subdivision of the boundary of
the attribute simplex, i.e., as $\sd(\bndry{(Y)})$, with the empty
simplex modeling a zone of indecision.  This process appears in
Figure~\ref{TwoSensorSpace}.  One obtains a map from a stochastic
sensor's observations to the poset $\PR \union \{\topone\}$, with
$\topone$ representing an inability to decipher any attribute.
\ These calculations suggest that existing stochastic results
fit naturally into this report's combinatorial framework.

\begin{figure}[h]
\begin{center}
\ifig{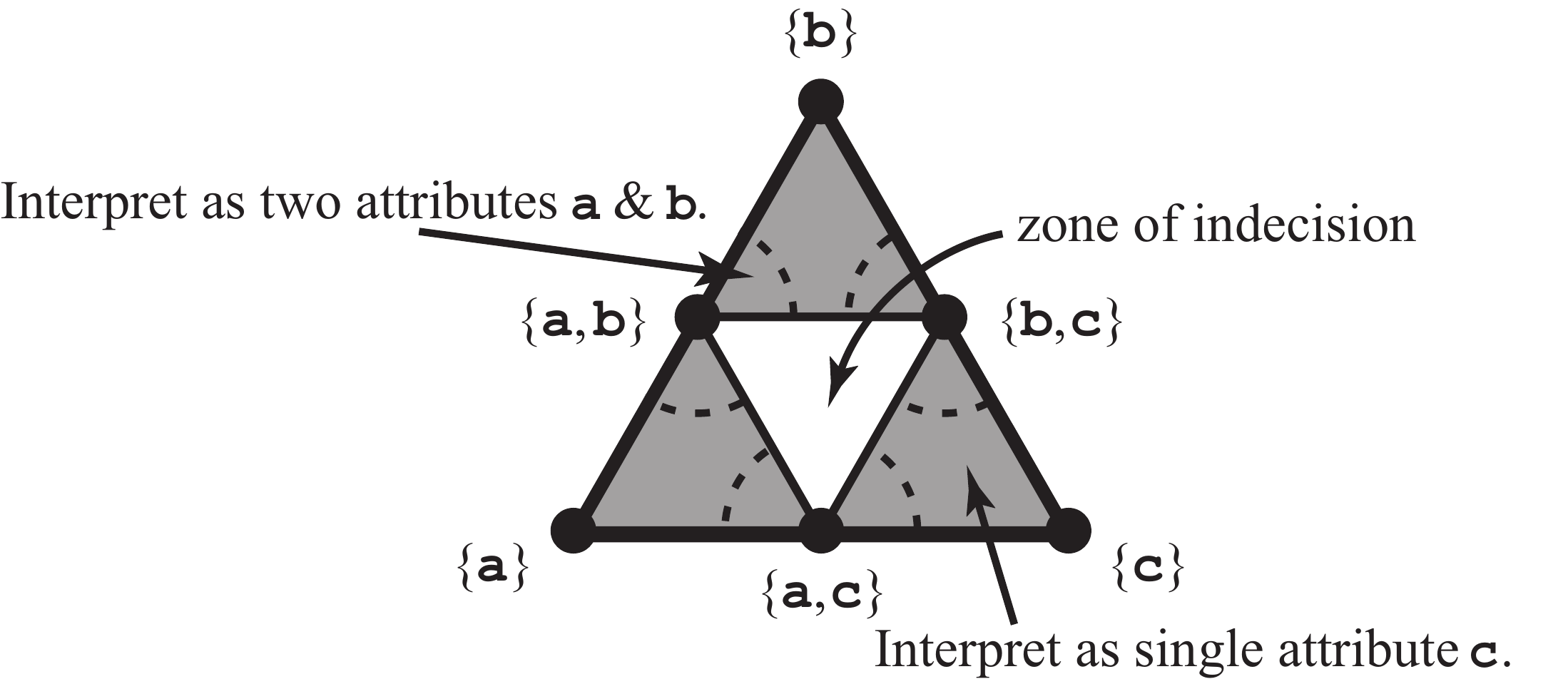}{scale=0.55}
\end{center}
\vspace*{-0.2in}
\caption[]{Two stochastic sensor observations may be modeled as a
  first observation followed by a conditional second observation.  The
  resulting decision space has a representation isomorphic to the
  first barycentric subdivision of the boundary of the original
  probability simplex, with the empty simplex representing a zone of
  indecision.  (The two observations may be understood as two points,
  $\vp$ and $\vq$, with $\vp$ in the encompassing triangle and $\vq$
  in a shaded subtriangle.)}
\label{TwoSensorSpace}
\end{figure}

\clearpage
\section*{Acknowledgments}
\addcontentsline{toc}{section}{Acknowledgments}

We are grateful to Professor Robert Ghrist and his students and
post-docs, particularly Pawe{\l} D{\l}otko, Vidit Nanda, and Greg
Henselman, for their topological advice and suggestions, as well as
their discussions related to this project.

We are also grateful for the publicly available {\tt Perseus} software
previously written at the University of Pennsylvania, which we used
for our homology computations:
{\tt http://www.sas.upenn.edu/{$\sim$}vnanda/perseus/}.

\vspace*{0.2in}

\markright{References}

\clearpage
\appendix

\section{Preliminaries}
\label{prelim}

{\bf Assumption:} All simplicial complexes, relations, posets, and
lattices in this report are finite.

\subsection{Simplicial Complexes}
\markright{Simplicial Complexes}
\label{complexesApp}

We largely follow the notation and definitions in \cite{tpr:munkres}
and \cite{tpr:bjorner}.

\begin{itemize}

\item An {\em (abstract) simplicial complex} $\Sigma$ {\em with
  underlying vertex set} $X$ is a collection of finite subsets of $X$,
  such that if $\sigma$ is in $\Sigma$, then so is every subset of
  $\sigma$.  An element of $\Sigma$ is called a {\em simplex}.  We
  allow the empty set $\emptyset$ to be a simplex in $\Sigma$, for
  combinatorial reasons.  An element of a simplex is called a {\em
  vertex}.  It is also convenient to refer indistinguishably to any
  singleton simplex as a {\em vertex}.  Not all elements of $X$ need
  to be vertices of $\Sigma$. \label{simpcpxdef}

\item The {\em dimension} of a simplex $\sigma$ is one less than its
  cardinality.  The empty simplex $\emptyset$ has dimension $-1$.  \ If
  a simplex has dimension $k$ we sometimes call it a {\em $k$-simplex}.

\item If $\Sigma$ is a simplicial complex with underlying vertex set
  $X$, we let $\verts{\Sigma}$ denote the set of elements in $X$ that
  actually appear as vertices in $\Sigma$.  Viewed as a set of
  0-simplices, this set is called the {\em zero-skeleton} of $\Sigma$.
  \ [The standard notation for the zero-skeleton is $\Sigma^{(0)}$ but
  that conflicts with some iterative notation in the proof of
  Theorem~\ref{manychains}.]

\item The {\em void complex\,} $\emptyset$ has no simplices in it.  We
  view it as a {\em degenerate} space.  The {\em empty complex}
  $\{\emptyset\}$ consists solely of the empty simplex.  The empty
  complex represents the empty topological space.  It is also the
  sphere of dimension $-1$, written $\Smo$.
  \label{emptycomplex} (There could be be different instances of the
  void or empty complex, depending on the underlying vertex set $X$,
  though frequently one takes $X$ to be empty in these situations.)

\item A simplex $\sigma$ of a simplicial complex $\hspt\Sigma\hspt$ is
  a {\em\hspt free face\hspt} of $\hspt\Sigma\hspt$ if it is a proper
  subset of exactly one maximal simplex $\tau$ of $\hspt\Sigma$.
  \ (The empty simplex $\emptyset$ can sometimes be a free face.)

\vspace*{-0.425in}

\paragraph{}$\phantom{0}$

\item Suppose $\Sigma$ is a simplicial complex.  Then $C_k(\Sigma;
  \Z)$ denotes the group of simplicial \hbox{{\em $k$-chains}} over
  $\Sigma$, with integer coefficients.  A $k$-chain $c\in C_k(\Sigma;
  \Z)$ is a function that assigns to each oriented $k$-simplex $\tau$
  of $\hspt\Sigma\hspt$ an integer.  When $k>0$, one further requires
  that $c(-\tau)=-c(\tau)$.  \ Here $-\tau$ refers to the same
  combinatorial set as $\tau$ but with opposite orientation.  (When
  $k=0$ or $k=-1$, each $k$-simplex has only one possible
  orientation.)

  Caution: We later also use the word ``chain'' in the poset sense;
  there should be no ambiguity given context.
  \label{kchainAppdef}

\item Suppose $\Sigma$ is a simplicial complex and $c \in C_k(\Sigma;
  \Z)$.  Assume all simplices have been assigned an orientation in
  $\Sigma$.  One can write $c=\sum_in_i\tau_i$ uniquely, for some
  subcollection $\{\tau_i\}$ of the $k$-simplices in $\Sigma$,
  such that $n_i \neq 0$ for each $i$.  (Any $k$-simplex $\tau$ of
  $\Sigma$ may appear at most once in the sum, with its assigned
  orientation.)  This means $c(\tau_i)=n_i$ for each $\tau_i$ that
  appears in the sum and, if $k>0$, $c(-\tau_i)=-n_i$.  For all other
  oriented $k$-simplices $\tau$ of $\Sigma$, $c(\tau)=0$.

  We define the {\em support} of $c$ as $\supp{c} = \union_i\tau_i$.
  [This is not standard notation.]  The support is the set of all
  vertices that appear in any of the simplices $\tau$ for which
  $c(\tau)$ is nonzero.

\item We let $\bndry{}$ and $\redbndry$ stand for ``boundary''.  There
  are two contexts:

  \begin{enumerate}
  \item When $V$ is a nonempty finite set, then \label{bndrycpxAppdefStart}
    $\bndry{(V)}$ means the simplicial complex whose underlying vertex
    set is $\mskip1mu{}V\mskip-1mu$ and whose simplices are all the
    proper subsets of $V$.  We refer to this complex as the {\em
    boundary complex of the full simplex on vertex set $\,V$}.  It has
    the homotopy type of a sphere, specifically $\Snt$, with
    $n=\abs{V}$, for all $n \geq 1$.
    \label{bndrycpxAppdefEnd}

\vspace*{-0.075in}

\addtolength{\baselineskip}{1pt}

  \item We also designate the {\em simplicial boundary operator} by
    $\partial{}$ and the {\em reduced boundary operator}\hspt\ by
    \hspace*{-6pt}
    $\phantom{\Big|}\redbndry$.  \ These operators are families of
    maps, describing for each dimension $k$ a group homomorphism
    $C_k(\Sigma; \Z) \rightarrow C_{k-1}(\Sigma; \Z)$, defined on
    basis elements by:

    \vspace*{0.05in}

    When $\sigma = \{x_0, \ldots, x_k\}$ is an oriented $k$-simplex,
    with $k \geq 1$, $\redbndry_k(\sigma) = \partial_k(\sigma) =
    \sum_{i=0}^k(-1)^i\tau_i$, where $\tau_i$ is the oriented
    $(k-1)$-simplex formed from $\sigma$ by removing vertex $x_i$ and
    using the induced orientation of $\sigma$ on $\tau_i$.
    \ \ (See \cite{tpr:munkres, tpr:hatcher} for details.)

    \vspace*{0.05in}

    For $k=0$, \ $\partial_0 : C_0(\Sigma; \Z) \rightarrow 0$, while
    $\redbndry_0 : C_0(\Sigma; \Z) \rightarrow C_{-1}(\Sigma; \Z)$, with
    $\redbndry_0(\{v\}) = \emptygen$, for each vertex $\{v\}\in\Sigma$.
    (Here $\emptygen$ represents the generator of $C_{-1}(\Sigma; \Z)$
    when $\Sigma$ is nonvoid.  If $\Sigma$ is void, then $\redbndry_0 = 0$.)
    \ There is also a map $\redbndry_{-1} :
    C_{-1}(\Sigma; \Z) \rightarrow 0$.\label{redbndryAppdef}

\addtolength{\baselineskip}{-1pt}

    \vspace*{0.1in}

    We are mainly interested in the reduced boundary operator $\redbndry$.

    \vso

    We may write $\redbndry$ in place of $\redbndry_k$ when the
    dimensional context $k$ is clear.

    \vspace*{0.1in}

    Elements of the subgroup ${\rm ker}(\redbndry_k)$ are called {\em
      reduced $k$-cycles}.

    \vst

    Elements of the subgroup ${\rm img}(\redbndry_{k+1})$ are called
    {\em reduced $k$-boundaries}.

  \end{enumerate}

\vspace*{0.045in}

\item Given a simplicial complex $\Sigma$, $\widetilde{H}_k(\Sigma;
  \Z)$ is the {\em reduced homology group in dimension} $k$ based on
  simplicial chains over $\Sigma$ with integer coefficients.  It is a
  quotient group, measuring the reduced $k$-cycles that are not
  reduced $k$-boundaries.

  Formally, $\widetilde{H}_k(\Sigma; \Z) = {\rm ker}(\redbndry_k)/{\rm
    img}(\redbndry_{k+1})$.  \label{homolAppDef}
  \quad (That makes sense since $\redbndry_k \circ \redbndry_{k+1} = 0$.)

\vspace*{-0.3in}

\paragraph{}$\phantom{0}$

\item Given a simplicial complex $\Sigma$ and a set $\sigma$, we
  define the following three simplicial subcomplexes of $\Sigma$ in
  the standard way:\label{subcomplexes}

  \vspace*{-0.04in}

  \begin{itemize}
  \item The {\em link} of $\sigma$ in $\Sigma$:\quad
    $\lk(\Sigma,\sigma) =
    \setdef{\tau\in\Sigma}{\tau\inter\sigma=\emptyset\,\;\hbox{and}\;\tau\union\sigma\in\Sigma}$.
  \item The {\em deletion} of $\sigma$ in $\Sigma$:\quad
    $\dl(\Sigma,\sigma) =
    \setdef{\tau\in\Sigma}{\tau\inter\sigma=\emptyset}$.
  \item The {\em closed star} of $\sigma$ in $\Sigma$:\quad
    $\st(\Sigma,\sigma) =
    \setdef{\tau\in\Sigma}{\tau\union\sigma\in\Sigma}$.
  \end{itemize}

  The definitions make sense even when $\sigma$ is not itself a simplex
  in $\Sigma$, though in that case both $\lk(\Sigma,\sigma)$ and 
  $\st(\Sigma,\sigma)$ are instances of the void complex $\emptyset$.

  \vspace*{0.05in}

  Observe that $\dl(\Sigma,\sigma) \,\inter\; \st(\Sigma, \sigma) \;=\;
  \lk(\Sigma,\sigma)$ \ and \ $\st(\Sigma, \sigma) \;=\;
  \lk(\Sigma,\sigma) * \hbox{$<$$\sigma$$>$}$.

  Here $*$ means simplicial join (described on
  page~\pageref{joinAppdef}) and \hbox{$<$$\sigma$$>$} is the {\em
  simplicial complex generated by} $\sigma$ (defined to be the
  collection of all subsets of $\sigma$).

  \vspace*{0.05in}

  When $\sigma$ consists of a single element $v$, i.e., $\sigma=\{v\}$,
  we tend simply to write $\lk(\Sigma,v)$, $\dl(\Sigma,v)$, $\st(\Sigma,v)$.
  \ Aside: For a singleton $v$, it is further true that $\dl(\Sigma,v)
  \union \st(\Sigma,v) = \Sigma$.

\paragraph{}$\phantom{0}$

\vspace*{-0.3in}

\item One may associate a {\em geometric realization} to a finite
  nonvoid abstract simplicial complex $\Sigma$ by embedding $\Sigma$
  into a finite-dimensional Euclidean space.  One may therefore think
  of $\Sigma$ as a topological space in a well-defined way
  \cite{tpr:munkres, tpr:bjorner}.\label{geomrealiz}

\paragraph{}$\phantom{0}$

\vspace*{-0.45in}

\item Suppose $\Sigma$ and $\Gamma$ are two simplicial complexes with
  underlying vertex sets $X$ and $Y$, respectively.  A set function $f
  : X \rightarrow Y$ is said to be a {\em simplicial map} if it
  satisfies the following condition: \ If $\sigma\in\Sigma$, then
  $f(\sigma)\in\Gamma$.

  In that case, one may view $f$ as a map of simplicial complexes, $f :
  \Sigma \rightarrow \Gamma$.

  A simplicial map may further be viewed as a continuous function
  between the geometric realizations of $\Sigma$ and $\Gamma$
  \cite{tpr:munkres}.

\item When $X_1$ and $X_2$ are topological spaces, the notation $X_1
  \homot X_2$ means that $X_1$ and $X_2$ have the same {\em homotopy
  type} \cite{tpr:bjorner, tpr:hatcher}.  One may also say that $X_1$
  and $X_2$ are {\em homotopic} or {\em homotopy equivalent}.  A
  topological space homotopic to a point is said to be {\em
  contractible}.\label{homotAppdef}

\item When $X_1$ and $X_2$ are topological spaces, $X_1 \join X_2$
  means a {\em wedge sum} of $X_1$ and $X_2$
  \cite{tpr:hatcher}.\label{wedgeAppdef}

\paragraph{}$\phantom{0}$

\vspace*{-0.45in}

\item Suppose $\calU$ is a nonempty collection of (not necessarily
  distinct) topological subspaces of some nonempty ambient topological
  space.  One may define a simplicial complex $\nerveU$, called the
  {\em nerve\!  of $\,\calU$}, whose simplices are the finite
  subcollections $\calW$ of $\calU$ for which
  $\biginter_{\,W\in\,\calW}W$ is not the empty space.
  If $\biginter_{\,W\in\,\calW}W$ is contractible for each nonempty
  simplex $\calW$ of $\nerveU$, then, under a variety of additional
  finiteness conditions \cite{tpr:bjorner, tpr:hatcher}, the nerve has
  the same homotopy type as the union of all the spaces in $\calU$:
  \
  $\nerveU \, \homot\; \bigunion_{\,U\in\,\calU}U$.

\item Suppose $\Sigma$ and $\Gamma$ are simplicial complexes with
  disjoint underlying vertex sets.  The {\em simplicial join}
  \cite{tpr:wachs} of $\Sigma$ and $\Gamma$ is the simplicial complex
  $$\Sigma * \Gamma \;=\; \setdef{\spc\sigma \union \gamma}{\sigma\in\Sigma
    \ \,\hbox{and}\ \gamma\in\Gamma}.$$
  The underlying vertex set of $\Sigma * \Gamma$ is the union of the
  underlying vertex sets of $\Sigma$ and $\Gamma$.\label{joinAppdef}

\end{itemize}

\subsection{Partially Ordered Sets (Posets)}
\markright{Posets, Semi-Lattices, and Lattices}

We largely follow the notation of \cite{tpr:wachs}.

\begin{itemize}

\item A {\em poset} $P$ is a set of elements with a partial order,
  sometimes written simply as ``$\leq$'' other times as ``$\leq_P$''.
  The symbols ``$\geq$'', ``$<$'', ``$>$'' and ``$=$'' are defined
  accordingly.\label{posetAppdef}

\item A {\em chain} $c$ in a poset $P$ is a totally ordered subset of
  $P$, which we often write as $c = \{p_0 < p_1 < \cdots < p_\ell\}$.
  The {\em length} $\ell(c)$ of chain $c$ is $\ell$, one less than the
  number of elements in the chain (analogous to simplex dimension).
  The length of the empty chain is $-1$.  The length $\ell(P)$ of a
  poset $P$ is the maximum length of any chain in $P$.\label{posetchainAppdef}

\item The {\em face poset\,} $\F(\Sigma)$ of a nonvoid simplicial
  complex $\Sigma$ consists of all nonempty simplices of $\Sigma$,
  partially ordered by set inclusion. \label{faceposetAppdef}
  \ (If $\Sigma$ is void, we leave $\F(\Sigma)$ {\em undefined}.)

\item The {\em order complex\,} $\Delta(P)$ of a poset $P$ is the
  simplicial complex whose simplices are given by all finite chains
  $\{p_0 < p_1 < \cdots < p_\ell\}$ in $P$.  \label{ordercpxAppdef}
  \ (If $P=\emptyset$, then $\Delta(P)=\{\emptyset\}$.)

\item One may speak of the {\em topology of a poset}: One says that a
  poset $P$ has a topological property when its order complex
  $\Delta(P)$ has that property and the property is an invariant of
  homeomorphism type.  For instance, to say that a poset is contractible
  means that its order complex is contractible.  To say that two posets
  $P$ and $Q$ are homotopic means that $\Delta(P)$ and $\Delta(Q)$ have
  the same homotopy type.  Etc.

\item For nonvoid $\Sigma$, it is a fact that $\Delta(\F(\Sigma))$ is
  homeomorphic to $\Sigma$.  Indeed, $\Delta(\F(\Sigma))$ may be
  viewed as the {\em first barycentric subdivision of $\hspt\Sigma$},
  which we write as $\sd(\Sigma)$.  See \cite{tpr:rotman, tpr:wachs}.

\item A set function $\theta : P \rightarrow Q$ between two posets $P$
  and $Q$ is said to be a {\em poset map} if it is either {\em
  order-preserving\,} or {\em order-reversing}.
  \quad That means:

  \vspace*{-0.2in}

  \begin{eqnarray*}
    \hbox{order-preserving:}\ \ & \hbox{For all $x,y \in P$, \ \ if \ $x \,\leq_P\, y$,} & \!\hbox{then \ $\theta(x) \,\leq_Q\, \theta(y).$} \\
    \hbox{order-reversing:}\ \  & \hbox{For all $x,y \in P$, \ \ if \ $x \,\leq_P\, y$,} & \!\hbox{then \ $\theta(x) \,\geq_Q\, \theta(y).$} \\
  \end{eqnarray*}

  \vspace*{-0.2in}

\item A poset map $\theta : P \rightarrow Q$ between two posets $P$ and
  $Q$ induces a simplicial map between the associated order complexes
  $\theta : \Delta(P) \rightarrow \Delta(Q)$.

\item An order-preserving poset self-map $\theta : P \rightarrow P$ is
  said to be a {\em closure operator} when $x \leq_P \theta(x)$, for all
  $x\in P$, and $\theta \circ \theta = \theta$.  A closure operator
  $\theta$ induces a homotopy equivalence between $P$ and the image
  $\theta(P)$.  \ See \cite{tpr:bjorner, tpr:wachs, tpr:quillenH,
  tpr:quillenK}.

\end{itemize}

\subsection{Semi-Lattices and Lattices}
\label{lattices}

We largely follow the development in \cite{tpr:wachs} and \cite{tpr:bjorner}.
\ \ Let $L$ be a partially ordered set:

\begin{itemize}

\item Suppose $p,q \in L$.  If $p$ and $q$ have a unique least upper
  bound, then one writes $p \join q$ to mean that least upper bound.
  (One may also write $p \,{\join_{\!L}}\, q$.)  If every pair of
  elements in $L$ has a unique least upper bound in $L$, then one
  refers to $L$ as a {\em join semi-lattice}.

\item Suppose $p,q \in L$.  If $p$ and $q$ have a unique greatest
  lower bound, then one writes $p \meet q$ (or possibly 
  $p \,{\meet_L}\, q$) to mean that greatest lower bound.  If every
  pair of elements in $L$ has a unique greatest lower bound in $L$,
  then one refers to $L$ as a {\em meet semi-lattice}.

\item A poset that is both a join semi-lattice and a meet semi-lattice
  is known as a {\em lattice}.  \label{latticeAppdef}

\item If $L$ has a unique top (i.e., maximal) element, we may
  designate that element by $\topone$ or $\oneL$.

\item If $L$ has a unique bottom (minimal) element, we may designate
  that element by $\botzero$ or $\zeroL$.

\item If $L$ is a finite join semi-lattice with a unique bottom
  element, then $L$ is a lattice.  Similarly, if $L$ is a finite meet
  semi-lattice with a unique top element, then $L$ is a lattice.

\item A lattice $L$ is called {\em bounded\hspt} if it has a unique
  top element $\topone$ and a unique bottom element $\botzero$.
  \ (These are same element if $L$ is a singleton.)

\item When $L$ is a bounded lattice, the {\em proper part of $L$} is
  the poset $\Lprop = \sdiff{L}{\{\botzero, \topone\}}$.

\paragraph{}$\phantom{0}$

\vspace*{-0.31in}

\item Suppose $L$ is a bounded lattice and $p\in L$.  Then the {\em
  complements of $p$} are given by the set $\comp(p)=\setdef{q\in
  L}{\hbox{$q\join{p}=\topone$ and $q\meet{p}=\botzero$}}$.

\item A bounded lattice $L$ is said to be {\em noncomplemented}\hspt\
  if $\comp(p)=\emptyset$ for at least one $p\in L$.  If $L$ is a
  noncomplemented bounded lattice with $\Lprop\neq\emptyset$, then
  $\Lprop$ is contractible \label{noncomplemented}\cite{tpr:bjorner}.

\item Suppose $L$ is a bounded lattice with $\Lprop\neq\emptyset$.
  The elements of $L$ immediately below $\topone$ are called {\em
  co-atoms}.  These are the maximal elements of $\Lprop$.  The
  elements immediately above $\botzero$ are called {\em atoms}.  These
  are the minimal elements of $\Lprop$.

\end{itemize}

\subsection{Relations}
\markright{Relations}
\label{relationsApp}

Let $R$ be a relation on $\XxY$, with $X$ and $Y$ finite
discrete spaces.\\[1pt]
We use the following notation and conventions (see also
page~\pageref{basicdefs}):

\vspace*{-0.025in}

\begin{itemize}

\item $R$ is a set of ordered pairs, namely a subset of the cross
  product $\XxY$.  It is convenient sometimes to view $R$ as a
  matrix of $0$s and $1$s, perhaps drawn as a matrix of blank and
  nonblank entries, representing the characteristic function of this
  set of ordered pairs.\label{reldefApp}

\item Even if $X\!\neq\emptyset$ and $Y\!\neq\emptyset$, it is possible
  that $R=\emptyset$, in which case we say that $R$ is an {\em empty
  relation}.

\item If $X\!=\emptyset\,$ and/or $\,Y\!=\emptyset$, then we say that
  $R\hspc\hspc$ is a {\em void relation}.\label{voidrelation}

  On some occasions, we may treat a void relation $R$ much like an
  empty relation, in the sense that we will let the {\em Dowker
  complexes} defined below (and on page~\pageref{basicdefs}) be empty
  rather than void.  That view will sometimes be convenient when $R$
  is derived from some encompassing relation as a link or deletion in
  a simplicial complex.

\item We often refer to elements of $\hspt{}X\mskip-1.5mu$ as {\em
  individuals} and elements of $\hspt{}Y\mskip-2.3mu$ as {\em attributes}.

\item For each $x\in{X}$, $Y_x$ is the set of attributes that individual
  $x$ has (in relation $R$).  Viewing $R$ as a matrix, one may think of
  $Y_x$ as the row of $R$ indexed by $x$.  We say that {\em the row is
  blank\,} when $Y_x=\emptyset$.

\paragraph{}$\phantom{0}$

\vspace*{-0.5in}

\item For each $y\in{Y}$, $X_y$ is the set of individuals who have
  attribute $y$ (in relation $R$).  Viewing $R$ as a matrix, one may
  think of $X_y$ as the column of $R$ indexed by $y$.  We say that {\em
  the column is blank\,} when $X_y=\emptyset$.

\vspace*{0.04in}

\item $\dowy$ is the {\em Dowker attribute complex\,} determined by
  $R$.  It is a simplicial complex with underlying vertex set $Y$. \ A
  nonempty subset $\gamma$ of $Y$ is a simplex in $\dowy$ precisely
  when there exists $x\in X$ such that $(x,y)\in R\hspt$ for all
  $y\in\gamma$.  \ We refer to $x$ as a {\em witness for
  $\gamma$}.\label{dowkerYAppdef}

  When $R\hspt$ is void, we let $\dowy$ be void as well, except as
  otherwise indicated in the text.

  When $R\hspt$ is nonvoid, $\dowy$ contains at least the empty simplex.
  We then may view $\dowy$ as generated by the rows of $R$, so
  $\dowy=\bigunion_{x\in X}\hbox{$<$$Y_x$$>$}$.
  \ Thus, $Y_x \in \dowy$ for each $x\in X$.

  \vspace*{0.05in}

\paragraph{}$\phantom{0}$

\vspace*{-0.31in}

\item $\dowx$ is the {\em Dowker association complex\,} determined by
  $R$.  It is a simplicial complex with underlying vertex set $X$. \ A
  nonempty subset $\sigma$ of $X$ is a simplex in $\dowx$ precisely
  when there exists $y\in Y$ such that $(x,y)\in R\hspt$ for all
  $x\in\sigma$.  \ We refer to $y$ as a {\em witness for
  $\sigma$}.\label{dowkerXAppdef}

\vso

  When $R\hspt$ is void, we let $\dowx$ be void as well, except as
  otherwise indicated in the text.

\vso

  When $R\hspt$ is nonvoid, $\dowx$ contains at least the empty simplex.
  We then may view $\dowx$ as generated by the columns of $R$, so
  $\dowx=\bigunion_{y\in Y}\hbox{$<$$X_y$$>$}$.
  \ Thus, $X_y \in \dowx$ for each $y\in Y$.

\paragraph{}$\phantom{0}$

\vspace*{-0.36in}

\item There exist homotopy equivalences \label{phipsiAppdef}$\phi_R :
  \dowx \rightarrow \dowy$ and $\psi_R : \dowy \rightarrow \dowx$.

  Viewed as (order-reversing) poset maps $\phi_R : \F(\dowx)
  \rightarrow \F(\dowy)$ and $\psi_R : \F(\dowy) \rightarrow
  \F(\dowx)$, one obtains explicit formulas, sending nonempty
  simplices to nonempty simplices:

\vspace*{-0.1in}

  $$\phi_R(\sigma) = \biginter_{x\in\sigma}Y_x
\quad\hbox{and}\quad
    \psi_R(\gamma) = \biginter_{y\in\gamma}X_y.$$

\vspace*{-0.05in}

  Suppose $X\!\neq\emptyset$ and $Y\!\neq\emptyset$.  Then the
  intersections appearing in the previous formulas comprise the
  witnesses for the respective simplex arguments.  Consequently, one
  may use the formulas more generally as tests for membership in the
  Dowker complexes:

\vspace*{-0.03in}

\begin{itemize}
\item For any $\sigma\subseteq X$, $\sigma\in\dowx$ if and only if $\phi_R(\sigma)\neq\emptyset.$
\item For any $\gamma\subseteq Y$, $\gamma\in\dowy$ if and only if $\psi_R(\gamma)\neq\emptyset.$
\end{itemize}

\vspace*{-0.03in}

 These tests also make sense for the empty set, that is, when
 $\sigma=\emptyset$\, or $\gamma=\emptyset$.  In particular,
 $\phi_R(\emptyset) = Y$ and $\psi_R(\emptyset) = X$.

\paragraph{}$\phantom{0}$

\vspace*{-0.385in}

\item Composing $\phi_R$ and $\psi_R$ as $\clsx : \F(\dowx) \rightarrow
  \F(\dowx)$ and $\clsy : \F(\dowy) \rightarrow \F(\dowy)$ produces
  closure operators.  See Appendix~\ref{basictools} for further details.

\vspace*{0.05in}

\item $\PR$ is the {\em doubly-labeled poset} associated with $R$ as
  per Definition~\ref{defPR} on page~\pageref{defPR}.  Each element in
  $\PR$ is of the form $(\sigma, \gamma)$, with $\sigma\neq\emptyset$
  and $\gamma\neq\emptyset$, such that $\sigma=\psi_R(\gamma)$ and
  $\gamma=\phi_R(\sigma)$.

  \vst

  One may view $\PR$ either as the image $(\clsx)(\F(\dowx))$ or as the
  image $(\clsy)(\F(\dowy))$.\label{PRAppDeff}

  \vst

  We mention some special cases:\label{PRspecialcases}

\vspace*{-0.05in}

\begin{itemize}

  \item If $\dowx$ and $\dowy$ are instances of the empty complex
  $\{\emptyset\}$, then $\PR=\emptyset$.  This occurs when $R$ is an
  empty relation, or when $R$ is void but we let $\dowx=\{\emptyset\}$
  and $\dowy=\{\emptyset\}$.

  \item If $\dowx$ and $\dowy$ are instances of the void complex
  $\emptyset$, then $\PR$ is {\em undefined}.

\end{itemize}

\vspace*{0.05in}

\item $\PRplus$ is the {\em Galois lattice} formed from $\PR$ as per
  Definition~\ref{galoislattice} on
  page~\pageref{galoislattice}.\label{galoislatticeApp}

\vst

  If $R$ is an empty relation, then $\PR=\emptyset$ and so $\PRplus$
  consists simply of $\zeroR$ and $\oneR$.

\vst

  Definition~\ref{galoislattice} assumes that the underlying spaces
  $\hspt{}X$ and $Y$ of $R\hspt$ are both nonempty.  One could imagine
  extending the definition, perhaps as follows: (i) When the Dowker
  complexes are void, leave $\PR$ undefined and let
  $\PRplus=\emptyset$.  (ii) When $R$ is technically void but the
  Dowker complexes are artificially empty, with one of $\hspt{}X$ or
  $Y\!$ empty, let $\PR=\emptyset$ and $\PRplus=\{(X, Y)\}$.
  \ Fortunately, we will not need these boundary cases.

  (Different perspectives often suggest conflicting interpretations in
  null situations \cite{tpr:nullgraph}.  This report chooses to
  preserve the validity of Dowker's Theorem, meaning
  $\dowx\homot\dowy$.)

\vspace*{0.05in}

\addtolength{\baselineskip}{0.75pt}

\item We sometimes view $\PR$ as ``almost a join-based lattice'', as per
  Definition~\ref{almostlattice} on page~\pageref{almostlattice}.  That
  amounts to adjoining a single new element $\topone$ above $\PR$, then
  inducing a join operation on $\PR \union \{\topone\}$ from the join
  operation on $\PRplus$.  Thus $\PR \union \{\topone\}$ is a join
  semi-lattice.  If we further adjoin a new bottom element $\botzero$,
  then $\PR \union \{\botzero, \topone\}$ is a lattice.

\addtolength{\baselineskip}{-0.75pt}

\vspace*{0.05in}

\item One may speak of the {\em topology of a relation (modulo
  homotopy equivalence)}: \ One says that a relation $R$ has a
  topological property when any and all of $\dowy$, $\dowx$, and
  $\Delta(\PR)$ have that property and the property is an invariant of
  homotopy type.  (This convention makes sense by Dowker's Theorem on
  page~\pageref{dowker} and the nature of $\PR$.)  Connectivity is an
  example of such a property.

\end{itemize}

\clearpage
\section{Basic Tools}
\markright{Basic Tools}
\label{basictools}

This appendix reviews some basic facts about relations, their Dowker
complexes, and the Galois connection.  \ Recall the formulas for
$\phi_R$ and $\psi_R$ from page~\pageref{phipsiAppdef}.

\vspace*{0.05in}

Although we do not always say so explicitly, there are dual statements
for the lemmas and corollaries in this appendix, for each of the two
perspectives offered by Dowker's Theorem, by inverting the roles of
individuals and attributes.

\begin{lemma}\label{orderreversing}
Let $R$ be a relation on $\XxY$.
Then $\phi_R$ is inclusion-reversing.
\end{lemma}

\vspace*{-0.16in}

\begin{proof}
Let $\sigma^\prime\subseteq\sigma\subseteq{X}$.  \quad Then:

\vspace*{-0.35in}

$$\hspace*{1.4in}\phi_R(\sigma^\prime)
  \;=\;
  \biginter_{x\in\sigma^\prime}Y_x
  \;\supseteq\;
  \biginter_{x\in\sigma}Y_x
  \;=\;
  \phi_R(\sigma).$$

Just to be careful: if $\sigma^\prime = \emptyset$, then
$\phi_R(\sigma^\prime)=Y$, which does indeed contain $\phi_R(\sigma)$.
\end{proof}

\vspace*{0.05in}

Each of $\phi_R$ and $\psi_R$ is inclusion-reversing, so $\clsy$ is
inclusion-preserving.  Lemmas~\ref{upward} and \ref{idempotent}
establish that $\clsy$ is a closure operator when viewed as a poset map
$\F(\dowy) \rightarrow \F(\dowy)$:

\begin{lemma}\label{upward}
Let $R$ be a relation on $\XxY$.
For all $\gamma\subseteq{Y}$, \ $\gamma\subseteq(\clsy)(\gamma)$.
\end{lemma}

\vspace*{-0.2in}

\begin{proof}
$$(\clsy)(\gamma) = \biginter_{x\in\sigma}Y_x,
   \quad \hbox{with} \;
  \sigma = \biginter_{y\in\gamma}X_y.$$

The assertion is clear if $\gamma=\emptyset$ or $\sigma=\emptyset$.
Otherwise, let $y\in\gamma$ and $x\in\sigma$.
Then $x\in X_y$, so $y\in Y_x$.  Since $x$ is arbitrary in $\sigma$, we
see that $y\in(\clsy)(\gamma)$ and thus $\gamma\subseteq(\clsy)(\gamma)$.
\end{proof}

\begin{corollary}\label{maximal}
Let $R$ be a relation on $\XxY$, with both $\hspt{}X\!$ and
$\hspt{}Y\!$ nonempty.

If $\gamma$ is a maximal simplex of $\hspt\dowy$, then
$(\clsy)(\gamma)=\gamma$.
\end{corollary}

\vspace*{-0.16in}

\begin{proof}
When $\gamma\mskip-0.75mu{}\neq\emptyset$, this assertion follows from
Lemma~\ref{upward} and maximality of $\gamma$.  Otherwise, apparently
$\dowy=\{\emptyset\}$ and so $(\clsy)(\emptyset)=\phi_R(X)=\emptyset$
(since $\phi_R$ must map $X$ into $\dowy$).
\end{proof}

\begin{lemma}\label{idempotent}
Let $R$ be a relation on $\XxY$.

For all $\gamma\subseteq{Y}$, \ $\big((\clsy)\circ(\clsy)\big)(\gamma)=(\clsy)(\gamma)$.
\end{lemma}

\vspace*{-0.16in}

\begin{proof}
Consider: \quad $\gamma
  \;\xmapsto{\phantom{0}\psi_R\phantom{0}}\;
  \sigma
  \;\xmapsto{\phantom{0}\phi_R\phantom{0}}\;
  \gamma^\prime
  \;\xmapsto{\phantom{0}\psi_R\phantom{0}}\;
  \sigma^\prime
  \;\xmapsto{\phantom{0}\phi_R\phantom{0}}\;
  \gamma^{\prime\prime}.$

\vspace*{0.1in}

We need to show that $\gamma^\prime=\gamma^{\prime\prime}$.

By Lemma~\ref{upward} and its dualization, 
$\gamma\subseteq\gamma^\prime\subseteq\gamma^{\prime\prime}$
and $\sigma\subseteq\sigma^\prime$.

By Lemma~\ref{orderreversing}, $\phi_R$ is inclusion-reversing, so
$\sigma\subseteq\sigma^\prime$ implies
$\gamma^\prime\supseteq\gamma^{\prime\prime}$, and thus
$\gamma^\prime=\gamma^{\prime\prime}$.

Comment: By the dual of Lemma~\ref{orderreversing}, $\psi_R$ is
inclusion-reversing, so in fact also $\sigma=\sigma^\prime$.
\end{proof}

\begin{corollary}\label{mappedsimplex}
Let $R$ be a relation on $\XxY$.   \ For all $\sigma\subseteq{X}$,
\ $(\clsy)(\phi_R(\sigma)) = \phi_R(\sigma)$.
\end{corollary}

\vspace*{-0.16in}

\begin{proof}
This follows from a dual version of the comment at the end of the
proof of Lemma~\ref{idempotent}.
\end{proof}

\begin{corollary}\label{generatorface}
Let $R$ be a relation on $\XxY$.
\ For all $x\in X$, $(\clsy)(Y_x)=Y_x$.
\end{corollary}

\vspace*{-0.16in}

\begin{proof}
The assertion follows from Corollary~\ref{mappedsimplex}, with
$\sigma=\{x\}$.

\hspace*{0.25in}(This includes the case $Y_x=\emptyset$.)
\end{proof}

\begin{lemma}\label{propersubsets}
Let $R$ be a relation on $\XxY$ and suppose $\eta\subseteq Y$.\\
Then the following two conditions are equivalent:

\vspace*{0.1in}

\hspace*{0.4in}\begin{minipage}{5in}
\begin{enumerate}
\addtolength{\itemsep}{-1pt}
\item[(a)] $(\clsy)(\chi) = \chi$, for every proper subset $\chi$ of $\eta$.
\item[(b)] $(\clsy)(\gamma) = \gamma$, for all $\gamma$ of the form $\gamma =
\eta\setminus\ys$ with $y\in\eta$.
\end{enumerate}
\end{minipage}
\end{lemma}

\vso

\begin{proof}
Certainly (a) implies (b).  Suppose (b) holds, but there is some
$\chi\subsetneq\eta$ such that $\chi\subsetneq(\clsy)(\chi)$.
\ Since (b) holds, $(\clsy)(\chi)\subseteq\eta$.
\ Let $y\in(\clsy)(\chi)\setminus\chi$ and consider $\gamma=\eta\setminus\ys$.

\vspace*{0.03in}

Observe that $\chi\subseteq\gamma$, so
$y\in(\clsy)(\chi)\subseteq(\clsy)(\gamma)$.
Consequently,

\vspace*{-0.1in}

$$\eta
  \;=\; \gamma \union \ys 
  \;\subseteq\; (\clsy)(\gamma)
  \;=\; \gamma
  \;\subsetneq\; \eta, \hspace{0.4in}\hbox{which is a contradiction.}$$

\vspace*{-0.24in}

\end{proof}

\vspace*{0.025in}

\begin{definition}[Connected]\label{connected}
A relation $R$ on $\,\XxY$ is \,\mydefem{connected} if $R$ is
connected when viewed as an undirected bipartite graph on the vertex
sets $X\!$ and $Y$. \ (This definition regards $X$ and $\,Y\!$ as
disjoint.)
\end{definition}

\vspace*{0.025in}

\begin{definition}[Tight]\label{tight}
A relation $R$ on $\,\XxY$ is \,\mydefem{tight}\, if it has no
blank rows or columns.
\end{definition}

\vspace*{-0.175in}

\paragraph{Comment:} \ As mentioned on page~\pageref{geomrealiz}, one
can view an abstract simplicial complex as a topological space, via
its geometric realization.  In particular, one may ask whether a
simplicial complex is {\em path-connected}.

\vspace*{0.025in}

\begin{lemma}[Connectedness]\label{connectedcplx}
Let $R$ be a tight relation on $\XxY$, with both $\hspc{}X\!$
and $\hspt{}Y\!$ nonempty.\\
Then the following three conditions are equivalent:

\vspace*{0.1in}

\hspace*{1.2in}\begin{minipage}{4in}
\begin{enumerate}
\addtolength{\itemsep}{-1pt}
\item[(a)] $R$ is connected.
\item[(b)] $\dowx$ is path-connected.
\item[(c)] $\dowy$ is path-connected.
\end{enumerate}
\end{minipage}
\end{lemma}

\vspace*{0.01in}

\begin{proof}
We will show that (a) and (b) are equivalent.  The proof for (a) and
(c) is similar, or one can simply invoke Dowker duality.

\vspace*{0.05in}

I.  Suppose $R$ is connected.  Consider two vertices $x_0$ and $x_f$
of $\dowx$.  Since $R$ is connected as a bipartite graph, there exists
a path $x_0, y_1, x_1, y_2, \ldots, y_n, x_n=x_f$ in this graph.
Observe that each $y_i$ is a witness for the simplex $\{x_{i-1},
x_i\}\in\dowx$.  We can assume without loss of generality that
$x_{i-1} \neq x_i$, for each relevant $i$.  So in $\dowx$ there exist
edges $\{x_0, x_1\}, \ldots, \{x_{n-1}, x_n\}$.  Since $\dowx$ is a
simplicial complex, we see that it is path-connected.

\vspace*{0.05in}

II.  Suppose $\dowx$ is path-connected.  Since $R$ is tight, each
$y\in{Y}$ appears as the vertex of an edge $(x,y)$ in the bipartite
graph $R$.  To show that $R$ is connected, it therefore is enough to
show that any two elements $x_0$ and $x_f$ of $X$ may be connected by a
path in the bipartite graph.  Since $R$ is tight, $x_0$ and $x_f$ are
each vertices of $\dowx$.  Since $\dowx$ is path-connected, there exists
a path between $x_0$ and $x_f$ in $\dowx$.  Since $\dowx$ is a finite
simplicial complex, we can deform that path so that it consists of
finitely many edges $\{x_0, x_1\}, \ldots, \{x_{n-1}, x_n\}$, with each
$x_i$ a vertex of $\dowx$ and $x_n=x_f$.  Each edge $\{x_{i-1}, x_i\}$
has some witness $y_i\in Y$.  So $x_0, y_1, x_1, y_2, \ldots, y_n, x_f$
is a path connecting $x_0$ and $x_f$ in the bipartite graph $R$.
\end{proof}

\clearpage

\begin{lemma}[Components]\label{components}
Let $R$ be a tight relation on $\,\XxY$, with both $\hspt{}X\!$ and
$\mskip1mu\hspt{}Y\!$ nonempty.  Suppose $R = R_1 \union \cdots \union
R_\ell$, with the $\{R_i\}$ pairwise disjoint and each $R_i$ a
connected component of $R$ viewed as a bipartite graph on $X\!$ and $Y$.
Then $X$, $Y$, $\dowx$, and $\dowy$ decompose as follows:

\vspace*{0.125in}

\begin{minipage}{5.9in}
\begin{enumerate}
\addtolength{\itemsep}{-1.5pt}
\item[(a)] $X = X_1 \union \cdots \union\, X_\ell$, with the $\{X_i\}$ pairwise disjoint and each $X_i$ not empty.
\item[(b)] $Y = Y_1 \union \cdots \union\, Y_\ell$, with the $\{Y_i\}$ pairwise disjoint and each $Y_i$ not empty.
\item[(c)] $R_i$ is the restriction of $R$ to $X_i \times Y_i$, and is tight, for $i=1, \ldots, \ell$.
\item[(d)] $\dowx = \Psi_{R_{\scriptstyle 1}} \union \cdots \union\, \Psi_{R_{\scriptstyle \ell}}$, with pairwise disjoint face posets and each $\Psi_{R_{\scriptstyle i}}$ path-connected.
\item[(e)] $\dowy = \Phi_{R_{\scriptstyle 1}} \union \cdots \union\, \Phi_{R_{\scriptstyle \ell}}$, with pairwise disjoint face posets and each $\Phi_{R_{\scriptstyle i}}$ path-connected.
\end{enumerate}
\end{minipage}

\end{lemma}

\vspace*{0.025in}

\begin{proof}
Let
$X_i = \setdef{x}{(x,y)\in R_i\; \hbox{for some}\; y\in Y}$ and 
$Y_i = \setdef{y}{(x,y)\in R_i\; \hbox{for some}\; x\in X}$,
for $i = 1, \ldots, \ell$.
These sets are nonempty since the components of $R$ are necessarily nonempty.

To see that $X_i \inter X_j = \emptyset$ unless $i=j$, suppose $x \in
X_i \inter X_j$.  Then $(x, y)\in R_i$ for some $y\in Y$ and $(x,
y^\prime)\in R_j$ for some $y^\prime\in Y$.  Since $R_i$ and $R_j$ are
connected components of $R$, $i=j$.  Next observe that each $x$ of $X$
must appear in some $X_i$ since $R$ has no blank rows.  Point (a)
follows.  \quad Point (b) is similar.

For (c), observe that if $(x,y)\in R_i \subseteq R$ then $x\in X_i$ and
$y\in Y_i$, so $(x,y)$ is in the restriction of $R$ to $X_i \times Y_i$.
Conversely, if $(x,y)\in R$ with $x\in X_i$ and $y\in Y_i$, then
$(x,y)\in R_j$ for some $j$.  By the previous reasoning, $i=j$.
Tightness follows by definition of $X_i$ and $Y_i$.

For (d), $\Psi_{R_{\scriptstyle i}} \subseteq \dowx$ since $R_i
\subseteq R$, for each $i=1, \ldots, \ell$.  \ Now let
$\emptyset\neq\sigma\in \dowx$.  Then there exists $y\in Y$ such that
$(x,y)\in R$ for every $x\in \sigma$.  For some $i$, $y\in Y_i$.  Since
$R_i$ is a connected component of $R$, $(x,y)\in R_i$ for every $x\in
\sigma$, so $\sigma\in \Psi_{R_{\scriptstyle i}}$.
The collections $\{\F(\Psi_{R_{\scriptstyle i}})\}$ are pairwise
disjoint since the underlying vertex sets $\{X_i\}$ are pairwise
disjoint.
Path-connectedness follows from Lemma~\ref{connectedcplx}, since each
$R_i$ is tight and connected.
\quad Point (e) is similar.
\end{proof}

\begin{corollary}[Component Maps]\label{componentmaps}
Assume the hypotheses and constructions as in Lemma~\ref{components} and
its proof.  \ Then:

\vspace*{-0.4in}

\begin{eqnarray*}
\hspace*{1in}\psi_{R_{\scriptstyle i}}(\gamma) &=& \psi_R(\gamma), \quad \hbox{for each $\;\emptyset\neq\gamma\in\Phi_{R_{\scriptstyle i}}$},\\[2pt]
\hspace*{1in}\phi_{R_{\scriptstyle i}}(\sigma) &=& \phi_R(\sigma), \quad \hbox{for each $\;\emptyset\neq\sigma\in\Psi_{R_{\scriptstyle i}}$},
       \quad i=1, \ldots, \ell.
\end{eqnarray*}
\end{corollary}

\begin{proof}
By direct computation (be aware, the subscripts in $X_y$ and $X_i$
have different meanings):

\vspace*{-0.1in}

$$\psi_{R_{\scriptstyle i}}(\gamma)
  \;=\;
  \biginter_{y\in\gamma}(X_y \inter X_i)
  \;=\;
  \biginter_{y\in\gamma}X_y
  \;=\;
  \psi_R(\gamma).$$

The second equality comes from the fact that each $X_y$ can touch only
$X_i$, since $R_i$ is a connected component of $R$.
\quad
The argument for the $\phi_{\ldots}$ maps is similar.
\end{proof}

\begin{corollary}[Component Privacy]\label{componentprivacy}
Assume the hypotheses and constructions as in Lemma~\ref{components} and
its proof.  Let $i\in\{1, \ldots, \ell\}$.

If $\clsx$ is the identity on $\dowx$ and
$Y_i\not\in\Phi_{R_{\scriptstyle i}}$,
then $\psi_{R_{\scriptstyle i}} \circ \phi_{R_{\scriptstyle i}}$ is the identity on
$\Psi_{R_{\scriptstyle i}}$.

If $\clsy$ is the identity on $\dowy$ and
$X_i\not\in\Psi_{R_{\scriptstyle i}}$,
then $\phi_{R_{\scriptstyle i}} \circ \psi_{R_{\scriptstyle i}}$ is the identity on
$\Phi_{R_{\scriptstyle i}}$.

\end{corollary}

\begin{proof}
Suppose $\emptyset\neq\sigma\in\Psi_{R_{\scriptstyle i}}$. Then
$\emptyset\neq\phi_{R_{\scriptstyle i}}(\sigma)\in\Phi_{R_{\scriptstyle
i}}$, so by Corollary~\ref{componentmaps},
$(\psi_{R_{\scriptstyle i}} \circ \phi_{R_{\scriptstyle i}})(\sigma) =
(\clsx)(\sigma) = \sigma$.\quad
And $(\psi_{R_{\scriptstyle i}} \circ \phi_{R_{\scriptstyle
i}})(\emptyset) = \psi_{R_{\scriptstyle i}}(Y_i) = \emptyset$, since
$Y_i\not\in\Phi_{R_{\scriptstyle i}}$.

The argument for $\phi_{R_{\scriptstyle i}} \circ \psi_{R_{\scriptstyle
i}}$ is similar.
\end{proof}

\clearpage
\section{Links, Deletions, and Inference}
\markright{Links, Deletions, and Inference}
\label{linksandinference}

This appendix provides some technical tools for modeling inference,
particularly in links, ending with some instances in which inference
is unavoidable.

\subsection{Links, Deletions, and Induced Maps}
\label{linksanddeletions}

\vspace*{0.1in}

{\bf Intuition:} The link $\lk(\dowy,\gamma)$ of a set of attributes
$\gamma$ in the Dowker complex $\dowy$ can be understood as a
description of what may yet be observed or inferred, {\em conditional}
on having already observed $\gamma$.

\begin{lemma}\label{yLink}
Let $R$ be a relation on $\XxY$, with both $\hspt{}X\!$ and
$\hspt{}Y\!$ nonempty.
Suppose $\gamma\in\dowy$.
Define relation $Q$ as a restriction of $R$ by

\vspace*{-0.1in}

$$\qquad\qquad\quad
    Q \;=\; R\,|_{\sigma \times \tY}, \quad \hbox{with} \quad
            \sigma = \psi_R(\gamma) \quad \hbox{and} \quad
	    \tY = \bigunion_{x \in \sigma}\sdiff{Y_x}{\gamma}.$$

(See the comments below for the case in which $\,\tY=\emptyset$.)

\vspace*{0.1in}

Then \ $\lk(\dowy,\gamma) = \dowqy$, \ as collections of simplices
(i.e., ignoring underlying vertex sets).

\end{lemma}

\paragraph{Comments:}\  (a) Observe that $\sigma\neq\emptyset$.
\ (b) If $\,\tY=\emptyset$, then technically $Q$ is void, but it is
convenient to let both $\dowqy$ and $\dowqx$ be instances of the empty
complex $\{\emptyset\}$, as in Definition~\ref{linkgamma} on
page~\pageref{linkgamma}.
\ (c) In a standard link, one might define $\tY = \sdiff{Y}{\gamma}$.
$\phantom{R^{1^1}}\hspace*{-18pt}$
With $\tY\!$ as above, $Q$ always discards blank columns of $R$, even
when $\gamma=\emptyset$.

\begin{proof}
Observe that $\gamma \subseteq Y_x$ if and only if $x\in\sigma$.

\vst

We discuss the case $\tY=\emptyset$ separately, for clarity.  We need
to show that $\lk(\dowy,\gamma) = \{\emptyset\}$.
If $\lk(\dowy,\gamma) \neq \{\emptyset\}$, then there exists some
$\ybar\in\verts{\lk(\dowy,\gamma)}$.  By definition of link,
$\ybar\not\in\gamma$ and there exists $\xbar\in{X}$ such that
$(\xbar,y)\in{R}$ for all $y\in\gamma\union\{\ybar\}$.  That means
$\xbar\in\sigma$, so $\ybar\in\tY$, a contradiction.

\vspace*{0.025in}

The converse is true as well: If $\lk(\dowy,\gamma) = \{\emptyset\}$,
then $\tY=\emptyset$.  For if some $x\in\sigma$ has an attribute $y$ in
addition to all those in $\gamma$, then $y$ would be a vertex in the
link.

\vspace*{0.1in}

Now suppose $\tY\neq\emptyset$:

\vspace*{0.05in}

I. If $\xi\in\lk(\dowy,\gamma)$, then $\xi\inter\gamma=\emptyset$ and
there exists $x\in X$ such that $(x,y)\in R$ for every
$y\in\xi\union\gamma$.  So $\xi\subseteq{Y_x}\setminus\gamma$ and
$x\in\psi_R(\gamma)=\sigma$.  Thus $(x,y)\in Q$ for every $y\in\xi$,
meaning $\xi\in\dowqy$.

\vspace*{0.05in}

II. Conversely, if $\xi\in\dowqy$, then there exists $x\in\sigma$ such
that $(x,y)\in Q \subseteq R$ for every $y\in\xi$.  By definition of
$\sigma$, $(x,y)\in R$ for every $y\in\gamma$.  Combining these two
assertions, we see that $(x,y)\in R$ for every $y\in\xi\union\gamma$.
And $\xi\inter\gamma=\emptyset$, since $\xi\subseteq\tY$.  So
$\xi\in\lk(\dowy,\gamma)$.
\end{proof}

\paragraph{Additional Comment:}\ There is a dual version of this lemma
for links of individuals $\sigma$, modeling $\lk(\dowx, \sigma)$ by
$\dowqx$, for an appropriate relation $Q$.  We see instances of that
construction in Theorems~\ref{privacysingle} and
\ref{privacymultiple}, as well as in Lemma~\ref{interplocal}, on
pages~\pageref{privacysingleAppPage}--\pageref{localOperAppPage}
(previously stated on page~\pageref{privacysingle}), including the
case in which $\sigma$ consists of a single individual $x$.

\clearpage

\paragraph{Link Witness Formulas.}\ With notation and construction as
in Lemma~\ref{yLink}, the following formulas hold, assuming
$\tY\neq\emptyset$:

\vspace*{-0.05in}

\begin{itemize}
\item Suppose $\xi\subseteq\tY$ and define $\tau=\xi\union\gamma$.  Then
\label{psiLinkformulas}
$$\psi_Q(\xi) 
  \;=\; \biginter_{y\in\xi}(X_y\inter\sigma)
  \;=\; \Big(\biginter_{y\in\xi}X_y\Big) \;\biginter\; \Big(\biginter_{y\in\gamma}X_y\Big)
  \;=\; \biginter_{y\in(\xi\union\gamma)}X_y
  \;=\; \psi_R(\tau).$$

Notes: We allow $\xi = \emptyset$, since $\psi_Q(\emptyset) = \sigma =
\psi_R(\gamma)$.  We do not require $\xi\in\dowqy$.  The equalities
hold regardless.  Of course, $\xi\in\dowqy$ if and only if
$\psi_Q(\xi)\neq\emptyset$.

\item Suppose $\emptyset\neq\kappa\subseteq\sigma$.  Then
\label{phiLinkformulas}
$$\phi_Q(\kappa)
  \;=\; \biginter_{x\in\kappa}(Y_x\inter\tY)
  \;=\; \Big(\biginter_{x\in\kappa}Y_x\Big)\setminus\gamma
  \;=\; \phi_R(\kappa)\setminus\gamma.$$
And thus also $\phi_R(\kappa) = \phi_Q(\kappa) \union \gamma$, since
  $\gamma\subseteq Y_x$ for all $x\in\sigma$.

Notes: Here we do {\em not\,} allow $\kappa = \emptyset$, since
$\phi_Q(\emptyset) = \tY$ whereas $\phi_R(\emptyset) = Y$.  It need
not be true that $Y = \tY \union \gamma$. \ Again, $\kappa\in\dowqx$
if and only if
$\phi_Q(\kappa)\neq\emptyset$, this valid also for $\kappa=\emptyset$.
\end{itemize}

Comment: If $\tY=\emptyset$, the previous formulas still hold, albeit
trivially.  However, testing for membership in $\dowqx$ via the question
``Is $\phi_Q(\kappa)$ nonempty?'' no longer makes sense.

\vspace*{0.15in}

\begin{lemma}\label{yDel}
Let $R$ be a relation on $\XxY$, with both $\hspt{}X\!$ and
$\hspt{}Y\!$ nonempty.
Suppose $\gamma\subseteq Y$.
Then $\dl(\dowy,\gamma) = \dowqpy$, with $Q^\prime$ formed from $R$ by
removing the columns corresponding to $\gamma$, that is,
$Q^\prime \;=\; R\,|_{X \times (Y\setminus\gamma)}.$
\quad (Here we let $\dowqpx$ and $\dowqpy$ each be an empty complex if
$\gamma=Y$.)
\end{lemma}
\begin{proof}
An individual $x\in X$ is a witness to a set of attributes
$\xi\subseteq{Y}\setminus\gamma$ in $R$ if and only if $x$ is a
witness to $\xi$ in $Q^\prime$. \ (If $\gamma=Y$, then
$\dl(\dowy,\gamma) = \{\emptyset\} = \dowqpy$.)
\end{proof}

\paragraph{Deletion Witness Formulas.}\ With notation and construction
as in Lemma~\ref{yDel}, the following formulas hold, assuming
$\gamma\neq Y$:

\vspace*{-0.05in}

\begin{itemize}
\item If $\xi\subseteq(Y\setminus\gamma)$,  then $\psi_{Q^\prime}(\xi) =
  \biginter_{y\in\xi}X_y = \psi_R(\xi)$.
\label{delformulas}

\item If $\kappa\subseteq X$,  then
  $\phi_{Q^\prime}(\kappa) =
  \biginter_{x\in\kappa}(Y_x\setminus\gamma) =
  \phi_R(\kappa)\setminus\gamma$.

\vst

  Caution:  It need {\em not} be true that $\phi_R(\kappa) =
  \phi_{Q^\prime}(\kappa) \union \gamma$.

\vspace*{0.05in}

\end{itemize}

Comments: (1) The first formula holds for $\xi=\emptyset$ and
the second formula holds for $\kappa=\emptyset$.
(2) The simplex tests hold: For $\xi\subseteq(Y\setminus\gamma)$,
$\xi\in\dowqpy$ if and only if $\psi_{Q^\prime}(\xi)\neq\emptyset$;
and, for $\kappa\subseteq X$, $\kappa\in\dowqpx$ if and only if
$\phi_{Q^\prime}(\kappa)\neq\emptyset$.
\ (3) If $\gamma=Y$, the formulas still hold, but testing for membership
in $\dowqpx$ via the question ``Is $\phi_{Q^\prime}(\kappa)$ nonempty?''
no longer makes sense.

\subsection{Privacy Preservation in Links and Deletions}
\markright{Privacy Preservation in Links and Deletions}

\vspace*{0.05in}

{\bf Recall:}\ A relation $R$ preserves attribute privacy when the
closure operator $\clsy$ is the identity on $\dowy$ and it preserves
association privacy when the closure operator $\clsx$ is the identity
on $\dowx$ \,(see page~\pageref{defAttribPriv}).

\begin{lemma}\label{ClIdLinkDel}
Let $R$ be a relation on $\XxY$, with both $\hspt{}X\!$ and
$\hspt{}Y\!$ nonempty.  \ Suppose $\gamma\in\dowy$.

If $\hspt\clsy$ is the identity on $\dowy$, then the corresponding closure
operators for the relations modeling $\mskip2mu\lk(\dowy,\gamma)$ and
$\mskip2.5mu\dl(\dowy,\gamma)$ are also identities.

\vso

(The assertion for $\mskip2mu\dl(\dowy,\gamma)$ holds even if $\gamma$
is merely a subset of $Y\!$.)
\end{lemma}

Technical reminder: The operators are formally defined as self-maps on
the face posets of the simplicial complexes mentioned in the lemma,
but we can extend each operator to the empty simplex and therefore
think of it as a self-map on a simplicial complex viewed as a
collection of simplices. \ See again
pages~\pageref{dowker}--\pageref{AttribAssocPriv} and
page~\pageref{phipsiAppdef}.

\begin{proof}
Define $Q$ as in Lemma~\ref{yLink}.
That lemma tells us $\dowqy=\lk(\dowy,\gamma)$.

\vst

Given $\xi\in\dowqy$, let
$\tau=\xi\union\gamma$ and calculate:

\vspace*{-0.07in}

$$(\clsqy)(\xi)
     \;=\; \phi_Q(\psi_R(\tau))
     \;=\; \phi_R(\psi_R(\tau))\setminus\gamma
     \;=\; \tau\setminus\gamma
     \;=\; \xi.$$

\vspace*{0.05in}

Define $Q^\prime$ as in Lemma~\ref{yDel}.
That lemma tells us $\dowqpy=\dl(\dowy,\gamma)$.

Given $\xi\in\dowqpy$, calculate:

\vspace*{-0.07in}

$$(\phi_{Q^\prime} \circ \psi_{Q^\prime})(\xi)
     \;=\; \phi_{Q^\prime}(\psi_R(\xi))
     \;=\; \phi_R(\psi_R(\xi))\setminus\gamma
     \;=\; \xi\setminus\gamma
     \;=\; \xi.$$

\vspace*{-0.25in}
\end{proof}

\vspace*{0.1in}

Here is a variation, in which one again computes a link of attributes,
but then considers the closure operator on the dual association
complex, modeling individuals consistent with the attributes:

\begin{lemma}\label{CldualLink}
Let $R$ be a tight relation on $\XxY$, with both $\hspt{}X\!$
and $\hspt{}Y\!$ nonempty.
\,Let $\gamma\in\dowy$.

Define $Q$, $\sigma$, and $\,\tY\!$ as in the construction of
Lemma~\ref{yLink}.  \ Assume $\abs{\sigma} > 1$ and
$\,\tY\mskip-1mu\neq\emptyset$.

If $\clsx$ is the identity on $\dowx$, then $\clsqx$ is the identity on
$\dowqx$.
\end{lemma}

\begin{proof}
Suppose $\emptyset\neq\kappa\in\dowqx$.
Observe that $\gamma\subseteq\phi_R(\kappa)$ and calculate:

\vspace*{-0.1in}

$$(\clsqx)(\kappa)
   \;=\; \psi_Q(\phi_R(\kappa)\setminus\gamma)
   \;=\; \psi_R(\phi_R(\kappa))
   \;=\; \kappa.$$

Additionally,

\vspace*{-0.12in}

$$(\clsqx)(\emptyset)
   \;=\; \psi_Q(\tY)
   \;=\; \psi_R(\tY \union \gamma)
   \;=\; \psi_R\Big(\bigunion_{x\in\sigma}Y_x\Big) \;=\;$$

$$ \;=\; \biginter_{x\in\sigma}\psi_R(Y_x)
   \;=\; \biginter_{x\in\sigma}(\clsx)(\{x\})
   \;=\; \biginter_{x\in\sigma}\{x\}
   \;=\; \emptyset.$$

\vspace*{0.05in}

The last equality holds since $\abs{\sigma} > 1$.
\ The equality before that holds since $\clsx$ is the identity on
$\dowx$ and since $R\hspt$ has no blank rows.

So we see that $(\clsqx)(\kappa) = \kappa$ for all $\kappa\in\dowqx$.
\end{proof}

\noindent{\bf Comment:}\ Assume the setting of the previous two
lemmas, but suppose $\tY\!=\emptyset$.  We would then take $\dowqy$
and $\dowqx$ to be instances of the empty simplicial complex
$\{\emptyset\}$.  It is sensible to say that $\clsqy$ is the identity
on $\dowqy$, since
$\phi_Q(\psi_Q(\emptyset))=\phi_Q(\sigma)=\emptyset$.  It could be
confusing to say that $\clsqx$ is the identity on $\dowqx$, since
$\psi_Q(\phi_Q(\emptyset))=\psi_Q(\tY)=\psi_Q(\emptyset)=\sigma$.  On
the other hand, one could argue that one may nonetheless say that
there is no association inference in $Q$, since there are no
attributes that could witness associations.

\vspace*{0.2in}

\begin{corollary}\label{linkprivacy}
Let $R$ be a tight relation on $\XxY$, with both $\hspt{}X\!$
and $\hspt{}Y\!$ nonempty.
\,Let $\gamma\in\dowy$.

Define $Q$ and $\,\tY\!$ as in the construction of Lemma~\ref{yLink}.
\ Assume $\,\tY\mskip-2mu\neq\emptyset$.

If $R$ preserves both attribute and association privacy, then so does
$Q$.

\end{corollary}

\begin{proof}
Relation $Q$ preserves attribute privacy by Lemma~\ref{ClIdLinkDel}.
Let $\sigma = \psi_R(\gamma)$.  If we can show that $\abs{\sigma} > 1$,
then $Q$ preserves association privacy by Lemma~\ref{CldualLink}.

\vst

Observe that $\abs{\sigma} > 0$, since $\gamma\in\dowy$.  If
$\psi_R(\gamma)$ consists of a single individual $x\in X$, then

\vspace*{-0.075in}

$$\gamma
  \;=\; (\clsy)(\gamma)
  \;=\; \phi_R(\sigma)
  \;=\; Y_x
  \;=\; \tY \union \gamma.$$

\vspace*{0.02in}

That is impossible for nonempty $\tY$, since $\tY \inter \gamma = \emptyset$.
\end{proof}

\vspace*{0.3in}

The following lemma formalizes the intuition that a set of attributes
$\gamma$ implies another attribute $y$ precisely when the columns
corresponding to $\gamma$ have nonempty intersection and that
intersection is a subset of the column corresponding to $y$.

\begin{lemma}\label{privacycolumns}
Let $R$ be a relation on $\XxY$, with both $\hspt{}X\!$ and
$\hspt{}Y\!$ nonempty.

$R$ preserves attribute privacy if and only if the following condition
is true:

For all $\gamma\in\dowy$ and all $y\in Y$, if $\psi_R(\gamma)\subseteq
\psi_R(\ys)$ then $y\in\gamma$.
\end{lemma}

\begin{proof}
I. Suppose there exist $\gamma\in\dowy$ and $y\in Y$ such that
$\psi_R(\gamma)\subseteq \psi_R(\ys)$ but $y\not\in\gamma$.  Since
$\clsy$ is a closure operator, $y\in(\clsy)(\ys)$ and
$\gamma\subseteq(\clsy)(\gamma)$.  Now observe that $(\clsy)(\ys)
\subseteq (\clsy)(\gamma)$ by supposition and because $\phi_R$ is
inclusion-reversing.  Consequently, $(\clsy)(\gamma)$ must be a proper
superset of $\gamma$, telling us there is attribute inference.

\vspace*{0.1in}

II. If there is attribute inference, then for some $\gamma\in\dowy$,
$\gamma\subsetneq(\clsy)(\gamma)$.  Pick some
$y\in(\clsy)(\gamma)\setminus\gamma$.  Then $y\not\in\gamma$ but

\vspace*{-0.1in}

$$\psi_R(\gamma)
      \;=\;  \psi_R\big((\clsy)(\gamma)\big)
      \;\subseteq\;  \psi_R\big(\,(\clsy)(\gamma)\,\setminus\, \gamma\,\big)
      \;\subseteq\;  \psi_R(\ys).$$

(The equality holds by associativity of $\circ$ and the dual version
of Corollary~\ref{mappedsimplex} on page~\pageref{mappedsimplex}.  The
two subset relations hold by inclusion-reversal of $\psi_R$.)

\vspace*{0.1in}
\noindent (Technical comment: In both parts above, $\gamma=\emptyset$ is permissible.)
\end{proof}

\clearpage
\subsection{Unique Identifiability, Free Faces, and Privacy Preservation}
\markright{Unique Identifiability, Free Faces, and Privacy Preservation}
\label{uniqueident}

Recall the following definition:\label{uniqueIDAppPage}

\setcounter{currentThmCount}{\value{theorem}}
\setcounter{theorem}{\value{DefUniqueID}}
\begin{definition}[Unique Identifiability]
\ Let $R$ be a relation on $\XxY$ and suppose $x \in X$.\\
We say that $x$ is \,\mydefem{uniquely identifiable via relation $R$}\, when
$\,\psi_R(Y_x) = \{x\}$.
\end{definition}
\setcounter{theorem}{\value{currentThmCount}}

\paragraph{Comment:}\ It is entirely possible that one or more proper subsets
 $\gamma$ of $Y_x$ already {\em identifies\,} $x$, meaning
 $\psi_R(\gamma) = \{x\}$.  Certainly $x$ is uniquely identifiable in
 that case.  Moreover, the attributes $Y_x\setminus\gamma$ can be
 inferred from $\gamma$.

\vspace*{0.1in}

\begin{lemma}\label{noproperident}
Let $R$ be a relation on $\XxY$ that preserves attribute
privacy.  Let $x\in X$.  Then no proper subset of $Y_x$ identifies
$x$.
\end{lemma}

\begin{proof}
Suppose, for some $\gamma\subsetneq Y_x$, $\psi_R(\gamma) = \{x\}$.
\ We obtain a contradiction as follows:

$$\gamma
  \;\subsetneq\;  Y_x
  \;=\;  \phi_R(\{x\})
  \;=\;  (\clsy)(\gamma)
  \;=\;  \gamma.$$

\vspace*{-0.25in}
\end{proof}

\vspace*{0.25in}

We turn now to proving the assertions of Section~\ref{faceshape}
regarding free faces.

\begin{lemma}\label{nofreefacesPrivacy}
Let $R$ be a relation on $\XxY$, with both $\hspt{}X\!$ and
$\hspt{}Y\!$ nonempty.  If $\hspt\dowy\mskip-1mu$ contains no free
faces, then $R$ preserves attribute privacy.
\end{lemma}

\begin{proof}
We will show that $\clsy$ is the identity on $\dowy$.

\vst

To build intuition, we treat the empty simplex separately.
As usual, $(\clsy)(\emptyset) = \phi_R(X)$.  Therefore, we will show
that $\phi_R(X)=\emptyset$.  Observe that every maximal simplex of
$\dowy$ contains $\phi_R(X)$, since any witness for such a simplex must
have all the attributes in $\phi_R(X)$.
Pick some maximal simplex $\eta$ of $\dowy$ and consider $\gamma =
\eta \setminus \phi_R(X)$.  Let $\eta^\prime$ be any maximal
simplex of $\dowy$ containing $\gamma$.  Then

\vspace*{-0.15in}

$$\eta
  \;=\;  \gamma \union \phi_R(X)
  \;\subseteq\; \eta^\prime \union \phi_R(X)
  \;=\; \eta^\prime.$$

So $\eta=\eta^\prime$ by maximality.  Since $\dowy$ has no free faces,
$\gamma$ cannot be a proper subset of $\eta$, meaning
$\phi_R(X)=\emptyset$, as desired.

\vspace*{0.1in}

Now consider $\emptyset \neq \gamma \in \dowy$.  Suppose $\gamma$ is a
proper subset of $(\clsy)(\gamma)$.
By Corollary~\ref{maximal} and Lemma~\ref{propersubsets} on
pages~\pageref{maximal} and \pageref{propersubsets}, respectively, we can
assume without loss of generality that $\gamma = \eta\setminus\ys$ for
some maximal $\eta$ of $\dowy$ and some $y\in\eta$.  Observe that

\vspace*{-0.1in}

$$\sdiff{\eta}{\{y\}} = \gamma \subsetneq (\clsy)(\gamma) \subseteq
  (\clsy)(\eta) = \eta,$$

so $\eta = (\clsy)(\gamma)$.
\ Now let $\eta^\prime$ be any maximal simplex of $\dowy$ containing
$\gamma$.  Then

\vspace*{-0.1in}

$$\eta
  \;=\;  (\clsy)(\gamma)
  \;\subseteq\; (\clsy)(\eta^\prime)
  \;=\; \eta^\prime.$$

(Note: The last equality in each of the lines of comparisons above
follows from Corollary~\ref{maximal} by maximality.)

So $\eta=\eta^\prime$ by maximality.  That says $\gamma$ is a free
face of $\dowy$, a contradiction.
\end{proof}

\clearpage

The converse of Lemma~\ref{nofreefacesPrivacy} need not hold if there
exists an individual who can hide, with attributes that form a strict
subset of some other individual's attributes.  However:

\begin{lemma}\label{privacyNofreefaces}
Let $R$ be a relation on $\XxY$, with both $\hspt{}X\!$ and $\hspt{}Y\!$
nonempty.  If $R$ preserves attribute privacy and if every $x\in X$ is
uniquely identifiable via $R$, then $\dowy$ contains no free faces.
\end{lemma}

\begin{proof}
Suppose that $\gamma$ is a free face of $\dowy$.  We can assume without
loss of generality that $\gamma=\eta\setminus\ys$ for some maximal
$\eta\in\dowy$ and $y\in\eta$.  Since a Dowker attribute complex is
generated by the rows of the underlying relation, it must be that
$\eta=Y_x$ for at least one $x\in X$.  By Lemma~\ref{noproperident},
there is at least one $x^\prime$ besides $x$ in $\psi_R(\gamma)$.  Then

\vspace*{-0.1in}

$$\gamma
   \;=\; (\clsy)(\gamma)
   \;\subseteq\; \phi_R(\{x, x^\prime\})
   \;=\; Y_x \inter Y_{x^\prime}.$$

Since we have assumed that $\gamma$ is free and $Y_x$ is maximal, we
see that $Y_{x^\prime}$ must be a subset of $Y_x$.  That means
$x^\prime$ is not uniquely identifiable, a contradiction.

\vst

(Technical comment: $\gamma=\emptyset$ is permissible throughout this argument.)
\end{proof}

\vspace*{0.4in}

The following lemma will help us later in Appendix~\ref{privacyspheres},
to establish the assertions of Sections~\ref{faceshape} and
\ref{holemeaning} regarding relations that preserve both attribute and
association privacy:

\begin{lemma}\label{EqIdent}
Let $R$ be a relation on $\XxY$ such that $\abs{X} = \abs{Y} > 1$.
If $R$ has no blank columns and preserves attribute privacy, then
every $x\in X$ is uniquely identifiable via $R$.
\end{lemma}

\begin{proof}
The proof is by induction on $n=\abs{X}=\abs{Y}$.

\vspace*{0.1in}

I. The base case $n=2$ implies that $R$ is isomorphic to

$$\begin{array}{c|cc}
R   & y_1 & y_2 \\[2pt]\hline
x_1 & \one &       \\
x_2 &      & \one  \\
\end{array}$$

\vspace*{0.1in}

(Any other type of $2 \times 2$ relation without blank columns would
allow for attribute inference.)

\vspace*{0.1in}

Each $x_i$ is uniquely identifiable in $R$ above.

\vspace*{0.2in}

II.  For the induction step, assume that, for some $n > 2$, the lemma
holds for all relations with $X$ and $Y$ spaces of size strictly less
than $n$ (and bigger than 1).  We need to establish the lemma for all
relations with $X$ and $Y$ spaces of size $n$.

\paragraph{}$\phantom{0}$

\vspace*{-0.3in}

\ \begin{minipage}[t]{5.5in}\label{subclaim}
{\bf Subclaim:}\ $R$ has no blank rows.

\vspace*{0.1in}

To see this, suppose that $Y_{\xtild} = \emptyset$ for some $\xtild\in
X$.  Let $Q$ be the restriction of $R$ to $\Xp\mskip-0.7mu\times{Y}$,
with $\Xp = \sdiff{X}{\{\xtild\}}$.  There is no significant difference
between $R$ and $Q$; in particular, $Q$ also preserves attribute
privacy.

\vspace*{0.1in}

(Perhaps the empty simplex is slightly tricky: $(\clsqy)(\emptyset) =
\biginter_{x\in\Xp}Y_x$.  If this intersection is nonempty, it
contains some $y_1\in Y$.  Pick $y_2\in Y$ with $y_2 \neq y_1$;
this is possible since $\abs{Y} > 2$.  Note that $\emptyset \neq
X_{y_2} \subseteq X^\prime$, so
$y_1\in\biginter_{x\in\Xp}Y_x \subseteq \biginter_{x\in
  X_{\scriptstyle y_2}}Y_x = (\clsy)(\{y_2\}) = \{y_2\}$,
a contradiction.  So $(\clsqy)(\emptyset) = \emptyset$.)

\vspace*{0.1in}

Now let $Q^\prime$ be the further restriction of $R$ to
$\Xp\mskip-0.8mu\times\Yp$, where $\Yp = \sdiff{Y}{\{\ytild\}}$, with
$\ytild$ {\em any\,} attribute in $Y$.  By Lemma~\ref{ClIdLinkDel} on
page~\pageref{ClIdLinkDel}, $Q^\prime$ preserves attribute privacy.  The
underlying spaces $\Xp$ and $\Yp$ of $Q^\prime$ each have size $n-1$ and
$Q^\prime$ has no blank columns.  The induction hypothesis therefore
tells us that every individual in $\Xp$ is uniquely identifiable via
$Q^\prime$.  Bearing in mind that $\xtild$ does not appear in any $X_y$,
one sees that for each $x\in\Xp$, there is some $\gamma\subseteq\Yp$
such that $\biginter_{y\in\gamma}X_y = \{x\}$. That intersection is a
column vector all of whose entries are $0\phantom{\big|}$(blank) except
for the entry indexed by $x$.  Since $R$ preserves attribute privacy and
$x$ is arbitrary in $\Xp$, Lemma~\ref{privacycolumns} implies that in
fact $X_{\ytild}=\emptyset$, contradicting the assumption that $R$ has
no blank columns.
\end{minipage}

\vspace*{0.2in}

Next, pick $\xbar\in X$.  We will show that $\xbar$ is uniquely
identifiable via $R$.  Without loss of generality, write $R$ as in
Figure~\ref{Rblocks} (the figure indicates blank entries by ``$0$''s):

\begin{figure}[h]
\vspace*{-0.1in}
\begin{center} 
\ifig{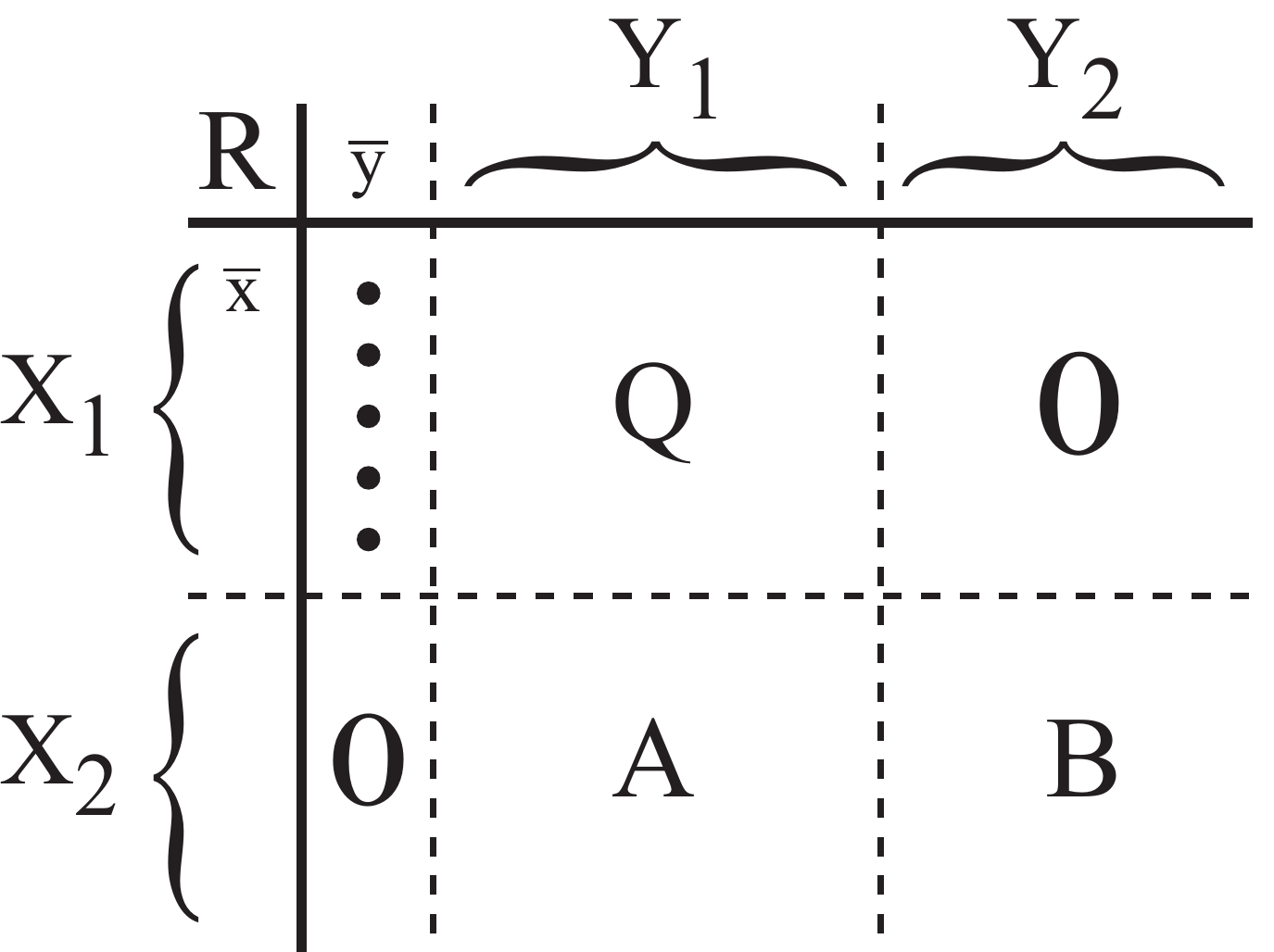}{scale=0.5}
\end{center}
\vspace*{-0.2in}
\caption[]{Relation $R$ decomposed into blocks for the proof of
  Lemma~\ref{EqIdent}, as described below.}
\label{Rblocks}
\end{figure}

\vst

Specifically, pick some $\ybar\in Y$ such that $(\xbar,\ybar)\in R$.
This is possible since $R$ has no blank rows.  Then decompose $X=X_1
\union X_2$, with $X_1=X_\ybar$ and $X_2 = X\setminus{X_1}$.  Since
$R$ preserves attribute privacy, $X_2 \neq \emptyset$.

Let $Q$ model $\lk(\dowy, \ybar)$.  So $Q$ is $R$ restricted to $X_1
\times Y_1$, with $Y_1 = \bigunion_{x\in{X_1}}Y_x\setminus\{\ybar\}$.
If $Y_1\neq\emptyset$, then $Q$ preserves attribute privacy, by
Lemma~\ref{ClIdLinkDel}, and $Q$ has no blank columns.

\vspace*{0.1in}

Now write $Y$ as the disjoint union $Y = \{\ybar\} \union Y_1 \union
Y_2$, with $Y_2 = Y\setminus(Y_1\union\{\ybar\})$. 

\clearpage

\paragraph{}$\phantom{0}$

\vspace*{-0.17in}

\label{uniqueIdentCases}

Observe that no individual in $X_2$ has attribute $\ybar$.  Observe
further that every individual in $X_1$ has attribute $\ybar$ but has no
attributes in $Y_2$, by construction.

\vso

Let $A$ be the restriction of $R$ to $X_2 \times Y_1$ and let $B$ be
the restriction of $R$ to $X_2 \times Y_2$.

\vso

If $Y_2\neq\emptyset$, then $B$ has no blank columns and $\dowby =
\dl(\dowy, Y_1 \union \{\ybar\})$.  If $\abs{Y_2}\geq 2$, then the
blank rows indexed by $X_1$ that remain after deleting from $R$ the
columns indexed by $Y_1 \union \{\ybar\}$ are irrelevant and so $B$
preserves attribute privacy (by Lemma~\ref{ClIdLinkDel} and by an
argument similar to that appearing in the proof of the Subclaim on
page~\pageref{subclaim}).

\vspace*{0.05in}

Let us look at some cases:

\vspace*{-0.05in}

\begin{itemize}
\item \underline{$\abs{Y_2} \geq \abs{X_2} = 1$}: \ Then any attribute
  of $Y_2$ identifies the one individual in $X_2$.  Since $R$ preserves
  attribute privacy, this implies both that $\abs{Y_2}=1$ and that
  relation $A$ is blank.  Consequently, every attribute in $Y_1$ implies
  $\ybar$ in $R$. Since $R$ preserves attribute privacy, we conclude
  that $Y_1=\emptyset$.  That means we are actually in the base case,
  with $n=2$.

\item \underline{$\abs{Y_2} > \abs{X_2} \geq 2$}: \ By removing some
  columns of $B$, we obtain a square relation to which we can apply
  the induction hypothesis.  That means every $x\in{X_2}$ is uniquely
  identifiable by the remaining columns.  Since $B$ preserves
  attribute privacy that means the columns removed must have been
  blank, a contradiction.

\item \underline{$\abs{Y_2} = \abs{X_2} \geq 2$}: \ We can apply the
  induction hypothesis directly to $B$.  That again tells us that
  every $x\in{X_2}$ is uniquely identifiable by $Y_2\,$-indexed columns,
  both in $B$ and in $R$.  We conclude that relation $A$ must be blank
  and so $Y_1=\emptyset$, arguing as above.  Thus $\abs{X_2} =
  \abs{Y_2} = n - 1$, implying $\abs{X_1} = 1$.  So $\ybar$ uniquely
  identifies $\xbar$, as desired.

\item \underline{$\abs{Y_2} < \abs{X_2}$}: \ This means $\abs{Y_1}
  \geq \abs{X_1}$.  Additionally, $\abs{X_1} \geq 2$, as otherwise
  $\ybar$ implies all the attributes $Y_1$.  If actually $\abs{Y_1}
  > \abs{X_1}$, then we could argue as above to see that some columns
  of $Q$ are blank, contrary to the construction of $Q$.  So we have
  that $n > \abs{Y_1} = \abs{X_1} \geq 2$ and the induction hypothesis
  applies.  Consequently, $\xbar$ is uniquely identifiable via $Q$, say
  as $\{\xbar\} = \psi_Q(\gamma)$, for some $\gamma \subseteq Y_1$.
  If we adjoin $\ybar$, we get that $\psi_R(\gamma \union \{\ybar\}) =
  \psi_Q(\gamma) = \{\xbar\}$, as desired.
\end{itemize}

\vspace*{-0.33in}

\end{proof}

\vspace*{0.05in}

\begin{theorem}[Too Many Attributes]\label{toomanyattrib}
Let $R$ be a relation on $\mskip1mu\XxY\mskip-2mu$ with no blank columns.

Suppose $\abs{Y} > \abs{X} \geq 1$.
\quad Then $R$ does not preserve attribute privacy.
\end{theorem}

\vspace*{-0.15in}

\begin{proof}
The proof is a corollary to Lemma~\ref{EqIdent}:

If $\abs{Y} > \abs{X} = 1$, then all attributes in
$\mskip1mu{}Y\mskip-1.6mu$ are inferable from nothing (in the context
of $R$).

Otherwise, suppose $R$ preserves attribute privacy.  We have $\abs{Y}
> \abs{X} > 1$, so we can delete some columns of $R$ and apply
Lemmas~\ref{ClIdLinkDel} and \ref{EqIdent} to a resulting square
relation.  Every individual in $X$ is therefore uniquely identifiable
via the columns retained.  Consequently, either there is attribute
inference in $R$ or the discarded columns were blank, a contradiction.
\end{proof}

\vspace*{0.025in}

\noindent{\bf Comment:}\ One implication of this result and those in
Appendix~\ref{privacyspheres} is the old detective show mantra
``eliminate suspects'': Reduce the number of relevant individuals
sufficiently, and some attribute inference is assured.  This amounts
to moving from relation $R$ to a subrelation $Q$ representing
$\dl(\dowx, \sigma)$, with $\sigma$ a {\em set}\, of ``eliminated
suspects''.

\vspace*{0.075in}

\noindent{\bf An additional conclusion:}\ The proof of
Lemma~\ref{EqIdent} suggests that perhaps $R$ also preserves
association privacy.  Indeed, we will see that to be true in
Appendix~\ref{symmetric}.

\clearpage
\section{Inference Hardness}
\markright{Inference Hardness}
\label{nphardness}

Recall the following definition:\label{individualprivacyAppPage}

\setcounter{currentThmCount}{\value{theorem}}
\setcounter{theorem}{\value{IndivPrivacyDef}}
\begin{definition}[Individual Privacy]
\ Let $R$ be a relation on $\XxY$ and suppose $x \in X$.

We say that \mydefem{$R$ preserves attribute privacy for $x$} whenever
$(\clsy)(\gamma) = \gamma$ for all $\gamma \subseteq Y_x$.
\end{definition}
\setcounter{theorem}{\value{currentThmCount}}

\vspace*{0.15in}

\noindent We have seen the following basic result within the proofs of other lemmas:

\begin{lemma}\label{privacyindiv}
Let $R$ be a relation on $\XxY$, with both $\hspc{}X\!$ and
$\hspt{}Y\!$ nonempty.   Let $x \in X$.
Then:

\vspace*{-0.15in}

\begin{center}
$R$ preserves attribute privacy for $x$

if and only if

$(\clsy)(\gamma) = \gamma$,\ for all $\gamma$ of the form $\gamma =
Y_x\setminus\ys$, with $y\in Y_x$.
\end{center}
\end{lemma}

\begin{proof}
I.  If $R$ preserves attribute privacy for $x$, then the condition is
satisfied by definition.

\vspace*{0.1in}

II. Suppose $R$ does not preserve attribute privacy for $x$.  Then for
some $\gamma\subseteq Y_x$, $\gamma\subsetneq(\clsy)(\gamma)$.  We know
$(\clsy)(Y_x) = Y_x$ by Corollary~\ref{generatorface} on
page~\pageref{generatorface}, so by Lemma~\ref{propersubsets} on
page~\pageref{propersubsets} we can assume that
$\gamma=Y_x\setminus\ys$, for some $y\in Y_x$.
\end{proof}

\vspace*{0.1in}

Lemma~\ref{privacyindiv} tells us that it is fairly easy to check
whether an individual's attribute privacy is preserved.  One merely
needs to check whether any one attribute is implied by all the
remaining attributes.  That may be done quickly since the maps
$\phi_R$ and $\psi_R$ amount to set intersections.  Harder is finding
a smallest set of attributes that implies another of the individual's
attributes.

\vspace*{0.1in}

Influenced by Lemma~\ref{privacycolumns} on
page~\pageref{privacycolumns}, we formulate the following problem:

\vspace*{-0.05in}

\begin{definition}[Minimal Inference]\label{mininf}

\ \underline{\MinInfnb} is the following decision problem:

\vst

Given relation $R$ on $\XxY$, $x\in X$, $y\in Y$, and $k \geq
0$, is there a simplex $\gamma\in\dowy$, with $\gamma\,\subseteq
Y_x\setminus\ys$, such that $\abs{\gamma} \leq k$ and $\psi_R(\gamma)
\subseteq \psi_R(\ys)$?
\end{definition}

\begin{lemma}\label{mininfhard}
\quad \MinInf is \np-complete
\end{lemma}

\begin{proof}
(A) Observe that the problem lies in \np:  Given some $\gamma$, one
  can verify the stated conditions in polynomial time.  The
  verifications amount to set intersection, cardinality, and subset
  computations, drawn from the columns and one row of $R$.

\vspace*{0.1in}

(B) We will establish \np-hardness by a reduction from {\sc Set
  Cover}.  Recall: Given a collection of sets $\{S_1, \ldots, S_m\}$,
  {\sc Set Cover} asks whether there is some subcollection of size at
  most $k$ such that the union of the subcollection is the overall
  union (often called the {\em universe}).

\vspace*{0.1in}

  Given an instance of the {\sc Set Cover} problem, we define the
  following relation:

\vspace*{-0.05in}

\begin{itemize}
\item $X \;=\; \{x_0\} \;\union\; \bigunion_{i=1}^mS_i$, with $x_0$ a
  new element distinct from any elements in the sets $S_i$.
\item $Y \;=\; \{0, 1, \ldots, m\}$.
\item $R\; \;= \big(\{x_0\} \times Y\big) \;\union\;
               \bigunion_{i=1}^m\setdefb{(x,i)\in \XxY}{x \in X\setminus{S_i}}.$
           
\end{itemize}

In words: The $0^{\hbox{\footnotesize th}}$ column of $R$ is the singleton set
$\{x_0\}$ and the $i^{\hbox{\footnotesize th}}$ column of $R$, for $i=1, \ldots, m$,
is $X\setminus{S_i}$, i.e., the complement of $S_i$ in the original
set cover universe, but now with $x_0$ added.  The row for $x_0$ has
entries for all possible attributes.  All other rows have no entry in
column 0.

\vspace*{0.1in}

\noindent {\bf Reduction:}\ Given an instance of {\sc Set Cover}, we
transform it into an instance of \MinInf using the relation
$R$ given above and by letting $x=x_0$ and $y=0$.  The parameter $k$
is the same for both problems.  Observe that $Y_x\setminus\ys = \{1,
\ldots, m\}$.

\vspace*{0.1in}

Observe further that $\abs{X} = \abs{\bigunion_{i=1}^mS_i} + 1 = n+1$
and $\abs{Y} = m + 1$, with $n$ the number of elements in the set cover
universe and $m$ the number of subsets specified for the set cover
problem.  The reduction can therefore be computed in polynomial time.

\vspace*{0.1in}

\noindent To avoid trivialities, we assume that $n > 0$ and $m > 0$.

\vspace*{0.1in}

\noindent To complete the proof, we will establish the following:

\vspace*{0.1in}

\noindent {\bf Claim:}\ The answer to {\sc Set Cover} is ``yes'' if
and only if the answer to \MinInf is ``yes''.

\vspace*{0.1in}

I. A ``yes'' answer to {\sc Set Cover} means that there is some set of
indices $\gamma\subseteq\{1, \ldots, m\}$, with $\abs{\gamma} \leq k$,
such that $\bigunion_{j\in\gamma}S_j = \bigunion_{i=1}^mS_i$.
Therefore, since $0\not\in\gamma$,

\vspace*{-0.15in}

$$\psi_R(\gamma)
  = \biginter_{j\in\gamma}X_j
  = \biginter_{j\in\gamma}(X\setminus{S_j})
  = X\setminus(\bigunion_{j\in\gamma}S_j)
  = X\setminus(\bigunion_{i=1}^mS_i)
  = \{x_0\}
  = \psi_R(\{0\})
  = \psi_R(\ys).$$

Consequently, $\emptyset\neq\psi_R(\gamma)\subseteq\psi_R(\ys)$
with $\gamma\subseteq{Y_x\setminus\ys}$ and $\abs{\gamma} \leq k$,
meaning that the answer to \MinInf is ``yes'' as well.

\vspace*{0.1in}

II.  A ``yes'' answer to \MinInf means there is some
$\gamma\subseteq\{1, \ldots, m\}$ such that $\abs{\gamma} \leq k$ and
$\emptyset\neq\psi_R(\gamma)\subseteq\psi_R(\ys)$.
\quad Observe that $\psi_R(\ys) = \psi_R(\{0\}) = \{x_0\}$ and that

\vspace*{-0.05in}

$$\psi_R(\gamma)
   \;=\; \biginter_{j\in\gamma}X_j
   \;=\; \biginter_{j\in\gamma}(X\setminus{S_j})
   \;=\; X\setminus(\bigunion_{j\in\gamma}S_j).$$

The middle equality holds as before because $0\not\in\gamma$.

So we see that $x_0 \in X\setminus(\bigunion_{j\in\gamma}S_j)
\subseteq \{x_0\}$, telling us

\vspace*{-0.05in}

$$\bigunion_{j\in\gamma}S_j \;=\; X\setminus\{x_0\} \;=\; \bigunion_{i=1}^mS_i.$$

That means $\gamma$ describes a set of indices sought for by {\sc Set
  Cover}, with $\abs{\gamma} \leq k$, so the answer to {\sc Set Cover} is
  also ``yes''.
\end{proof}

\clearpage
\section{Privacy Spheres}
\label{privacyspheres}

The aim of this appendix is to characterize privacy and inference in
terms of spheres.  Spheres exhibit homogeneity, which is good for
privacy, while still admitting a coordinate system for
identifiability.

We first prove a theorem characterizing individual attribute privacy,
then a generalization that holds for arbitrary elements of a
relation's doubly-labeled poset, and finally a characterization of
those relations that preserve both attribute and association privacy.

\subsection{Individual Attribute Privacy}
\markright{Individual Attribute Privacy}

\noindent We first state a lemma as a tool.  Recall also
Definitions~\ref{uniqueID} and \ref{individualprivacy} (see
pages~\pageref{uniqueIDAppPage} and \pageref{individualprivacyAppPage}).

\begin{lemma}\label{IdentPriv}
Let $R$ be a relation on $\XxY$, with both $\hspt{}X\!$ and
$\hspt{}Y\!$ nonempty.\\
Let $x\in X$ be uniquely identifiable via $R$.  \ Then:

$$\Big(\biginter_{y\in Y_x}X_y\Big)\setminus\{x\} \;=\; \emptyset.$$

\vspace*{0.1in}

Moreover, $R$ preserves attribute privacy for $x$ if and only if

$$\Big(\biginter_{y\in\gamma}X_y\Big)\setminus\{x\} \;\neq\; \emptyset,
\quad\hbox{for all $\;\gamma \subsetneq Y_x$}.$$

\end{lemma}

\begin{proof}
The first statement follows from the definition of unique
identifiability:  $\biginter_{y\in Y_x}X_y=\psi_R(Y_x) = \{x\}$.

\vspace*{0.1in}

For the second statement:

I.  Assume that $R$ preserves attribute privacy for $x$.
Let $\gamma\subsetneq{Y_x}$.  If
$\big(\biginter_{y\in\gamma}X_y\big)\setminus\{x\} = \emptyset$, then
$\psi_R(\gamma)=\biginter_{y\in\gamma}X_y = \{x\}$, since $x\in{X_y}$
whenever $y\in\gamma\subseteq{Y_x}$ (when $\gamma=\emptyset$, the
vacuous intersection is all of $X$, containing $x$).  That says a
proper subset of $Y_x$ identifies $x$, contradicting the proof of
Lemma~\ref{noproperident} on page~\pageref{noproperident}.

\vspace*{0.1in}

II.  Assume $\big(\biginter_{y\in\gamma}X_y\big)\setminus\{x\} \neq
\emptyset$ for all proper subsets $\gamma$ of $Y_x$.
If $R$ fails to preserve attribute privacy for $x$, then by
Lemma~\ref{privacyindiv} on page~\pageref{privacyindiv}
there is some $\gamma$ of the form
$Y_x\setminus\ys$, with $y\in{Y_x}$, such that
$\gamma\subsetneq(\clsy)(\gamma)=Y_x$.  Applying $\psi_R$ to both
sides of that last equality gives $\psi_R(\gamma)=\psi_R(Y_x)=\{x\}$,
by unique identifiability.  That is a contradiction, since
$\psi_R(\gamma)= \biginter_{y\in\gamma}X_y$.
\end{proof}

\vspace*{0.1in}

\noindent We now address our characterization of individual privacy,
proving a theorem stated previously:\label{privacysingleAppPage}

\setcounter{currentThmCount}{\value{theorem}}
\setcounter{theorem}{\value{localprivThm}}
\begin{theorem}[Individual Attribute Privacy]
\ Let $R$ be a relation on $\XxY$, with $\abs{X} > 1$.
Suppose $x\in X$ is uniquely identifiable via $R$.
\ Let $Q$ be the relation modeling $\lk(\dowx, x)$.\\
Then the following three conditions are equivalent:

\vspace*{0.125in}

\hspace*{1in}\begin{minipage}{4in}
\begin{enumerate}
\addtolength{\itemsep}{-1pt}
\item[(a)] $R$ preserves attribute privacy for $x$.
\item[(b)] $\Lk(\dowx, x) \;\homot\;\, \Skt$, with $k = \abs{Y_x}$.
\item[(c)] $\dowqy \;=\; \partial(Y_x)$.
\end{enumerate}
\end{minipage}

\end{theorem}
\setcounter{theorem}{\value{currentThmCount}}

\begin{proof}
The hypotheses ensure that $Y_x\neq\emptyset$ (and so also
$Y\!\neq\emptyset$).  They also ensure that $x$ is a vertex of $\dowx$,
so the link is not void.  It could be an empty complex
$\{\emptyset\}$, of course.

\vspace*{0.1in}

Observe that $Q$ is the restriction of $R$ to $\tX \times Y_x$, with
$\tX = \bigunion_{y \in Y_x}\sdiff{X_y}{\{x\}}$.

\vspace*{0.1in}

If $\tX=\emptyset$, then, reasoning as in the proof of
Lemma~\ref{yLink} on page~\pageref{yLink}, we see that $\lk(\dowx, x)
= \{\emptyset\} = \Smo$.  Furthermore, $x$ does not share any of its
attributes with any other individuals in $X$.  By convention, $\dowqy
= \{\emptyset\}$ as well.  If $k=\abs{Y_x}=1$, meaning $x$ has a
single attribute, then $R$ preserves attribute privacy for $x$, since
$\abs{X} > 1$.  Also, $\Skt=\Smo=\{\emptyset\}=\partial(Y_x)$.  So
conditions (a), (b), (c) all hold. If $k=\abs{Y_x}\geq 2$, then any
one attribute of $Y_x$ implies all the others, so condition (a) does not
hold.  Moreover, conditions (b) and (c) also do not hold.  In short,
the theorem holds when $\tX=\emptyset$.

\vspace*{0.1in}

We now assume that $\tX\neq\emptyset$.  We then know that
$\lk(\dowx,x)=\dowqx\homot\dowqy$ by a dual version of
Lemma~\ref{yLink} and by Dowker duality.  Definitionally,
$\partial(Y_x) \homot \Skt$, with $k = \abs{Y_x} > 0$.  We therefore
see that (c) implies (b).  To see that (b) implies (c), observe that
the underlying vertex set of $\dowqy$ is $Y_x$, so $\dowqy \homot
\Skt$ means $\dowqy = \partial(Y_x)$, since no proper subset of a
sphere can be homotopic to that same sphere.  To prove the theorem we
therefore only need to establish that conditions (a) and (c) are
equivalent.

\vspace*{0.15in}

Recall the formulas relating $\phi_Q$ and $\phi_R$ from
page~\pageref{phiLinkformulas} and dualize them here.  We see that:

\vspace*{-0.05in}

$$\psi_Q(\chi)
  \;=\; \psi_R(\chi)\setminus\{x\}
  \;=\; \Big(\biginter_{y\in\chi}X_y\Big)\setminus\{x\},
  \quad \hbox{for all $\emptyset\neq\chi\subseteq{Y_x}$}.$$

I. Assume that $R$ preserves attribute privacy for $x$.
By Lemma~\ref{IdentPriv} and the formula above we see that
$\psi_Q(\chi) \neq \emptyset$ for all nonempty proper subsets $\chi$
of $Y_x$ and that $\psi_Q(Y_x)=\emptyset$, since $x$ is uniquely
identifiable via $R$.  Consequently, $\dowqy$ contains every nonempty
proper subset of $Y_x$ as a simplex, but does not contain $Y_x$.
(Also, $\dowqy$ contains the empty simplex since the complex is not
void.)  Thus $\dowqy = \partial(Y_x)$.

\vspace*{0.2in}

II.  Assume that $\dowqy = \partial(Y_x)$.
Then $\psi_Q(\chi)\neq\emptyset$ for every nonempty proper subset
$\chi$ of $Y_x$.  By the formula above,
$\big(\biginter_{y\in\chi}X_y\big)\setminus\{x\} \;\neq\;
\emptyset$, for each such $\chi$.
Now suppose $\chi=\emptyset\subsetneq{Y_x}$.  Then:

\vspace*{-0.075in}

$$\emptyset
  \;\neq\; \tX 
  \;=\; \psi_Q(\emptyset)
  \;\subseteq\; X\setminus\{x\}
  \;=\; \Big(\biginter_{y\in\emptyset}X_y\Big)\setminus\{x\}.$$

So we see that $\;\big(\biginter_{y\in\chi}X_y\big)\setminus\{x\} \;\neq\;
\emptyset\;$ for every proper subset $\chi$ of $Y_x$, implying that $R$
preserves attribute privacy for $x$, by Lemma~\ref{IdentPriv}.
\end{proof}

Comment: It is impossible to satisfy the following three conditions
simultaneously:

\vso

(1) $x$ is uniquely identifiable,
(2) $\abs{Y_x} = 1$,
(3) $\tX\neq\emptyset$.

\clearpage
\subsection{Group Attribute Privacy}
\markright{Group Attribute Privacy}
\label{grppriv}

We now generalize the previous theorem to arbitrary elements $(\sigma,
\gamma)$ of the doubly-labeled poset $\PR$ associated with a relation
$R$.  We stated the generalized theorem previously in the report, as
Theorem~\ref{privacymultiple}, and replicate that below.  One may view
this generalized theorem as a characterization of the conditions under
which a set $\sigma$ of individuals (i.e., a {\em group} of
individuals, in the non-mathematical sense) has its attribute privacy
preserved, as a whole, not necessarily individually.
\ Theorem~\ref{privacysingle} is a special case of
Theorem~\ref{privacymultiple}, with the ``group'' a single individual
$x$, since $(\{x\}, Y_x) \in \PR$ whenever $x$ is uniquely
identifiable via $R$ and $Y_x \neq \emptyset$.

\vspace*{0.1in}

\setcounter{currentThmCount}{\value{theorem}}
\setcounter{theorem}{\value{multprivThm}}
\begin{theorem}[Group Attribute Privacy]
\ Let $R$ be a relation on $\XxY$.\\
Suppose $(\sigma, \gamma) \in \PR$, with $\sigma \neq X$.
\ Let $Q$ be the relation modeling $\lk(\dowx, \sigma)$.\\
Then the following three conditions are equivalent:

\vspace*{0.125in}

\hspace*{0.4in}\begin{minipage}{4in}
\begin{enumerate}
\addtolength{\itemsep}{-1pt}
\item[(a)] $(\clsy)(\gamma^\prime) \;=\; \gamma^\prime$,
   for every subset $\gamma^\prime$ of $\gamma$.
\item[(b)] $\Lk(\dowx, \sigma) \;\homot\; \Skt$, with $k = \abs{\gamma}$.
\item[(c)] $\dowqy \;=\; \partial(\gamma)$.
\end{enumerate}
\end{minipage}

\end{theorem}
\setcounter{theorem}{\value{currentThmCount}}

\vspace*{0.00in}

\begin{proof}
Reminder: Since $(\sigma, \gamma) \in \PR$,
\ $\emptyset\neq\sigma\in\dowx$, $\,\emptyset\neq\gamma\in\dowy$,
$\,\phi_R(\sigma) = \gamma$, and $\psi_R(\gamma) = \sigma$.

\vspace*{0.05in}

Thus also $(\clsy)(\gamma)=\gamma$, meaning we can focus on proper
subsets of $\gamma$ for part (a).

\vspace*{0.05in}

Recall also that $Q$ is the restriction of $R$ to $\tX \times \gamma$,
with $\tX = \bigunion_{y \in \gamma}\sdiff{X_y}{\sigma}$.

\vspace*{0.05in}

If $\tX=\emptyset$, then $\lk(\dowx, \sigma) = \{\emptyset\} = \Smo$.
By convention, $\dowqy = \{\emptyset\}$ as well.  If
$k=\abs{\gamma}=1$, then $\Skt=\Smo=\{\emptyset\}=\partial(\gamma)$.
The only proper subset of $\gamma$ in this case is
$\gamma^\prime=\emptyset$, and
$(\clsy)(\emptyset)=\phi_R(X)=\emptyset$.  (Reason: If
$y\in\phi_R(X)$, then $y\in\gamma$, so $\gamma=\ys$, implying
$\sigma=X$, which is disallowed.)  Thus conditions (a), (b), (c) all
hold. If $k = \abs{\gamma} \geq 2$, then conditions (b) and (c) cannot
hold.  Also, condition (a) does not hold since $(\clsy)(\ys)=\gamma$
for each $y\in\gamma$, bearing in mind that $\tX=\emptyset$ means
$X_y=\sigma$ for each $y\in\gamma$.  In short, the theorem holds when
$\tX=\emptyset$.

\vspace*{0.1in}

We now assume that $\tX\neq\emptyset$.  As in the proof of
Theorem~\ref{privacysingle}, we see readily that conditions (b) and (c)
are equivalent, so we will prove that conditions (a) and (c) are
equivalent.  And, as in the previous proof, dualizing a formula from
page~\pageref{phiLinkformulas} gives this formula:

\vspace*{-0.05in}

  $$\psi_Q(\chi)
    \;=\; \psi_R(\chi)\setminus\sigma,
    \quad \hbox{for all $\emptyset\neq\chi\subseteq{\gamma}$}.$$

\begin{itemize}

\item[I.] Assume that $(\phi_R \circ \psi_R)(\gamma^\prime) \;=\;
\gamma^\prime$, for every subset $\gamma^\prime$ of $\gamma$.\\
   We will establish that $\dowqy$ contains all proper subsets of
   $\gamma$ but not $\gamma$, telling us $\dowqy=\partial(\gamma)$.

   Since $\dowqy$ is not void, it contains the empty simplex.

   Pick some $\emptyset\neq\gamma^\prime\subsetneq\gamma$. \ 
   Since $(\phi_R \circ \psi_R)(\gamma^\prime) \;=\; \gamma^\prime$,
   $\psi_R(\gamma^\prime) \supsetneq \sigma$.

   The formula above therefore says
   $\psi_Q(\gamma^\prime)\neq\emptyset$, telling us $\gamma^\prime\in\dowqy$.

   Similarly, $\psi_Q(\gamma) \;=\; \psi_R(\gamma)\setminus\sigma
   \;=\; \sigma\setminus\sigma \;=\; \emptyset$,\ so $\gamma\not\in\dowqy$.

\vspace*{0.1in}

\item[II.]  Assume that $\dowqy = \partial(\gamma)$.

   Recall that $k = \abs{\gamma} > 0$.
   We look at two cases based on the value of $k$:

\begin{itemize}

   \item[$k=1$:] In this case, $\gamma = \ys$, for some $y\in Y$, so
   $\sigma = X_y$ and $\tX = \emptyset$, which we discussed above.

\vspace*{0.1in}

   \item[$k>1$:] Suppose, for the sake of contradiction, that
   $\gamma^\prime\subsetneq(\clsy)(\gamma^\prime)$, for some
   $\gamma^\prime \subsetneq \gamma$.  By Lemma~\ref{propersubsets} on
   page~\pageref{propersubsets}, we can assume $\gamma^\prime =
   \gamma\setminus\ys$, for some $y\in\gamma$.  Consequently,
   $(\clsy)(\gamma^\prime)=\gamma$, which implies
   $\psi_R(\gamma^\prime) = \sigma$.  The formula on the previous page
   then says $\psi_Q(\gamma^\prime)=\emptyset$, whereas
   $\gamma^\prime\in\dowqy$ means
   $\psi_Q(\gamma^\prime)\neq\emptyset$, a contradiction.

\end{itemize}
\end{itemize}

\vspace*{-0.25in}

\end{proof}

\vspace*{0.1in}

\paragraph{}$\phantom{0}$

\vspace*{-0.45in}

\label{localOperAppPage}

The following lemma, previously stated on page~\pageref{interplocal},
relates privacy preservation in a link to privacy preservation in the
encompassing relation.

\setcounter{currentThmCount}{\value{theorem}}
\setcounter{theorem}{\value{InterpLocalOperators}}
\begin{lemma}[Interpreting Local Operators]
\ Let $R$ be a relation on $\XxY$.

\vst

Suppose $(\sigma, \gamma) \in \PR$, with $\sigma \neq X$.

Let $Q$ be the relation on $\tX \times \gamma$ that models
$\lk(\dowx, \sigma)$ and suppose $\tX \neq \emptyset$.$\phantom{\Big|}$

\vspace*{0.05in}

Then, for every $\gamma^\prime \subseteq \gamma$:

\vspace*{-0.15in}

\hspace*{1.7in}\begin{minipage}{4in}
\begin{enumerate}

\item[(i)] If $\,\gamma^\prime \not\in\dowqy$,
              then $\psi_R(\gamma^\prime) = \sigma$.

\item[(ii)] If $\,\gamma^\prime \in\dowqy$,
              then $\psi_R(\gamma^\prime) \supsetneq \sigma$.

\vspace*{0.05in}

   Moreover, in this case:

For $\,\gamma^\prime=\emptyset$, \ $(\clsqy)(\emptyset) \supseteq
(\clsy)(\emptyset)$.

\vspace*{0.05in}

If $\,\gamma^\prime\neq\emptyset$, \,then
$(\clsqy)(\gamma^\prime) = (\clsy)(\gamma^\prime)$.

\end{enumerate}
\end{minipage}

\end{lemma}
\setcounter{theorem}{\value{currentThmCount}}

\vspace*{0.1in}

\begin{proof}

Observe that for every $\gamma^\prime \subseteq \gamma$, one has 
$\gamma^\prime \in \dowy$ and $\psi_R(\gamma^\prime) \supseteq
\psi_R(\gamma) = \sigma$.

\vspace*{0.1in}

By formulas from page~\pageref{phiLinkformulas} dualized,
\ $\psi_Q(\gamma^\prime) = \psi_R(\gamma^\prime)\setminus\sigma$
\;and\;
$\psi_R(\gamma^\prime) = \psi_Q(\gamma^\prime) \union \sigma$,
when $\emptyset\neq \gamma^\prime \subseteq \gamma$.

\vspace*{0.1in}

(i) Suppose $\gamma^\prime \not\in \dowqy$.  Then
$\gamma^\prime\neq\emptyset$, since $\emptyset\in\dowqy$.  Also,
$\psi_Q(\gamma^\prime)=\emptyset$, so by the second formula above,
$\psi_R(\gamma^\prime)=\sigma$.

\vspace*{0.1in}

(ii) Suppose $\gamma^\prime \in \dowqy$.
If $\gamma^\prime=\emptyset$, then $\psi_R(\emptyset) = X \supsetneq
\sigma$, by hypothesis.
If $\gamma^\prime\neq\emptyset$, then
$\psi_Q(\gamma^\prime)\neq\emptyset$, so by the formulas
above, $\psi_R(\gamma^\prime)\supsetneq\sigma$.

\vspace*{0.1in}

Turning to the ``Moreover'':

\vst

If $y\in(\clsy)(\emptyset)$, then $y$ is an attribute for all
individuals in $X$, so $y\in\gamma$ and
$y\in\phi_Q(\tX)=(\clsqy)(\emptyset)$.

\vspace*{0.1in}

Let $\emptyset\neq\gamma^\prime \!\in \dowqy$.
By another formula on page~\pageref{psiLinkformulas} dualized, if
$\kappa\subseteq\tX$, then $\phi_Q(\kappa) = \phi_R(\kappa\union\sigma)$.

\vspace*{0.1in}

Therefore, using the first formula above: $(\clsqy)(\gamma^\prime) =
\phi_Q(\psi_R(\gamma^\prime)\setminus\sigma)= (\clsy)(\gamma^\prime)$.
\end{proof}

{\bf Comment:} \ Also, $(\clsqy)(\emptyset) = \phi_Q(\tX) =
\phi_R\big(\bigunion_{y\in\gamma}X_y\,\big) = \biginter_{y\in\gamma}(\clsy)(\ys)$.

\clearpage
\subsection{Preserving Attribute and Association Privacy}
\markright{Preserving Attribute and Association Privacy}
\label{preservingboth}

In this subsection, we are interested in understanding relations that
preserve {\em \,both\,} attribute and association privacy.  We will
discover that this requirement is severely limiting.  As we
already see from Theorem~\ref{toomanyattrib} on
page~\pageref{toomanyattrib}, if $R$ is a nonvoid tight relation on
$\XxY$ that preserves both attribute and association privacy, then
$\abs{X}=\abs{Y}=n$.  \ What are the possibilities?

\begin{itemize}

\item[$n=0$:] Not relevant; this is a void relation.

\item[$n=1$:] Not possible; such a relation does not preserve privacy;
  one can infer the single individual or single attribute ``for free''
  (e.g., merely by knowing someone is covered by the relation).

\item[$n=2$:]  As we have seen before, such a relation must
be isomorphic to the following relation:\label{twoelementdiag}

\vspace*{-0.1in}

$$\begin{array}{c|cc}
R   & y_1 & y_2 \\[2pt]\hline
x_1 & \one &       \\
x_2 &      & \one  \\
\end{array}$$

Then both $\dowx$ and $\dowy$ are instances of the 0-sphere $\Szero$.

\item[$n\geq{3}$:]  Now there are several possibilities:

\begin{itemize}
\item The relation could be isomorphic to a {\em cyclic staircase relation}:

\vspace*{-0.1in}

$$\begin{array}{c|cccccc} R & y_1 & y_2 & \cdots & \cdots & y_{n-1} &
y_n \\[2pt]\hline x_1 & \one & \one & & & & \\ x_2 & & \one & \one & & &
\\
\vdots &   &      & \ddots & \ddots &         &      \\
\vdots &   &      &        & \ddots & \one    &      \\
x_{n-1} &  &      &        &        & \one    & \one \\
x_n & \one &      &        &        &         & \one \\
\end{array}$$

Then both $\dowx$ and $\dowy$ are homotopic to the 1-sphere $\Sone$.
Each is simply a linear cycle of edges, with vertices in one complex
dualizing to edges in the other.

\label{staircase}

\item The relation could be isomorphic to a {\em spherical boundary
relation}, in which every entry is present except that a diagonal is
blank.  For example, in the following relation all entries are present
except those of the form $(x_i, y_{n-i+1})$, $i=1, \ldots, n$:

\vspace*{-0.1in}

$$\begin{array}{c|cccccc}
R       & y_1    & y_2    & \cdots  & \cdots  & y_{n-1} & y_n    \\[2pt]\hline
x_1     & \one   & \one   & \one    & \cdots  & \one    &        \\
x_2     & \one   & \one   & \cdots  & \one    &         & \one   \\
\vdots  & \vdots & \vdots & \uddots &         & \one    & \vdots \\
\vdots  & \one   & \one   &         & \uddots & \vdots  & \one   \\
x_{n-1} & \one   &        & \one    & \cdots  & \one    & \one   \\
x_n     &        & \one   & \one    & \cdots  & \one    & \one   \\
\end{array}$$

Then $\dowx$ and $\dowy$ are each boundary complexes, namely
$\dowx=\partial(X)$ and $\dowy=\partial(Y)$.  Thus both are homotopic to
the $(n-2)$-sphere $\Snt$.

\item Finally, $R$ could have multiple components, each of which is
isomorphic to one of the following: A singleton, a cyclic staircase
relation, or a spherical boundary relation, all as above.  (Observe
that even though a nonblank $1 \times 1$ relation in and of itself
preserves no privacy, a relation containing a nonblank $1 \times 1$
subrelation can preserve privacy when that subrelation is one of
several components.)
\end{itemize}

(Comment: the staircase and spherical relations are isomorphic when $n=3$.)

\end{itemize}

The aim of this subsection is to prove that these are the only
possibilities.

\vspace*{0.2in}

\begin{lemma}\label{PrivacyLinks}
Let $R$ be a connected tight relation on $\XxY$, with $\abs{X} =
\abs{Y} \geq 3$, that preserves both attribute and association privacy.

Let $x\in X$ and define $Q$ to be the relation on $\tX \times Y_x$
that models $\lk(\dowx,x)$.

Then $\dowqx=\partial(\tX)$ and \,$\dowqy=\partial(Y_x)$, with
$\abs{\tX}=\abs{Y_x}$.
\end{lemma}

\begin{proof}
Observe that $Y_x\neq\emptyset$ since $R$ is tight.  Recall that
$\tX=\bigunion_{y \in Y_x}\sdiff{X_y}{\{x\}}$, which is nonempty
since $R$ is connected and $X$ contains not just $x$.

By Lemma~\ref{EqIdent} on page~\pageref{EqIdent}, $x$ is uniquely
identifiable via $R$, so Theorem~\ref{privacysingle} on
page~\pageref{privacysingleAppPage} says that $\dowqx\homot\Skt$ and
$\dowqy=\partial(Y_x)$, with $k=\abs{Y_x}$.  If we can show that
$\abs{\tX}=k$, then we can conclude that $\dowqx=\partial(\tX)$.
$\phantom{R^{1^{-}}}$(We also see that $k \geq 2$, since
$\tX\neq\emptyset$.)

\vso

The vertices of $\dowqx$ generate the maximal simplices of $\dowqy$.
In particular, there exist distinct $x_1, \ldots, x_k \in \tX$ such
that $\tY_{\!1}, \ldots, \tY_{\!k}$ are the maximal simplices of
$\dowqy$, with $\tY_{\!i} = Y_{x_{\scriptstyle i}} \inter Y_{x}$, and
$\abs{\tY_{\!i}}=k-1$, for $i=1, \ldots, k$.

\vst

Let $\xtild\in\tX$.  Then $Y_\xtild \inter Y_x \subseteq \tY_{\!i}
\subseteq Y_{x_{\scriptstyle i}}$, for some $i\in\{1, \ldots, k\}$.

\vso

That says $\emptyset\neq\phi_R(\{\xtild, x\}) \subseteq \phi_R(\{x_i\})$.

\vso

Since $R$ preserves association privacy, the dualization of
Lemma~\ref{privacycolumns} on page~\pageref{privacycolumns}
implies $\xtild=x_i$.
\quad Thus $\abs{\tX}=k$.
\end{proof}

Comment: Where did we use the assumption that each of $X$ and $Y$ has
at least three elements?  In fact, for much of the proof it is enough
to assume that $\abs{X} = \abs{Y} \geq 2$.  However, there is no
connected tight relation that preserves privacy when $\abs{X} =
\abs{Y} = 2$.

\vspace*{0.1in}

\begin{corollary}\label{notlinearcycle}
Let $R$ be a connected tight relation on $\XxY$, with $\abs{X} =
\abs{Y}$, that preserves both attribute and association privacy.

Let $y\in Y$ and suppose $\abs{X_y}\geq 4$.

Then $\lk(\dowy,y)$ is not a linear cycle.
\quad (In other words, the relation $Q$ that models $\lk(\dowy,y)$ is
\underline{{\em not}} isomorphic to a cyclic staircase relation.)
\end{corollary}

\begin{proof}
Arguing as in the proof of Lemma~\ref{PrivacyLinks}, now in dual form,
we see that $\lk(\dowy,y) \homot \Skt$, with $k=\abs{X_y}$.  \ Since
$k-2 \;\geq\; 2$, \ $\lk(\dowy,y)$ is not a linear cycle.
\end{proof}

\begin{corollary}\label{columnsizeequal}
Let $R$ be a connected tight relation on $\XxY$, with $\abs{X} =
\abs{Y} \geq 3$, that preserves both attribute and association privacy.

Suppose $\{x, x^\prime\}$, with $x\neq x^\prime$, is an edge (1-simplex) in $\dowx$.

Then $\abs{Y_x} = \abs{Y_{x^\prime}}$.
\end{corollary}

\begin{proof}

Let $k=\abs{Y_x}$ and $k^\prime=\abs{Y_{x^\prime}}$.

Observe that $x^\prime$ is a vertex of $\lk(\dowx,x)$ and $x$ is a
vertex of $\lk(\dowx,x^\prime)$.

By the proof of Lemma~\ref{PrivacyLinks}, each of $x^\prime$ and $x$
generates a maximal simplex in the attribute complex associated with the
other's link.  That simplex is ${Y_x \inter Y_{x^\prime}}$ in both
complexes.

So $k-1 = \abs{Y_x \inter Y_{x^\prime}} = k^\prime - 1$, hence $k=k^\prime$.
\end{proof}

\begin{corollary}\label{rowandcolsizesequal}
Let $R$ be a connected tight relation on $\XxY$, with $\abs{X} =
\abs{Y} \geq 3$, that preserves both attribute and association privacy.

Then all rows and columns have the same number of nonblank entries.
\end{corollary}

\begin{proof}
By Lemma~\ref{connectedcplx} on page~\pageref{connectedcplx} and
Corollary~\ref{columnsizeequal} above, all rows have the same number,
$k_r$, of nonblank entries.  Dualizing, one sees that all columns have
the same number, $k_c$\spc, of nonblank entries.  We claim that $k_c=k_r$.
This assertion follows from Lemma~\ref{PrivacyLinks} and its proof as follows:

Pick some $x\in X$ and let $Q$ be the relation modeling $\lk(\dowx,x)$.
By Lemma~\ref{PrivacyLinks}, $\dowqx$ and $\dowqy$ are each boundary
complexes, with $k_r = \abs{Y_x}$ vertices.  Moreover, each attribute
$y\in Y_x$ generates a maximal simplex $X_y \inter \tX$ in $\dowqx$,
which must have size $k_r - 1$.  The column $X_y$ contains one
additional individual, namely $x$. \quad
So $k_c = \abs{X_y} = (k_r -1 ) + 1 = k_r$.
\end{proof}

\begin{theorem}[Privacy as Sphere]\label{privbndrycmplx}
Let $R$ be a nonvoid connected tight relation on $\hspt{}\XxY\mskip-2mu$
that preserves both attribute and association privacy.

Then $\abs{X} = \abs{Y} \geq 3$ and $R$ is isomorphic to either a cyclic
staircase relation or a spherical boundary relation (each described on
page~\pageref{staircase}).
\end{theorem}

\begin{proof}
As we commented previously, Theorem~\ref{toomanyattrib} on
page~\pageref{toomanyattrib} implies that $\abs{X} = \abs{Y} = n$, for
some $n \geq 2$.  Connectedness further means that $n \geq 3$.

By Corollary~\ref{rowandcolsizesequal}, all rows and columns in $R$ have
the same number of nonblank entries.  In other words, $\abs{X_y} =
\abs{Y_x} = k$, for all $x\in X$ and all $y\in Y$, for some fixed $k$.
By connectedness, $k \geq 2$.

By Lemma~\ref{EqIdent} on page~\pageref{EqIdent}, each $x\in X$ is
uniquely identifiable via $R$.  Dualized, each $y\in Y$ is uniquely
identifiable via $R$ as well.

If $k=2$, then $\dowx$ and $\dowy$ contain vertices and edges but no
higher-dimensional simplices.  By duality, each vertex therefore has at
most two incident edges.  By unique identifiability, each vertex has
exactly two incident edges.  Thus, by connectedness, each complex is a
linear cycle.  So $R$ is isomorphic to a cyclic staircase relation.

Now assume that $k \geq 3$.

Pick a $\ybar\in Y$ and consider the decomposition of
Figure~\ref{Rblocksy}, similar to the one we saw in the proof of
Lemma~\ref{EqIdent}.

\begin{figure}[h]
\vspace*{-0.1in}
\begin{center} 
\ifig{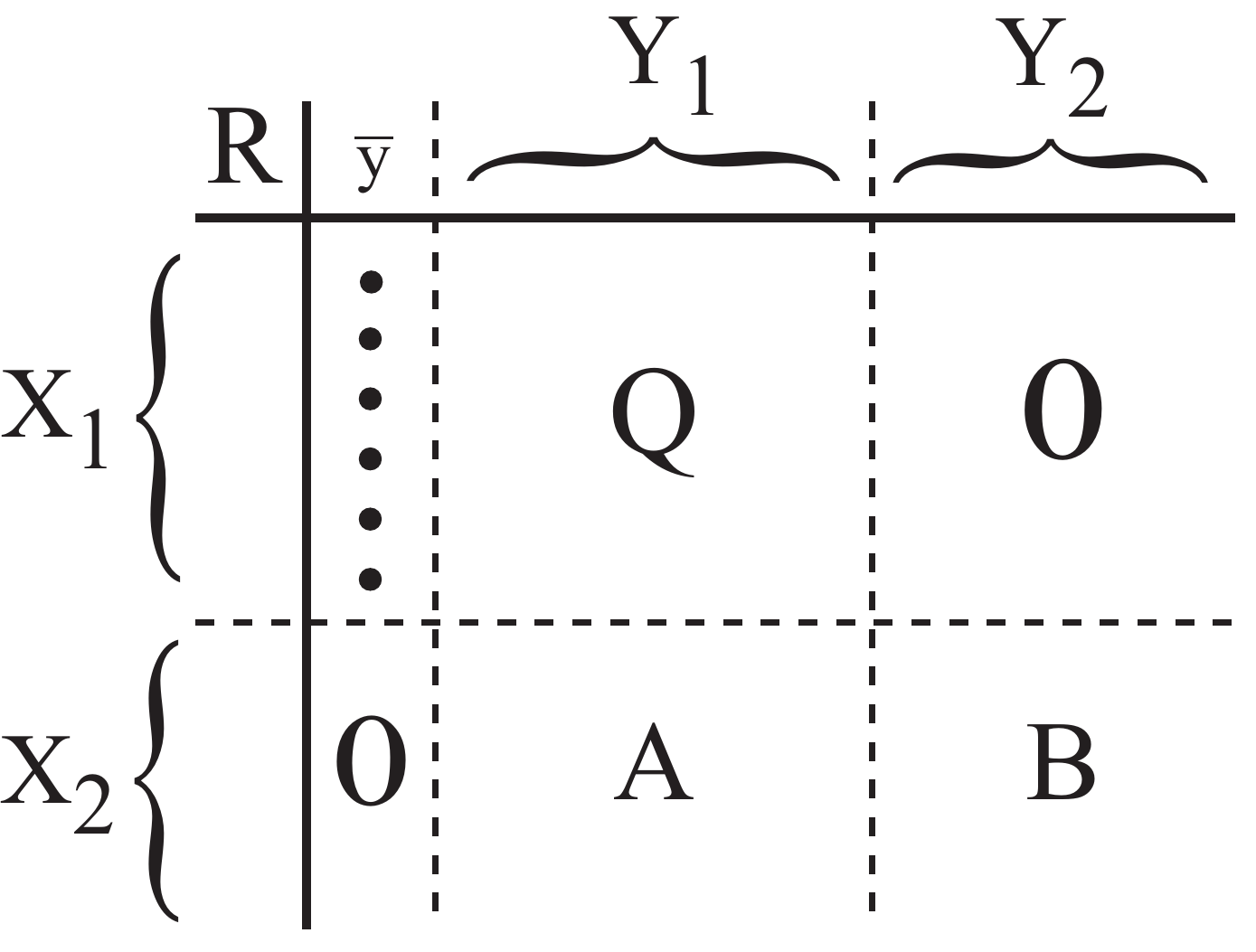}{scale=0.5}\hbox{\phantom{000000}}
\end{center}
\vspace*{-0.2in}
\caption[]{Relation $R$ decomposed into blocks for the proof of Theorem~\ref{privbndrycmplx}.}
\label{Rblocksy}
\end{figure}

Let $X_1=X_\ybar$ and write $X=X_1 \union X_2$ with $X_2 =
X\setminus{X_1}$.  $X_1\neq\emptyset$ since every column of $R$ has
$k$ nonblank entries and $X_2\neq\emptyset$ since $R$ preserves
attribute privacy.

Let $Q$ model $\lk(\dowy, \ybar)$.  So $Q$ is $R$ restricted to $X_1
\times Y_1$, with $Y_1 = \bigunion_{x\in{X_1}}Y_x\setminus\{\ybar\}$.
$Y_1\neq\emptyset$ because every row of $R$ has $k$ nonblank entries.
In particular, there are exactly $k-1$ entries in each row of $Q$, so at
least two entries in each row.

Now write $Y$ as the disjoint union $Y = \{\ybar\} \union Y_1 \union
Y_2$, with $Y_2 = Y\setminus(Y_1\union\{\ybar\})$.  Observe that every
individual in $X_1$ has attribute $\ybar$ but has no attributes in $Y_2$,
by construction.

By the dual to Lemma~\ref{PrivacyLinks}, we know that
$\dowqx=\partial(X_1)$ and $\dowqy=\partial(Y_1)$, with
$k=\abs{X_1}=\abs{Y_1}$.  Therefore, for each each $y\in Y_1$, column
$X_y$ of $R$ has $k-1$ entries that lie in $X_1$ and one entry that lies
in $X_2$.  We claim that the $X_2$ entry is the same across all columns
$X_y$ as $y$ varies over $Y_1$.  For otherwise, at least two such
columns would have an intersection (nonempty, since $k-2\geq 1$)
contained wholly within $X_\ybar$, implying that $R$ permits attribute
inference after all, by Lemma~\ref{privacycolumns} on
page~\pageref{privacycolumns}.  Call that common individual $\xbar$.
Observe that $Y_\xbar = Y_1$ since every row of $R$ has exactly $k$
attributes.  Consequently, the block diagram for $R$ becomes as in
Figure~\ref{Rblocksz}. \ (The figures now indicate blank entries either
by blanks or by explicit ``$0$''s.)

\begin{figure}[h]
\begin{center} 
\ifig{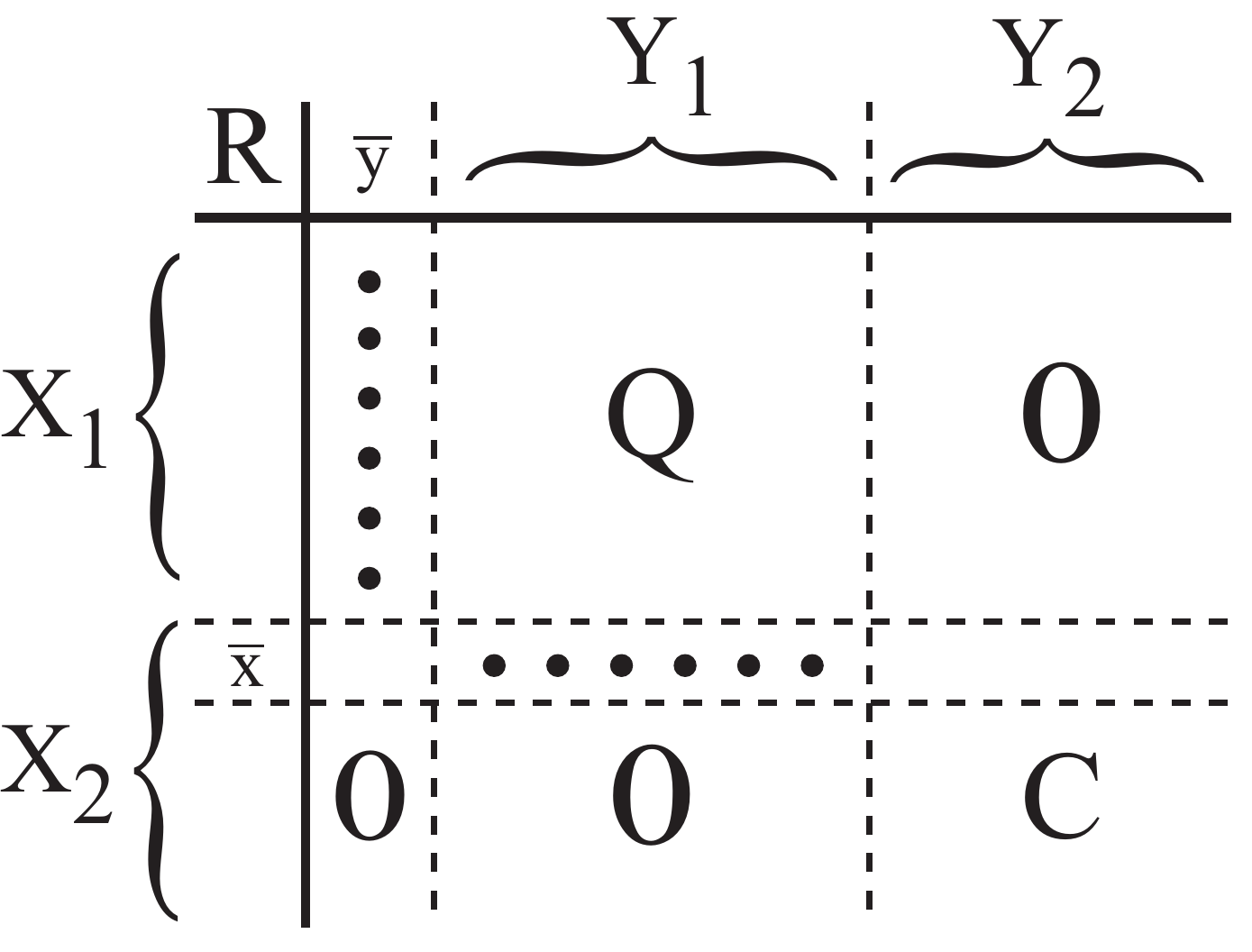}{scale=0.5}\hbox{\phantom{000000}}
\end{center}
\vspace*{-0.2in}
\caption[]{Relation $R$ decomposed further.}
\label{Rblocksz}
\end{figure}

Observe that no individual of $X_1 \union \{\xbar\}$ has any
attributes in $Y_2$ and that no individual of $X_2\setminus\{\xbar\}$
has any attributes in $Y_1 \union \{\ybar\}$, by the row and column
cardinality constraints.  That means relation $C$, which is the
restriction of $R$ to $(X_2\setminus\{\xbar\}) \times Y_2$, would be
disconnected from the rest of $R$, if $C$ were to exist.  We conclude
that $Y_2 = \emptyset$ and that $X_2 = \{\xbar\}$.  Thus, finally, $R$
must decompose as in Figure~\ref{Rblocksdiag}.  As we have seen, $Q$
is nearly a full relation, missing only a diagonal.  We now see that
$R$ is also nearly a full relation, missing only a diagonal.  Thus
$\dowx = \partial(X)$ and $\dowy = \partial(Y)$, meaning $R$ is
isomorphic to a spherical boundary relation, as claimed.

\begin{figure}[h]
\begin{center} 
\ifig{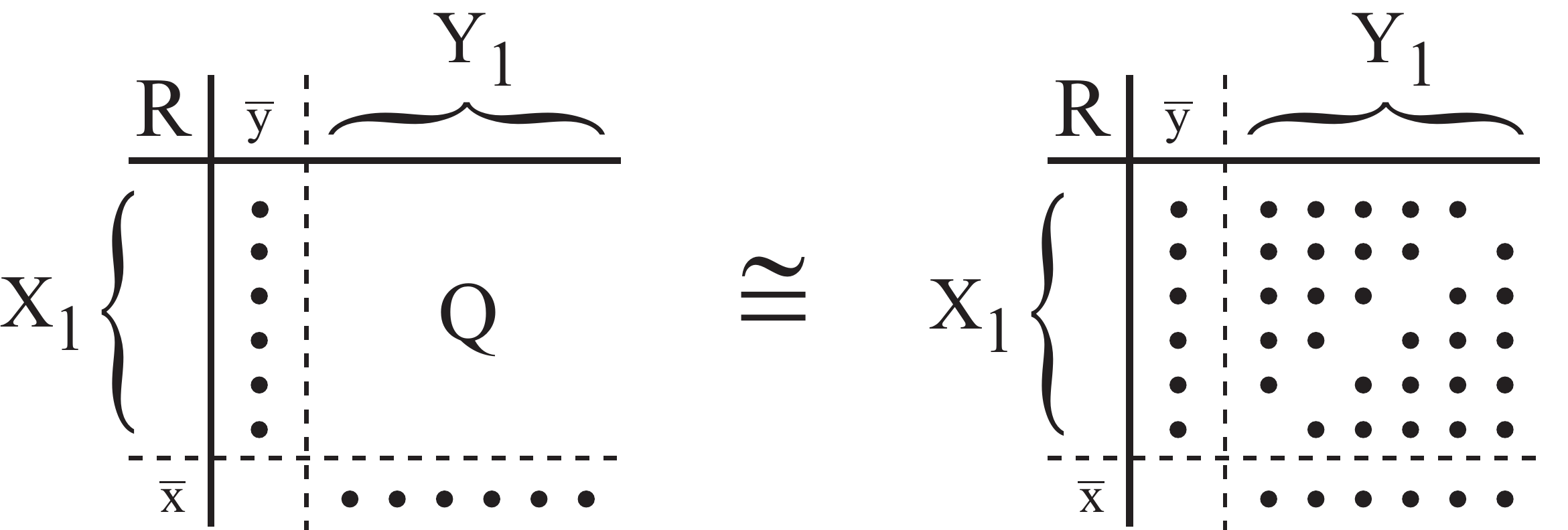}{scale=0.6}
\end{center}
\caption[]{Relation $R$ decomposes diagonally.}
\label{Rblocksdiag}
\end{figure}

\vspace*{-0.25in}

\end{proof}

\vspace*{1in}

\begin{corollary}\label{privbndrycmplxcomponents}
Let $R$ be a nonvoid tight relation that preserves both attribute and
association privacy.
Decompose $R$ into its connected components as $R = R_1 \union \cdots
\union R_\ell$, with each $\hspt{}R_i$ a nonvoid tight relation on
$X_i \times Y_i$, as per the proof of Lemma~\ref{components} on
page~\pageref{components}.  \ Then, for each $i \in \{1, \ldots,
\ell\}$,\, $R_i$ is isomorphic to a singleton or a cyclic staircase
relation or a spherical boundary relation, and $\abs{X_i}=\abs{Y_i}$.
\end{corollary}

Comment: When $\ell=2$ and each of $R_1$ and $R_2$ is a singleton, then
the Dowker complexes of $R$ itself, $\dowx$ and $\dowy$, are each
an instance of $\Szero$.

\begin{proof}
Consider $R_i$, for some $i \in \{1, \ldots, \ell\}$.

Suppose that $X_i\in\Psi_{R_{\scriptstyle i}}$.  Then some attribute
$y\in Y_i$ is shared by all individuals in $X_i$.  If there were any
other attributes in $Y_i$, then each of those would individually imply
$y$ in $R$.  Since $R$ preserves attribute privacy, $\abs{Y_i} = 1$.
Consequently, since $R$ also preserves association privacy, $\abs{X_i}
= 1$, so $R_i$ is a singleton.

If $R_i$ is not a singleton, then
$X_i\not\in\Psi_{R_{\scriptstyle i}}$ and similarly
$Y_i\not\in\Phi_{R_{\scriptstyle i}}$.

Consequently, Lemma~\ref{components} and
Corollary~\ref{componentprivacy} on page~\pageref{components}
tell us that $R_i$ is a nonvoid connected tight relation that
preserves both attribute and association privacy.
Theorem~\ref{privbndrycmplx} completes the proof.
\end{proof}

\paragraph{Comment:} \ The development leading to
Corollary~\ref{privbndrycmplxcomponents} used the language of
relations, privacy, and inference as proof tools, in part to build
intuition.  One can take an alternate, more directly simplicial and
combinatorial approach.  For instance, by counting vertices, maximal
simplices, and free faces that are just one vertex shy of being
maximal simplices, one can obtain an alternate proof of
Theorem~\ref{toomanyattrib} on page~\pageref{toomanyattrib}.

\clearpage
\subsection{Square Relations Preserve Privacy Symmetrically}
\markright{Square Relations Preserve Privacy Symmetrically}
\label{symmetric}

At the end of Appendix~\ref{uniqueident}, we observed that one could
perhaps strengthen the conclusions of Lemma~\ref{EqIdent} on
page~\pageref{EqIdent}.  According to the lemma, if a square relation
with no blank columns preserves attribute privacy, then each
individual is uniquely identifiable via the relation.  The proof of
the lemma further established that the relation necessarily has no
blank rows.  In fact, we will now prove that the relation also
preserves association privacy.

\vspace*{-0.05in}

\paragraph{Summary:}\ By Theorem~\ref{privbndrycmplx} and
Corollary~\ref{privbndrycmplxcomponents}, any nonvoid tight relation
preserving both attribute and association privacy must be a square
relation whose components are isomorphic to singletons, cyclic
staircase relations, or spherical boundary relations.  Complementing
this statement, by upcoming Theorem~\ref{squarerelations} and its dual
form, any nonvoid tight square relation must either preserve {\em
both\,} attribute and association privacy or fail to preserve both,
that is, allow some attribute inference {\em and\,} some association
inference.

\vspace*{0.15in}

\noindent We now state the theorem, but will need to develop some
tools before proving it.

\begin{theorem}[Privacy in Square Relations]\label{squarerelations}
Let $R$ be a relation on $\XxY$ with $\abs{X} = \abs{Y} > 1$.

\vso

If $R$ has no blank columns and preserves attribute privacy,
then these three conditions hold:

\vspace*{0.1in}

\hspace*{0.4in}\begin{minipage}{3.5in}
\begin{enumerate}
\addtolength{\itemsep}{-2pt}
\item[(i)] $R$ has no blank rows.
\item[(ii)] Every $x\in X$ is uniquely identifiable via $R$.
\item[(iii)] $R$ preserves association privacy.
\end{enumerate}
\end{minipage}
\end{theorem}

\vspace*{0.2in}

\noindent We now develop the tools:

\vspace*{0.1in}

\begin{definition}[Individuals with Maximal Attributes]\label{Qmax}
\ Let $R$ be a relation on $\XxY$.
\ The \mydefem{restriction of $R$ to its maximally attributed individuals}
is the relation $\,\Qmax = R\,|_{\,\tX\times\tY}$,\, with

\vspace*{-0.1in}

\begin{eqnarray*}
  \kmax &\;=\;& \max_{x\in X}\,\abs{Y_x},\\[3pt]
\tX &\;=\;& \setdefbV{x\in X}{\abs{Y_x}=\kmax},\\[4pt]
\hbox{\em and}\quad \tY &\;=\;& \bigunion_{x\in\tX}Y_x.
\end{eqnarray*}
\end{definition}

\begin{lemma}[Privacy Preservation in $\Qmax$]\label{Qmaxprivacy}
Let $R$ be a relation on $\XxY$, with $\abs{X} \geq
\abs{Y} > 1$.

Suppose that $R$ is tight, that $R$ preserves attribute privacy, and
that every $x\in X$ is uniquely identifiable via $R$.
\quad Let $\kmax$ and $\Qmax$ be as in Definition~\ref{Qmax}.

Then $\kmax \geq 1$, $\Qmax$ is tight, $\Qmax$ preserves attribute
privacy, and every individual $\xbar\in\tX$ is uniquely identifiable
via $\Qmax$.
\end{lemma}

\begin{proof}
Since neither $X$ nor $Y$ is empty and since $R$ is tight, $\kmax \geq
1$.  Consequently, neither $\tX$ nor $\tY$ is empty in the definition
of $\Qmax$, from which it follows that $\Qmax$ is tight by
construction.

Suppose $\kmax=1$.  Since $R$ has no blank rows, every individual in
$X$ has a single attribute in $Y$.  By unique identifiability,
distinct individuals have distinct attributes.  Consequently,
$\abs{X}=\abs{Y}$.  So $R$ is isomorphic to a square diagonal
relation, and $\Qmax=R$.  The lemma's assertions therefore hold.

\vst

Henceforth, assume that $\kmax > 1$.  \quad
Let $\xbar\in\tX \subseteq X$.

\vst

Observe that $\xbar$ is uniquely identifiable via $\Qmax$, since
$\xbar$ is uniquely identifiable via $R$, $Y_{\xbar} \subseteq
\tY$, and
$\psimax(Y_{\xbar}) \;=\; \psi_R(Y_{\xbar}) \inter \tX
  \;=\; \{\xbar\} \inter \tX \;=\; \{\xbar\}.\phantom{\Big|}$

\vspace*{0.05in}

By assumption, $R$ preserves attribute privacy, every $x\in X$ is
uniquely identifiable via $R$, and $\abs{X} > 1$.  Consequently,
Theorem~\ref{privacysingle} on page~\pageref{privacysingleAppPage}
says that $\dowqy=\bndry{(Y_{\xbar})}$, with $Q$ modeling $\lk(\dowx,
\xbar)$.  As in the proof of Lemma~\ref{PrivacyLinks} on
page~\pageref{PrivacyLinks}, this means there exist distinct
vertices $x_1, \ldots, x_\kmax$ in $\dowqx$ such that $\tY_{\!1},
\ldots, \tY_{\!\kmax}$ are the maximal simplices of $\dowqy$, with
$\tY_{\!i} = Y_{x_{\scriptstyle i}} \inter Y_{\,\xbar}$, and
$\abs{\tY_{\!i}}=\kmax-1$, for $i=1, \ldots, \kmax.$ \phantom{\Big|}
(Aside: Here, $\dowqx$ could contain additional vertices.)

\vst

Since each $x_i$ is uniquely identifiable via $R$, $Y_{x_{\scriptstyle
i}} \not\subseteq Y_{\xbar}$.  Bearing in mind the definition of
$\kmax$, this means each $Y_{x_{\scriptstyle i}}$ contains exactly one
attribute in $\sdiff{Y}{Y_{\xbar}}$.  Consequently,
$\abs{Y_{x_{\scriptstyle i}}}=\kmax$ and each $x_i$ is an individual
in $\tX$.

\vsr

We therefore see that each $x_i$ is a vertex as well of $\dowqpx$ and
that $\dowqpy=\bndry{(Y_{\xbar})}$, with $Q^\prime$ now modeling
$\lk(\dowqmx, \xbar)$.  We also see that $\tX$ must contain at least
$\kmax+1$ individuals, so $\abs{\tX} > 1$.
Theorem~\ref{privacysingle} then says that $\Qmax$ preserves
attribute privacy for $\xbar$.  Since $\xbar$ is arbitrary in $\tX$,
that means $\Qmax$ preserves attribute privacy generally.
\end{proof}

\vsr

\begin{lemma}[Square Uniform Relations]\label{uniform}
Let $R$ be a relation on $\XxY$, with $\abs{X} = \abs{Y} > 1$.

Suppose that $R$ is tight and that $R$ preserves attribute privacy.

Suppose further that every row has exactly $k$ nonblank entries, with
$k\geq 1$.

Then every column has exactly $k$ nonblank entries.
\end{lemma}

\begin{proof}
By Lemma~\ref{EqIdent} on page~\pageref{EqIdent}, every $x\in X$ is
uniquely identifiable via $R$.

\vst

We first claim that $\abs{X_y} \geq k$ for every $y\in Y$.  To see
this, let $y\in Y$ be arbitrary.  Pick some $x \in X$ such that
$(x,y)\in R$.  Such an $x$ exists since $R$ has no blank columns.
Since $R$ preserves attribute privacy for $x$, we can argue as in the
proof of Lemma~\ref{Qmaxprivacy}, concluding that at least $k-1$ other
individuals in $X$ must share attribute $y$ with $x$.  So $\abs{X_y}
\geq k$.

\vst

Counting the total number of nonblank entries in $R$ in two ways, we
obtain:

$$n\,k 
  \;=\; \sum_{x\in X}\,\abs{Y_x}
  \;=\; \sum_{y\in Y}\,\abs{X_y}
  \;\geq\; \sum_{y\in Y}\,k
  \;=\; n\,k, \qquad \hbox{with $n=\abs{X}=\abs{Y}$}.$$

Thus $\abs{X_y}=k$ for every $y\in Y$.
\end{proof}

\vsr

\begin{corollary}\label{kone}
Assume the hypotheses of Lemma~\ref{uniform} and that $k=1$.

Then $R$ is isomorphic to a square diagonal relation.
\end{corollary}

\begin{proof}
As we saw in the proof of Lemma~\ref{uniform}, every $x\in X$ is
uniquely identifiable via $R$.  The argument in the proof of
Lemma~\ref{Qmaxprivacy}, for $\kmax=1$, therefore establishes this
corollary.
\end{proof}

\begin{corollary}\label{ktwo}
Assume the hypotheses of Lemma~\ref{uniform}.

Suppose further that $R$ is connected and that $k=2$.

Then $R$ is isomorphic to a cyclic staircase relation.
\end{corollary}

\begin{proof}
As we saw in the proof of Lemma~\ref{uniform}, every $x\in X$ is
uniquely identifiable via $R$.

Observe as well that, in a dual sense, each $y\in Y$ is uniquely
identifiable via $R$, since $R$ has no blank columns and preserves
attribute privacy: \ $\phi_R(X_y) = (\clsy)(\ys) = \ys$.

By Lemma~\ref{uniform}, all rows and columns of $R$ have exactly two
nonblank entries.

Consequently, the argument in the proof of
Theorem~\ref{privbndrycmplx} on page~\pageref{privbndrycmplx}, for
$k=2$, establishes this corollary.
\end{proof}

\vst

\begin{corollary}\label{kthree}
Assume the hypotheses of Lemma~\ref{uniform}.

Suppose further that $R$ is connected and that $k\geq 3$.
\ Let $n=\abs{X}=\abs{Y}$.

Then $k=n-1$ and $R$ is isomorphic to a spherical boundary relation.
\end{corollary}

\begin{proof}
As we saw in the proof of Lemma~\ref{uniform}, every $x\in X$ is
uniquely identifiable via $R$.

Pick some such $x$ and let $Q$ model $\lk(\dowx, x)$.  We can again
argue as we did in the proof of Lemma~\ref{Qmaxprivacy} (and
elsewhere), that there exist distinct individuals (vertices) $x_1,
\ldots, x_k$ in $\dowqx$ such that $\tY_{\!1}, \ldots, \tY_{\!k}$ are
the maximal simplices of $\dowqy$, with $\tY_{\!i} =
Y_{x_{\scriptstyle i}} \inter Y_{\,x}$, and $\abs{\tY_{\!i}}=k-1$, for
$i=1, \ldots, k$.  (This time there are exactly $k$ individuals in
$\dowqx$, since every column of $R$ has exactly $k$ nonblank entries,
by Lemma~\ref{uniform}.)

Since every row of $R$ has exactly $k$ nonblank entries, each $x_i$
has one additional attribute in $\sdiff{Y}{Y_x}$.  We claim that this
additional attribute is the same $y$ for all $x_i$.  Given that claim
and the general row and column cardinality constraints, $R$ must be
isomorphic to the decomposition shown in Figure~\ref{RblocksQ}.  \
(This figure and the next indicate blank entries either by blanks or
by explicit ``$0$''s.)

\begin{figure}[h]
\begin{center} 
\vspace*{-0.1in}
\ifig{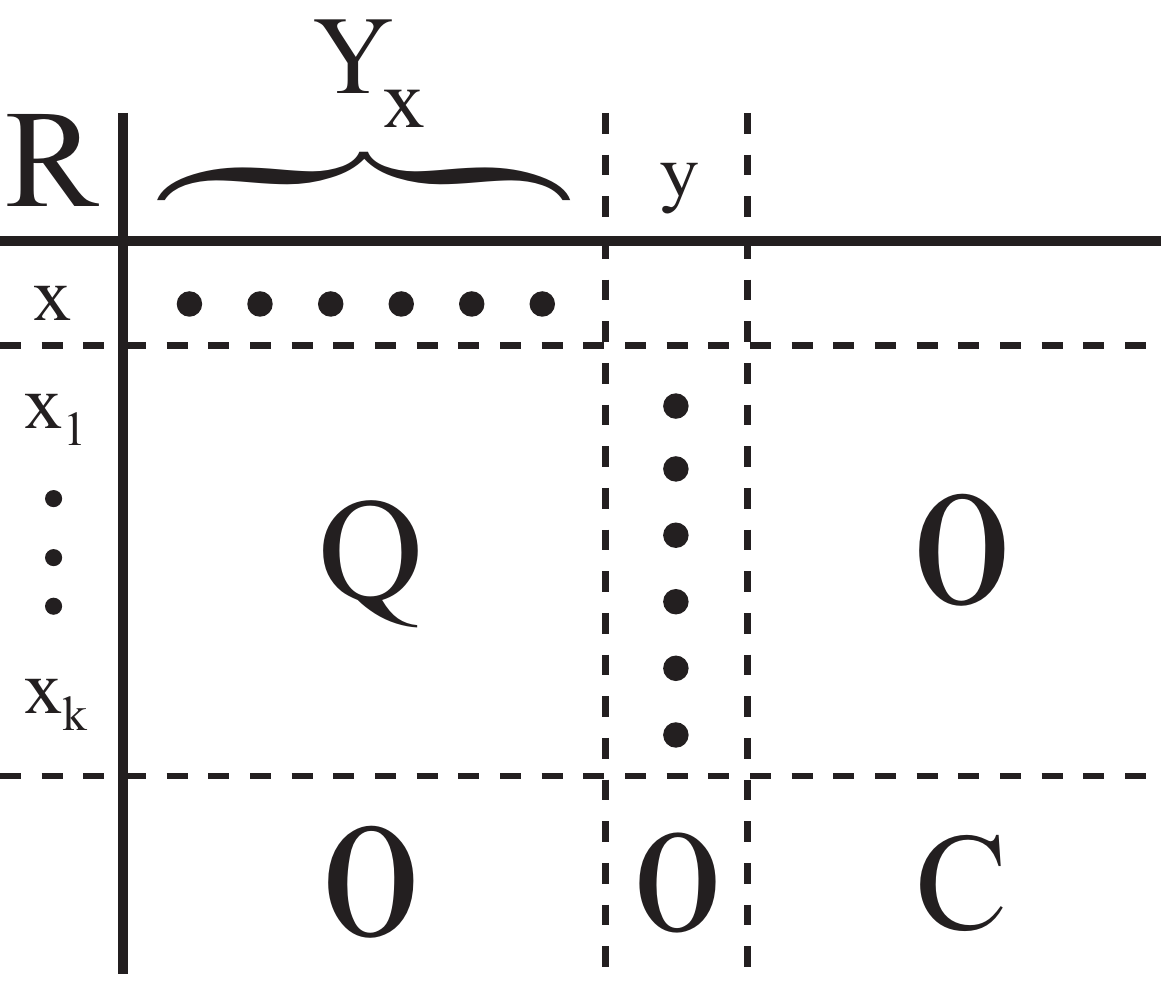}{scale=0.5}
\end{center}
\vspace*{-0.2in}
\caption[]{Relation $R$ decomposed as for the proof of
  Corollary~\ref{kthree}.}
\label{RblocksQ}
\end{figure}

Since $R$ is connected, $C$ cannot exist and the corollary follows,
somewhat as in the proof of Theorem~\ref{privbndrycmplx} as shown in
Figure~\ref{Rblocksdiag} on page~\pageref{Rblocksdiag}.

In order to establish the claim, suppose that $Y_x = \{y_1, \ldots,
y_k\}$.  Suppose further that for each $i=1, \ldots, k$, individual
$x_i$ has all the attributes of $Y_x$ except for $y_i$.  Now let $y$
be $x_2$'s attribute outside $Y_x$.  We can assume without loss of
generality that $x_1$ does not have this attribute, and then derive a
contradiction, as follows:

Let $\Ximply = X_{y_3} \,\inter\, X_y$.
The intersection is well-defined since $k \geq 3$.  Moreover, $x_2 \in
\Ximply$, since $x_2$ has attributes $y_3$ and $y$.  However, $x_1
\not\in \Ximply$, since $x_1$ does not have attribute $y$.

\vst

Observe that $X_{y_1} = \{x, x_2, x_3, \ldots, x_k\}$
\ and \ $X_{y_3} = \sdiff{\{x, x_1, x_2, \ldots, x_k\}}{\{x_3\}}$.

\vst

Thus $\emptyset \neq \Ximply = X_{y_3} \inter X_y \subseteq X_{y_1}$.
\ In other words, attributes $y_3$ and $y$ imply attribute $y_1$,
contradicting the assumption that $R$ preserves attribute privacy.
\end{proof}

\vspace*{0.2in}

\noindent We turn now to the proof of Theorem~\ref{squarerelations},
which we had stated previously on page~\pageref{squarerelations}:

\vspace*{-0.05in}

\begin{proof}
Part (i) follows from the Subclaim on page~\pageref{subclaim} and
part (ii) follows from Lemma~\ref{EqIdent} on page~\pageref{EqIdent}.
\quad We therefore focus on proving part (iii), assuming parts (i) and
(ii) hold:

\vspace*{0.1in}

The proof is by induction on $n=\abs{X}=\abs{Y}$.

\vspace*{0.1in}

I. The base case $n=2$ means $R$ is isomorphic to a standard two
element diagonal relation as on page~\pageref{EqIdent}, which
preserves association privacy.

\vspace*{0.1in}

II. For the induction step, assume that, for some $n > 2$, part (iii)
of the theorem holds for all relations with $X$ and $Y$ spaces of size
strictly less than $n$ (and bigger than 1).  We need to establish part
(iii) for all relations with $X$ and $Y$ spaces of size $n$.

\vst

As we observed in the proof of Lemma~\ref{Qmaxprivacy}, if $\kmax=1$
in Definition~\ref{Qmax}, then every individual has exactly one
attribute and $R$ is isomorphic to a square diagonal relation, hence
preserves association privacy.  We therefore assume that $\kmax>1$ for
the rest of the proof.

\vst

Let $\Qmax$ be as in Definition~\ref{Qmax} and consider the
decomposition of $R$ as in Figure~\ref{QmaxR}.

\begin{figure}[h]
\begin{center} 
\ifig{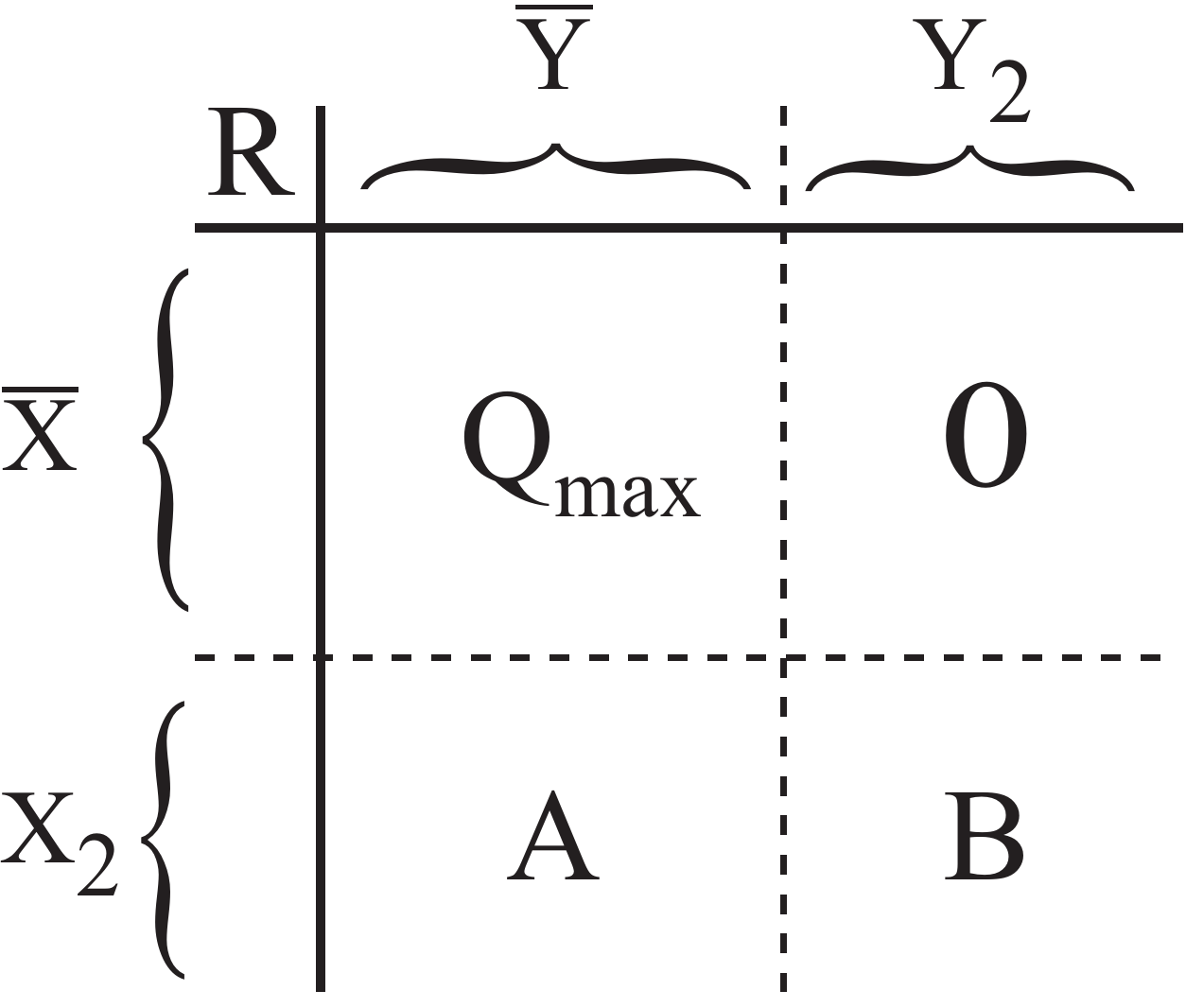}{scale=0.5}
\end{center}
\vspace*{-0.2in}
\caption[]{Relation $R$ decomposed into blocks by $\Qmax$.}
\label{QmaxR}
\end{figure}

Here $X_2 = \sdiff{X}{\tX}$ and $Y_2 = \sdiff{Y}{\tY}$, with $\tX$ and
$\tY$ as in Definition~\ref{Qmax}.  Then $A$ is the restriction of $R$
to $X_2 \times \tY$ and $B$ is the restriction of $R$ to $X_2 \times
Y_2$.

Given parts (i) and (ii), Lemma~\ref{Qmaxprivacy} tells us that
$\Qmax$ is tight and preserves attribute privacy.  By
Theorem~\ref{toomanyattrib} on page~\pageref{toomanyattrib}, we see
that $\abs{\tX}\geq\abs{\tY}$ and therefore that
$\abs{X_2}\leq\abs{Y_2}$, since $\abs{X}=\abs{Y}$.

\clearpage

Let us look at some cases:

\vspace*{-0.05in}

\begin{itemize}
\item \underline{$\abs{Y_2}=\abs{X_2}=1$}: \ Then $B$ is a singleton,
  so $A=\emptyset$, since $R$ preserves attribute privacy.  The
  induction hypothesis applies to $\Qmax$, telling
  us that $\Qmax$ preserves association privacy.  Since $R$ is the
  disjoint union of $\Qmax$ and $B$, both nonvoid, we see that $R$
  must also preserve association privacy.

\item \underline{$\abs{Y_2}=\abs{X_2}>1$}: \ Arguing as on
  page~\pageref{uniqueIdentCases}, we see that $B$ preserves attribute
  privacy.  Lemmas~\ref{EqIdent} and \ref{privacycolumns}, on
  pages~\pageref{EqIdent} and \pageref{privacycolumns}, respectively,
  then imply that $A=\emptyset$.  The induction hypothesis applies to
  each of $\Qmax$ and $B$, since neither is now a singleton (since
  $\kmax > 1$).  Again, $R$ is a disjoint union of these two
  relations, so we see that $R$ preserves association privacy.  (One
  can formalize that argument by using the dual version of
  Lemma~\ref{privacycolumns} on page~\pageref{privacycolumns}.)

\item \underline{$\abs{Y_2}>\abs{X_2}\geq 1$}: \ This case cannot occur,
  since $B$ would preserve attribute privacy but have more attributes
  than individuals (see again also Theorem~\ref{toomanyattrib} on
  page~\pageref{toomanyattrib}).

\item \underline{$\abs{X_2}=0$}: \ Then $Y_2 = \emptyset$ and
  $\Qmax=R$.

  If $R$ has more than one connected component, then one can apply the
  induction hypothesis to each component separately.  (In order to
  apply the induction hypothesis, one should first make a small
  argument that each component is a tight square relation, contains
  more than one entry, and preserves attribute privacy.  This is
  straightforward.)  One concludes that each component preserves
  association privacy and therefore that $R$ preserves association
  privacy.

  Otherwise, $R$ is square, tight, connected, and preserves attribute
  privacy.  Furthermore, every row of $R$ has exactly $\kmax$ entries.
  By assumption, $\kmax > 1$ in this part of the proof.  So
  Corollaries~\ref{ktwo} and \ref{kthree} tell us that $R$ is
  isomorphic to either a cyclic staircase relation or a spherical
  boundary relation.  Thus $R$ preserves association privacy.
\end{itemize}

\vspace*{-0.3in}

\end{proof}

\clearpage
\section{Poset Chains}
\markright{Chains in Posets and Links}
\label{posetchains}

Recall Definition~\ref{galoislattice}, on
page~\pageref{galoislattice}, of the Galois lattice $\PRplus$
associated with a relation $R$, and Definition~\ref{iars}, on
page~\pageref{iars}, defining informative attribute release sequences.
In this appendix we will explore connections between these two
concepts.

\subsection{Maximal Chains and Informative Attribute Release Sequences}

\hspace*{0.24in}Let $R$ be a relation on $\XxY$, with both
$\hspt{}X$ and $Y\mskip-1.5mu$ nonempty:

\vspace*{0.05in}

\hspace*{0.2in}\begin{minipage}{5.75in}

\vst

Suppose $\{(\sigma_k, \gamma_k) < \cdots < (\sigma_1, \gamma_1) <
(\sigma_0, \gamma_0)\}$, with $k\geq 1$, is a maximal chain in
$\PRplus$.

\vst

Then, \ for $1 \leq i \leq k$, \ $\sigma_i \subsetneq \sigma_{i-1}$
and $\gamma_i \supsetneq \gamma_{i-1}$.

\vst

Also, $\sigma_0 = X$ and $\gamma_k = Y$, so
$\gamma_0 = \phi_R(X)$ and $\sigma_k = \psi_R(Y)$.

\vst

Consequently, $\hspt\gamma_0 \neq \emptyset\hspt$ if and only if
$\mskip1mu{}X\mskip-2mu\in\dowx$, and $\sigma_k \neq \emptyset\hspt$
if and only if $\mskip1.5mu{}Y\mskip-2mu\in\dowy$.

\end{minipage}

\vspace*{0.1in}

We sometimes speak of a {\em maximal chain at and above} $(\sigma,
\gamma)$, by which we mean a chain $\{(\sigma, \gamma) < \cdots <
(\sigma_1, \gamma_1) < (\sigma_0, \gamma_0)\}$ in $\PRplus$ that is
maximal among all chains in $\PRplus$ containing $(\sigma, \gamma)$
as least element.  Such a chain is a prefix of a full maximal chain in
$\PRplus$ (``prefix'' with respect to our subscript ordering, which
starts at the top of a poset and moves downward).

\vspace*{0.1in}

Recall the following lemma, previously stated on
page~\pageref{chaintoiars} in
Section~\ref{informativerelease}:\label{chaintoiarsAppPage}

\setcounter{currentThmCount}{\value{theorem}}
\setcounter{theorem}{\value{ChainToIarsLemma}}
\begin{lemma}[Informative Attributes from Maximal Chains]
Let $R$ be a relation on $\XxY$, with both $\hspt{}X\!$ and
$\hspt{}Y\!$ nonempty.
Suppose $\{(\sigma_k, \gamma_k) < \cdots < (\sigma_1, \gamma_1) <
(\sigma_0, \gamma_0)\}$, with $k\geq 1$, is a maximal chain in
$\PRplus$.

\vso

Define $y_1, \ldots, y_k$ by selecting some $y_i \in
\gamma_i\setminus\gamma_{i-1}$, for each $i=1, \ldots, k$.

\vso

Then $y_1, \ldots, y_k$ is an informative attribute release sequence for
$R$.

\vso

Moreover, $(\clsy)(\{y_1, \ldots, y_i\}) = \gamma_i$, for each $i = 0, 1,
\ldots, k$.
\end{lemma}
\setcounter{theorem}{\value{currentThmCount}}

\begin{proof}
Establishing the ``Moreover'' also establishes the ``iars'' assertion.

The proof is by induction on $i$.

For the base case, $i=0$ and we need to show that $(\clsy)(\emptyset) =
\gamma_0$.

Calculating, $(\clsy)(\emptyset) = \phi_R(X) = \gamma_0$, by our earlier
comments about maximal chains.

\vst

For the induction step, we assume that, for some $1 \leq i \leq k$, the
assertion holds for indices smaller than $i$ and we need to show the
assertion holds for $i$.  First, observe:

\vspace*{-0.2in}

$$\psi_R(\{y_1, \ldots, y_i\})
  \;=\;
  \psi_R(\{y_1, \ldots, y_{i-1}\}) \;\inter\; X_{y_{\scriptstyle i}}
  \;=\;
  \psi_R(\gamma_{i-1}) \;\inter\; X_{y_{\scriptstyle i}}
  \;=\;
  \psi_R(\gamma_{i-1} \union \{y_i\}).
$$

(The middle equality follows from the induction hypothesis and a dual
  version of Corollary~\ref{mappedsimplex} from
  page~\pageref{mappedsimplex}, specifically because $(\clsy)(\{y_1,
  \ldots, y_{i-1}\}) = \gamma_{i-1}$ and $\clsx \circ \psi_R = \psi_R$.)

\vspace*{0.1in}

Since $\gamma_{i-1} \subsetneq \gamma_{i-1} \union \{y_i\} \subseteq
\gamma_i$,

\vspace*{-0.1in}

$$\gamma_{i-1}
  \;=\; (\clsy)(\gamma_{i-1})
  \;\subsetneq\; (\clsy)(\gamma_{i-1} \union \{y_i\})
  \;\subseteq\; (\clsy)(\gamma_i)
  \;=\; \gamma_i.$$

By maximality of the original chain and the nature of elements in
$\PRplus$, we see that $(\clsy)(\gamma_{i-1} \union \{y_i\}) =
\gamma_i$, so $(\clsy)(\{y_1, \ldots, y_i\}) = (\clsy)(\gamma_{i-1}
\union \{y_i\}) = \gamma_i$.
\end{proof}

\clearpage

Here is a partial converse (also previously stated in
Section~\ref{informativerelease}):

\setcounter{currentThmCount}{\value{theorem}}
\setcounter{theorem}{\value{IarsToChainLemma}}
\begin{lemma}[Chains from Informative Attributes]
Let $R$ be a relation on $\XxY$, with both $\hspt{}X$ and $\hspt{}Y\!$
nonempty.
Suppose $y_1, \ldots, y_k$ is an informative attribute release sequence
for $R$, with $k \geq 1$.

\vso

Let $\gamma_i = (\clsy)(\{y_1, \ldots, y_i\})$ and $\sigma_i =
\psi_R(\gamma_i)$, for $i=1, \ldots, k$.

\vso

Let $\gamma_0 = \phi_R(X)$.
\ Then $\{(\sigma_k, \gamma_k) < \cdots < (\sigma_1, \gamma_1) < (X,
\gamma_0)\}$ is a chain in $\PRplus$.
\end{lemma}
\setcounter{theorem}{\value{currentThmCount}}

Comment:  The resulting chain need not be maximal.

\begin{proof}
Observe that each $(\sigma_i, \gamma_i)\in\PRplus$ by construction, so
we need to establish the total ordering.
Letting $\sigma_0 = X$, we need to show that $\sigma_i \subsetneq
\sigma_{i-1}$, for each $i=1,\ldots,k$.

Since $\{y_1, \ldots, y_i\} \supseteq \{y_1, \ldots, y_{i-1}\}$, we
see that $\sigma_i \subseteq \sigma_{i-1}$.  If $\sigma_i =
\sigma_{i-1}$, then also $\gamma_i = \gamma_{i-1}$, contradicting the
fact that $y_i\in\gamma_i\setminus\gamma_{i-1}$ (which is true by the
nature of informative attribute release sequences).
\end{proof}

As a corollary to Lemmas~\ref{chaintoiars} and \ref{iarstochain}, one
sees that every informative attribute release sequence (iars) for $R$ is
a subsequence of an iars derived from a maximal chain in $\PRplus$.
(Technically, one needs to show that any nonempty subsequence of an iars
is itself an iars.  And one needs to show that extending any chain
obtained via Lemma~\ref{iarstochain} to a maximal chain retains the
original iars as a subsequence of one subsequently obtainable via
Lemma~\ref{chaintoiars}.  All that is straightforward.)

\subsection{Chains and Links}

We are interested in understanding how chains and informative
attribute release sequences behave as one passes to links.  (Small
caution: whereas we were looking at chains in $\PRplus$ before, we
focus here on $\PR$ (and $\PQ$).)

\begin{lemma}[Chains in Links]\label{linkchain}
Let $R$ be a relation on $\XxY$, with both $X\!$ and $\hspt{}Y\!$
nonempty, and suppose $(\sigma, \gamma)\in \PR$.
Let $Q$ be the relation modeling $\lk(\dowx, \sigma)$.  Then

\vspace*{-0.05in}

$$\PQ \;=\;
  \setdef{(\sdiff{\sigma^\prime}{\sigma},\, \gamma^\prime)}{%
          (\sigma, \gamma) < (\sigma^\prime, \gamma^\prime) \in \PR}.$$

\end{lemma}

Comments:

\vspace*{-0.1in}

\begin{itemize}

\item $Q$ is the restriction of $R$ to $\tX \times \gamma$, with
  $\tX = \bigunion_{y \in \gamma}\sdiff{X_y}{\sigma}$, as per
  Definition~\ref{linksigma} on page~\pageref{linksigma}.

\item $\PQ$ could be empty.  This occurs precisely when $(\sigma,
  \gamma)$ is a maximal element of $\PR$, which occurs precisely when
  $\lk(\dowx, \sigma) = \{\emptyset\}$, which occurs precisely when
  $\tX = \emptyset$.

\item If $\sigma = X$, then $\lk(\dowx, \sigma) = \{\emptyset\}$ and
  so $\PQ=\emptyset$, given Definition~\ref{linksigma} on
  page~\pageref{linksigma}.  (For future reference, observe that
  $P_{\qsg}$ is undefined when $\sigma = X$ and $\qsg$ is given by
  Definition~\ref{restrictedlink} on page~\pageref{restrictedlink}.
  See also page~\pageref{PRAppDeff} for comments about the
  doubly-labeled poset.)

\item $\PQ$ never contains the element $\zeroQ$ of $\PQplus$.  Indeed,
  $\zeroQ=(\emptyset, \gamma)$, corresponding to $(\sigma, \gamma)$ in
  $\PR$.  That value is consistent with the idea of
  Lemma~\ref{interplocal} on page~\pageref{interplocal} that one has
  ``localized to $\sigma$ upon observing $\gamma$''.  (See also
  Definition~\ref{DefMinIdent} on page~\pageref{DefMinIdent}.)

\item $\PQ$ could contain the element $\oneQ=(\tX, \chi)$ of $\PQplus$,
  for some $\chi \subsetneq \gamma$.  That happens precisely when
  $\tX\neq\emptyset$ and all individuals in $\tX$ share an attribute of
  $\gamma$, in which case $\chi\neq\emptyset$.

\end{itemize}

\begin{proof}
The proof relies on dual versions of the formulas appearing in
the top half of page~\pageref{psiLinkformulas}.

\vspace*{0.05in}

I.  Suppose $(\kappa, \eta) \in \PQ$.  So $\kappa\neq\emptyset$ and
    $\eta\neq\emptyset$.  Also, $\dowqx = \lk(\dowx, \sigma)$, so
    $\kappa\inter\sigma=\emptyset$ and $\kappa\union\sigma\in\dowx$.
    Let $\sigma^\prime = \kappa\union\sigma$.  So $\sigma \subsetneq
    \sigma^\prime$.  We can take $\gamma^\prime$ to be $\eta$ since
    $\eta = \phi_Q(\kappa) = \phi_R(\sigma^\prime)$.  Note that
    $\psi_R(\gamma^\prime) = \psi_Q(\eta) \union \sigma = \kappa \union
    \sigma = \sigma^\prime$.  We have shown that $(\sigma^\prime,
    \gamma^\prime)\in\PR$ and $(\sigma, \gamma) < (\sigma^\prime,
    \gamma^\prime)$.

\vspace*{0.05in}

II. Suppose $(\sigma^\prime, \gamma^\prime)\in\PR$ and $(\sigma, \gamma)
    < (\sigma^\prime, \gamma^\prime)$.  So $\sigma \subsetneq
    \sigma^\prime$ and $\gamma \supsetneq \gamma^\prime$.
    Let $\kappa=\sdiff{\sigma^\prime}{\sigma}$.  Note that
    $\kappa\neq\emptyset$ and $\gamma^\prime\neq\emptyset$.  Moreover,
    $\kappa\in\lk(\dowx,\sigma)$, so $\tX\neq\emptyset$.

    Verifying correspondence: $\phi_Q(\kappa) = \phi_R(\sigma^\prime) =
    \gamma^\prime$ and $\psi_Q(\gamma^\prime) =
    \psi_R(\gamma^\prime)\setminus\sigma =
    \sdiff{\sigma^\prime}{\sigma} = \kappa$.

    We have shown that $(\sigma^\prime\setminus\sigma,\, \gamma^\prime)\in\PQ$.
\end{proof}

\begin{corollary}[Order Preservation]\label{preserveorder}
Let $R$ and $Q$ be as in Lemma~\ref{linkchain}, with $(\sigma,
\gamma)\in\PR$.

Then $(\sigma, \gamma) < (\sigma_1, \gamma_1) < (\sigma_2, \gamma_2)$ in
$\PR$ if and only if \,$(\sigma_1\setminus\sigma,\, \gamma_1) <
(\sigma_2\setminus\sigma,\, \gamma_2)$ in $\PQ$.
\end{corollary}

\begin{proof}
By Lemma~\ref{linkchain} and because:\\[3pt]
\hspace*{1in}(a) $\sigma \subsetneq \sigma_1 \subsetneq
     \sigma_2$ \ implies \ $\emptyset\,\neq\,\sigma_1\setminus\sigma
     \,\subsetneq\, \sigma_2\setminus\sigma$\,;\\[3pt]
\hspace*{1in}(b) $\emptyset\neq\kappa_1
     \subsetneq \kappa_2$ \ and \ $\kappa_2 \inter \sigma = \emptyset$
     \; implies \ $\sigma \,\subsetneq\,
     (\kappa_1\union\sigma) \,\subsetneq\, (\kappa_2\union\sigma)$.
\end{proof}

\begin{corollary}[Maximal Chain Preservation]\label{preservemax}
Let $R$ and $Q$ be as in Lemma~\ref{linkchain}, with $(\sigma,
\gamma)\in\PR$.  Then
$\{(\sigma, \gamma) < (\sigma_k, \gamma_k) < \cdots < (\sigma_1,
\gamma_1)\}$ is a maximal chain at and above $(\sigma, \gamma)$ in $\PR$
if and only if
$\{(\sigma_k\setminus\sigma,\, \gamma_k) < \cdots <
(\sigma_1\setminus\sigma,\, \gamma_1)\}$ is a maximal chain in $\PQ$.
\end{corollary}

\begin{proof}
By Lemma~\ref{linkchain} and Corollary~\ref{preserveorder}, we know that
$\{(\sigma, \gamma) < (\sigma_k, \gamma_k) < \cdots < (\sigma_1,
\gamma_1)\}$ is a chain extending upward from $(\sigma, \gamma)$ in
$\PR$ if and only if $\{(\sigma_k\setminus\sigma,\, \gamma_k) < \cdots <
(\sigma_1\setminus\sigma,\, \gamma_1)\}$ is a chain in $\PQ$.

Maximality follows for the same reason: Refine or extend a chain in one
poset and one can refine or extend the corresponding chain in the other
poset as well.
\end{proof}

\paragraph{Comment about ``length'':}\ Recall that the length of a
chain in a poset is one less than the number of elements in the chain.
We also speak of the {\em length} of an informative attribute release
sequence $y_1, \ldots, y_k$, which is $k$, the actual number of
attributes in the sequence.

In the context of Lemmas~\ref{chaintoiars} and \ref{iarstochain},
there is a happy alignment of definitions: The length $k$ of a longest
iars for $R$ is the length $\ell(\PRplus)$.

In thinking about poset lengths, bear in mind that $\ell(\PRplus)$ may
be any of $\ell(\PR)$, $\ell(\PR) + 1$, or $\ell(\PR) + 2$, depending on
whether the top and/or bottom elements of $\PRplus$ already lie in
$\PR$.

\begin{corollary}[Longest Localization Sequences]\label{linklengths}
Let $R$ be a relation on $\XxY$, with both $X$ and $\hspt{}Y\!$
nonempty, and suppose $(\sigma, \gamma)\in \PR$.
Let $Q$ be the relation modeling $\lk(\dowx, \sigma)$.

\vst

If $X\not\in\dowx$, then the length of a longest informative attribute
release sequence for localizing to $\sigma$ in $R$ is $\ell(\PQ) + 2$.
\quad
If $X\in\dowx$ and $\sigma\neq X$, then that length is $\ell(\PQ) + 1$.

\vst

(Note: If $\,\sigma=X\in\dowx$, then the length is 0; one can localize
to $X\!$ in $R$ without observation.)
\end{corollary}

Comment: If $\PQ$ does not contain the top element $\oneQ$ of
$\PQplus$, then $\ell(\PQ) + 2 = \ell(\PQplus)$, since $\PQ$ never
contains the bottom element $\zeroQ$.  This occurs precisely when no
attribute is shared by all the individuals in the link.  \ Also, if
$\sigma\subsetneq X \in \dowx$, then $\ell(\PQ) + 1 = \ell(\PQplus)$.

\begin{proof}
Let us address one special case first, namely when $\lk(\dowx,
\sigma)$ is an empty complex.  We only care about the situation in
which $\sigma$ is not all of $X$, which implies $X \not\in \dowx$. \
Observe that $\PQ$ is empty, so $\ell(\PQ) = -1$ and $\ell(\PQ) + 2 =
1$.  Observe further that any $y\in\gamma$ identifies $\sigma$, as
otherwise $\tX$ in the definition of $Q$ would not be empty.  \ So the
Corollary holds in this case.

\vspace*{0.05in}

Suppose $\lk(\dowx, \sigma)$ is not an empty complex and that
$X\not\in\dowx$.  Lemmas~\ref{chaintoiars} and \ref{iarstochain} imply
that a longest informative attribute release sequence for localizing
to $\mskip1mu\sigma\mskip-1mu$ comes from a longest maximal chain in
$\PRplus$ at and above $(\sigma,\gamma)$.  Thus, by
Corollary~\ref{preservemax}, this sequence arises from a maximal chain
in $\PQ$, augmented by considering also $\zeroQ$ and $\oneR$.  The
length of the chain in $\PQ$ is two shorter than that in $\PRplus$.
Why?  Because
$(\sigma,\gamma)\in\PRplus$$\phantom{R^{1^1}}\hspace*{-18pt}$ becomes
$\zeroQ\in\PQplus$, which is not present in $\PQ$, and because the top
element $\oneR=(X, \emptyset)\in\PRplus$ disappears altogether
($\oneQ$ may or may not be in $\PQ$).  \ \ So $\ell(\PQ) + 2$ gives the
correct length of the iars for $R$.$\phantom{R^{1^1}}$

\vspace*{0.05in}

Suppose $\lk(\dowx, \sigma)$ is not an empty complex but that
$\sigma\subsetneq X\in\dowx$.  The argument proceeds as before except
that now the top element of $\PRplus$ looks like $\oneR=(X,\gamma_0)$,
with $\gamma_0\neq\emptyset$.  It appears in $\PR$.  Consequently,
$\oneQ=(X\setminus\sigma,\,\gamma_0)$ and so $\oneQ$ also appears in
$\PQ$.  So a maximal chain in $\PQ$ is now only one shorter than a
corresponding maximal chain in $\PRplus$ at and above $(\sigma,
\gamma)$, meaning $\ell(\PQ) + 1$ gives the correct length of a
longest iars.
\end{proof}

\subsection{Isotropy}
\markright{Isotropy}
\label{isotropyApp}

We turn now to the proof of our isotropy sphere theorem, with the
theorem replicated here from earlier in the report.  Recall also
Definitions~\ref{iars}, \ref{isotropydef}, \ref{DefMinIdent}, and
\ref{restrictedlink} from
pages~\pageref{iars}--\pageref{restrictedlink}.

\setcounter{currentThmCount}{\value{theorem}}
\setcounter{theorem}{\value{IsotropyThm}}
\begin{theorem}[Isotropy = Minimal Identification = Sphere]
Let $R$ be a relation and suppose $\emptyset\neq\gamma\in\dowy$.
\ Let $\sigma = \psi_R(\gamma)$.
Then the following four conditions are equivalent:

\vspace*{0.1in}

\hspace*{0.4in}\begin{minipage}{3.5in}
\begin{enumerate}
\addtolength{\itemsep}{-1pt}
\item[(a)] $\gamma$ is isotropic.
\item[(b)] $\gamma$ is minimally identifying (for $\sigma$).
\item[(c)] $\dqx \;\homot\; \Skt$, with $k = \abs{\gamma}$.
\item[(d)] $\dqy \;=\; \partial(\gamma)$.
\end{enumerate}
\end{minipage}
\end{theorem}
\setcounter{theorem}{\value{currentThmCount}}

\begin{proof}
Observe that $\sigma\in\dowx$$\phantom{\Big|}\!$ and $\gamma \subseteq
(\clsy)(\gamma) = \phi_R(\sigma)$, so constructing $\qsg$ is valid.
Also, (a)--(d) each imply $\sigma \neq X$.
Finally, observe
that $\gamma\not\in\dqy$.  For if there were some $x\in\tX$ such
that $(x,y)\in\qsg\subseteq R$ for every $y\in\gamma$, then
$x\in\sigma$, but $\sigma$ is disjoint from $\tX$.

\vspace*{0.1in}

If $\abs{\gamma} = 1$, then $\Skt = \Smo = \{\emptyset\} =
\partial(\gamma)$.  Write $\gamma=\ys$.  Then $\gamma$ is isotropic if
and only if $y$ constitutes an informative attribute release sequence,
if and only if $y\not\in\phi_R(X)$.
\addtolength{\baselineskip}{1.5pt}
If $y\in\phi_R(X)$, then $\sigma=X$, so our conventions say
$\dqx\!=\emptyset\not\homot\{\emptyset\}$ and
$\dqy=\emptyset\neq\{\emptyset\}$.  Moreover, $\psi_R(\emptyset) =
\sigma$, so $\gamma$ is not minimally identifying.  If
$y\not\in\phi_R(X)$, then $\sigma=X_y\subsetneq X$ and $\tX=\emptyset$,
so both $\dqx$ are $\dqy$ are instances of $\{\emptyset\}$, by our
conventions.$\phantom{R^{1^1}}\hspace*{-12pt}$
Moreover, $X=\psi_R(\emptyset)\supsetneq\sigma$.
\ So we see that (a), (b), (c), (d) are all equivalent when
$\abs{\gamma} = 1$.  $\phantom{R^{1^1}}$
\addtolength{\baselineskip}{-1.5pt}

\vspace*{0.1in}

Henceforth assume that $\abs{\gamma} > 1$.  It will be convenient to
write $\gamma = \{y_1, \ldots, y_k\}$, with $k > 1$, and with the
attribute indexing chosen arbitrarily.

As we have observed elsewhere, (c) and (d) are equivalent by Dowker
duality and the fact that only a boundary complex can produce
$\Skt$ homotopy type when the underlying vertex set has size $k$.

\vspace*{0.1in}

We will first show that (a) implies (d) and (b):

\vst

\underline{Suppose that $\gamma$ is isotropic.}

\vst

We wish to show that all proper subsets of $\gamma$ are simplices in
$\dqy$.  Without loss of generality, consider $\{y_1, \ldots,
y_{k-1}\}$.  If we can show that $\psi_R(\{y_1, \ldots,
y_{k-1}\})\setminus\sigma \neq \emptyset$, then that provides an
$x\in\tX$ such that $(x,y_i)\in R$ for $i=1, \ldots, k-1$, thereby
establishing that $\{y_1, \ldots, y_{k-1}\} \in \dqy$.  It also
establishes that $\psi_R(\{y_1, \ldots, y_{k-1}\})\supsetneq\sigma$.
Since the ``missing attribute'' $y_k$ is arbitrary in $\gamma$, we see
that $\dqy = \partial(\gamma)$ and that $\gamma$ is minimally
identifying.

Suppose otherwise: $\psi_R(\{y_1, \ldots, y_{k-1}\}) = \sigma =
\psi_R(\gamma)$, so also $(\clsy)(\{y_1, \ldots, y_{k-1}\}) =
(\clsy)(\gamma)\supseteq\gamma$.  That says $y_k\in(\clsy)(\{y_1,
\ldots, y_{k-1}\})$, violating the assumption that any ordering of
$\gamma$ is an informative attribute release sequence.

\vspace*{0.15in}

We will now show that (d) implies (a):

\vst

\underline{Suppose that $\dqy = \partial(\gamma)$.}

\vst

If some ordering of $\gamma$ is not an informative attribute release
sequence, then we can rearrange the sequence further to establish that
the last attribute is implied by all the others, i.e., that $y_k \in
(\clsy)(\{y_1, \ldots, y_{k-1}\})$.  Arguing as we did in the proof of
Lemma~\ref{chaintoiars} on page~\pageref{chaintoiarsAppPage}, we obtain:

\vspace*{-0.2in}

\begin{eqnarray*}
\psi_R\big(\{y_1, \ldots, y_{k-1}\}\big)
    &=& (\clsx)\Big(\psi_R\big(\{y_1, \ldots, y_{k-1}\}\big)\Big)\\[2pt]
    &=& \psi_R\Big((\clsy)\big(\{y_1, \ldots, y_{k-1}\}\big)\Big)\\[2pt]
    &=& \psi_R\Big(\{y_k\} \union (\clsy)\big(\{y_1, \ldots, y_{k-1}\}\big)\Big)\\[2pt]
    &=& X_{y_{\scriptstyle k}} \;\inter\;
         \psi_R\Big((\clsy)\big(\{y_1, \ldots, y_{k-1}\}\big)\Big)\\[2pt]
    &=& X_{y_{\scriptstyle k}} \;\inter\;
         \psi_R\big(\{y_1, \ldots, y_{k-1}\}\big)\\[2pt]
    &=& \psi_R\big(\{y_1, \ldots, y_k\}\big)\\[2pt]
    &=& \psi_R(\gamma)\\[2pt]
    &=& \sigma.
\end{eqnarray*}

On the other hand, since $\{y_1, \ldots, y_{k-1}\}\in\dqy$, there is a
witness $x\in\tX$, meaning $x\in\psi_R(\{y_1, \ldots, y_{k-1}\})$, which
contradicts $\tX \inter \sigma = \emptyset$.

\vspace*{0.15in}

Finally, we will show that (b) implies (d):

\vst

\underline{Suppose that $\gamma$ is minimally identifying.}

Observe that $\psi_R(\{y_1, \ldots, y_{k-1}\}) \supsetneq \sigma$.  As
above, this establishes $\{y_1, \ldots, y_{k-1}\} \in \dqy$, from which
we conclude that $\dqy = \partial(\gamma)$, since the missing attribute
$y_k$ was arbitrary.
\end{proof}

\clearpage
\section{Many Long Chains}
\markright{Many Long Chains}
\label{manylongchains}

This appendix provides a proof of Theorem~\ref{manychains} from
page~\pageref{manychains}.

\vspace*{0.1in}

\noindent First, we need some tools:

\vspace*{0.1in}

\noindent Recall what it means for a poset to be almost a join-based
lattice from Definition~\ref{almostlattice} on
page~\pageref{almostlattice}.

\begin{definition}[Join Completion]\label{joincompletion}
Suppose $P$ is almost a join-based lattice.  Let $S$ be a subset of
$P$.  The \mydefem{bounded join-completion of $S$ in $P$} is the set
$\Sv$ defined by:
\end{definition}

\vspace*{-0.3in}

$$\Sv = \setdef{p\in P}{\hbox{$p \leq s$, some $s \in S$, and $p =
    s_1 \join \cdots \join s_m$, with each $s_i\in S$, and $m  \geq 1$}}.$$

\vspace*{-0.05in}

{\em Here and in the rest of this appendix, ``\,$\leq$'' and ``$<$''
refer to the partial order on $P$, while ``$\,\join$'' denotes the
resulting join operation on $P \union \{\topone\}$.}  \ $S$ and $\Sv$
inherit this partial order.

\vspace*{0.05in}

{\em We also define $\Smax$ to consist of all the maximal elements of
$S$ relative to the partial order inherited from $P$.}

\vspace*{0.25in}

\noindent The following facts will be useful.  Assume $S \subseteq P$,
with $P$ almost a join-based lattice.  Then:

\vspace*{-0.05in}

\begin{enumerate}
\label{joinfacts}

\item $\Sv$ is almost a join-based lattice.  The join operation for
  elements $p,q \in \Sv$ is given by: 

\vspace*{-0.1in}

 $$ p\, \join_{\Sv}\, q \; = \; \left\{\,
       \begin{aligned}
         & p \join q, & & \hbox{if $\;p \join q \;\leq\; s$, for some $s\in{S}$;}\\[4pt]
         & \topone,   & & \hbox{otherwise.}\\
       \end{aligned}\right.
 $$

\item $S \subseteq \Sv$ and $\Smax = \Svmax$.

\item $(\Sv)^{\join} = \Sv$.

\item If $T \subseteq S$, then $\Tv \subseteq \Sv$.

\item If $T \subseteq \Sv$ such that $\Smax\setminus{T} \neq \emptyset$,
      then $\Tv \subsetneq \Sv$.

\newcounter{factneq}
\setcounter{factneq}{\value{enumi}}

\vspace*{0.1in}

\item Let $\emptyset \neq T \subseteq S$.  Then the poset

\label{factsSjoin}
\newcounter{factST}
\setcounter{factST}{\value{enumi}}

\vspace*{-0.15in}

      $$S_T=\setdef{p \in \Sv}{p \leq t,\,\, \hbox{for all $t\in T$}}$$

\vspace*{-0.04in}

  is almost a join-based lattice. The join operation for elements
  $p,q \in S_T$ is given by:

\vspace*{-0.1in}

 $$ p\, \join_{S_T}\, q \; = \; \left\{\,
       \begin{aligned}
         & p \join q, & & \hbox{if $\;p \join q \;\leq\; t$,\, for \!{\em all}\, $t\in T$;}\\[4pt]
         & \topone,   & & \hbox{otherwise.}\\
       \end{aligned}\right.
 $$

\vspace*{0.1in}

\item Fact~{\the\value{factST}} holds as well for the poset
   $S^{\prime}_T=\setdef{\,p \in \Sv}{p < t,\,\, \hbox{for all $t\in T$}}$,\\[2pt]
   now using ``$<$'' in place of ``$\leq$'' throughout.

\newcounter{factSTlt}
\setcounter{factSTlt}{\value{enumi}}

\end{enumerate}

\vspace*{0.1in}

\begin{lemma}[Contractibility of Closed Semi-Intervals]\label{intervalcontract}
Suppose $\emptyset \neq T \subseteq S \subseteq P$, with $P$ almost a
join-based lattice.
Define the poset $S_T$ as in Fact~{\the\value{factST}} on
page~\pageref{joinfacts}.

If $S_T \neq \emptyset$, then $S_T$ is contractible.

\end{lemma}

\begin{proof}

Suppose $p$ and $q$ are arbitrary elements of $S_T$.  Every element of
$T$ is an upper bound for both $p$ and $q$.  Since $T$ is not empty,
this means $p \join q$ exists in $P$ and $p \join q \leq t$ for all $t
\in T$.
Since $t\in S$, we have that $p \join q \in \Sv$ and thus $p \join q \in
S_T$ as well.  Consequently, the lattice $S_T \union \{\botzero,
\topone\}$ is noncomplemented, implying that $S_T$ is contractible, by a
fact on page~\pageref{noncomplemented}.
\end{proof}

\vsr

Intuitively: \ $\Delta(S_T)$ is a cone with apex $\bigwedge{T}$, the
meet in $\Sv$ of all the upper bounds $T$.

\vst

Caution: \ The lemma need {\em not}\, hold for $S^{\prime}_T$ as
defined in Fact~{\the\value{factSTlt}} on page~\pageref{joinfacts}.

\vspace*{0.1in}

We now specialize a topological tool to our current setting.  We refer
to the lemma as ``cycle tightening'' because we will apply the lemma
with $p\in\Smax$ and with $z$ a reduced homology generator for
$\Delta(P)$.  The lemma will allow us to move that generator downward
in $P$.

\begin{lemma}[Cycle Tightening]\label{tighten}
Let $P$ be almost a join-based lattice. 
Suppose $z$ is a nontrivial reduced $k$-cycle for $\Delta(P)$, i.e.,
$0 \neq z \in C_k(\Delta(P); \Z)$ and $\redbndry{z}=0$, for some $k
\geq 0$.

\vspace*{0.05in}

Define $S=\supp{z}\,$ and $\,K=\setdef{\tau\in\Delta(P)}{\tau\subseteq\Sv}$.

\vspace*{0.03in}

Let $p\in S$.

\vspace*{0.07in}

If $\widetilde{H}_{k-1}(\lk(K, p); \Z) = 0$, then there exists $\eta
\in C_{k+1}(\st(K,p); \Z)$ such that $p \notin \supp{z + \redbndry\eta}$,
now viewing $\eta \in C_{k+1}(\Delta(P); \Z)$.

\end{lemma}

\begin{proof}

Let $W = \st(K,p)$ and $A = \lk(K,p)$.  Note that $A$ is not an empty
complex (that observation follows from the reduced homology assumption
when $k=0$ and the fact that $p$ is part of a simplex containing at
least one other element when $k>0$).

The long exact sequence for a pair \cite{tpr:munkres, tpr:hatcher}
therefore gives us the following exact sequence:

\vspace*{-0.1in}

$$
0 = \widetilde{H}_k(W; \Z) \longrightarrow \widetilde{H}_k(W, A;\, \Z) \longrightarrow \widetilde{H}_{k-1}(A; \Z) = 0
.$$

The left $0$ comes from $W$ being a cone and the right $0$ comes from
the lemma's hypotheses.   Consequently, $\widetilde{H}_k(W, A;\, \Z) = 0$.

Suppose $z = \sum_in_i\tau_i$, for some collection $\{\tau_i\}$ of
(oriented) $k$-simplices such that $n_i \neq 0$ for each $i$.  Let
$z_S$ consist of the part of $z$ that lies within $W$, so:

\vspace*{-0.12in}

$$\hspace*{2in}z_S = \sum_{\tau_{\scriptstyle i}\in{W}}n_i\tau_i
  \qquad\qquad \hbox{(with each $n_i$ and $\tau_i$ as in $z$)}.$$

Since $z$ is a reduced $k$-cycle with support in $\verts{K}$, $z_S$ is
a reduced relative $k$-cycle for the pair $(W, A)$.
\ 
Since $\widetilde{H}_k(W, A;\, \Z) = 0$, $z_S$ must be a reduced relative
  boundary, so there exists $\kappa \in C_{k+1}(W; \Z)$ such that $z_S
  = \redbndry\kappa + \gamma$, with $\gamma\in C_k(A; \Z)$.

Now let $\eta = -\kappa$ and view $\eta \in C_{k+1}(\Delta(P); \Z)$.

Observe that $\supp{z_S + \redbndry\eta} \subseteq \verts{A} \subseteq
  \verts{\dl(K,p)}$.  Consequently, $p \not\in \supp{z + \redbndry\eta}$.
\end{proof}

\begin{lemma}[Maximal Element Cardinality]\label{maxelemcard}
Let $P$ be almost a join-based lattice. 
Suppose $P$ has reduced integral homology in dimension $k \geq 0$,
that is, $\widetilde{H}_k(\Delta(P); \Z) \neq 0$.

\vsr

Let $S = \supp{z}$, with $z \in C_k(\Delta(P); \Z)$ a reduced homology
generator for $\widetilde{H}_k(\Delta(P); \Z)$.

\vsr

Then $\abs{\Smax} \geq k+2$.

\end{lemma}

\begin{proof}

Since $S \subseteq \Sv$, we can view $z \in C_k(\Delta(\Sv); \Z)$.  If
there were to exist $\eta \in C_{k+1}(\Delta(\Sv); \Z)$ such that
$\redbndry\eta = z$, then $z$ would also be a reduced boundary in
$\Delta(P)$.  So, $\widetilde{H}_k(\Delta(\Sv); \Z) \neq 0$ and $z$ is
a reduced homology generator for $\Delta(\Sv)$.

\vspace*{0.1in}

Recall the notation $S_T$ in Fact~{\the\value{factST}} on
page~\pageref{factsSjoin}.  \ Observe that

\vspace*{-0.05in}

$$\bigunion_{\phantom{000}t\in{\Ssmax}}\Delta(S_{\{t\}}) \;=\; \Delta(\Sv).$$

To see this, first observe that the empty simplex $\emptyset$ appears
in both these sets.  Then:

\vspace*{-0.075in}

\begin{itemize}

\item[I.]  Suppose $\emptyset \neq \sigma \in \Delta(S_{\{t\}})$ for some
  $t\in\Smax$.  Being a chain in $S_{\{t\}}$, we can write $\sigma$ as
  $\{p_0 < p_1 < \cdots < p_\ell\}$, for some $\ell \geq 0$, with each $p_i
  \in \Sv$ (and $\,p_\ell \leq t \in \Smax \subseteq S \subseteq \Sv$).

  Consequently, $\sigma \in \Delta(\Sv)$ as well.

\item[II.]  Suppose $\emptyset \neq \sigma \in \Delta(\Sv)$.
  Then $\sigma = \{p_0 < p_1 < \cdots < p_\ell\}$, for some $\ell \geq 0$,
  with each $p_i \in \Sv$.
  By definition of $\Sv$ and $\Smax$, $p_\ell \leq s \leq t$, for some
  $s\in S$ and $t\in\Smax$.

  Consequently, $\sigma \in \Delta(S_{\{t\}})$ as well, for that $t$.

\end{itemize}

Similarly, one sees that, for any $\emptyset\neq T \subseteq S$,

\vspace*{-0.05in}

$$\biginter_{t\in{T}}\Delta(S_{\{t\}}) \;=\; \Delta(S_T).$$

The complex on the right is either an empty complex or it is
contractible, by Lemma~\ref{intervalcontract}.

A variation of the Nerve Lemma now implies that $\Delta(\Sv)$ and the
nerve of the simplicial complexes
$\big\{\Delta(S_{\{t\}})\big\}_{t\in\Ssmax}$ have the same homotopy
type (see Theorem 10.6(i) in \cite{tpr:bjorner}).

\vspace*{0.05in}

\addtolength{\baselineskip}{1.5pt}

\vst

Since $\Delta(\Sv)$ has reduced homology in dimension $k$, so does the
nerve of $\big\{\Delta(S_{\{t\}})\big\}_{t\in\Ssmax}$.

\vst

The nerve of $\big\{\Delta(S_{\{t\}})\big\}_{t\in\Ssmax}$ is
isomorphic to a simplicial complex with underlying vertex set $\Smax$.
In order for a simplicial complex to have reduced homology in
dimension $k$, with $k \geq 0$, the complex must have at least $k+2$
vertices.  Thus $\abs{\Smax}\geq k+2$.
\end{proof}

We now turn to the proof of the main theorem, the statement of which
is replicated here:

\setcounter{currentThmCount}{\value{theorem}}
\setcounter{theorem}{\value{manychainThm}}
\begin{theorem}[Many Maximal Chains]
Let $P$ be almost a join-based lattice.
Suppose $P$ has reduced integral homology in dimension $k \geq 0$,
that is, $\widetilde{H}_k(\Delta(P); \Z) \neq 0$.

Then there are at least $(k+2)!$ maximal chains in $P$ of length at
least $k$.
\end{theorem}
\setcounter{theorem}{\value{currentThmCount}}

\begin{proof}
The proof is by induction on $k$.

\vspace*{0.1in}

I. For the base case, $k=0$, observe that $\Delta(P)$ must have at
least two vertices that are incomparable in $P$, as otherwise
$\Delta(P)$ would be either empty or contractible.  Each vertex sits
inside a maximal chain of $P$.  The chains are different since the
vertices are incomparable.

\vspace*{0.15in}

II. For the induction step, assume that, for some $k \geq 1$, the
theorem holds for all relevant $P$ with reduced homology in dimension
$k-1$.  We need to establish the theorem for all relevant $P$ with
reduced homology in dimension $k$.

Let $z = \sum_in_i\tau_i$ be a reduced homology generator for
$\widetilde{H}_k(\Delta(P); \Z)$, with $n_i \neq 0$ for each $i$.

\vspace*{0.05in}

Define $S$ and $K$ by $S = \supp{z}$ and $K =
\setdef{\tau\in\Delta(P)}{\tau\subseteq\Sv}$.  Interpretation: \ $S$
is the support of the reduced homology generator $z$ and $K$ is the
subcomplex of $\Delta(P)$ formed by restricting to the bounded
join-completion of $z$'s support.

\vspace*{0.1in}

We now have an inner induction, which we will describe as an iterative
loop:

(Notation: superscript $(j)$ indicates the $j^{\hbox{\footnotesize th}}$
 iteration.)

\begin{enumerate}
\item Initialize with $z^{(0)} = z$, \ $S^{(0)} = S$, \ and $K^{(0)} = K$.

\item Suppose $z^{(j)}$, \ $S^{(j)}$, \ and $K^{(j)}$ have been
  defined, with $z^{(j)}$ a reduced homology generator for
  $\widetilde{H}_k(\Delta(P); \Z)$,$\phantom{\Big|}$
  and with $S^{(j)}$ and $K^{(j)}$
  similar in meaning to $S$ and $K$, now based on $z^{(j)}$.
  \ In particular, $z^{(j)}$ has support $S^{(j)}$ and all of
  $K^{(j)}$'s vertices lie in $(S^{(j)})^{\join}$.

  Pick some $p\in(S^{(j)})_{\max}$ such that
  $\widetilde{H}_{k-1}(\lk(K^{(j)}, p); \Z) = 0$.

  If no such $p$ exists, then the loop ends.

\item Otherwise, invoke Lemma~\ref{tighten} to find an $\eta \in
  C_{k+1}(\st(K^{(j)},p); \Z)$ such that $p \notin \supp{z^{(j)} + \redbndry\eta}$.

  Let $z^{(j+1)} = {z^{(j)} + \redbndry\eta}$, \ so $z^{(j+1)}$ is
  again a generator of reduced homology in dimension $k$.
  \ \ Further, let

\vspace*{-0.35in}

  $$\hspace*{0.85in}S^{(j+1)} = \supp{z^{(j+1)}}
  \hspace*{0.12in}\hbox{and}\hspace*{0.1in}
  K^{(j+1)} \;=\; \setdef{\tau\in\Delta(P)}{\tau\subseteq(S^{(j+1)})^{\join}}.$$

\end{enumerate}

Observe that $S^{(j+1)} \;\subseteq\; \supp{z^{(j)}} \union
\supp{\redbndry\eta} \;\subseteq\; (S^{(j)})^{\join}$.

\vso

On the other hand, $p \in (S^{(j)})_{\rm max}\setminus{S^{(j+1)}}$.
So by Fact~{\the\value{factneq}} on page~\pageref{joinfacts},
$(S^{(j+1)})^{\join} \subsetneq (S^{(j)})^{\join}$.

In other words, the possible vertex set for the simplicial complex
shrinks with each iteration, and so the loop must eventually end, $P$
being finite.

\vspace*{0.1in}

Given this iterative algorithm, we can now assume without loss of
generality that $\widetilde{H}_{k-1}(\lk(K, p); \Z) \neq 0$ for each
$p$ that is a maximal element in the support $S$ of the given
reduced homology generator $z$.

Observe that $\lk(K,p) = \setdef{\tau\in\Delta(P)}{\tau\subseteq\Sv
  \;\hbox{and $s < p$ for every $s\in\tau$}}$, when $p \in \Smax$.

\vspace*{0.1in}

Consequently, $\lk(K,p) = \Delta(Q_p)$, where $Q_p$ is the subposet of
  $P$ given by

\vspace*{-0.08in}

  $$Q_p \;=\; \setdef{s\in\Sv}{s < p}.$$

By Fact~{\the\value{factSTlt}} on page~\pageref{joinfacts}, $Q_p$ is
itself almost a join-based lattice.

$Q_p$ has reduced integral homology in dimension $k-1$, so by the
induction hypothesis, there are at least $(k+1)!$ maximal chains in
$Q_p$ of length at least $k-1$.  As the description of $Q_p$ makes
clear, we can extend each of these chains in $P$ by adding $\,p\,$ as
a top element, then further refine and/or extend each chain as needed
into a maximal chain in $P$. \label{manylongsubchains}
Distinct chains remain distinct after this augmentation since the
process only adds elements of $P$ that lie outside $Q_p$.

Consequently, we obtain, for each $p\in\Smax$, at least $(k+1)!$
distinct maximal chains in $P$ of length at least $k$, each touching
$p$.  A maximal chain in $P$ cannot contain more than one element of
$\Smax$, since such elements are necessarily incomparable.  Letting $p$
vary over $\Smax$ therefore produces at least $\abs{\Smax}\cdot(k+1)!$
distinct maximal chains in $P$ of length at least $k$.

By Lemma~\ref{maxelemcard}, $\abs{\Smax} \geq k+2$.  So $P$ contains
at least $(k+2)!$ distinct maximal chains of length at least $k$.
\end{proof}

Here are two corollaries, previously stated on
page~\pageref{holesreduceinference} in Section~\ref{bubbles}:

\vspace*{0.1in}

\setcounter{currentThmCount}{\value{theorem}}
\setcounter{theorem}{\value{holeinferCor}}
\begin{corollary}[Holes Reduce Inference]
Let $R$ be a nonvoid relation.  Suppose $\PR$ has reduced integral
homology in dimension $k \geq 0$.  Then there are at least $(k+2)!$
maximal chains in $\PR$ of length at least $k$.
\end{corollary}
\setcounter{theorem}{\value{currentThmCount}}

\begin{proof}
The assertion follows from Theorem~\ref{manychains}, since $\PR$ is
almost a join-based lattice.

(The join operation is exactly that of $\PRplus$.  In particular, the
top element $\oneR$ of $\PRplus$ is not already in $\PR$, since $\PR$
has homology, so we may adjoin that as the upper bound $\topone$ for
$\PR$.)
\end{proof}

\vspace*{0.2in}

Recall informative attribute release sequences from
Section~\ref{iarsnarrative} and Appendix~\ref{posetchains}.

\setcounter{currentThmCount}{\value{theorem}}
\setcounter{theorem}{\value{holerecCor}}
\begin{corollary}[Holes Defer Recognition]
Let $R$ be a nonvoid relation and let $(\sigma, \gamma) \in \PR$.

\vspace*{0.05in}

Define $Q = Q(\sigma,\gamma)$ as per Definition~\ref{restrictedlink}
and recall Definition~\ref{DefIdentLengths}, from
pages~\pageref{DefIdentLengths}--\pageref{restrictedlink}.

\vspace*{0.05in}

Suppose $\PQ$ is well-defined and has reduced integral homology in
dimension $k \geq 0$.

\vspace*{0.05in}

Then there are at least $(k+2)!$ distinct informative attribute
release sequences $y_1, \ldots, y_\ell$ for $R$, each with $\ell \geq
k+2$, such that $\psi_R(\{y_1, \ldots, y_\ell\}) = \sigma$.
Consequently, $\rslow(\sigma) \geq k+2$.

\end{corollary}
\setcounter{theorem}{\value{currentThmCount}}

\begin{proof}
By Corollary~\ref{holesreduceinference}, $\PQ$ contains at least
$(k+2)!$ maximal chains of length at least $k$.

\vsr

The rest of the argument is much like that in the proof of
Corollary~\ref{linklengths} from page~\pageref{linklengths}:

\vspace*{-0.1in}

\begin{itemize}

\addtolength{\itemsep}{-3pt}

\item Each maximal chain in $\PQ$ gives rise to a maximal chain in
$\PRplus$ at or above $(\sigma, \gamma)$.

\item Distinctness in $\PQ$ carries over to $\PRplus$.

\item In moving from $\PQ$ to $\PRplus$ one adds two elements:

\vspace*{-0.1in}

\begin{enumerate}

\item One adds $(\sigma, \gamma)$, corresponding to $\zeroQ$ in
$\PQplus$.

\item $\PQ$ has reduced homology, so no attribute is shared by all
individuals, either in $Q$ or in $R$.
\ One thus also adds the top element $\oneR$ of $\PRplus$.

\end{enumerate}

\end{itemize}

\vspace*{-0.05in}

Summarizing the length argument:\ Each distinct maximal chain in $\PQ$
of length at least $k$ gives rise to a distinct maximal chain at or
above $(\sigma, \gamma)$ in $\PRplus$ of length at least $k+2$, and
therefore a distinct informative attribute release sequence for $R$ of
length at least $k+2$.  By construction of $Q$ and by
Lemma~\ref{chaintoiars} on page~\pageref{chaintoiarsAppPage}, each such
iars identifies $\sigma$ in $R$.  So by
Definition~\ref{DefIdentLengths}, $\rslow(\sigma) \geq k+2$.

\vspace*{0.1in}

(How do we know that distinct maximal chains produce distinct iars?
 Because if two iars are the same, the chains must be the same, by the
 ``Moreover'' of Lemma~\ref{chaintoiars}.
 It is true that one may be able to obtain different iars from the same
 maximal chain, but our counting is over maximal chains, so provides a
 lower bound for the number of distinct iars.)
\end{proof}

\paragraph{Comment:}  Since $\PQ$ is well-defined and has reduced
 homology in nonnegative dimension, $Q$'s Dowker complexes are neither
 void nor empty.  Thus $\sigma\neq{X}$.  Along with the assumption
 $(\sigma, \gamma) \in \PR$, that means relation $Q(\sigma,\gamma)$
 models the link $\Lk(\dowx, \sigma)$.

\clearpage
\section{Obfuscating Strategies}
\label{obfuscatingstrategies}

Recall the discussion and terminology of Section~\ref{obfuscation}.

\vspace*{0.05in}

\noindent The primary goal of this appendix is to provide a proof of
Theorem~\ref{stratdelay} appearing on page~\pageref{stratdelay}.  In
addition, this appendix provides proof of some of the assertions in
the bullets on pages~\pageref{cycletheorems}--\pageref{cycletheoremsend}.

\vspace*{0.05in}

\noindent Once again, we first need to develop some tools:

\vspace*{0.05in}

\subsection{Source Complex}
\markright{Source Complex of a Graph with Uncertain Transitions}

Subsection~\ref{stratcomplex} introduced the strategy complex $\DG$ of
a graph $G=(V,\frakA)$.  Recall that every action $a\in\frakA$ has a
unique source state in $V$.  Given a set of actions $\calA \subseteq
\frakA$, we define the {\em start region} of $\calA$, denoted by
$\src(\calA)$, as

\vspace*{-0.1in}

$$\src(\calA) \;=\; \setdef{v\in V}{\hbox{$v$ is the source state of
    some action $a\in\calA$}}.$$ \label{SrcCmplxAppdef}

\vspace*{-0.05in}

One obtains another simplicial complex from $G$ via $\src$,
now on underlying vertex set $V$:

\vspace*{-0.05in}

$$\SG \;\; = \setdef{\src(\sigma)}{\sigma\in\DG}.$$

\vspace*{0.1in}

We refer to this complex as $G$'s {\em source complex}.

\vspace*{0.05in}

\begin{figure}[h]
\begin{center} 
\vspace*{0.05in}
\ifig{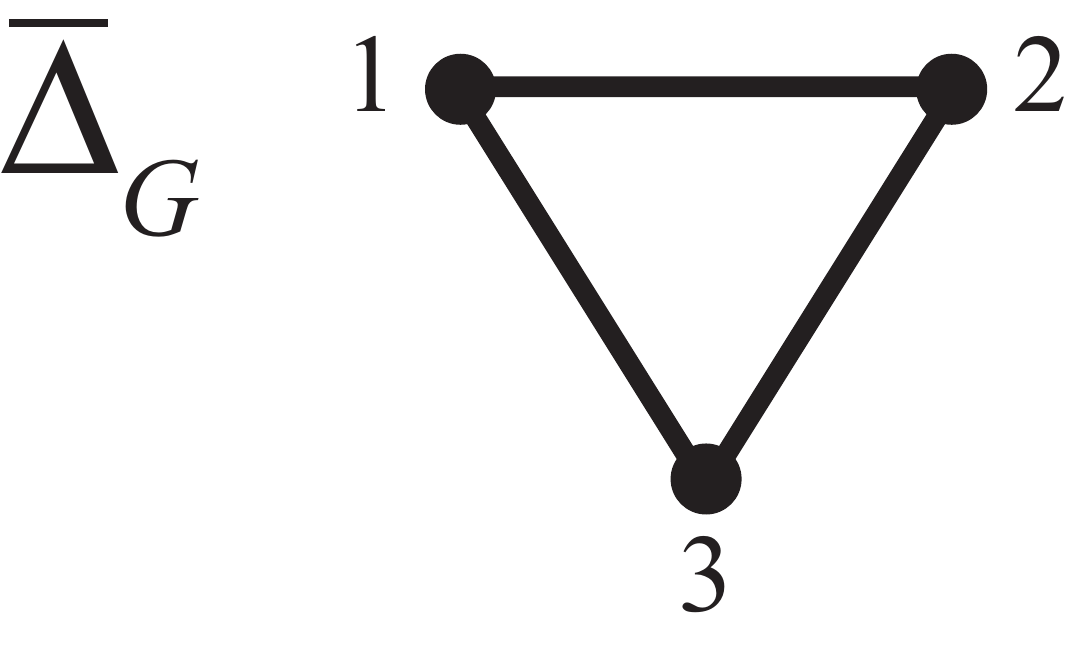}{height=1.25in}
\end{center}
\vspace*{-0.3in}
\caption[]{Source complex for the graph of Figure~\ref{graph214} on
  page~\pageref{graph214}.}
\label{src214}
\end{figure}

The map $\src\!: \F(\DG) \rightarrow \F(\SG)$ is a homotopy
equivalence, so $\DG \homot \SG$ \cite{tpr:strategies, tpr:plans}.
Consequently, the source complex of a fully controllable graph $G$ is
equal to the boundary complex of the full simplex on the graph's state
space $V$, that is, $\SG = \bndry{(V)}$.  For the graph of
Figure~\ref{graph214} on page~\pageref{graph214}, the source complex
is the boundary of a triangle, as shown in Figure~\ref{src214}.

\vspace*{0.1in}

\begin{figure}[h]
\begin{center}
\begin{minipage}{1.8in}{$\begin{array}{c|ccc}
\hbox{\Large $B$} & 1 & 2 & 3 \\[2pt]\hline
\sigma_1 &      & \one & \one \\
\sigma_2 & \one &      & \one \\
\sigma_3 & \one & \one &      \\
\sigma_4 & \one & \one &      \\
\end{array}$}
\end{minipage}
\quad
\begin{minipage}{0.5in}{$\begin{array}{c}
\hbox{Goal} \\[2pt]\hline
1 \\
2 \\
3 \\
3 \\
\end{array}$}
\end{minipage}
\end{center}
\vspace*{-0.15in}
\caption[]{Relation $B$ describes the source complex $\SG$ of the graph
  of Figure~\ref{graph214}.  Each row describes the start region of a
  maximal simplex of $\DG$, with $\DG$ as in Figure~\ref{strat214} on
  page~\pageref{strat214}.  The rightmost column again shows each
  maximal strategy's goal. (See also Figure~\ref{Arel214} on
  page~\pageref{Arel214}.)}
\label{Brel214}
\end{figure}

In Lemma~\ref{dowkerstrats} on page~\pageref{dowkerstrats} we saw that
$\DG = \dowAy$ for the action relation $A$ defined there.  We now see
that $\SG = \Phi_B$, with relation $B$ as defined in the next lemma.  As
example, Figure~\ref{Brel214} shows relation $B$ for the graph of
Figure~\ref{graph214} from page~\pageref{graph214}.

\begin{lemma}\label{dowkerbegin}
Let $G=(V,\frakA)$ be a graph as discussed in
Section~\ref{obfuscation}.  \,Let \,$\maxDG$ be the set of maximal
simplices of $\,\DG$.
\ Define relation $B$ on $\maxDG \times V$ by
$B = \setdef{(\sigma, v)}{v \in \src(\sigma) \;\hbox{and}\; \sigma \in \maxDG}.$\quad

\vso

Then $\dowBy = \SG$.  
\end{lemma}

\vspace*{-0.05in}

(Again, the proof is nearly definitional, so we omit it.)

\vspace*{0.02in}

(The ``$B$'' stands for ``Beginning'' --- while ``$S$'' for ``source''
might be desirable, we have already used $S$ to mean ``support''
elsewhere.)

\vspace*{0.1in}

How should we interpret the remaining Dowker complexes, $\dowAx$ and
$\dowBx$, for relations $A$ and $B$?  To answer this question, let us
look at the semantics of simplices in these complexes.  Suppose $\DG$
is not void or empty.  A nonempty simplex in $\dowAx$ represents a
{\em collection}\, of maximal simplices of $\DG$, namely maximal
simplices that have at least one action in common.  A nonempty simplex
in $\dowBx$ again represents a collection of maximal simplices of
$\DG$, now with at least one source state in common.  \ Thus $\dowAx
\subseteq \dowBx$.  \ Moreover, Dowker duality gives:

\begin{lemma}\label{inclusionequiv}
Let $G=(V,\frakA)$ be a graph as discussed in Section~\ref{obfuscation},
with $\mskip1mu{}V\!\neq\emptyset$.

Then the inclusion $\iota : \F(\dowAx) \rightarrow \F(\dowBx)$ is a
homotopy equivalence.

\end{lemma}

Comment: The assumption $V\!\mskip-1mu\neq\emptyset$ means $\DG$ and
$\SG$ are not void, so relation $B$ is not void.  If
$\mskip2mu{}V\!\neq\emptyset$ but $\frakA=\emptyset$, then technically
relation $A$ is void, but is is convenient to think of it as an
instance of the empty relation instead, with associated empty Dowker
complexes.

\begin{proof}
Consider the following diagram:

\vspace*{-0.12in}

{\Large $$
\begin{CD}
\F(\dowAx)   @. \;\xhookrightarrow{\phantom{01}\iota\phantom{10}}\;  @. \F(\dowBx)\phantom{.} \\
@A{\psi_A}AA @.         @A{\psi_B}AA  \\
\F(\dowAy)   @.         @.       \F(\dowBy)\phantom{.} \\[2pt]
\hspace*{4pt}\roteq       @.         @.   \hspace*{4pt}\roteq\phantom{.}     \\[-6pt]
\F(\DG)      @. \;\xrightarrow{\phantom{0}\src\phantom{0}}\;  @. \F(\SG).  \\
\end{CD}
$$}

\vspace*{0.02in}

Recall that $\psi_A$, $\psi_B$, and $\src$ are homotopy equivalences.

Let $\maxDG$ denote the maximal simplices of $\DG$.
Observe the following, for each $\sigma\in\F(\DG)$:

\begin{center}
\begin{minipage}{4in}

 $(\iota\circ\psi_A)(\sigma) =
  \setdef{\sigma^{\prime}\in\maxDG}{\sigma\subseteq\sigma^{\prime}}$.

\vspace*{0.1in}

 $(\psi_B\circ\src)(\sigma) =
  \setdef{\sigma^{\prime}\in\maxDG}{\src(\sigma)\subseteq\src(\sigma^{\prime})}$.

\vspace*{0.1in}

 If $\sigma\subseteq\sigma^{\prime}$, then $\src(\sigma)\subseteq\src(\sigma^{\prime})$.

\end{minipage}
\end{center}

Consequently, $(\iota\circ\psi_A)(\sigma) \leq
(\psi_B\circ\src)(\sigma)$ for every $\sigma\in\F(\DG)$, where
``$\leq$'' refers to the partial order on $\F(\dowBx)$.

We conclude that the two order-reversing poset maps $\iota\circ\psi_A$
and $\psi_B\circ\src$ are homotopic (see \cite{tpr:bjorner}, Theorem
10.11) and therefore that $\iota$ is a homotopy equivalence.
\end{proof}

\begin{lemma}\label{posetequiv}
Let $G=(V,\frakA)$ be a graph as discussed in Section~\ref{obfuscation},
with $\mskip1mu{}V\!\neq\emptyset$. \\
Then $\,\src$ induces a homotopy equivalence of
posets $\mskip2mu\PA \rightarrow \PB$ with explicit formula

\vspace*{-0.1in}

    $$(\tau, \sigma) \mapsto 
    \big((\psi_B\circ\src)(\sigma), (\phi_B\circ\psi_B\circ\src)(\sigma)\big).$$

\end{lemma}

\begin{proof}
Let $\clAy$ denote the image of the closure operator
$\phi_A \circ \psi_A : \F(\dowAy) \rightarrow \F(\dowAy)$
and let $\clBy$ denote the image of the closure operator
$\phi_B \circ \psi_B : \F(\dowBy) \rightarrow \F(\dowBy)$.
We then have the following diagram of homotopy equivalences:

\vspace*{-0.1in}

$$\PA \xrightarrow{\phantom{01}\pi_2\phantom{10}}
  \clAy \xhookrightarrow{\phantom{01}\iota\phantom{10}}
    \F(\dowAy) = \F(\DG) \xrightarrow{\phantom{0}\src\phantom{0}}
      \F(\SG) = \F(\dowBy) \xrightarrow{\phi_B \,\circ\, \psi_B}
        \clBy \xhookrightarrow{\phantom{01}\iota\phantom{10}}
          \PB.$$

\vspace*{0.05in}

(Here $\pi_2$ is projection onto the second coordinate, i.e.,
$\pi_2(\tau,\sigma)=\sigma$, and each of the occurrences of $\iota$ is
an inclusion.)

\vspace*{0.05in}

The composition of all these maps is an order-preserving poset map
with the specified formula.  The overall map is a homotopy equivalence
because each of its constituent maps is a homotopy equivalence.
\end{proof}

\begin{corollary}\label{posetformula}
If $G$ is fully controllable in Lemma~\ref{posetequiv}, then the
formula for the poset map becomes
    $(\tau, \sigma) \mapsto 
    \big((\psi_B\circ\src)(\sigma), \,\src(\sigma)\big).$
\end{corollary}

\begin{proof}
Since $G$ is fully controllable, $\dowBy = \SG = \bndry{(V)} \homot
\Snt$, with $n=\abs{V}$.  So $\dowBy$ has no free faces, implying that
$\phi_B\circ\psi_B$ is the identity, by Lemma~\ref{nofreefacesPrivacy}
on page~\pageref{nofreefacesPrivacy}.
\end{proof}

{\bf Two Observations:} \ Suppose that $G$ is a fully controllable
graph $(V,\frakA)$, with both $V$ and $\frakA$ nonempty. \ (i) No
action can appear in all maximal simplices of $\DG$, as that would mean
$\DG$ would be a cone, so not homotopic to a sphere.  Consequently,
$\oneA=(\maxDG, \gamma)$ has $\gamma=\emptyset$ (recall that $\maxDG$
is the collection of all maximal simplices of $\DG$).  \ (ii) Even if
all actions of $\frakA$ appear individually as vertices of $\DG$,
$\zeroA=(\tau, \frakA)$ has $\tau=\emptyset$, since $\src(\frakA) = V$
and $V\notin\bndry{(V)}$.

\vspace*{0.1in}

\noindent These observations mean that $\PA$ does {\em not}\, contain
either the top element $\oneA$ or the bottom element $\zeroA$ of
$\PAplus$, when $G$ is fully controllable.

\vspace*{0.1in}

\subsection{Delaying Strategy Identification}
\markright{Delaying Strategy Identification}
\label{delaystratident}

We now turn to the proof of the main theorem, the statement of which
is replicated here:

\setcounter{currentThmCount}{\value{theorem}}
\setcounter{theorem}{\value{stratIDdelay}}
\begin{theorem}[Delaying Strategy Identification]
Let $G=(V,\frakA)$ be a fully controllable graph, with $n=\abs{V} > 1$.
Let $A$ be the relation constructed as in Lemma~\ref{dowkerstrats} on
page~\pageref{dowkerstrats} and let $\PA$ be its associated
doubly-labeled poset.  Then:

For each $v\in V$, there exists a maximal strategy $\sigma_v \in \DG$
for attaining singleton goal state $v$ such that $\PA$ contains at least
$(n-1)!$ distinct maximal chains for identifying $\sigma_v$, with each
chain consisting of at least $n-1$ elements.
\end{theorem}
\setcounter{theorem}{\value{currentThmCount}}

\begin{proof}
Let $\PAop$ be $\PA$ but with the opposite partial order.  Then $\PAop$ is
almost a join-based lattice, with join operation for elements of $\PAop$
given by

\vspace*{-0.01in}

 $$ (\tau_1, \sigma_1) \,\join\, (\tau_2, \sigma_2)
          \; = \; \left\{\,
       \begin{aligned}
         & \big(\tau_1\inter\tau_2,\; (\phi_A\circ\psi_A)(\sigma_1\union \sigma_2)\big),
                      & & \hbox{when $\tau_1\inter\tau_2 \neq \emptyset$;}\\[4pt]
         & \topone,   & & \hbox{otherwise.}\\
       \end{aligned}\right.
 $$

\vspace*{0.07in}

The maximal elements of $\PAop$ are of the form $(\{\sigma\},
\sigma)$, with $\sigma$ varying over the maximal simplices of $\DG$.
Each minimal element of $\PAop$ is of the form
$\big(\psi_A(\tas), (\phi_A \circ \psi_A)(\tas)\big)$, with action $\ta$
some vertex of $\DG$.
\ (Aside: not every element of that form is necessarily minimal.)

\vspace*{0.06in}

Since $G$ is fully controllable, $\Delta(\PAop) \homot \Snt$.  So
$\Delta(\PAop)$ has reduced homology in dimension $k=n-2 \geq 0$.
\ By the proof of Theorem~\ref{manychains}, on
page~\pageref{manylongsubchains}, there exists a reduced homology
generator $z$ for $\Delta(\PAop)$, with support $S=\supp{z}$, such
that $\PAop$ contains, for each $p\in\Smax$, a collection of maximal
chains passing through $p$ with the following property: Even if one
merely considers the portions of the chains at and below $p$, the
collection contains at least $(n-1)!$ distinct such subchains and each
subchain has length at least $n-2$.  Each full chain, being maximal,
must be a path in $\PAop$ between some maximal element $(\{\sigma\},
\sigma)$ and some minimal element $\big(\psi_A(\tas), (\phi_A \circ
\psi_A)(\tas)\big)$.  Working upward from the bottom in $\PAplusop$
(which is equivalent to working downward from the top in $\PAplus$),
each such chain therefore gives rise to an informative action release
sequence for identifying $\sigma$, consisting of at least $n-1$
actions.  Moreover, there are at least $(n-1)!$ different such
sequences for that same strategy $\sigma$; we can hold fixed the
portion of any chain at and above $p$ in $\PAop$, while varying the
portion below $p$ in at least $(n-1)!\mskip0.5mu$ different ways.

\vso

Let $p\in\Smax$ and suppose $c$ is some maximal chain of $\PAop$ that
passes through $p$ and touches maximal element $(\{\sigma\},\, \sigma)$.
Pick $q\in{S}$, with $q \leq p$ (here, ``$\leq$'' is the partial order
on $\PAop$).  Write $q=(\tau_q, \sigma_q)$ and $p=(\tau_p, \sigma_p)$.
Even though $q$ may not be part of chain $c$, we can still conclude
that $\sigma_q \subseteq \sigma_p \subseteq \sigma$.  If additionally
$\src(\sigma_q) = V\setminus\{v\}$, then $\sigma$ at the top of $c$
must be a maximal strategy for attaining singleton goal state $v$.  In
order to prove the theorem, it is therefore enough to show that, for
any $v\in{V}$, some such $q\in{S}$ (and thus $p\in\Smax$) exists.

\vst

Recall the source relation $B$ from Lemma~\ref{dowkerbegin}.  Let
$\PBop$ be $\PB$ but with the opposite partial order.  Referring back
to the notation in the proof of Lemma~\ref{posetequiv}, and using the
fact that $G$ is fully controllable, one sees that
$\Delta(\PBop)\cong\Delta(\clBy)=\Delta(\F(\dowby))=\sd(\bndry{(V)})$,
with ``$\cong$'' meaning ``isomorphic'' and ``$\sd$'' meaning ``first
barycentric subdivision''.  The isomorphism holds by definition of
$\PB$.  The first equality holds because $\phi_B\circ\psi_B$ is the
identity when $G$ is fully controllable, as we saw in the proof of
Corollary~\ref{posetformula}.  The second equality amounts to the
definition of first barycentric subdivision, bearing in mind that
$\dowby=\SG=\bndry{(V)}$.

\vso

The homotopy equivalence of Lemma~\ref{posetequiv} carries over to
this setting as $\srceq : \Delta(\PAop) \rightarrow \sd(\bndry{(V)})$.
\ Corollary~\ref{posetformula} (or inspection of the diagram in the
proof of Lemma~\ref{posetequiv}) provides an explicit formula.  \
Specifically, for vertices $(\tau,\sigma)$ of
$\hspt\Delta(\PAop)$, one has
$\srceq(\tau,\sigma) = \src(\sigma)$.

\vst

Since $\srceq$ is a homotopy equivalence, the induced homomorphism
$\srceq_*$ on reduced homology
must map the reduced homology generator $z$ to a reduced homology
generator for the triangulated $(n\mskip-1mu{}-\mskip-1mu{}2)$-sphere
$\sd(\bndry{(V)})$.  Consequently, $\supp{\srceq_*(z)}$ must consist
of all vertices in $\sd(\bndry{(V)})$, meaning all nonempty proper
subsets of $V$.  In particular, for each $v\in{V}$, there is some $q =
(\tau_q, \sigma_q)\in\supp{z}$ such that $\src(\sigma_q) = \srceq(q) =
V\setminus\{v\}$, as desired.
\end{proof}

\clearpage
\subsection{Delaying Goal Recognition}
\label{delaygoalrecog}

The next lemma establishes the ``yes'' assertion in the bullet that
starts near the top of page~\pageref{cycletheoremsend}.

\vspace*{0.05in}

\begin{definition}[Complete Strategy]
Let $G=(V,\frakA)$ be a graph as discussed in Section~\ref{obfuscation}.
  A \mydefem{complete strategy for attaining state $v$} is a strategy
  $\sigma$ that has at least one action at every state other than $v$.
  \ In other words, $\sigma\in\DG$ and $\,\src(\sigma) = V\setminus\{v\}$.
\end{definition}

\vspace*{0.1in}

\begin{lemma}[Delaying Goal Recognition]\label{cyclechoices}
Let $G=(V,\frakA)$ be a fully controllable graph.  \  Let $n=\abs{V}$.
\ Suppose $n > 1$.  \ Let $s\in V$ be some desired goal state.

\vspace*{0.05in}

There exists a sequence of actions $a_1, a_2, \ldots, a_{n-1}$ in
$\frakA$ satisfying the following conditions:

\vspace*{0.1in}

\begin{minipage}{5in}
\begin{itemize}
\addtolength{\itemsep}{-3pt}

\item[(i)] $\{a_1, \ldots, a_{n-1}\}$ is a complete strategy for
  attaining $s$.

\item[(ii)] For each $i=1, \ldots, n-1$, let $\tau_i = \{a_1, \ldots,
  a_i\}$ and $W_i = \src(\tau_i)$.  Then, for each such $i$ and each
  $v \in V\setminus{W_i}$, there exists a complete strategy $\sigma$
  for attaining $v$, such that $\tau_i \subseteq \sigma \in \DG$.

\end{itemize}
\end{minipage}

\end{lemma}

\vspace*{0.1in}

\noindent {\bf Comments:}

\vso

(a) Condition (i) implies that no two of the actions $a_1, \ldots,
a_{n-1}$ have the same source state.

\vst

(b) Condition (ii) implies that an observer cannot predict the final
goal after seeing only a proper prefix of the sequence $a_1, a_2,
\ldots, a_{n-1}$.

\vst

(c) Condition (ii) further implies that the sequence $a_1, \ldots,
a_{n-1}$ forms an informative attribute release sequence for the
relation $A$ defined in Lemma~\ref{dowkerstrats} on
page~\pageref{dowkerstrats}.  Again, the reason is that an observer
cannot even predict any specific source state for the remaining
actions to be released after seeing only a proper prefix of the
sequence $a_1, a_2, \ldots, a_{n-1}$.

\vspace*{0.05in}

\begin{proof}
For the proof, we assume that $\frakA$ contains only deterministic and
nondeterministic actions, not stochastic ones.  The proof generalizes
to graphs that include stochastic actions (possibly in addition to
deterministic and nondeterministic actions), by an argument in
\cite{tpr:plans}.  The essence of that argument is that the source
complex of a graph does not change if one replaces stochastic
transitions by deterministic ones.

\vspace*{0.1in}

We sketch the rest of the proof, assuming all actions are deterministic
or nondeterministic.

\vspace*{0.1in}

Since $G$ is fully controllable, for each state in $V$ there must be a
deterministic transition {\em to} that state (from some other state).
Backchaining such transitions gives rise to a directed cycle of
deterministic actions, since the graph is finite.  If that cycle is
Hamiltonian, then we may choose $a_1, \ldots, a_{n-1}$ to be any
ordering of those $n$ deterministic actions except that we omit the
action whose source state is $s$.

Suppose instead that the directed cycle of deterministic actions
covers only a proper subset $W$ of the state space $V$.  Form a
quotient graph with state space $V^\prime = \{\wrep\} \union
\sdiff{V}{W}$, where $\wrep$ represents all of $W$ collapsed to a
point.  Inductively, the lemma's assertions hold for the quotient
graph.  One then needs to show how to combine the actions determined
by the quotient graph with the cycle on $W$ in order to satisfy the
lemma's assertions for the original graph $G$.  That argument is
straightforward if a bit tedious, so we omit it.
\end{proof}

\clearpage
\subsection{Hamiltonian Flexibility for Strategy Obfuscation}
\markright{Hamiltonian Flexibility for Strategy Obfuscation}
\label{hamiltoniangraph}

\noindent The next lemma establishes the Hamiltonian ``yes'' in the
bullet that starts near the bottom of page~\pageref{cycletheorems}.

\vspace*{0.05in}

\begin{definition}[Hamiltonian Action Cycle]
Let $G=(V,\frakA)$ be a graph as in Section~\ref{obfuscation},
  possibly with a mix of deterministic, nondeterministic, and
  stochastic actions.  Let $n=\abs{V}$ and assume $n > 1$.
  A sequence of actions $a_1, \ldots, a_n$ in $\frakA$ is a
  \mydefem{Hamiltonian cycle of actions}\, whenever all three of the
  following conditions hold:

\vspace*{0.1in}

\begin{minipage}{5.5in}
\begin{itemize}
\addtolength{\itemsep}{-3pt}
\item[(i)]  No two of the actions $a_1, \ldots, a_n$ have the same
  source state.

\item[(ii)] Each action $a_i$ is either deterministic or stochastic.

\item[(iii)] The source of action $a_{i+1}$ is a target of action
  $a_i$, for all $i=1, \ldots, n-1$, and the source of $a_1$ is a
  target of $a_n$.

\end{itemize}
\end{minipage}
\end{definition}

\paragraph{Observe:} Any proper subset of a Hamiltonian cycle of actions
  is a simplex in $\DG$.

(That observation requires understanding the definition of $\DG$ when
  stochastic actions are involved:  stochastic cycles are fine, so long
  as they are not recurrent.  See \cite{tpr:plans} for details.)

\vspace*{0.1in}

\begin{lemma}[Delaying Identification of a Given Strategy]\label{hamiltonianhelp}
Let $G=(V,\frakA)$ be a fully controllable graph.
\ Assume $\frakA$ contains a Hamiltonian cycle of actions $a_1, \ldots,
a_n$, with $n=\abs{V} > 1$.

\vso

Let $v\in V$ and suppose $\sigma_v$ is a maximal and complete strategy
in $\DG$ for attaining $v$.  Then $\sigma_v\!$ contains actions
$b_{n-1}, \ldots, b_1$ that constitute a complete strategy for
attaining $v$ and that form an informative attribute release sequence
for relation $A$.

\vso

(Recall: Relation $A$ was defined in Lemma~\ref{dowkerstrats} on
page~\pageref{dowkerstrats}; it models the maximal simplices of $\DG$ in
terms of their constituent actions.)

\end{lemma}

\begin{proof}
Let $v$ and $\sigma_v$ be as specified.

\vst

We can assume without loss of generality that $V=\{1, \ldots, n\}$,
that the source of $a_i$ is $i$ for all $i\in V$, and that $v = n$.

\vst

Now let $b_1, \ldots, b_{n-1}$ be any actions in $\sigma_v$ chosen so
that the source of $b_i$ is $i$, for $i = 1, \ldots, n-1$.
\ (If $b_i = a_i$ for some or all $i$, that is fine.)

\vso

Then $\{b_1, \ldots, b_{n-1}\}$ is itself a complete strategy for
attaining $v$.

\vso

We claim that the release order $b_{n-1}, \ldots, b_1$ constitutes an
informative attribute release sequence for relation $A$.
In fact, we will prove the stronger assertion:

\vspace*{0.1in}

\hspace*{0.35in}{\bf Claim:}\ \ \begin{minipage}[t]{4.8in}Pick some
$i\in\{1, \ldots, n\}$. \ Then:

For each $s\in\{n\}\union\{1, \ldots, i-1\}$, there exists a complete
strategy $\sigma_s\in\DG$ for attaining $s$, with $\{b_i, b_{i+1},
\ldots, b_{n-1}\}\subseteq\sigma_s$.

\vst

(Notation: $\{b_i, b_{i+1}, \ldots, b_{n-1}\}=\emptyset$ when $i=n$.
Similarly for other sets.)
\end{minipage}

\vspace*{0.1in}

Consequently, an observer cannot predict a specific source state for
the remaining actions to be released after seeing a proper prefix of
$b_{n-1}, \ldots, b_1$, so the action sequence is informative.

\vst

The claim certainly holds for $s=n$, using the original $\sigma_v$.
Now consider an $s\in\{1, \ldots, i-1\}$ and let
$\sigma_s = \{a_1, \ldots, a_{s-1}\}
  \union \{b_{s+1}, \ldots, b_{n-1}\}
  \union \{a_n\}$.
By arguments from \cite{tpr:plans}, $\sigma_s\in\DG$.  Finally,
observe that $\src(\sigma_s) = \sdiff{V}{\{s\}}$ and that $\sigma_s$
contains $\{b_i, b_{i+1}, \ldots, b_{n-1}\}$.
\end{proof}

{\bf Caution:}\ As mentioned on page~\pageref{chaininferencecaution},
just because $b_{n-1}, \ldots, b_1$ as produced by
Lemma~\ref{hamiltonianhelp} is an informative attribute release
sequence for $A$, that does not mean one should always release actions
in that fashion.  If the release protocol were so rigid, an adversary
familiar with the protocol would be able to infer much about the goal.
In particular, the target set of $a_{n-1}$ includes the goal state, so
if that action is deterministic and if always $b_{n-1}=a_{n-1}$, then
the adversary would be able to infer at least the goal from the first
action released.

\subsection{Example: A Rapidly Inferable Strategy}
\markright{Example: A Rapidly Inferable Strategy}
\label{shortiars}

\begin{figure}[h]
\begin{center}
\vspace*{-0.1in}
\ifig{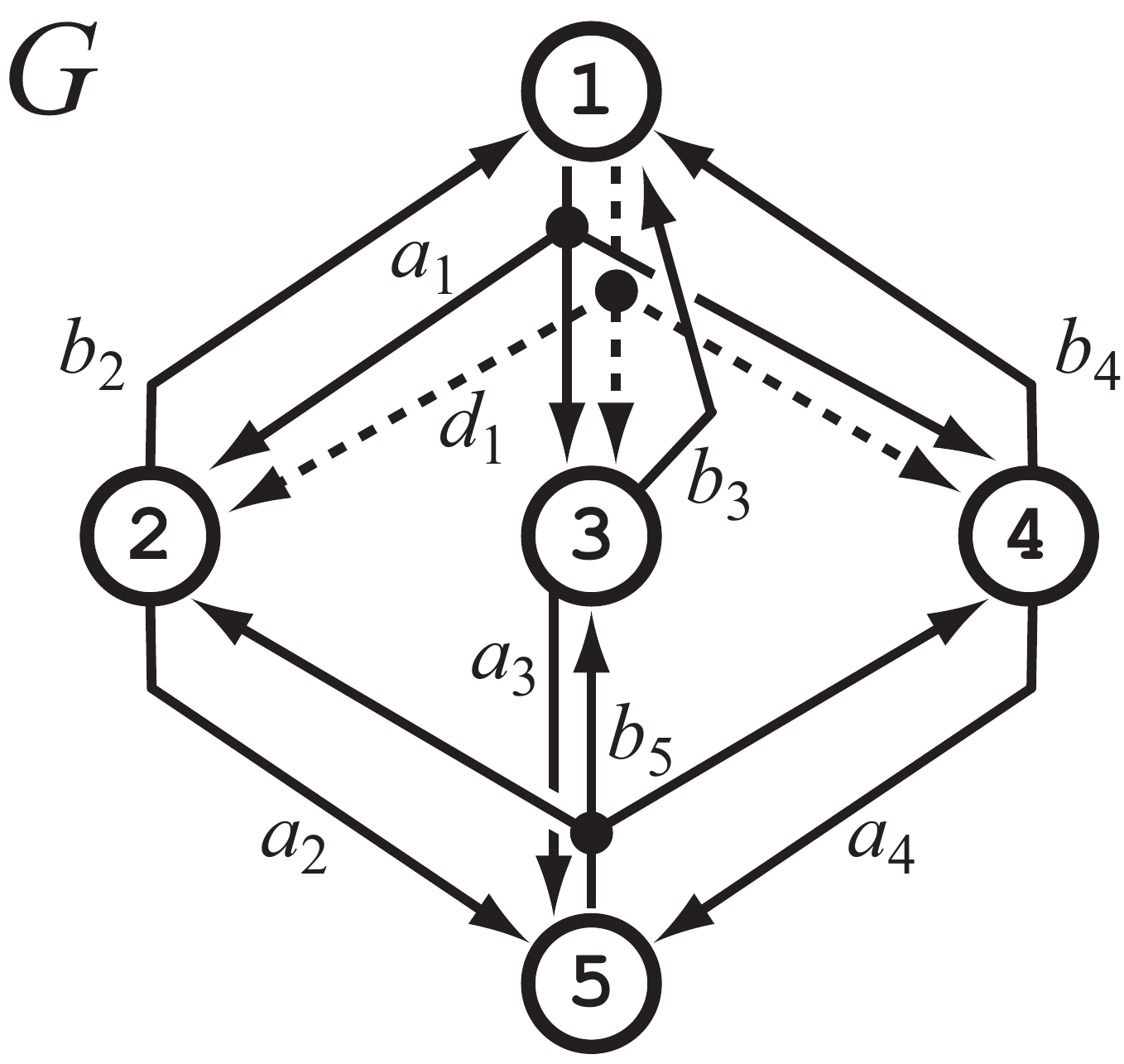}{height=3in}
\end{center}
\vspace*{-0.3in}
\caption[]{A graph with five states $\{1,2,3,4,5\}$, six deterministic
actions $\{a_2, a_3, a_4, b_2, b_3, b_4\}$, two nondeterministic actions
$\{a_1, b_5\}$, and one stochastic action $\{d_1\}$.}
\label{graph238}
\end{figure}

Figure~\ref{graph238} shows a fully controllable graph on fives states.
The graph contains six deterministic actions, two nondeterministic
actions, and one stochastic action.  (See \cite{tpr:strategies,
tpr:plans} to learn more about graphs that contain deterministic,
nondeterministic, and stochastic actions.)
The action relation $A$ for the graph appears in Figure~\ref{Arel238}.

\begin{figure}[h]
\begin{center}
\begin{minipage}{3.5in}{$\begin{array}{c|ccccccccc}
       A     & a_1  & a_2  & a_3  & a_4  & d_1  & b_2  & b_3  & b_4  & b_5  \\[2pt]\hline
\sigma_1     &      &      &      &      &      & \one & \one & \one & \one \\[2pt]
\sigma_2     &      &      &      &      & \one &      & \one & \one & \one \\[2pt]
\sigma_3     &      &      &      &      & \one & \one &      & \one & \one \\[2pt]
\sigma_4     &      &      &      &      & \one & \one & \one &      & \one \\[2pt]
\sigma_{234} & \one &      &      &      & \one &      &      &      & \one \\[2pt]
\sigma_5     & \one & \one & \one & \one & \one &      &      &      &      \\[2pt]
\sigma_{54}  &      & \one & \one & \one & \one & \one & \one &      &      \\[2pt]
\sigma_{53}  &      & \one & \one & \one & \one & \one &      & \one &      \\[2pt]
\sigma_{52}  &      & \one & \one & \one & \one &      & \one & \one &      \\[2pt]
\sigma_{15}  &      & \one & \one & \one &      & \one & \one & \one &      \\[2pt]
\end{array}$}
\end{minipage}
\quad
\begin{minipage}{0.5in}{$\begin{array}{c}
\hbox{Goal} \\[2pt]\hline
1 \\[2pt]
2 \\[2pt]
3 \\[2pt]
4 \\[2pt]
\{2, 3, 4\} \\[2pt]
5 \\[2pt]
5 \\[2pt]
5 \\[2pt]
5 \\[2pt]
\{1,5\} \\[2pt]
\end{array}$}
\end{minipage}
\end{center}
\vspace*{-0.15in}
\caption[]{Relation $A$ describes the strategy complex for the graph
  of Figure~\ref{graph238} in terms of its maximal strategies and
  their constituent actions.  The rightmost column further shows each
  maximal strategy's goal.}
\label{Arel238}
\end{figure}

Strategy $\sigma_5$ is a maximal and complete strategy for attaining
state $\{5\}$, consisting of actions $\{a_1, a_2, a_3, a_4, d_1\}$.
Observe that a longest informative action release sequence for this
strategy has length 3.  For instance, $d_1, a_2, a_1$ is such a
sequence.  Why is there no informative action release sequence longer
than 3?  Answer: As soon as one releases any one of the three actions
$\{a_2, a_3, a_4\}$, an observer can infer that the strategy cannot
contain the action $b_5$ and therefore, being maximal, must contain the
other two actions in the set $\{a_2, a_3, a_4\}$ as well.

\vspace*{0.025in}

On page~\pageref{cycletheorems} we asked whether one can always find an
informative action release sequence of length $n-1$ for a maximal
complete strategy.  (Here $n$ is the number of states in the graph.)  We
see now that the answer is ``no'', not in general.  Key to the current
counterexample are two characteristics:

\begin{enumerate}
\item The graph contains three equivalent actions, namely $\{a_2, a_3,
  a_4\}$: The actions have completely identical columns in relation
  $A$.  One could prune the graph to remove redundant such actions
  while preserving full controllability, but this would not ensure a
  ``yes'' answer in general.  The reason is that action equivalence is
  merely one of an infinite family of inferences one can construct
  with graphs.  For instance, we could build a subgraph with an
  ``any-2-actions-imply-4'' character and construct a new
  counterexample.

\item The actions $a_1$ and $d_1$ both have source state $1$ and target
  states $\{2, 3, 4\}$, yet are not equivalent.  The difference is that
  $a_1$ is nondeterministic, meaning that potentially an adversary could
  determine the target state attained any time the action is executed.
  In contrast, $d_1$ is merely stochastic, meaning each target state has
  some fixed nonzero probability of occurring any time the action is
  executed, independent over repeated execution instances.

  If we disallowed the stochastic action $d_1$, full controllability
  would necessitate more deterministic actions in the graph, thereby
  changing both the collection of maximal simplices and consequent
  inferences.  For instance, we might replace $d_1$ by three
  deterministic transitions.  Doing so would similarly change the
  problematic strategy $\sigma_5$, providing more actions and
  increasing the length of an informative action release sequence for
  identifying $\sigma_5$.  We then would indeed obtain an informative
  action release sequence of length 4.  Alternatively, if we
  disallowed the nondeterministic action $a_1$, then the problematic
  strategy $\sigma_5$ would simply disappear.  We would not even need
  to add any other actions, as the graph would remain fully
  controllable.  Whether and when such alternatives are realistic is
  application dependent.  We leave for future work exploration of more
  detailed conditions under which each maximal complete strategy in a
  graph has an informative action release sequence of length at least
  $n-1$.  Such characterizations might be useful in designing systems
  that either are or are not obfuscation-friendly.
  Appendix~\ref{puregraphs} initiates this exploration with a
  discussion of pure graphs.
\end{enumerate}

\subsection{Pure Nondeterministic Graphs and Pure Stochastic Graphs}
\label{puregraphs}

Let $G=(V,\frakA)$ be a graph whose actions may have uncertain
outcomes.  We say that $G$ is {\em pure nondeterministic} if each
action in $\frakA$ is nondeterministic (for the purposes of this
definition, we regard a deterministic action as a special instance of
a nondeterministic action).  We say that $G$ is {\em pure stochastic}
if each action in $\frakA$ is stochastic (now viewing a deterministic
action as a special instance of a stochastic action).  We may simply
say that $G$ is {\em pure\,} if it is either pure nondeterministic or
pure stochastic.

\vst

One can prove that every maximal strategy in the strategy complex of a
fully controllable pure graph contains an informative action release
sequence of length at least $n-1$, with $n = \abs{V}$.  This result
means that one can release at least $n-1$ actions informatively before
an observer can identify the strategy.

\vst

We will omit proof details, merely sketch the approach.  First, one
can show that if the Hamiltonian cycle in Lemma~\ref{hamiltonianhelp}
is formed from deterministic actions, then every maximal strategy with
a multi-state goal actually contains an informative action release
sequence (iars) of length $n$.  More specifically, suppose $\sigma$ is
such a strategy.  The iars for $\sigma$ consists of all the
Hamiltonian cycle edges that are in $\sigma$, along with one action in
$\sigma$ for each cycle edge that is missing from $\sigma$.  At least
two such cycle edges are missing, since $\sigma$ has at least two goal
states.  Consequently, the cycle breaks nicely into at least two
intervals.  For each missing cycle edge $e$, let $I_e$ denote the
interval of states between $e$'s target and the next missing edge's
source (based on the circular ordering induced by the Hamiltonian
cycle).  Include the endpoints (these could be identical, which is
fine).  One knows that every minimal nonface of $\DG$ containing $e$
must include some other action whose source lies in $I_e$ and whose
targets do {\em not}\, all lie inside $I_e$.  One can therefore
release one such action from $\sigma$ informatively, in place of $e$.

\vst

Next, observe that every fully controllable pure nondeterministic
graph contains a hierarchical decomposition of nested directed cycles,
whose union covers $V$, such that each cycle is deterministic
Hamiltonian in the quotient graph formed when one regards each of that
cycle's subcycles as a singleton state.  There are some details to
verify, but, given a maximal strategy $\sigma$ in the graph's strategy
complex, this hierarchical decomposition and the previous Hamiltonian
result yield an informative action release sequence of length at least
$n-1$ for $\sigma$.

\vsr

The approach is different for pure stochastic graphs.  A nice property
of minimal nonfaces in $\DG$, when $G$ is pure stochastic, is that
they form irreducible recurrent Markov chains and thus define fully
controllable subgraphs of $G$.  Moreover, each nonempty proper subset
of a minimal nonface defines an isotropic simplex of actions with respect
to $G$'s action relation.  When $G$ itself is fully controllable, one
can patch such subgraphs together expansively.  In particular, given a
maximal strategy $\sigma$ in $\DG$, one can choose as building blocks
fully controllable subgraphs that each consist solely of actions in
$\sigma$ plus one action not in $\sigma$.  One starts by considering
an action that moves off a goal state of $\sigma$.  That action cannot
be in $\sigma$, so gives rise to a minimal nonface all of whose other
actions do lie in $\sigma$.  One then expands outward repeatedly.
Once no additional expansion is possible, one passes to a quotient
graph by identifying all states covered thus far.  Inductively, one
can then repeat the expansion in the quotient graph.  Again, there are
some details to verify, but ultimately this process covers the graph's
state space and produces an informative action release sequence of
length at least $n-1$ for $\sigma$.

\clearpage
\subsection{Strategy Obfuscation Summary}
\markright{Strategy Obfuscation Summary}

We summarize the key points of this appendix with the following
theorem:

\begin{theorem}[Obfuscation]
\label{stratsummary}
Let $G=(V,\frakA)$ be a fully controllable graph, with $n=\abs{V} > 1$.

\vspace*{0.05in}

Let $\mskip2mu\sigma\mskip-1mu$ be a maximal strategy in $\DG$.

\vspace*{-0.05in}

\begin{enumerate}
\item[(a)] If $\,G$ is pure, then $\,\sigma$ contains at least
  \,$n-1$ actions \,$a_1, \ldots, a_{n-1}$\, that form an informative
  attribute release sequence (iars) with respect to $G$'s action
  relation.

  (This means that releasing the actions in the order $a_1, \ldots,
  a_{n-1}$ reduces the possible maximal strategies consistent with the
  actions released each time an action is released, but prevents
  identification of $\,\sigma$ until at least all $\,n-1$ actions have
  been released.)

\item[(b)]  This result can fail if \hspt$G$ contains a mix of
  deterministic, nondeterministic, and stochastic actions.

\item[(c)] Even if \hspt$G$ contains such a mix, the following is true:

Let $v\in{V}$.  Then there exist maximal strategies $\,\sigma_v$ and
$\,\tau_v$ in $\DG$ such that each strategy is a complete strategy for
attaining state $v$ and:
\begin{enumerate}
\item[(i)] $\sigma_v$ contains at least $(n-1)!$ distinct iars of
  length at least $n-1$ each.  (These iars may or may not be
  permutations of the same underlying $n-1$ actions.)
\item[(ii)] $\tau_v$ contains an iars of length at least $n-1$ that
  narrows the possible goal states consistent with the actions
  released by at most one state with each action released.
\end{enumerate}
\end{enumerate}
Caution:  The release order of the actions in an iars need not
correspond to the order in which actions might be executed at runtime.
\end{theorem}

\clearpage
\section{Morphisms and Lattice Generators}
\label{morphismappendix}

The aim of this appendix is to prove the claims of
Section~\ref{category}, ending with Theorem~\ref{latsurj}.  That
theorem shows how a surjective morphism of relations can use lattice
operations to fully cover its codomain's poset even when the poset
maps induced by the morphism are not themselves surjective.

\subsection{Morphisms}
\markright{Morphisms}
\label{morphismpropertiesApp}

Notation reminder: We frequently will be working with two relations:
$R$ is a relation on $\XR \times \YR$ and $Q$ is a relation on $\XQ
\times \YQ$.  In order to distinguish rows and columns between the two
relations, we also use notation of the form $\XRy$, $\YRx$, $\XQy$,
and $\YQx$.

Also, recall that a {\em set map\,} is a function between two sets.

\vspace*{0.1in}

Now recall the definition of {\em morphism} from page~\pageref{morphism}:

\setcounter{currentThmCount}{\value{theorem}}
\setcounter{theorem}{\value{morphismDef}}
\begin{definition}[Morphism]
Let $R$ be a relation on $\XR \times \YR$ and let $Q$ be a relation on
$\XQ \times \YQ$.
\ A \mydefem{morphism of relations} \ $f : R \rightarrow Q$ is a pair
of set maps:

\vspace*{-0.2in}

\begin{eqnarray*}
\fX &:& \XR \rightarrow \XQ\\[2pt]
\fY &:& \YR \rightarrow \YQ
\end{eqnarray*}

\noindent such that $\big(\fX(x), \,\fY(y)\big) \in Q$ whenever $(x,y) \in R$.
\end{definition}
\setcounter{theorem}{\value{currentThmCount}}

\vspace*{0.1in}

Throughout this appendix, {\tt{\char13}morphism{\char13}} refers to
Definition~\ref{morphism}.  When the time comes, we will refer to
{\tt{\char13}G-morphism{\char13}} explicitly (see again
Definitions~\ref{Gmorphism} and \ref{inducedgmorphisms} on
pages~\pageref{Gmorphism} and \pageref{inducedgmorphisms}).

\paragraph{Morphism Equality:}
Before proving properties about morphisms, we should give a notion of
morphism equality.  Suppose $g,h : R \rightarrow Q$ are two morphisms
of relations.  We say that $g=h$ if and only if
$\big(\gX(x), \,\gY(y)\big) = \big(\hX(x), \,\hY(y)\big)$ for all
$(x,y)\in R$.  In particular, we do not care what the constituent set
maps do on elements that are not relevant to the relations viewed as
sets of ordered pairs.  (Note: The condition stated is equivalent to
requiring $\gX(x)=\hX(x)$ and $\gY(y)=\hY(y)$ for all $(x,y)\in R$.)

\vspace*{0.1in}

The following lemma shows that the component maps of a morphism between
relations may be viewed as simplicial maps:

\setcounter{currentThmCount}{\value{theorem}}
\setcounter{theorem}{\value{SimplicialMorphisms}}
\begin{lemma}[Induced Simplicial Maps]
A morphism $f : R \rightarrow Q$ between nonvoid relations induces
simplicial maps between the Dowker complexes:

\vspace*{-0.2in}

\begin{eqnarray*}
\fX &:& \dowx \rightarrow \dowqx\\[2pt]
\fY &:& \dowy \rightarrow \dowqy
\end{eqnarray*}
\end{lemma}
\setcounter{theorem}{\value{currentThmCount}}

\begin{proof}
We need to show that $\fX(\sigma)\in\dowqx$ for all $\sigma\in\dowx$.

\vst

If $\sigma=\emptyset$, then $\fX(\sigma) = \emptyset \in \dowqx$,
since $Q$ is nonvoid.

\vst

If $\sigma = \{x_1, \ldots, x_k\}$, then $\fX(\sigma) =
\{\fX(x_1), \ldots, \fX(x_k)\}$.

\vst

Since $\emptyset\neq\sigma\in\dowx$, there exists $y\in\YR$ such that
$(x,y)\in R$ for all $x\in\sigma$.  Thus $(\fX(x), \,\fY(y))\in Q$ for
all $x\in\sigma$, by the definition of morphism.  So $\fY(y)$ is a
witness for $\fX(\sigma)$ in $Q$, telling us $\fX(\sigma)\in\dowqx$.

\vspace*{0.1in}

The argument for the map $\fY : \dowy \rightarrow \dowqy$ is similar.
\end{proof}

\paragraph{Comment:} The nonvoid requirement is an artifact, arising
because we sometimes regard void relations as having empty rather than
void Dowker complexes, in the context of links (see
Definitions~\ref{linkgamma} and \ref{linksigma} on
page~\pageref{linkgamma}, Definition~\ref{restrictedlink} on
page~\pageref{restrictedlink}, the comments about void relations on
page~\pageref{voidrelation}, and the hypotheses of Lemma~\ref{yLink}
on page~\pageref{yLink}).  The nonvoid requirement of
Lemma~\ref{simpmaps} avoids having to worry about mapping from an
artificially empty complex into a void one.

\vspace*{0.05in}

\setcounter{currentThmCount}{\value{theorem}}
\setcounter{theorem}{\value{morphismproperties}}
\begin{lemma}[Morphism Properties]
Assume the notation from before and that all relevant relations are nonvoid.
Let $f : R \rightarrow Q$ be a morphism of relations.  Then:

\begin{itemize}

\item[(i)] $\fX$ {\em and} $\fY$ are injective set maps $\implies$ $f$ is
            injective $\iff$ $f$ is a monomorphism.

  \vspace*{0.05in}

\item [(ii)] $f$ surjective $\implies$ $f$ epimorphism $\iff$
  $\fX$ {\em and} $\fY$ are surjective set maps.

  \vspace*{0.05in}

  (Additional conditions for that last $\iff$: The $\Longrightarrow$
  direction assumes that $Q$ has no blank rows or columns, while the
  $\Longleftarrow$ direction assumes that $R$ has no blank rows or
  columns.)

\end{itemize}

\vspace*{-0.05in}

\noindent The two uni-directional implications $\implies$ above are strict.

\vspace*{0.075in}

\begin{itemize}

\item[(iii)] If $\fX : \dowx \rightarrow \dowqx$ is surjective and $Q$
  has no blank rows, then $\fX : \XR \rightarrow \XQ$ is surjective.

  Similarly for $\fY$, now assuming that $Q$ has no blank columns.

The converses need not hold.  Indeed, $f$ itself can be surjective but
the maps of simplicial complexes need not be (as we saw with the maps
of page~\pageref{MTproj} and as one can see with simpler examples as
well).

\item[(iv)] If $\fX : \XR \rightarrow \XQ$ is injective, then $\fX :
  \dowx \rightarrow \dowqx$ is injective.   The converse holds if $R$
  has no blank rows.

  Similarly for $\fY$, now assuming that $R$ has no blank columns for
  the converse.

\end{itemize}

\end{lemma}
\setcounter{theorem}{\value{currentThmCount}}

\begin{proof}
We will prove the various implications.  Strictness, i.e., failure of
converses, where mentioned above, can be seen readily with simple
examples.

\vspace*{0.1in}

\noindent\underline{Part (i):}

\vspace*{0.05in}

\noindent (a) Let $\fX$ and $\fY$ be injective set maps.

\vspace*{0.05in}

Suppose $(\fX(x^\prime), \fY(y^\prime)) = (\fX(x), \fY(y))$.  Then
$\fX(x^\prime)=\fX(x)$, so $x^\prime=x$.

\vst

And $\fY(y^\prime)=\fY(y)$, so $y^\prime=y$.  So $f$ is injective as
a set map of ordered pairs.

\vspace*{0.1in}

\noindent (b) Let $f$ be injective as a set map of ordered pairs.

\vspace*{0.05in}

Suppose $g,h : S \rightarrow R$ are morphisms such that $f \circ g = f
\circ h$.

Suppose $(x,y) \in S$.
\ 
By assumption, \ $\big(\fX(\gX(x)),\, \fY(\gY(y))\big) \,=\, \big(\fX(\hX(x)),\, \fY(\hY(y))\big)$.

Since $f$ is injective, \ $(\gX(x),\, \gY(y)) \,=\; (\hX(x),\, \hY(y))$.

So $g = h$, by our notion of equality.  Consequently, $f$ is a monomorphism.

\clearpage

\noindent (c) Let $f$ be a monomorphism.

\vspace*{0.05in}

Suppose $f(x,y) = f(x^\prime, y^\prime)$ but
$(x,y)\neq(x^\prime,y^\prime)$.
Let $S$ be the relation consisting of the single element
$\{(I, \alpha)\}$, with $I$ and $\alpha$ new symbols:

\vspace*{-0.3in}

$$\hspace*{2in}\begin{array}{c|c}
S & \alpha \\[2pt]\hline
I & \one \\
\end{array}$$

Define two morphisms $g,h : S \rightarrow R$ by:

\vspace*{-0.05in}

$$\begin{matrix}
\gX : I \mapsto x,      & & & \hX : I \mapsto x^\prime, \cr
\gY : \alpha \mapsto y, & & & \hY : \alpha \mapsto y^\prime. \cr
\end{matrix}$$

Then $g\neq h$, but $f \circ g = f \circ h$, a contradiction.  So $f$
is injective.

\vspace*{0.15in}

\noindent\underline{Part (ii):}

\vspace*{0.05in}

\noindent (a) Let $f$ be surjective as a set map of ordered pairs.

\vspace*{0.05in}

Suppose $g,h : Q \rightarrow S$ are morphisms such that $g \circ f = h
\circ f$.

Suppose $(x^\prime, y^\prime) \in Q$.

By surjectivity, there exists $(x, y) \in R$ such that $(\fX(x),
\fY(y)) = (x^\prime, y^\prime)$.
\ So:

\vspace*{-0.25in}

$$\hspace*{0.25in}\big(\gX(x^\prime),\, \gY(y^\prime)\big)
  \,=\, \big(\gX(\fX(x)),\, \gY(\fY(y))\big)
  \,=\, \big(\hX(\fX(x)),\, \hY(\fY(y))\big)
  \,=\, \big(\hX(x^\prime),\, \hY(y^\prime)\big).$$

\vspace*{-0.05in}

Thus $g=h$ and we see that $f$ is an epimorphism.

\vspace*{0.12in}

\noindent (b) Assume $Q$ has no blank rows or columns and let $f$ be
an epimorphism.

\vspace*{0.05in}

Suppose $\fY$ is not surjective, so there exists $\ystar \in \YQ
\setminus (\fY(\YR))$.

Let $S$ be the relation consisting of two elements
$\{(I, \alpha), (I, \beta)\}$,
with $I$, $\alpha$, $\beta$ new symbols:

\vspace*{-0.1in}

$$\begin{array}{c|cc}
S & \alpha & \beta \\[2pt]\hline
I & \one   & \one \\
\end{array}$$

\vst

Define two morphisms $g,h : Q \rightarrow S$ by:

\vspace*{-0.05in}

$$\begin{array}{lclcl}
\gX(x) = I      & \quad\hbox{and}\quad & \hX(x) = I,      & & \hbox{for every $\;x\;\in\;\XQ$}; \\[6pt]
\gY(y) = \alpha & \quad\hbox{and}\quad & \hY(y) = \alpha, & & \hbox{for every $\;y\;\in\;\YQ\setminus\{\ystar\}$}; \\[3pt]
\gY(\ystar) = \alpha & \quad\hbox{and}\quad & \hY(\ystar) = \beta. & & \\
\end{array}$$

\vst

Since $\ystar\in\YQ$ and $Q$ has no blank columns there is at least
one $\xstar\in\XQ$ such that $(\xstar, \ystar)\in Q$.  So $g \neq h$.

Observe that $g \circ f = h \circ f$ since $\ystar$ does not appear in
the image of $\fY$, contradicting $f$ being an epimorphism.

\vst

The argument showing that $\fX$ is surjective is similar.

\vspace*{0.1in}

\noindent (c) Assume $R$ has no blank rows or columns and let $\fX$
and $\fY$ be surjective.

\vspace*{0.05in}

Suppose $g,h : Q \rightarrow S$ are morphisms such that $g \circ f = h
\circ f$.

Suppose $(x,y)\in Q$.  We need to show that $\gX(x)=\hX(x)$ and
$\gY(y)=\hY(y)$, as that means $g=h$, given our definition of equality.
We will make the argument for the $X$ coordinate; the $Y\mskip-2mu$
argument is similar.

Since $\fX$ is surjective, there exists $\xbar\in\XR$ such that
$\fX(\xbar)=x$.  Since $R$ has no blank rows, there exists
$\ybar\in\YR$ such that $(\xbar, \ybar)\in R$.

Since $(g \circ f)(\xbar, \ybar) = (h \circ f)(\xbar, \ybar)$, one
obtains  $\gX(x) = g_X(f_X(\xbar)) = h_X(f_X(\xbar)) = \hX(x)$.

\clearpage

\noindent\underline{Part (iii):}

\vspace*{0.05in}

Suppose $Q$ has no blank rows and suppose $\fX : \dowx \rightarrow
\dowqx$ is surjective as a simplicial map.

\vspace*{0.05in}

Suppose $x\in \XQ$.  Since $Q$ has no blank rows, $\{x\}$ is a vertex
of $\dowqx$, so there is some simplex $\sigma\in\dowx$ such that
$\fX(\sigma) = \{x\}$, with $\fX$ viewed as a simplicial map.
Necessarily, $\sigma\neq\emptyset$, so pick some $\xbar\in\sigma$.
Then $\fX(\xbar) = x$, with $\fX$ now viewed as a set map.  Thus the
set map $\fX : \XR \rightarrow \XQ$ is surjective.

\vspace*{0.05in}

The argument for $\fY$ assuming $Q$ has no blank columns is similar.

\vspace*{0.2in}

\noindent\underline{Part (iv):}

\vspace*{0.05in}

\noindent (a) Let $\fX$ be injective as a set map $\XR \rightarrow \XQ$.
\ 
Consider $\fX$ as a simplicial map $\dowx \rightarrow \dowqx$.

\vspace*{0.05in}

Suppose $\fX(\sigma) = \kappa = \fX(\tau)$, with $\sigma,\tau\in\dowx$
and $\kappa\in\dowqx$.

\vso

If $\kappa=\emptyset$, then necessarily $\sigma=\tau=\emptyset$.
Otherwise, $\sigma\neq\emptyset$ and $\tau\neq\emptyset$, so
let $x\in\sigma$.  Then $\fX(x)\in\kappa$.  So there exists
$x^\prime\in\tau$ such that $\fX(x^\prime)=\fX(x)$.  Since $\fX$ is
injective as a set map, that says $x^\prime=x$.  Thus
$\sigma\subseteq\tau$.  A similar argument shows the reverse
inclusion, so $\sigma=\tau$.  Thus $\fX$ is injective as a simplicial
map.

\vspace*{0.1in}

\noindent (b) Assume $R$ has no blank rows and let $\fX$ be injective
as a simplicial map $\dowx \rightarrow \dowqx$.

\vspace*{0.05in}

Consider $\fX$ as a set map $\XR \rightarrow \XQ$ and suppose
$\fX(x)=\fX(x^\prime)$.  Since $R$ has no blank rows, both $\{x\}$ and
$\{x^\prime\}$ are vertices in $\dowx$.  That means
$\fX(\{x\})=\fX(\{x^\prime\})$ when we view $\fX$ as a simplicial map,
so $\{x\}=\{x^\prime\}$ by injectivity, i.e., $x=x^\prime$.  So we see
that $\fX$ is injective as a set map.

\vspace*{0.05in}

A similar argument holds for the assertions regarding $\fY$.
\end{proof}

\subsection{G-Morphisms}
\markright{G-Morphisms}
\label{gmorphisms}

Recall the material of Section~\ref{gmorphismsection}, starting on
page~\pageref{gmorphismsection}.

\setcounter{currentThmCount}{\value{theorem}}
\setcounter{theorem}{\value{containmentlemma}}
\begin{lemma}[Witness Containment]
Let $f : R \rightarrow Q$ be a morphism of nonvoid relations.  Then:

\vspace*{0.1in}

\qquad
\begin{minipage}{3.7in}
\begin{itemize}

\item[(a)] $(\fY \circ \phi_R)(\sigma) \;\subseteq\; (\phi_Q \circ \fX)(\sigma)$,
           \ for every $\sigma \in \dowx$,

\item[(b)] $(\fX \circ \psi_R)(\gamma) \;\subseteq\; (\psi_Q \circ \fY)(\gamma)$,
           \ for every $\gamma \in \dowy$.

\end{itemize}
\end{minipage}
\end{lemma}
\setcounter{theorem}{\value{currentThmCount}}

\begin{proof}
Observe that
$(\fY \circ \phi_R)(\emptyset) = \fY(\YR) \subseteq \YQ =
\phi_Q(\emptyset) = (\phi_Q \circ \fX)(\emptyset)$.

\vspace*{0.05in}

Now let $\emptyset \neq \sigma \in \dowx$.  Let
$y\in\phi_R(\sigma)\neq\emptyset$.  Then $(x,y)\in R$ for every
$x\in\sigma$.  Thus $(\fX(x), \fY(y))\in Q$ for every $x\in\sigma$.
So $\fY(y) \in \phi_Q(\fX(\sigma))$.  This is true for all
$y\in\phi_R(\sigma)$, telling us $\fY(\phi_R(\sigma)) \subseteq
\phi_Q(\fX(\sigma))$.

\vspace*{0.05in}

The argument for assertion (b) is similar.
\end{proof}

\setcounter{currentThmCount}{\value{theorem}}
\setcounter{theorem}{\value{homotopicfacemaps}}
\begin{corollary}[Homotopic Face Maps]
Let $f : R \rightarrow Q$ be a morphism of nonvoid relations.  Then:

\qquad
\begin{minipage}{4.75in}
\begin{itemize}

\item[(a)] $\fX$ and $\,\psi_Q \circ \fY \circ \phi_R$ are homotopic
  poset maps $\,\Fdowx \rightarrow \Fdowqx$,

\item[(b)] $\fY$ and $\,\phi_Q \circ \fX \circ \psi_R$ are homotopic
  poset maps $\,\Fdowy \rightarrow \Fdowqy$.

\end{itemize}
\end{minipage}
\end{corollary}
\setcounter{theorem}{\value{currentThmCount}}

\begin{proof}
Let $\sigma\in\Fdowx$.

\vspace*{0.05in}

By Lemma~\ref{containment},
$(\fY \circ \phi_R)(\sigma) \;\subseteq\; (\phi_Q \circ \fX)(\sigma)$.

\vspace*{0.05in}

Therefore
$(\psi_Q \circ \fY \circ \phi_R)(\sigma) \;\supseteq\; (\psi_Q \circ \phi_Q \circ \fX)(\sigma)$.

\vso

So $(\psi_Q \circ \fY \circ \phi_R)$ and $(\psi_Q \circ \phi_Q \circ \fX)$
are homotopic maps (see \cite{tpr:bjorner}, Theorem 10.11).

\vso

Since $\clsqx$ is homotopic to the identity on $\Fdowqx$, part (a)
follows.

\vst

The proof of (b) is similar.
\end{proof}

\vspace*{0.05in}

\setcounter{currentThmCount}{\value{theorem}}
\setcounter{theorem}{\value{homotopicposetmaps}}
\begin{corollary}[Homotopic G-Morphisms]
Let $f : R \rightarrow Q$ be a morphism of nonvoid relations.  The
induced G-morphisms, as given by the poset maps $\fXg, \fYg : \PR
\rightarrow \PQ$ of Definition~\ref{inducedgmorphisms} on
page~\pageref{inducedgmorphisms}, are homotopic.
\end{corollary}
\setcounter{theorem}{\value{currentThmCount}}

\paragraph{}$\phantom{0}$

\vspace*{-0.4in}

\label{posethomotopyAppPage}
\begin{proof}
See Figure~\ref{morphismdiagram} on page~\pageref{morphismdiagram} for
the underlying maps comprising the G-morphisms.  The G-morphisms are
defined as follows:

\vspace*{0.1in}

For all $(\sigma,\gamma)\in\PR$:

\vspace*{-0.1in}

$$\fXg(\sigma,\gamma) \;=\;
  (\sigma^\prime, \gamma^\prime), \quad
\hbox{with}\quad \sigma^\prime \;=\; (\psi_Q \circ \fY \circ \phi_R)(\sigma)
\quad \hbox{and}\quad \gamma^\prime \;=\; \phi_Q(\sigma^\prime).$$

\vspace*{-0.15in}

$$\fYg(\sigma,\gamma) \;=\;
  (\sigma^{\prime\prime}, \gamma^{\prime\prime}), \quad
\hbox{with}\quad \gamma^{\prime\prime} \;=\; (\phi_Q \circ \fX \circ \psi_R)(\gamma)
\quad \hbox{and}\quad \sigma^{\prime\prime} \;=\; \psi_Q(\gamma^{\prime\prime}).$$

\vspace*{0.1in}

These definitions make sense because $\fX$ and $\fY$ map nonempty
simplices to nonempty simplices and because the images of $\psi_Q$ and
$\phi_Q$ may be viewed as lying in $\PQ$, by
Corollary~\ref{mappedsimplex} on page~\pageref{mappedsimplex}.
(Similarly, the images of $\psi_R$ and $\phi_R$ may be viewed as lying
in $\PR$ --- In fact, as used above, these maps are simply switching
between the $\sigma$ and $\gamma$ components (``labels'') of the given
element $(\sigma, \gamma)$ in $\PR$.) Observe that $\fXg$ and $\fYg$ are
order-preserving poset maps.

\vspace*{0.1in}

Applying Lemma~\ref{containment} and since $(\sigma,\gamma)\in\PR$:

\vspace*{-0.1in}

$$(\fY \circ \phi_R)(\sigma) 
   \;\subseteq\; (\phi_Q \circ \fX)(\sigma)
   \;=\; (\phi_Q \circ \fX \circ \psi_R)(\gamma) 
   \;=\; \gamma^{\prime\prime}.$$

Consequently:

\vspace*{-0.15in}

$$\sigma^\prime 
  \;=\; (\psi_Q \circ \fY \circ \phi_R)(\sigma)
  \;\supseteq\; \psi_Q(\gamma^{\prime\prime})
  \;=\; \sigma^{\prime\prime}.$$

\vspace*{0.05in}

So the maps are homotopic (see \cite{tpr:bjorner}, Theorem 10.11).
\end{proof}

\subsection{Lattice Generators}
\markright{Lattice Generators}
\label{latticegen}

We turn now to the main result.

\vspace*{0.05in}

(Recall that a relation is {\em tight} when it has no blank rows or
columns.)

\begin{lemma}[Generators in Image]\label{generatorimages}
Let $f : R \rightarrow Q$ be a surjective morphism between nonvoid
tight relations.

\vst

Suppose $q \in \PQ$ is of the form $\big(\XQy,\, (\clsqy)(\ys)\big)$,
for some $y\in\YQ$.

\vspace*{0.05in}

Then there exist $q_1, \ldots, q_k$ in the image of $\fXg : \PR
\rightarrow \PQ$, with $k\geq 1$, such that $q=\bigvee_{i=1}^kq_i$.

\vst

(Here, $\bigvee$ is the join operation of $\PQplus$.)
\end{lemma}

\begin{proof}
By Lemma~\ref{morphismprop}(ii), the component functions $\fX : \XR
\rightarrow \XQ$ and $\fY : \YR \rightarrow \YQ$ are surjective.
\quad
Since $\fY$ is surjective, $\fYinv(\ys) = \{y_1, \ldots, y_k\}
\subseteq \YR$, for some $k \geq 1$.

\vspace*{0.1in}

For each $i=1,\ldots,k$, observe and define the following:

\begin{itemize}

\addtolength{\itemsep}{-1pt}

\item Since $R$ has no blank columns, $\XRyi \neq \emptyset$, so
$\big(\XRyi,\, (\clsy)(\ysi)\big) \in \PR$.

\item Define $\sigma_i$ as the ``$\sigma^\prime$-component'' of
$\,\fXg\big(\XRyi,\, (\clsy)(\ysi)\big)$, meaning:
                 
\vspace*{-0.15in}

$$\sigma_i
    \;=\; (\psi_Q \circ \fY \circ \phi_R)\big(\XRyi\big)
    \;=\; \psi_Q(\gamma)
    \;=\; \biginter_{\ybar\in\gamma}\XQybar,
          \quad \hbox{with $\gamma = \fY\big((\clsy)(\ysi)\big)$}.$$

\vspace*{-0.1in}

\item Observe that $y\in\gamma$, since $y=\fY(y_i)$ and
  $y_i\in(\clsy)(\ysi)$.  Therefore $\sigma_i \subseteq \XQy$.

\item Define $q_i = (\sigma_i, \gamma_i) \in \PQ$, with $\gamma_i =
  \phi_Q(\sigma_i)$.  So $q_i$ is in the image of $\fXg : \PR
  \rightarrow \PQ$.

\end{itemize}

We need to show that $q = \bigvee_{i=1}^kq_i$.  \quad Expanding, we see:

$$\bigvee_{i=1}^kq_i 
   \;\;=\;\; \bigg(\big(\clsqx\big)\Big(\bigunion_{i=1}^k\sigma_i\Big),\;\;
                            \biginter_{i=1}^k\gamma_i\bigg).$$

\vspace*{0.05in}

By the third bullet above, $\bigunion_{i=1}^k\sigma_i \subseteq \XQy$, so:

$$\bigunion_{i=1}^k\sigma_i
  \;\subseteq\;
  \big(\clsqx\big)\Big(\bigunion_{i=1}^k\sigma_i\Big) 
  \;\subseteq\;
     (\clsqx)\big(\XQy)
  \;=\; \XQy.$$

\vspace*{0.05in}

We will establish $\XQy \subseteq \bigunion_{i=1}^k\sigma_i$, thereby
completing the proof.

\vspace*{0.1in}

Let $\xbar\in\XQy$.
\ So $(\xbar, y) \in Q$.

\vst

By surjectivity of $f$, there exists $(\xhat, \yhat) \in R$ such that
$\fX(\xhat)=\xbar$ and $\fY(\yhat)=y$.

\vspace*{0.05in}

Now $\yhat=y_j$, for some $j\in\{1, \dots, k\}$, as defined earlier.
Thus $\xhat\in\XRyj$.

\vst

Consequently, for every $z\in(\clsy)(\ysj)$,
$(\xhat,z)\in R$ and so $(\fX(\xhat), \fY(z))\in Q$.

\vsr

That means $(\xbar,\ybar)\in Q$ for every $\ybar\in\fY\big((\clsy)(\ysj)\big)$.

\vsr

Therefore, $\xbar\in\sigma_j\subseteq\bigunion_{i=1}^k\sigma_i$
and we conclude that $\XQy \subseteq \bigunion_{i=1}^k\sigma_i$.

\vsr

(Note: $\XQy$ need not lie in a single $\sigma_j$, since $j$ depends on $\xbar$.)
\end{proof}

\begin{corollary}\label{directimage}
Assume the hypotheses of Lemma~\ref{generatorimages}.

Suppose further that for some $y_i\in\fYinv(\ys)$,
$(\clsy)(\ysi)=\ysi$.

Then $q$ is itself in the image of $\fXg : \PR \rightarrow \PQ$.
\end{corollary}

\begin{proof}
In the proof of Lemma~\ref{generatorimages}, we see that now
$\fY\big((\clsy)(\ysi)\big) = \ys$, so $\sigma_i = \XQy$.
\end{proof}

\paragraph{Comment:} Corollary~\ref{directimage} helps to explain the
example of pages~\pageref{surjqueststop} and \pageref{surjectivity},
in which a surjective morphism generated the entire poset of its
codomain even though the induced maps on the Dowker complexes were not
surjective.  Namely:

\vst

In the M\"obius relation $M$ of page~\pageref{moebius}, singletons are
unmoved by the closure operators.  In the tetrahedral relation $T$ of
page~\pageref{tetrahedra2}, maximal simplices are dual to singletons.
Intersections of maximal simplices in the tetrahedral relation generate
all of $\PT$.  These maximal simplices come from dualizing images of
singletons of the M\"obius relation.  \ Consequently:

\vspace*{-0.05in}

\begin{itemize}

\item Even though one merely has $\fX(\{1,2,5\})=\{1,4\}$, one further
sees that $\fXg(\{1,2,5\}, \tas) = (\{1,3,4\}, \tas)$.  The G-morphism
$\fXg$ therefore supplies the apparently uncovered simplex $\{1,3,4\}$
of $\dowTx$.
\ \  Similarly, $\fYg$ supplies $\{\ta,\tb,\td\}$ in $\dowTy$.

\item Previously, in Table~\ref{MTposetmaps} on
page~\pageref{MTposetmaps}, we saw that the element $(13, \ta\tc)$ of
$\PT$ did not itself appear in the images of the maps $\fYg$ and
$\fXg$.  However, $(13, \ta\tc)$ {\em \hspc{}does\,} appear as the
join or meet of elements in the images:

\vspace*{-0.25in}

\begin{eqnarray*}
(13, \ta\tc) &\,=\,& (134, \ta) \,\meet\, (123, \tc),
     \quad\hbox{with both arguments to $\meet$ in the image of $\fXg$;}\\[2pt]
(13, \ta\tc) &\,=\,& (1, \ta\tb\tc) \,\join\, (3, \ta\tc\td),
     \quad\hbox{with both arguments to $\join$ in the image of $\fYg$.}\\
\end{eqnarray*}

\end{itemize}

\vspace*{-0.1in}

\noindent More generally, the following theorem describes the process:

\setcounter{currentThmCount}{\value{theorem}}
\setcounter{theorem}{\value{LatticeGenerators}}
\begin{theorem}[Lattice Surjectivity]
Let $R$ and $Q$ be nonvoid tight relations.  
Suppose $f : R \rightarrow Q$ is a surjective morphism (in the sense
of Definition~\ref{morphism}).
 \ For any $q \in \PQ$:

\vspace*{-0.05in}

$$ q \;=\; \bigwedge_j \bigvee_i q_{ji}, \quad \hbox{with each
  $q_{ji}$ in the image of $\fXg : \PR \rightarrow \PQ$,}$$

\vspace*{-0.05in}

$$ q \;=\; \bigvee_k \bigwedge_\ell q^{\prime}_{k\ell}, \quad \hbox{with each
  $q^\prime_{k\ell}$ in the image of $\fYg : \PR \rightarrow \PQ$.}$$

\vspace*{0.05in}

(Here, $\bigvee$ and $\bigwedge$ are the lattice operations of $\PQplus$.)

\end{theorem}
\setcounter{theorem}{\value{currentThmCount}}

\vspace*{0.025in}

\begin{proof}
Write $q=(\sigma,\gamma)\in\PQ$.   \quad Then
$\sigma=\psi_Q(\gamma)=\biginter_{y\in\gamma}\XQy$.

\vspace*{0.1in}

So $q \;=\; \bigwedge_{y\in\gamma}q_y$, with each $q_y\in\PQ$ of the
form $\big(\XQy,\, (\clsqy)(\ys)\big)$.

\vspace*{0.1in}

By Lemma~\ref{generatorimages}, for each $y\in\gamma$, we have that
$q_y = \bigvee_{\!i}q_{y,i}$ with each $q_{y,i}$ in the image of $\fXg :
\PR \rightarrow \PQ$ and with $i$ in some finite index set ${\cal
I}(y)$.  $\phantom{\Big|}$ \ Thus:

\vspace*{-0.05in}

$$q = \bigwedge_{y\,\in\,\gamma} \;\, \bigvee_{\phantom{1}i\,\in\,{\cal I}(y)}\,q_{y,i}.$$

\vst

The other form follows by dualizing the previous arguments.
\end{proof}

\clearpage
\section{A Few More Examples}
\label{moreexamples}

\subsection{Local Spheres versus Global Contractibility}
\markright{Local Spheres versus Global Contractibility}
\label{duncehat}

The reader may wonder whether preservation of attribute privacy always
requires a relation to exhibit homology in its Dowker complexes.  The
answer is that links of individuals must have homology, by
Theorems~\ref{privacysingle} and \ref{privacymultiple} on
page~\pageref{privacysingle}, but the overall relation need not.

\begin{figure}[h]
\begin{center} 
\hspace*{0in}\begin{minipage}{2in}{$\begin{array}{c|cccccccc}
 D &  \ta &  \tb &  \tc &  \td &  \te &  \tf &  \tg & \tth \\[2pt]\hline
 1 &      &      &      & \one &      & \one &      & \one \\
 2 &      &      &      & \one & \one & \one &      &      \\
 3 &      &      &      &      &      & \one & \one & \one \\
 4 &      &      & \one &      &      & \one & \one &      \\
 5 & \one &      &      &      & \one & \one &      &      \\
 6 & \one &      & \one &      & \one &      &      &      \\
 7 &      &      & \one & \one &      &      &      & \one \\
 8 &      & \one &      & \one & \one &      &      &      \\
 9 & \one & \one &      &      &      &      & \one &      \\
10 &      & \one & \one &      &      &      & \one &      \\
11 & \one &      &      &      &      &      & \one & \one \\
12 & \one & \one &      &      &      &      &      & \one \\
13 &      & \one & \one &      &      &      &      & \one \\
14 & \one &      & \one & \one &      &      &      &      \\
15 & \one & \one &      & \one &      &      &      &      \\
16 &      & \one & \one &      & \one &      &      &      \\
17 & \one &      & \one &      &      & \one &      &      \\
\end{array}$}
\end{minipage}
\hspace*{0.4in}
\begin{minipage}{2.9in}
\ifig{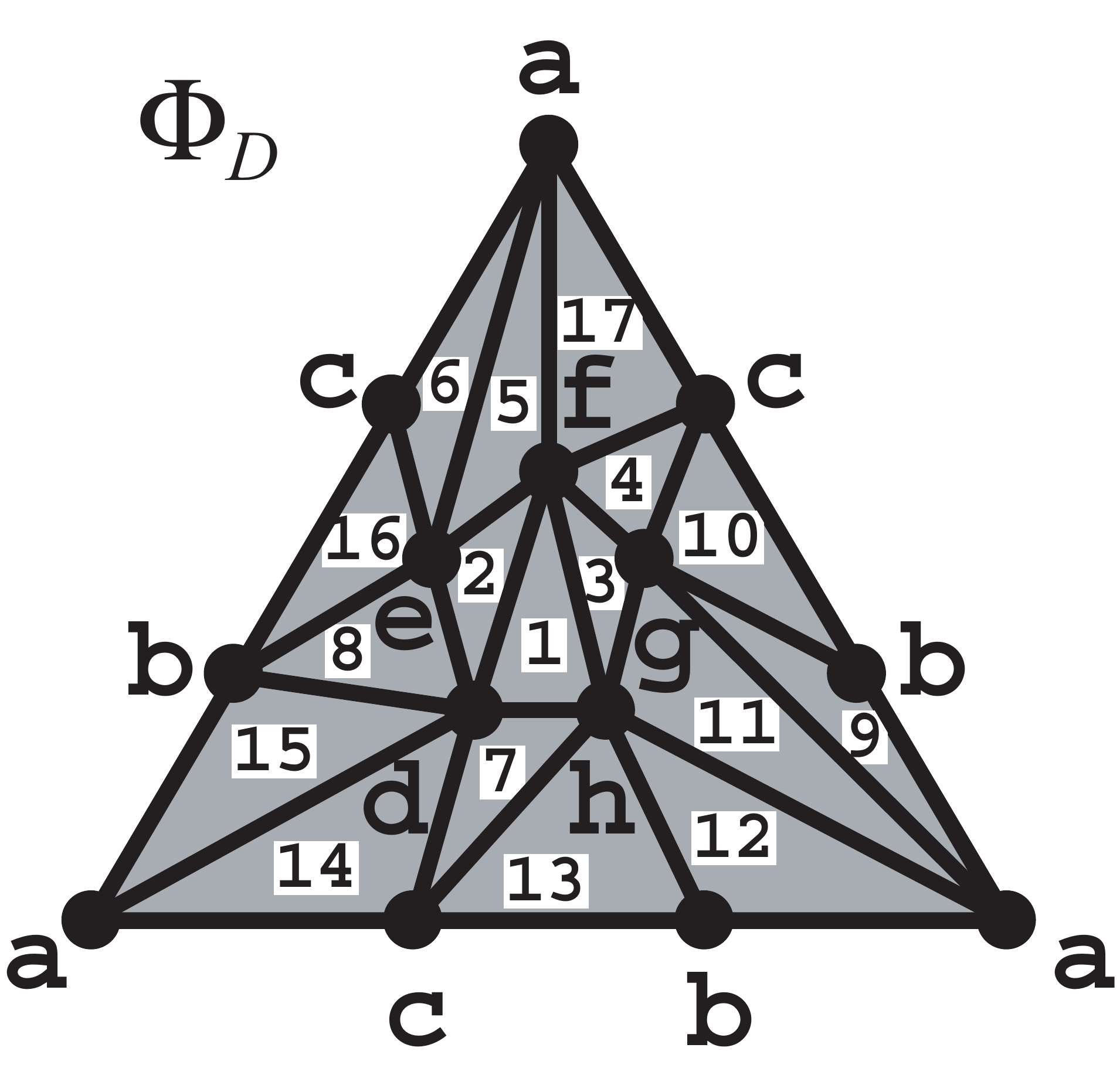}{width=2.75in}
\end{minipage}
\end{center}
\vspace*{-0.1in}
\caption[]{Relation $D$ and its Dowker complex $\dowdy$.  The complex
  is a triangulation of the Dunce Hat, a contractible space (the
  seemingly bounding edges actually touch, as suggested by the vertex
  labels).  The Dunce Hat has no free faces, indicating that $D$
  preserves attribute privacy. \ (Vertices of $\dowdy$ are attributes.
  Triangles are labeled with their generating individuals.)}
\label{dunce}
\end{figure}

Consider for example the relation $D$ of Figure~\ref{dunce}.  There
are 17 individuals, each with three attributes.  The figure also shows
$\dowdy$.  We can see that there are no free faces, so the relation
preserves attribute privacy by Lemma~\ref{nofreefacesPrivacy} on
page~\pageref{nofreefacesPrivacy}.  Moreover, each link $\lk(\dowdx,
x)$ is homotopic to a circle $\Sone$.  Indeed, viewed from attribute
space, that link is exactly the boundary of a triangle for each
individual.  Figure~\ref{duncelink10} shows such a link for individual
\#10.  The link relation has a large number of individuals but only
three attributes.  So Theorem~\ref{privacysingle} holds and there is
homology in the link.  There is however no homology in the attribute
complex of the relation $D$ itself; the simplicial complex $\dowdy$ is
a triangulation of the Dunce Hat, a nontrivially contractible space.

Although $R$ preserves attribute privacy, it does not preserve
association privacy.  For example: Individuals \#1 and \#12 share
exactly one attribute (namely $\tth$), but do so with four additional
individuals (namely \#3, \#7, \#11, and \#13).  If attributes
represent shared dinners, then in some cases one can infer all the
guests at a dinner after having seen as few as two guests.  (Attribute
privacy means that one cannot definitively infer additional dinners
attended by a guest simply from having observed that guest at a
particular dinner or two.)

\begin{figure}[t]
\begin{center} 
\hspace*{0.3in}\begin{minipage}{1.2in}{$\begin{array}{c|ccc}
 Q &  \tb &  \tc &  \tg \\[2pt]\hline
 3 &      &      & \one \\
 4 &      & \one & \one \\
 6 &      & \one &      \\
 7 &      & \one &      \\
 8 & \one &      &      \\
 9 & \one &      & \one \\
11 &      &      & \one \\
12 & \one &      &      \\
13 & \one & \one &      \\
14 &      & \one &      \\
15 & \one &      &      \\
16 & \one & \one &      \\
17 &      & \one &      \\
\end{array}$}
\end{minipage}
\hspace*{0.8in}
\begin{minipage}{1.6in}
\ifig{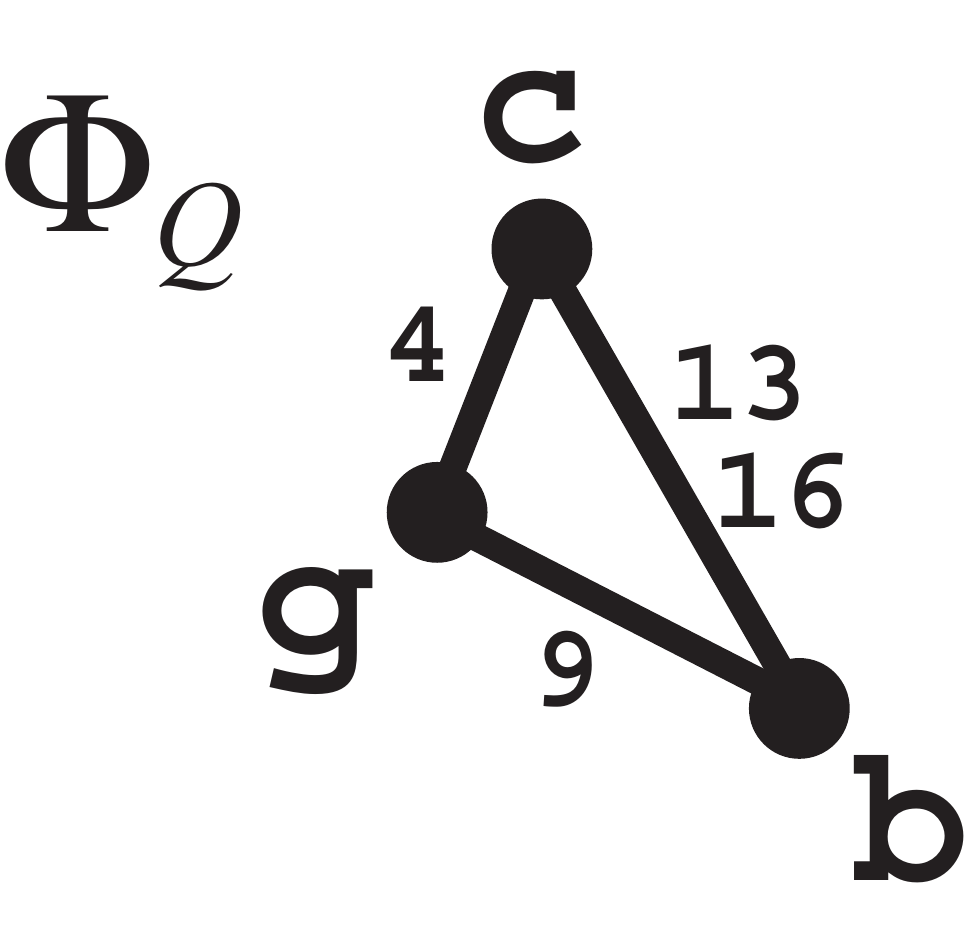}{width=1.7in}
\end{minipage}
\end{center}
\vspace*{-0.1in}
\caption[]{Relation $Q$ represents $\lk(\dowdx, 10)$ from
  Figure~\ref{dunce}.  Also shown is the attribute Dowker complex
  $\dowqy$.  It is the boundary of a triangle, so homotopic to $\Sone
  = \Skt$.  Since individual \#10 has three attributes and $1 = 3 - 2
  $, that means relation $R$ preserves attribute privacy for
  individual \#10.  (Vertices of $\dowqy$ are attributes.  Edges are
  labeled with their generating individuals.  Notice that the edge
  $\{\tb, \tc\}$ is generated by two individuals.  Whereas most edges
  in $\dowdy$ are shared by only two triangles, edge $\{\tb, \tc\}$ is
  shared by three triangles; it is one of those edges that are glued
  to two others in the Dunce Hat representation.  --- Individuals who
  generate just vertices are not shown in the drawing of $\dowqy$ here.)}
\label{duncelink10}
\end{figure}

\subsection{Disinformation}
\markright{Disinformation}
\label{disinformation}

Privacy loss is possible when there is a free face in the relevant
Dowker complex.  One way to preserve privacy is to eliminate such free
faces.  Earlier in the report, we studied morphisms between relations
as a possible way to transform data so as to reduce privacy loss.
Ideally, for attribute privacy, the goal of such a transformation is
to map onto a relation whose attribute complex has no free faces.  We
saw that such transformations need not always exist, for topological
reasons, unless one is willing to introduce discontinuities, that is,
discard knowledge of some relationships in the underlying spaces.

\vspace*{0.1in}

Alternatively, one could imagine embedding a relation within another
that does preserve privacy.  Of course, at the extreme, one simply
embeds the given relation in a huge relation that looks like a perfect
sphere.  Now there is privacy but the same mechanism that provides
privacy reduces utility.  Nonetheless, one has not discarded
relationships, merely surrounded them with disinformation.  We saw an
example of that early on, when we added a single attribute to relation
$H$ in the example of Section~\ref{ToyExample}, in order to remove the
inference that someone had cancer.  If one has a separate mechanism for
discerning fake entries from legitimate entries, then one can see past
the disinformation --- in the earlier example that would entail having a
(presumably safely encrypted) memory of which single entry in the
relation is false.

\vspace*{0.1in}

\begin{figure}[h]
\begin{center}
\quad
\begin{minipage}{1.5in}{$\begin{array}{c|ccccc}
\hbox{\large $M$} & \ta  & \tb  & \tc  & \td  & \te  \\[2pt]\hline
              1   & \one & \one &      &      & \one \\
              2   & \one & \one & \one &      &      \\
              3   &      & \one & \one & \one &      \\
              4   &      &      & \one & \one & \one \\
              5   & \one &      &      & \one & \one \\
\end{array}$}
\end{minipage}
\hspace*{0.5in}
\begin{minipage}{3in}
\ifig{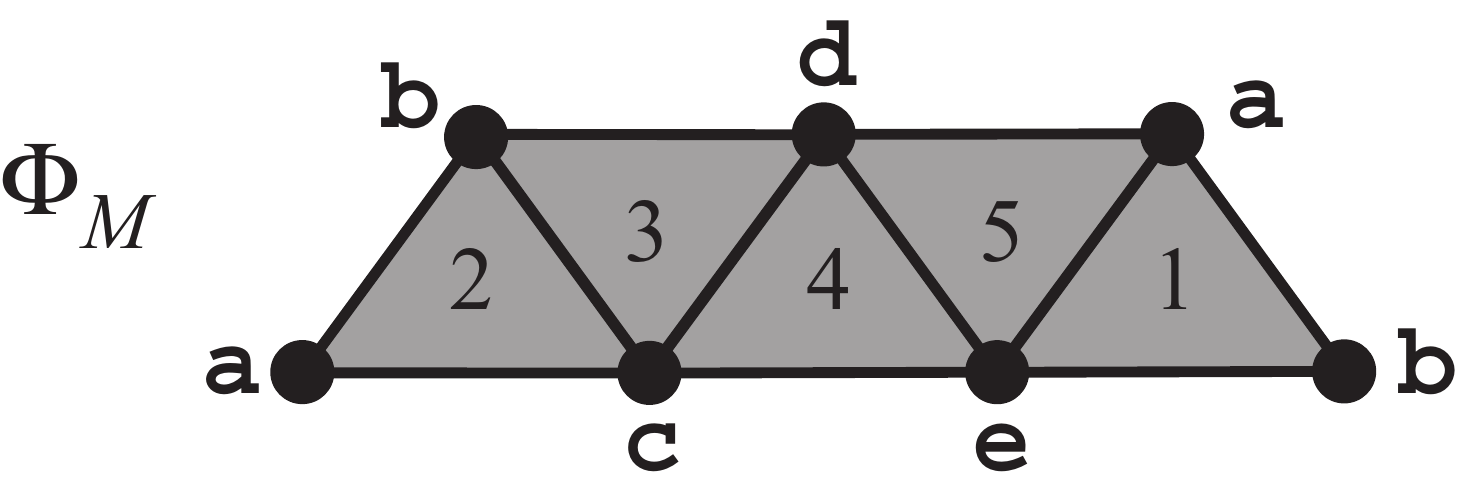}{scale=0.5}
\end{minipage}
\end{center}
\vspace*{-0.1in}
\caption[]{Relation $M$ revisited, along with its attribute complex $\dowMy$.}
\label{moebiusdup}
\end{figure}

Figure~\ref{moebiusdup} revisits our earlier M\"obius strip relation,
showing the relation $M$ and its attribute complex $\dowMy$.  Loss of
attribute privacy occurs when someone observes two attributes that
form a free edge on the boundary of the M\"obius strip, such as the
edge $\{\tb, \td\}$.  Given the relation, the observer can then infer
a third attribute and identify the underlying individual, in this case
infer attribute $\tc$ and identify individual \#3.

In order to preserve attribute privacy, one might consider adding
decoy individuals whose so-called attributes include those edges,
making the edges nonfree, thus removing that inference mechanism.
Relation $MM$ in Figure~\ref{moebiusdouble} does so by doubling the
number of individuals.

\begin{figure}[h]
\begin{center}
\quad
\begin{minipage}{1.5in}{$\begin{array}{c|ccccc}
\hbox{\large $MM$} & \ta  & \tb  & \tc  & \td  & \te  \\[2pt]\hline
              1   & \one & \one &      &      & \one \\
              2   & \one & \one & \one &      &      \\
              3   &      & \one & \one & \one &      \\
              4   &      &      & \one & \one & \one \\
              5   & \one &      &      & \one & \one \\
              6   & \one & \one &      & \one &      \\
              7   &      & \one & \one &      & \one \\
              8   & \one &      & \one & \one &      \\
              9   &      & \one &      & \one & \one \\
              10  & \one &      & \one &      & \one \\
\end{array}$}
\end{minipage}
\hspace*{0.55in}
\begin{minipage}{3in}
\ifig{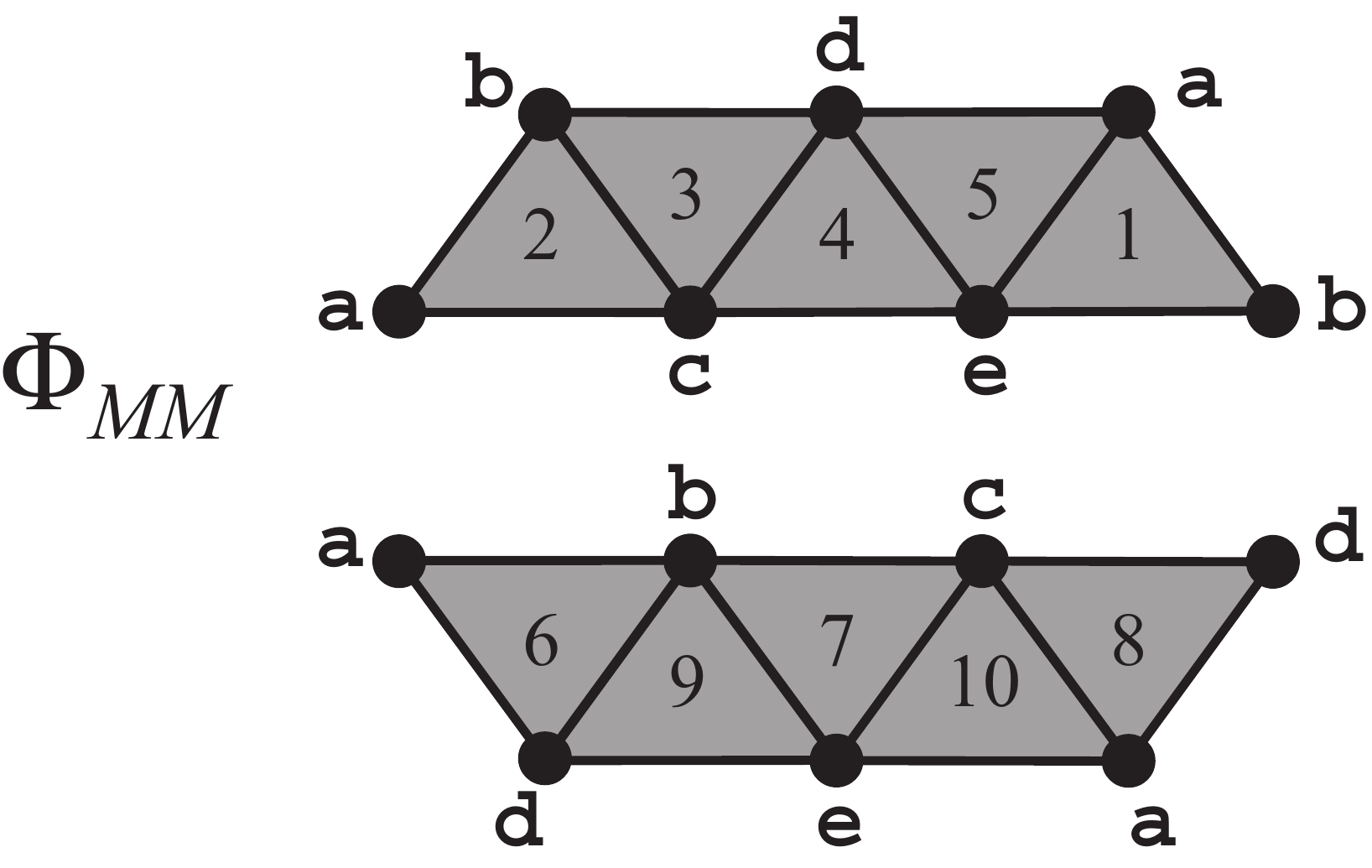}{scale=0.45}
\end{minipage}
\end{center}
\vspace*{-0.1in}
\caption[]{Relation $MM$ adds five decoy individuals. The attribute
  complex $\dowMMy$ entails gluing two M\"obius strips together.}
\label{moebiusdouble}
\end{figure}

The additional five individuals form their own M\"obius strip.  The
figure therefore describes the overall attribute complex $\dowMMy$ as
two M\"obius strips, with edges shared between the two strips, as
suggested by the vertex labels.  The overall attribute complex amounts
to gluing the two M\"obius strips together, boundary to zigzag spine.
The resulting attribute complex is the 2-skeleton of the full complex on
the attribute set $\{\ta, \tb, \tc, \td, \te\}$.  It therefore is
homotopic to a wedge sum of four 2-spheres:
$\Stwo \join \Stwo \join \Stwo \join \Stwo$.

Each of $\dowMMy$'s edges is now shared by three triangles.
There are no free faces.

No attribute inference is possible. \ (Association inference is possible.)

Moreover, the complex is sufficiently isotropic that one cannot say {\em
\hspace*{-0.3pt}a priori\hspace*{1.3pt}} which individuals are real
and which are decoys, even if one knows that there might be decoys.
Of course, the curator of the relation likely would want some secure
mechanism to separate truth from fiction, that is, to peel apart the
gluing.  Regardless, real individuals may be identified via $MM$ upon
seeing all their attributes (and only then).

\subsection{Insufficient Representation}
\label{insufficientrep}

In this subsection we show that if there are fewer than $2^k$
individuals in a nonvoid relation with $2k$ attributes that model $k$
bits for each individual, then the relation cannot preserve attribute
privacy for everyone.  The reason is that fewer than $2^k$ individuals
amounts to removing some generating simplices from the potential
attribute complex
$\Szero * \mskip2mu \Szero * \cdots * \mskip2mu \Szero$, thereby
creating free faces in $\dowy$.  By similar intuition, it may be
possible to preserve attribute privacy even if there are fewer than,
say, $3^k$ individuals in a relation with $3k$ attributes representing
$k$ trivalent pieces of information.  After all, bits are a special
case of tri-values, so one can preserve attribute privacy with certain
$2^k$ individuals.  Thinking simplicially, the potential attribute
complex for tri-values is
$(\Szero \join \mskip2mu \Szero) * (\Szero \join \mskip2mu \Szero) * \cdots * (\Szero \join \mskip2mu \Szero)$.
Removing some generating simplices from that space does not
necessarily create free faces, as one can see by simple example.

\begin{definition}[Binary Attribute Pair]
By a \mydefem{binary attribute pair} we mean two mutually exclusive
attributes, written $y$ and $\ybar$.  No individual can have both
attributes.  Moreover, in what follows we will assume that every
individual has exactly one attribute from any such pair.
\end{definition}

\begin{lemma}[Privacy Requires Many Individuals]
Suppose $Y = \{y_1, \ybar_1, y_2, \ybar_2, \ldots, y_k, \ybar_k\}$,
with $\,\{y_i, \ybar_i\}$ a binary attribute pair, for $i=1, \ldots, k$,
and $k \geq 1$.

Let $R$ be a relation on $\XxY$, with $X\!\neq\emptyset$, such that
every individual $x\in X$ has as attributes exactly one of $\{y_i,
\ybar_i\}$, for each $i = 1, \ldots, k$. \quad
Let $n$ be the number of distinct rows of $R$.

\vst

Then $R$ preserves attribute privacy if and only if $\hspt{}n=2^k$.

\end{lemma}

\begin{proof}
Observe that each row of $R$ has exactly $k$ nonblank entries, so each
maximal simplex of $\dowy$ consists of exactly $k$ vertices.
Moreover, no row of $R$ is contained in another row unless the two
rows are identical.  We may therefore assume, without loss of
generality, that all rows of $R$ are distinct and incomparable.
Consequently, every $x \in X$ is uniquely identifiable.  We can think
of each individual $x\in X$ as defining a unique and identifying
$k$-bit number, with one bit per binary attribute pair, as determined
by that individual's row, $Y_x$.  All possible $k$-bit numbers are
represented by $X$ if and only if $n=2^k$.

\vspace*{0.1in}

I.  \underline{Suppose that $n=2^k$.}

\vspace*{0.05in}

Showing that $\dowy$ contains no free faces would establish that $R$
preserves attribute privacy, by Lemma~\ref{nofreefacesPrivacy} on
page~\pageref{nofreefacesPrivacy}.  To show that $\dowy$ contains no
free faces, it is enough to show that, for every maximal
$\gamma\in\dowy$ and every $y\in\gamma$, the simplex
$\gamma\setminus\ys$ is contained in some maximal simplex of $\dowy$
other than just $\gamma$.

Write $\chi = \gamma\setminus\ys$.  Since $y$ is part of a binary
attribute pair, we can construct a new set $\gamma^\prime$ from
$\gamma$ by replacing $y$ with its ``opposite''.  Specifically:
$\gamma^\prime = \chi \union \{y_i\}$, if $y=\ybar_i$; and
$\gamma^\prime = \chi \union \{\ybar_i\}$, if $y=y_i$.
Since $n=2^k$, there is an $x\in X$ for which $Y_x=\gamma^\prime$.  So
$\gamma^\prime\in\dowy$, telling us $\chi$ is not free.

II.  \underline{Suppose that that $R$ preserves attribute privacy.}

\vspace*{0.05in}

By Lemma~\ref{privacyNofreefaces} on page~\pageref{privacyNofreefaces},
$\dowy$ contains no free faces.

Let $\gamma$ be a maximal simplex of $\dowy$ and $y\in\gamma$.  Define
$\chi = \gamma\setminus\ys$.  Construct $\gamma^\prime$ as in part I
above.
Consider the collection
$\Gamma = \setdef{\eta\in\dowy}{\chi\subsetneq\eta}$.
The only possible set that might be in $\Gamma$ besides $\gamma$ is
$\gamma^\prime$.  Since $\dowy$ contains no free faces, $\Gamma =
\{\gamma, \gamma^\prime\}$.

Now vary $y$ across $\gamma$ and then repeat the process for all
$\gamma^\prime$ thus constructed.  The transitive closure of this
process generates $2^k$ distinct maximal simplices in $\dowy$, each of
which corresponds to a unique $x\in X$.  So $n=2^k$.
\end{proof}

\subsection{A Structural Inference Example: Passengers on Ferries}
\markright{A Structural Inference Example: Passengers on Ferries}
\label{ferries}

\begin{figure}[h]
\vspace*{-0.15in}
\begin{center}
\ifig{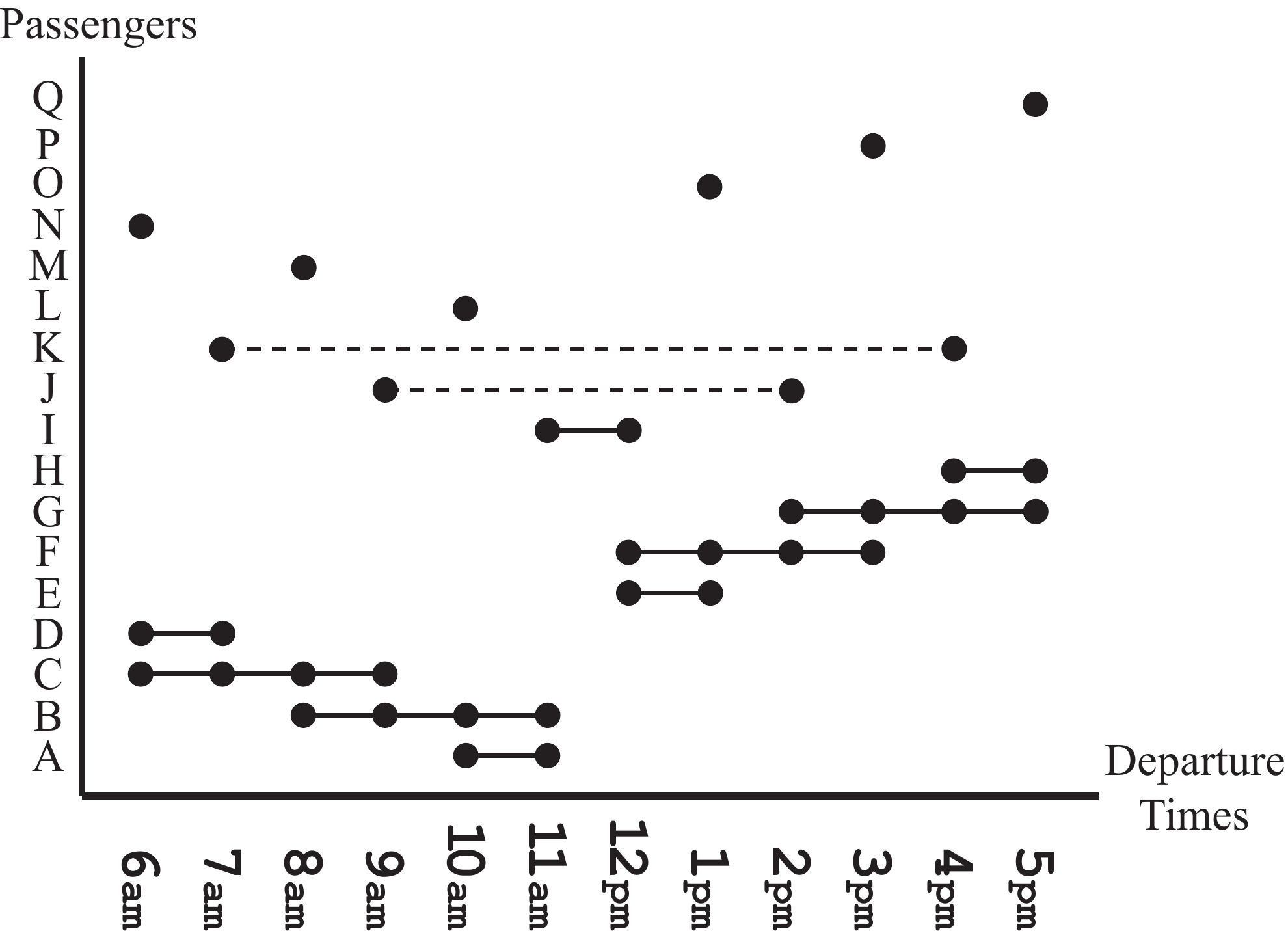}{scale=0.5}
\end{center}
\vspace*{-0.25in}
\caption[]{This time series represents 17 different passengers
  on 12 different ferry crossings.  Each dot represents a passenger on
  a crossing.  As a visual aid, solid lines connect multiple crossings
  by the same passenger at consecutive departure times, while dashed
  lines connect multiple crossings by the same passenger at
  non-consecutive departure times.}
\label{crossings}
\end{figure}

Imagine a commuter ferry that crosses back and forth between downtown
and an island.  Passengers pay electronically as they enter the ferry,
so there is a record of who is on which crossing.  Figure
\ref{crossings} shows a hypothetical time series for 12 crossings
during a day in which 17 passengers took the ferry, some of whom
crossed several times.  Figure~\ref{crossingsPsi} shows the
corresponding $\dowx$ complex: vertices are individuals; each triangle
represents a particular crossing.  (Each ferry crossing had three
passengers in this simplified example.)

\begin{figure}[h]
\begin{center}
\hspace*{-0.5in}\ifig{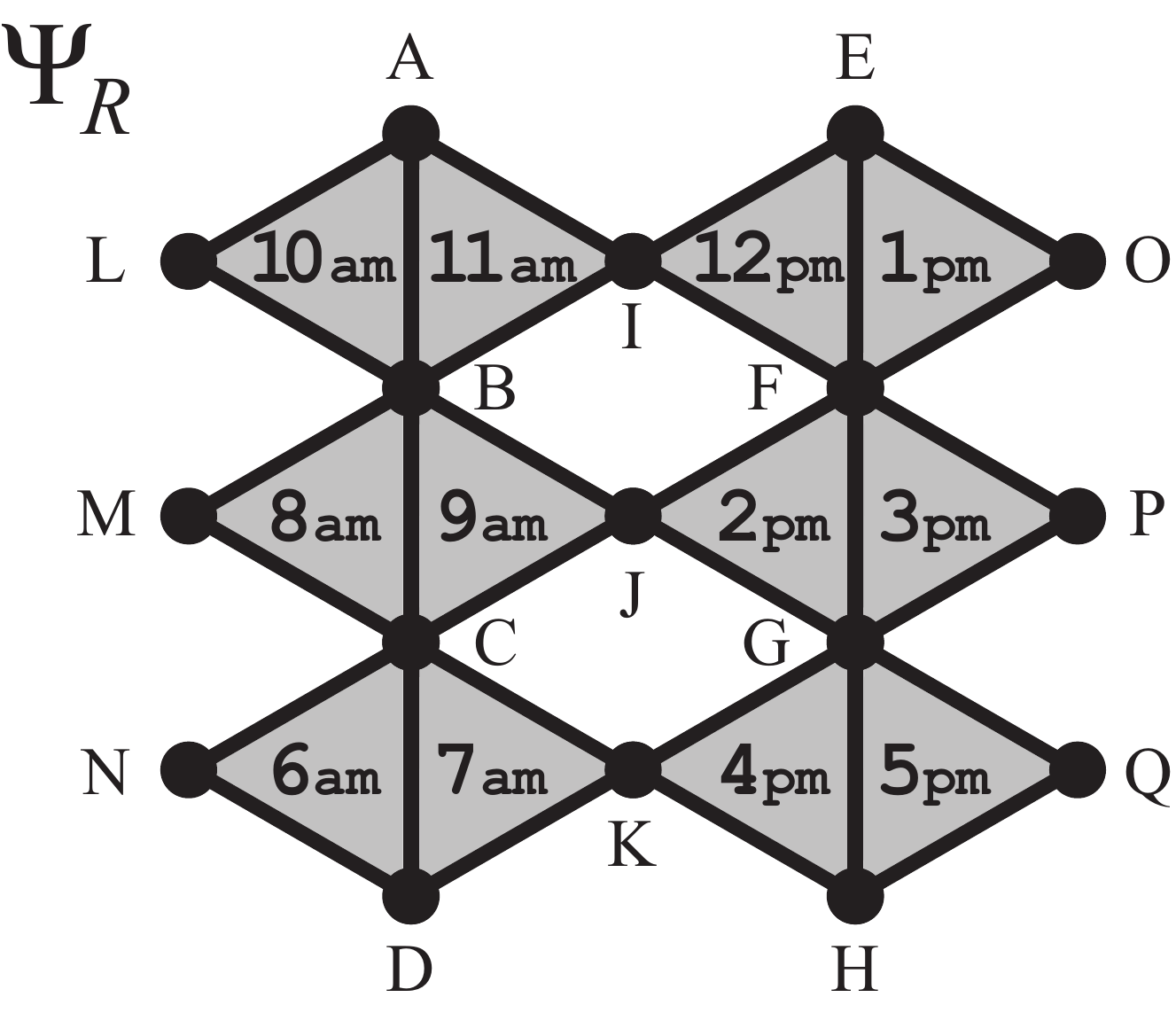}{scale=0.5}
\end{center}
\vspace*{-0.3in}
\caption[]{Simplicial complex $\dowx$ determined by viewing the
  time series of Figure~\ref{crossings} as a relation $R$.
  Vertices represent passengers, labeled with letters.  Triangles
  represent ferry crossings, labeled with departure times.}
\label{crossingsPsi}
\end{figure}

\begin{figure}[h]
\vspace*{0.1in}
\begin{center}
\hspace*{0.8in}\ifig{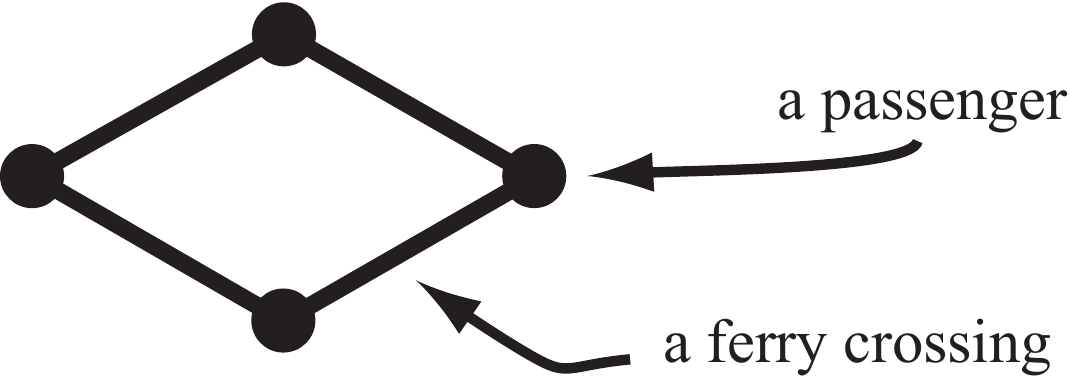}{scale=0.5}
\end{center}
\vspace*{-0.25in}
\caption[]{A waitress's observations of passengers drinking
  coffee together at various times, represented by a simplicial
  complex.  Vertices represent unknown but distinct passengers.  Edges
  represent unknown but distinct crossing times.}
\label{observations}
\end{figure}

The waitress in the ferry's coffee shop observes four individuals
ordering coffee and conversing during the day, appearing in pairs on
four different crossings.  She remembers seeing four distinct pairs,
but does not remember the crossing times.  Who are the individuals?

It is convenient to also represent the waitress's observations as a
simplicial complex.  Figure~\ref{observations} does so.  Vertices are
now the four unknown individuals; edges are their (unknown) common
crossing times.
One can interpret who the individuals are by embedding the complex of
Figure~\ref{observations} into the complex of Figure~\ref{crossingsPsi},
using injective maps in both the passenger and time domains.  There are
exactly two such embeddings (modulo index permutations), given by the two
ways one can wrap a rectangle around the two holes in the complex of
Figure~\ref{crossingsPsi}.  (Those are the only two ``diamonds''
touching four different crossing times.)
Thus the individuals are either
$\{\hbox{C}, \hbox{G}, \hbox{J}, \hbox{K}\}$
or
$\{\hbox{B}, \hbox{F}, \hbox{I}, \hbox{J}\}$,
as indicated by Figure~\ref{imbeddings}.
Either way, one knows for sure that individual ``\hbox{J}'' twice had
a conversation over coffee that day.

\begin{figure}[h]
\begin{center}
\ifig{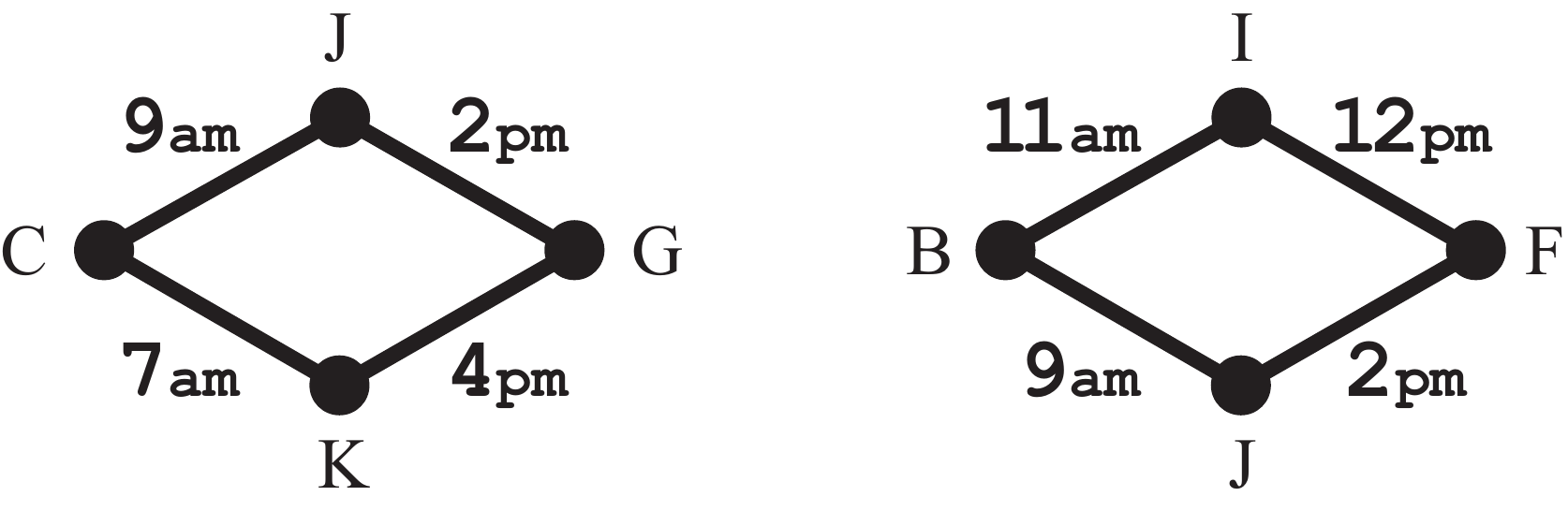}{scale=0.5}
\end{center}
\vspace*{-0.25in}
\caption[]{The two possible embeddings of the complex of
  Fig.~\ref{observations} into the complex of
  Fig.~\ref{crossingsPsi}.}
\label{imbeddings}
\end{figure}

Moreover, each of these embeddings places a time ordering on the
embedded edges, from which one can make inferences as to who might
have transmitted information to whom.  For instance, for the embedding
involving individuals
$\{\hbox{B}, \hbox{F}, \hbox{I}, \hbox{J}\}$, one sees that individual
``\hbox{J}'' could have been both the initial source and final
recipient of information.



\addcontentsline{toc}{section}{References}

\begin{thebibliography}{10}

\bibitem{tpr:bjorner}
A.~Bj{\"{o}}rner.
\newblock Topological methods.
\newblock In R.~L. Graham, M.~Gr{\"{o}}tschel, and L.~Lov{\'a}sz, editors, {\em
  Handbook of Combinatorics}, volume~II, pages 1819--1872. Elsevier, Amsterdam,
  1995.

\bibitem{tpr:DinurNissim}
I.~Dinur and K.~Nissim.
\newblock Revealing information while preserving privacy.
\newblock In {\em Proceedings of the 22nd ACM Symposium on Principles of
  Database Systems}, pages 202--210, 2003.

\bibitem{tpr:dowker}
C.~H. Dowker.
\newblock Homology groups of relations.
\newblock {\em Annals of Mathematics}, 56(1):84--95, 1952.

\bibitem{tpr:Dwork08}
C.~Dwork.
\newblock Differential privacy: {A} survey of results.
\newblock In {\em Proceedings of the 5th International Conference on Theory and
  Applications of Models of Computation}, pages 1--19, 2008.

\bibitem{tpr:dworkcacm}
C.~Dwork.
\newblock A firm foundation for private data analysis.
\newblock {\em Communications of the ACM}, 54(1):86--95, 2011.

\bibitem{tpr:strategies}
M.~A. Erdmann.
\newblock On the topology of discrete strategies.
\newblock {\em International Journal of Robotics Research}, 29(7):855--896,
  2010.

\bibitem{tpr:plans}
M.~A. Erdmann.
\newblock On the topology of discrete planning with uncertainty.
\newblock In A.~Zomorodian, editor, {\em Advances in Applied and Computational
  Topology}, pages 147--193. AMS, 2012.

\bibitem{tpr:rappor}
{\'{U}}.~Erlingsson, V.~Pihur, and A.~Korolova.
\newblock {{RAPPOR:} Randomized Aggregatable Privacy-Preserving Ordinal
  Response}.
\newblock In {\em Proceedings of the 21st ACM Conference on Computer and
  Communications Security}, pages 1054--1067, 2014.

\bibitem{tpr:forman-guide}
R.~Forman.
\newblock A user's guide to discrete {Morse} theory.
\newblock {\em S{\'e}minaire Lotharingien de Combinatoire}, 48:B48c, 2002.

\bibitem{tpr:ganterwille}
B.~Ganter and R.~Wille.
\newblock {\em {Formal Concept Analysis: Mathematical Foundations}}.
\newblock Springer-Verlag, Berlin, 1999.

\bibitem{tpr:ghristlipsky}
R.~Ghrist, D.~Lipsky, J.~Derenick, and A.~Speranzon.
\newblock Topological landmark-based navigation and mapping.
\newblock Preprint, 2012.

\bibitem{tpr:gleiserdanon}
P.~M. Gleiser and L.~Danon.
\newblock Community structure in jazz.
\newblock {\em Advances in Complex Systems}, 6(4):565--573, 2003.

\bibitem{tpr:nullgraph}
F.~Harary and R.~Read.
\newblock Is the null-graph a pointless concept?
\newblock In {\em Proceedings of the Capital Conference on Graph Theory and
  Combinatorics}, pages 37--44, 1973.

\bibitem{tpr:hatcher}
A.~Hatcher.
\newblock {\em {Algebraic Topology}}.
\newblock Cambridge University Press, Cambridge, 2002.

\bibitem{tpr:missedcollab}
M.~Hoang, R.~Ramanathan, and A.~Singh.
\newblock Structure and evolution of missed collaborations in large networks.
\newblock In {\em The Sixth IEEE International Workshop on Network Science for
  Communication Networks, in conjunction with IEEE Infocom}, pages 849--854,
  2014.

\bibitem{tpr:munkres}
J.~R. Munkres.
\newblock {\em {Elements of Algebraic Topology}}.
\newblock Addison-Wesley, Menlo Park, CA, 1984.

\bibitem{tpr:netflix}
A.~Narayanan and V.~Shmatikov.
\newblock Robust de-anonymization of large sparse datasets.
\newblock In {\em Proceedings of the IEEE Symposium on Security and Privacy},
  pages 111--125, 2008.

\bibitem{tpr:quillenK}
D.~Quillen.
\newblock Higher algebraic {K}-theory: {I}.
\newblock In {\em Lecture Notes in Mathematics}, volume 341, pages 85--147.
  Springer-Verlag, Berlin, 1973.

\bibitem{tpr:quillenH}
D.~Quillen.
\newblock Homotopy properties of the poset of nontrivial $p$-subgroups of a
  group.
\newblock {\em Advances in Mathematics}, 28(2):101--128, 1978.

\bibitem{tpr:rotman}
J.~J. Rotman.
\newblock {\em {An Introduction to Algebraic Topology}}.
\newblock Springer-Verlag, Berlin, 1988.

\bibitem{tpr:sweeneykanon}
L.~Sweeney.
\newblock k-anonymity: A model for protecting privacy.
\newblock {\em International Journal of Uncertainty, Fuzziness and
  Knowledge-based Systems}, 10(5):557--570, 2002.

\bibitem{tpr:wachs}
M.~L. Wachs.
\newblock {\em {Poset Topology: Tools and Applications}}.
\newblock IAS/Park City Mathematics Institute, Summer 2004.
\newblock \ Also available here: {\tt http://arxiv.org/abs/math/0602226}.

\bibitem{tpr:rrWarner}
S.~L. Warner.
\newblock {Randomized Response: A} survey technique for eliminating evasive
  answer bias.
\newblock {\em Journal of the American Statistical Association},
  60(309):63--69, 1965.

\bibitem{tpr:warrenbrandeis}
S.~D. Warren and L.~D. Brandeis.
\newblock The right to privacy.
\newblock {\em Harvard Law Review}, IV(5):193--220, 1890.

\bibitem{tpr:wille}
R.~Wille.
\newblock Concept lattices and conceptual knowledge systems.
\newblock {\em Computers \& Mathematics with Applications}, 23(6--9):493--515,
  1992.

\end{thebibliography}
\end{document}
